\documentclass[10pt,reqno]{article}
\usepackage[margin=1in]{geometry}
\usepackage{amsmath, amsthm, amssymb}
\usepackage[colorlinks=true, pdfstartview=FitV, linkcolor=blue,citecolor=blue, urlcolor=blue]{hyperref}
\usepackage[abbrev,lite,nobysame]{amsrefs}
\usepackage{times}
\usepackage[usenames,dvipsnames]{color}
\usepackage{mathtools,enumitem}
\usepackage[compact]{titlesec}

\usepackage{graphicx}

\allowdisplaybreaks

\newcommand{\Real}{\mathbb R}
\newcommand{\Torus}{\mathbb T}
\newcommand{\Complex}{\mathbb C}

\newcommand{\norm}[1]{\left\|#1\right\|}
\newcommand{\abs}[1]{\left\vert#1\right\vert}
\newcommand{\set}[1]{\left\{#1\right\}}

\newcommand{\hPhi}{Y}
\newcommand{\tPhi}{\widetilde{\Phi}}

\newcommand{\tP}{\widetilde{P}}
\newcommand{\tw}{\widetilde{w}}

\def\eps{\varepsilon}
\def\e{{\rm e}}
\def\dd{{\rm d}}
\def\Ray{{\textsc{Ray}}}
\def\Re{{\rm Re}\,}
\def\Im{{\rm Im}\,}

\def\tuu {\boldsymbol{\tilde u}}

\def\RR {\mathbb{R}}
\def\ZZ {\mathbb{Z}}

\def\KK {\mathbb{K}}

\def\cE {{\mathcal E}}
\def\cB {{\mathcal B}}
\def\cO {{\mathcal O}}
\def\cL {{\mathcal L}}
\def\cG {{\mathcal G}}

\def\de {{\partial}}

\newcommand{\brak}[1]{\langle #1 \rangle} 
\newcommand{\indicator}[1]{{\bf 1}_{#1}}

\newtheorem{theorem}{Theorem}[section]
\theoremstyle{lemma}

\newtheorem{proposition}[theorem]{Proposition}
\newtheorem{corollary}[theorem]{Corollary}
\theoremstyle{definition}
\newtheorem{definition}[theorem]{Definition}

\newtheorem{remark}[theorem]{Remark}

\theoremstyle{lemma}
\newtheorem{lemma}[theorem]{Lemma}

\numberwithin{equation}{section}

\newcommand{\addperiod}[1]{#1.}
\titleformat{\paragraph}[runin]{\normalfont\bfseries}{\theparagraph}{1em}{\addperiod}

\begin{document}

\title{Vortex axisymmetrization, inviscid damping, and vorticity depletion in the linearized 2D Euler equations} 

\author{Jacob Bedrossian\footnote{\textit{jacob@cscamm.umd.edu}, University of Maryland, College Park. The author was partially supported by NSF CAREER grant DMS-1552826, NSF DMS-1413177, and a Alfred P. Sloan research fellowship. Additionally, the research was supported in part by NSF RNMS \#1107444 (Ki-Net).} \quad  Michele Coti Zelati\footnote{\textit{m.coti-zelati@imperial.ac.uk}, Imperial College London. 
The author was partially supported by NSF grant DMS-1713886.} \quad Vlad Vicol\footnote{\textit{vvicol@math.princeton.edu}, Princeton University. The author was partially supported by the NSF CAREER grant DMS-1652134, and an Alfred P.~Sloan research fellowship.}}

\maketitle

\begin{abstract}
Coherent vortices are often observed to persist for long times in turbulent 2D flows even at very high Reynolds numbers and are observed in experiments and computer simulations to potentially be asymptotically stable in a weak sense for the 2D Euler equations. 
We consider the incompressible 2D Euler equations linearized around a radially symmetric, strictly monotone decreasing vorticity distribution. 
For sufficiently regular data, we prove the inviscid damping of the $\theta$-dependent radial and angular velocity fields with the optimal rates $\norm{u^r(t)} \lesssim \brak{t}^{-1}$ and $\norm{u^\theta(t)} \lesssim \brak{t}^{-2}$ in the appropriate radially weighted $L^2$ spaces. 
We moreover prove that the vorticity  weakly converges back to radial symmetry as $t \rightarrow \infty$, a phenomenon known as \emph{vortex axisymmetrization} in the physics literature, and characterize the dynamics in higher Sobolev spaces. 
Furthermore, we prove that the $\theta$-dependent angular Fourier modes in the vorticity are ejected from the origin as $t \rightarrow \infty$, resulting in faster inviscid damping rates than those possible with passive scalar evolution. This non-local effect is called \emph{vorticity depletion}. Our work appears to be the first to find vorticity depletion relevant for the dynamics of vortices. 
\end{abstract}

\setcounter{tocdepth}{1}
\tableofcontents

\newpage

\section{Introduction and statements of results}
In polar coordinates $(r,\theta)\in [0,\infty)\times \Torus$, the two-dimensional Euler equations in vorticity formulation 
read
\begin{align}\label{eq:radialEuler}
\de_t\tilde\omega +\tilde u^r \de_r \tilde\omega+\frac{\tilde u^\theta}{r}\de_\theta\tilde\omega=0,
\end{align}
where the velocity vector $\boldsymbol{\tilde u}=(\tilde u^r,\tilde u^\theta)$  is recovered from the vorticity $\tilde\omega$ by means
of the streamfunction $\tilde\psi$, via the relations
\begin{align*}
\tilde\omega=-\frac1r\de_\theta\tilde u^r+ \frac1r \de_r (r\tilde u^\theta),\qquad \tuu=\left(\frac1r\de_\theta\tilde\psi,-\de_r\tilde\psi\right),
\qquad-\left(\de_{rr}+\frac1r \de_r +\frac{1}{r^2}\de_{\theta\theta}\right)\tilde\psi=\tilde\omega.
\end{align*}
Any radially symmetric configuration $\Omega=\Omega(r)$ is a stationary solution of \eqref{eq:radialEuler} and the above relations simplify to
\begin{align}\label{eq:pappa1}
\Omega=\frac1r \de_r (rU),\qquad U=-\de_r	\Psi,
\qquad-\left(\de_{rr}+\frac1r \de_r\right)\Psi=\Omega. 
\end{align}
In what follows, we denote 
\begin{align}\label{eq:ubeta}
u(r)=\frac{U(r)}{r}=-\frac{\de_r\Psi(r)}{r}, \qquad \beta(r)=-\frac{\de_r\Omega(r)}{r}.
\end{align}
Writing $\tilde\omega(t,r,\theta)=\omega(t,r,\theta)+\Omega(r)$ and dropping terms quadratic in $\omega$ gives the linearized 2D Euler equations which are the main object of study in this paper:  
\begin{subequations} \label{eq:EU}
\begin{align}
& \de_t\omega +u(r)\de_\theta \omega-\beta(r)\de_\theta\psi =0,  \\ 
& -\left(\de_{rr}+\frac1r \de_r +\frac{1}{r^2}\de_{\theta\theta}\right)\psi=\omega, \\ 
& \omega(0,r,\theta) = \omega^{in}(r,\theta). 
\end{align}
\end{subequations}
By expanding the solution $\omega$ to \eqref{eq:EU} as a Fourier series in the angular $\theta$ variable, namely
\begin{align}
\omega(t,r,\theta)=\sum_{k\in \ZZ} \omega_k(t,r)\e^{ik\theta}, \qquad  \psi(t,r,\theta)=\sum_{k\in \ZZ} \psi_k(t,r)\e^{ik\theta},
\end{align} 
we perform a $k$-by-$k$ analysis of the linearized equations, since for any integer $k$ 
we have that
\begin{subequations}
\begin{align} \label{eq:EUfourier1}
& \de_t \omega_k+ ik u(r) \omega_k-ik\beta(r) \psi_k =0, \\ 
& -\Delta_k \psi_k := -\left(\de_{rr} + \frac{1}{r}\de_r - \frac{k^2}{r^2}\right) \psi_k = \omega_k. 
\end{align} 
\end{subequations}
Note that $\omega_0(t,r) = \omega_0(0,r)$ (i.e., the $\theta$-average of the solution), and therefore we restrict to $k\neq 0$ without loss
of generality. Similarly, by reality we only consider $k \geq 1$ without loss of generality. 

\subsection{Background}

The stability of vortices is one of the most fundamental problems in the theory of hydrodynamic stability and has been considered by many authors, starting with Kelvin \cite{Kelvin1880} and Orr \cite{Orr07} and continuing to present day in both mathematics and physics (see e.g. \cite{Briggs70,MelanderEtAl1987,Koumoutsakos97,SchecterEtAl00,BalmforthEtAl01,HallBassomGilbert03,GallayWayne05,LiWeiZhang2017,Gallay2017} and the references therein). 
Nonlinear stability results in weighted $L^2$ spaces (of the vorticity) are available using energy-Casimir methods \cite{Dritschel88}, however, these results do not provide a clear description of the long-time dynamics. Experimental observations, computer simulations, and formal asymptotics (see e.g. \cite{BassomGilbert98,SchecterEtAl00,BalmforthEtAl01,HallBassomGilbert03} and the references therein) suggest that a vortex subjected to a sufficiently small disturbance might return to radial symmetry as $t \rightarrow \infty$ in a weak sense. Specifically, it is observed that the vorticity in the angle-dependent modes is stirred up around the vortex into a spiral pattern (sometimes called `filamentation') and eventually the angle-dependent modes weakly converge to zero as $t \rightarrow \infty$. This weak convergence is called \emph{vortex axisymmetrization} and is thought to be relevant to understanding coherent vortices in 2D turbulence \cite{BraccoEtAl2000}, atmospheric dynamics \cite{SmithMontgomery95,MontgomeryKallenback97}, and various settings in plasma physics \cite{SchecterEtAl00,YuDriscoll02,CerfonEtAl13}. Our paper appears to be the first mathematically rigorous confirmation of this behavior for vortices in the linearized 2D Euler equations and the first paper to obtain an accurate prediction of the decay rates for general initial data. 

When studying the stability of the planar Couette flow, Orr \cite{Orr07} identified another effect associated with vorticity mixing: the strong convergence (in $L^2$) of the velocity field to equilibrium as $t \rightarrow \infty$. This effect is now often called \emph{inviscid damping}.  Further studies of the 2D Euler equations linearized around planar shear flows were made by Case \cite{Case1960}, Dikii \cite{Dikii1960}, and Stepin \cite{Stepin95}. 
More recently, the linearized problem was revisited in \cite{LinZeng11} and optimal rates were deduced by Zillinger in \cite{Zillinger2016,Zillinger2017} for shear flows close to Couette flow and later by Wei, Zhang, and Zhao in \cite{WeiZhangZhao15} for more general strictly monotone shear flows in a channel. See also \cite{CZZ17,Zillinger2016circ} for inviscid damping of Taylor-Couette in an annulus and \cite{YangLin16} for inviscid damping in stratified shear flows.    
The recent works of Wei, Zhang, and Zhao \cite{WeiZhangZhao2017,WZZK2017} deduce optimal inviscid damping rates for the some shear flows with non-degenerate critical points. 
This latter works also confirmed the predictions of Bouchet and Morita \cite{BouchetMorita10} that the linearized 2D Euler equations can have a faster inviscid damping rate than if the vorticity were evolving under passive scalar dynamics. We prove a similar effect here as well; see the discussion following Theorem \ref{thm:Main} for more details. 
Finally, see \cite{LinXu2017} for an approach to the problem which is well-suited to treating general problems but obtains less precise decay estimates. 

Inviscid damping is closely related to Landau damping in the Vlasov equations, which arises in the kinetic theory of plasmas and galactic dynamics. 
Landau damping involves the rapid decay of the self-generated electric field in a plasma in the absence of any dissipative mechanisms and was first predicted by Landau in the Vlasov equations linearized around a homogeneous Maxwellian \cite{Landau46}. The predictions matched with experiments \cite{MalmbergWharton64} and many works on the linearized Vlasov equations followed (see \cite{Penrose,VKampen55,Degond86,Glassey94} and the references therein). 
In Vlasov, the decay is caused by the mixing of particles traveling at different velocities whereas in 2D Euler it is caused by the mixing of vorticity.
 Due to the special structure of the Vlasov equations, inviscid damping for the linearized 2D Euler equations (with the exception of the Couette flow \cite{Kelvin87,Orr07}) is significantly harder than the linearized theory of Landau damping near \emph{homogeneous} equilibria (and in general the rates are much faster -- on $\Torus^n \times \Real^n$ the decay can be exponential).

All of the above mentioned works on inviscid and Landau damping only apply to the \emph{linearized} Euler or Vlasov equations (respectively).   
The work of \cite{CagliotiMaffei98} first demonstrated the existence of (analytic) Landau damping solutions to the nonlinear Vlasov equations (see also \cite{HwangVelazquez09}). Later, Mouhot and Villani \cite{MouhotVillani11} demonstrated that on $\Torus^n \times \Real^n$, all perturbations small enough in a sufficiently regular Gevrey class give rise to nonlinear dynamics that matches the linearized dynamics (and in particular, rapid Landau damping). 
This work was followed subsequently by a variety of others on Landau damping in nonlinear Vlasov and related models (see e.g. \cite{BMM13,Young14,MR3437866,BMM16,FGVG} and the references therein). 
 In \cite{BM13}, it was shown that Orr's inviscid damping predictions for Couette flow hold also for the nonlinear 2D Euler equations in $\Torus \times \Real$, provided the perturbation is small in  a sufficiently regular Gevrey class.
At sufficiently low regularities, it was proved in \cite{LZ11b,LinZeng11} that linearized and nonlinear dynamics do not necessarily agree (for both 2D Euler and Vlasov). 
High regularity does not play a special role in the linear theory, however, it was shown in \cite{B16} that on $\Torus \times \Real$ one cannot (in general) extend the linearized theory of Landau damping to the nonlinear Vlasov equations in \emph{any} Sobolev space (however, see \cite{BMM16}). 
This is due to the \emph{plasma echoes}, a nonlinear oscillation observed in experiments  in \cite{MalmbergWharton68}. 
The work of \cite{B16} showed the existence of solutions to Vlasov with arbitrarily many, arbitrarily small, plasma echoes. 
Similar nonlinear echoes are observed in experiments on vortices in the 2D Euler equations \cite{YuDriscoll02,YuDriscollONeil} (see also the analyses of \cite{VMW98,Vanneste02}). 
Hence, we expect the linear and nonlinear regularity requirements to be drastically different. 
For this reason, it is important to study the linearized Euler equations is in as high regularity as possible (preferably Gevrey class), and determine if such high regularity can be propagated in a suitable sense (see Remark \ref{rmk:profile} below). 
If the answer is `no', then it might be possible to introduce nonlinear instabilities even for Gevrey or analytic data in the nonlinear equations.    

Mixing involves a transfer of vorticity from large to small scales.  
When there is diffusion present, it has been shown that this can enhance the dissipative time-scale.
For example, Kelvin showed in \cite{Kelvin87} that $x$-dependent modes of the linearized Couette flow in $\Torus \times \Real$ decay on a time-scale like $O(\nu^{-1/3})$ as opposed to the natural heat equation time-scale of $O(\nu^{-1})$ (denoting $\nu$ to be the inverse Reynolds number).  This effect has been called the `shear-diffuse mechanism', `relaxation enhancement', or `enhanced dissipation' and has been studied many times in linear and some nonlinear settings in both mathematics \cite{CKRZ08,BeckWayne11,Deng2013,BCZGH15,BCZ15,BMV14,BVW16,BGM15I,BGM15II,BGM15III,LiWeiZhang2017,Gallay2017,IMM17,WZZK2017} and physics \cite{RhinesYoung83,DubrulleNazarenko94,LatiniBernoff01,BajerEtAl01}.
Thus far, it has also proved challenging to obtain enhanced dissipation estimates on the linearized Navier-Stokes. 
Ideally, the goal is to obtain both simultaneously on Navier-Stokes; it seems the best result in this direction for non-trivial shear flows is \cite{WZZK2017}. See \cite{Deng2013,LiWeiZhang2017,Gallay2017} and the references therein for the most recent results on the 2D Navier-Stokes equations linearized around the Oseen vortex.

\subsection{Statement of main results}
Throughout the article, we assume the following conditions on the background vortex: 
\begin{enumerate}[label={(V\arabic*)}] 
\item\label{V1} $0 \leq \Omega(r) \lesssim \brak{r}^{-6}$;
\item\label{V2} $\abs{(r \partial_r)^j \Omega(r)} \lesssim_j \brak{r}^{-6}$ for all $j \geq 0$;  
\item\label{V3} $\de_r\Omega(r) < 0$, $\forall r > 0$,
\end{enumerate}
We additionally take the following orthogonality condition on the initial condition of \eqref{eq:EU}: 
\begin{align}\label{eq:ortho}
\int_0^\infty \omega_{\pm1}^{in}(r) r^2 \dd r = 0,
\end{align}
which removes the neutral modes that arise due to the translation invariance (see Lemma \ref{eq:ortho}). 
Recall that for smooth functions $\omega$, the asymptotic expansion of $\omega_k(r)$ at the origin is necessarily in the form (see e.g. \cite{BassomGilbert98}) 
\begin{align}
\omega_k(r) \sim r^k\left(a_0 + a_1 r^2 + a_2 r^4 + \cdots \right) \,\, \textup{as} \,\, r \rightarrow 0. \label{eq:smthExp}
\end{align}
Let $\chi$ be a smooth cut-off which is $1$ for $r < 1/2$ and $0$ for $r> 3/4$ and denote
\begin{align}\label{def:omegak0F} 
\omega_{k,0}^{in}  = \lim_{r \rightarrow 0}\frac{\omega^{in}_k(r)}{r^k}, \qquad F_k(r)  = \omega_k^{in}(r) \sqrt{r} - \frac{ \beta(r)}{\beta(0)}\chi(r)r^{k+1/2}\omega_{k,0}^{in}.
\end{align}
By \eqref{eq:smthExp}, for smooth $\omega_k^{in}$, we $F_k(r) \sim r^{k+2+1/2}$ as $r \rightarrow 0$ and thus we may use a stronger weight for $F_k$ at the origin than for $\omega_{k}^{in}$. 
To state our main result, for $\delta>0$ we define smooth weights $w_{\psi,\delta},w_{f,\delta},w_{F,\delta}$ and corresponding $L^2$-norms which satisfy the following asymptotics:  
\begin{subequations} \label{def:weights} 
\begin{align}
w_{\psi,\delta}(r)& \approx \min\left\{r^{k+1/2-\delta},r^{-k+1/2+\delta}\right\},
\qquad \norm{g}_{L^2_{\psi,\delta}}  := \left(\int_0^\infty \frac{\abs{g(r)}^2}{w_{\psi,\delta}^2(r)} \dd r\right)^{1/2}, \label{eq:weightpsi} \\ 
w_{f,\delta}(r)& \approx \min\left\{r^{k+1/2-\delta},r^{-k+1/2-6 + \delta}\right\},
\qquad  \norm{g}_{L^2_{f,\delta}}  := \left(\int_0^\infty \frac{\abs{g(r)}^2}{w_{f,\delta}^2(r)} \dd r\right)^{1/2},\label{eq:weightf} \\ 
w_{F,\delta}(r) & \approx \min\left\{r^{k+1+2-\delta}, r^{-k+1 -6+\delta}\right\},
\qquad \norm{g}_{L^2_{F,\delta}}  := \left(\int_0^\infty \frac{\abs{g(r)}^2}{w_{F,\delta}^2(r)} \dd r\right)^{1/2}. \label{eq:weightF}
\end{align}
\end{subequations} 
We assume that $(r\partial_r)^j w_{\ast,\delta}$ satisfies the same upper bounds as $w_{\ast,\delta}$ (up to a $j$-dependent constant). 
Noting that \ref{V3} ensures $\beta > 0$, we also make use of the $L^2$-space normed by
\begin{align*}
\norm{g}_{L^2_\beta} := \left(\int_0^\infty \frac{\abs{g(r)}^2}{\beta(r)} r \dd r\right)^{1/2},
\end{align*}
which is the natural energy space for the equations (see Lemma \ref{lem:betaspace}).
The main result of the paper is as follows. 
\begin{theorem} \label{thm:Main}  
Let $k\neq 0$, and assume  \ref{V1}-\ref{V3}. For all $1 \gg \delta \gg \eta_0 > 0$  and any smooth $\omega^{in}_{k} \in L^2_{\beta} \cap L^2_{f,\delta}$ satisfying 
the orthogonality condition \eqref{eq:ortho}, the solution $\omega,\psi$ to \eqref{eq:EU} obeys the following inviscid damping estimates 
\begin{subequations} \label{ineq:ID}
\begin{align*}
\norm{\psi_k(t)}_{L^2_{\psi,\delta}} & \lesssim_{\delta,\eta_0} \frac{1}{\brak{kt}^2}\left(k^{5+\eta_0} \abs{\omega_{k,0}^{in}} + \sum_{j=0}^2 k^{6 -2j + \eta_0} \norm{(r \partial_r)^j F_k}_{L^2_{F,\delta/4}} \right) \\
\brak{kt} \norm{r u^r_k(t)}_{L^2_{\psi,\delta}}  + \norm{r u^\theta_k(t)}_{L^2_{\psi,\delta}} & \lesssim_{\delta,\eta_0}  \frac{1}{\brak{kt}}\left(k^{5+\eta_0} \abs{\omega_{k,0}^{in}} + \sum_{j=0}^2 k^{6-2j+\eta_0} \norm{(r \partial_r)^j F_k}_{L^2_{F,\delta/4}} \right). 
\end{align*}
\end{subequations}
Furthermore, there exist $f_{k;1}(t,r)$ and $f_{k;2}(t,r)$ such that
\begin{align}
\omega_k(t,r) & = \e^{-iktu(r)}f_{k;1}(t,r) + \e^{-iktu(r)}f_{k;2}(t,r), \label{eq:vdecomp}
\end{align}
and the following vorticity depletion estimates hold
\begin{subequations} \label{ineq:VortDep}
\begin{align}
  \norm{\sqrt{r} (r \partial_r)^nf_{k;1}(t)}_{L^2_{F,\delta}}  &\lesssim_{\delta,n,\eta_0} \left(k^{2n+1 + \eta_0} \abs{\omega_{k,0}^{in}} + \sum_{j=0}^n k^{2(n-j)+\eta_0} \norm{(r \partial_r)^j F_k}_{L^2_{F,\delta/4}} \right), \label{ineq:VortDep:1}\\
 \norm{(r \partial_r)^n f_{k;2}(t)}_{L^2_{f,\delta}} &\lesssim_{\delta,n,\eta_0} \frac{1}{\brak{kt}}\left(k^{2n+4+\eta_0} \abs{\omega_{k,0}^{in}} + \sum_{j=0}^n k^{2(n-j)+ 3 + \eta_0} \norm{(r \partial_r)^j F_k}_{L^2_{F,\delta/4}} \right), \label{ineq:VortDep:2}
\end{align}
\end{subequations} 
for all $0 \leq n \leq \max(2,k)$ in \eqref{ineq:VortDep:1} and for all $0 \leq n \leq \max(1,k-1)$ in \eqref{ineq:VortDep:2}.
\end{theorem}

\begin{remark} 
By density we can extend the results to cover any $\omega^{in}_k \in L^2$ (satisfying \eqref{eq:ortho}) for which the norms appearing on the right-hand sides above are finite. 
\end{remark}

\begin{remark} 
The $L^2$ norms we are using in Theorem~\ref{thm:Main}, namely \eqref{def:weights},  are natural in light of \eqref{eq:smthExp} and are well-suited for studying vorticity depletion. However, these norms are quite strong at the origin (and infinity). Note that $\norm{\Delta_k^{-1}}_{L^2_{f,\delta} \rightarrow L^2_{\psi,\delta}} \approx k^{-1}$ (as opposed to $k^{-2}$ as one might expect), which explains why some of the powers of $k$ in Theorem \ref{thm:Main} are slightly higher than might be at first expected. Similarly, note that $F_k$ contains information about the second term in the expansion \eqref{eq:smthExp}. 
\end{remark} 

\begin{remark} \label{rmk:profile}
The correct analogue of propagation of regularity for mixing problems is the regularity of: $\e^{iktu(r)}\omega_k(t,r)$, the object which measures the difference between the passive scalar and full linearized (or nonlinear) dynamics. Regularity of this object is often studied in dispersive equations and it is sometimes called the `profile'; see e.g. \cite{BM13} for more discussions (note that regularity of this type was called `gliding regularity' in \cite{MouhotVillani11}). 
Understanding regularity of the profile plays a major role in all of the works involving nonlinear inviscid/Landau damping \cite{CagliotiMaffei98,MouhotVillani11,BM13,BMM13,BMV14,BGM15I,BGM15II,BGM15III,B16} including those which obtain results in Sobolev spaces \cite{MR3437866,BMM16}.
Theorem \ref{thm:Main} deduces higher regularity of the vorticity profile than is necessary to prove the \eqref{ineq:ID}, at least for $k \geq 3$. However, as regularity plays a crucial role in the nonlinear theory, it seems appropriate to study it as carefully as possible in the linear problem. 
This goal has motivated many of the primary aspects of our approach. 
\end{remark} 

\begin{remark} \label{rmk:FixedSob}
Because in this work we were only concerned with obtaining finite Sobolev regularity of the vorticity profile, we have not carefully quantified how the constants in \eqref{ineq:VortDep} depend on $n$. This is sufficient for any fixed Sobolev space of interest, but e.g. for the end-point cases such as $n=k-1$ and $n=k$, we have not quantified the rate at which the constants grow in $n$ as $n=k \rightarrow \infty$, an issue which becomes important at infinite regularity.
\end{remark} 

A direct consequence of \eqref{ineq:VortDep} is the following weak convergence result which shows that the solution weakly converges back to a radially symmetric vortex. 
\begin{corollary}[Vortex axisymmetrization]
For all $k \neq 0$, $\omega_k(t,r) \rightharpoonup 0$ in $L^2$ as $t \rightarrow \pm\infty$. 
\end{corollary}

Another direct corollary of Theorem \ref{thm:Main} shows that the vorticity behaves as a passive scalar evolution in the limit $t \rightarrow \pm \infty$ (the analogue of `scattering' in dispersive equations): 
\begin{corollary}[Scattering to passive scalar evolution]
There exists $\omega_{k,\pm \infty} \in L^2_{f,\delta}$ such that 
\begin{align*}
\lim_{t \rightarrow \pm \infty}\norm{\e^{iktu}\omega_k(t) - \omega_{k,\pm \infty}}_{L^2_{f,\delta}} = 0. 
\end{align*}
\end{corollary}
\begin{remark} 
If we additionally proved that $(r\partial_r)^j f_{k,1}(t)$ converged as $t \rightarrow \pm \infty$ in $L^2_{f,\delta}$ then of course scattering in stronger spaces would follow immediately. This seems likely with some mild technical refinements of our method, but we did not pursue this direction.   
\end{remark}

\begin{remark}
Vortices with a Gaussian distribution of vorticity constitute an important class that satisfies our assumptions \ref{V1}--\ref{V3}. Specifically,
we can consider vortices with
\begin{align*}
\Omega_{\lambda,L}(r)= \frac{\lambda}{4\pi L^2} \exp \left( -\frac{r^2}{4L^2}\right), 
\qquad \Psi_{\lambda,L}(r)= - \frac{\lambda}{2\pi}\int_0^r \frac1r \left[1-  \exp \left( -\frac{r^2}{4L^2}\right) \right]\dd r
\end{align*}
having length scale $L>0$ and total circulation $\lambda>0$. In view of the notation introduced in \eqref{eq:ubeta},
we can compute
\begin{align}
u_{\lambda,L}(r)=\frac{\lambda}{2\pi} \frac{1}{r^2}\left[1-  \exp \left( -\frac{r^2}{4L^2}\right) \right] , 
\qquad \beta_{\lambda,L}(r)=\frac{\lambda}{8\pi L^4}  \exp \left( -\frac{r^2}{4L^2}\right). 
\end{align}
\end{remark}

\begin{remark} 
The restriction $j \leq \max(k,2)$ and the loss of $k^2$ per $r\partial_r$ derivative  (as opposed to $k$) are due to difficulties specific to the vortex case. 
We expect that our methods can easily be adapted to get boundedness of $\e^{iktu(y)}\omega_k(y)$ in $H^s$ for all $s \geq 0$ for strictly monotone shear flows on $\Torus \times \Real$. Our methods may also be able to shed further light on higher derivatives of  $\e^{iktu(y)}\omega_k(y)$ in the case of a channel $\Torus \times [-1,1]$ (see \cite{Zillinger2016,WeiZhangZhao15}).  
\end{remark}

\begin{remark} 
The strict monotonicity \ref{V3} plays a crucial role. See e.g. the studies \cite{SmithRosenbluth1990,NolanMont2000,BalmforthEtAl01,HallBassomGilbert03} showing various kinds of pathologies in non-strictly monotone vortices, including embedded neutral modes (as occur e.g. in Rankine vortices) and non-normal algebraic instabilities. See also the recent nonlinear counter examples to inviscid damping around a vortex constructed in \cite{CastroEtAl2016} without monotonicity. 
\end{remark}

The angular velocity of the background vortex $u(r)$ satisfies $u'(0) = 0$, which indicates that the mixing is weak at the vortex core. 
For well-localized data supported near the origin, one can show that the passive scalar evolution (e.g $\omega_k(t,r) = \e^{-iktu(r)}\omega_k^{in}(r)$) generally cannot give rise to inviscid damping faster than $\norm{u^r(t)} \lesssim \brak{t}^{-1/2}$ and $\norm{u^\theta(t)} \lesssim \brak{t}^{-1}$ (see the proof of Lemma \ref{lem:VortDepID} below for more details). 
Hence, the non-local term in \eqref{eq:EUfourier1} \emph{improves} the inviscid damping rate in \eqref{ineq:ID}. A related effect was conjectured for shear flows with non-degenerate critical points (e.g. points such that $u'(y_c) = 0$ but $u''(y_c) \neq 0$) by Bouchet and Morita \cite{BouchetMorita10}. 
Bouchet and Morita predict that vorticity will be ejected from the critical point, allowing the break-down of the mixing there to have less effect than naively predicted.  
Specifically, Bouchet and Morita predict that the vorticity should behave as 
\begin{align}
\omega_k(t,y) = \e^{-iktu(y)}\omega_{k,\infty}(y) + O(t^{-\gamma}) \textup{ for some } \gamma > 0 \textup{ such that } \omega_{k,\infty}(y_c) = 0. \label{def:vortdep}
\end{align}
In \cite{WeiZhangZhao2017,WZZK2017}, the authors prove that indeed the inviscid damping for such shear flows can be faster than passive scalar. 
The authors directly study the evolution of the streamfunction via methods somewhat different from our approach (though a number of common themes exist); our methods and theirs each have their own advantages and disadvantages. 
Specifically, our methods obtain more precise information about the vorticity directly, and thus the inviscid damping \eqref{ineq:ID} is a straightforward consequence of our vorticity decomposition \eqref{ineq:VortDep} (see Lemma \ref{lem:VortDepID}).  
A vortex analogue of the depletion effect \eqref{def:vortdep} (more carefully quantified) is described by \eqref{ineq:VortDep}: 
 \begin{align*}
\omega_k(t,r) & = \e^{-iktu(r)}f_{k,1}(t,r)   + O\left(\frac{r^{k}}{t}\right) \quad \textup{ as } r \rightarrow 0, \\ 
f_{k,1}(t,r) & = O(r^{k+2})\quad \textup{ as } r \rightarrow 0. 
\end{align*}
Although a hint of the vorticity depletion effect can be seen in the numerics and formal asymptotics of \cite{BassomGilbert98}, our work appears to be the first to precisely connect this depletion effect to vortex dynamics.

In what follows, we denote $\brak{r} = \sqrt{1 + r^2}$. 
We use the notation $f \lesssim g$ if there exists a constant $C > 0$ such that $f \leq Cg$ (and analogously $f \gtrsim g$) 
and $f \approx g$ if $f \lesssim g$ and $f \gtrsim g$. 
We use the notation $f \lesssim_{a,b,...} g$ to emphasize that the constant $C$ depends on $a,b,...$.  
The implicit constants will \emph{never} depend on the quantities $c$, $k$, $r$, $\eps$, or $\omega_k^{in}$ or similar variables except where otherwise noted (see below for the appearance of these quantities). 
Finally, we let $\chi$ be a smooth, non-negative function which satisfies $\chi(r) = 1$ for $\abs{r} < 1/2$ and $0$ for $\abs{r}> 3/4$.

\section{Preliminaries and outline of the proof}

\subsection{Skew-symmetric structure, neutral modes, and contour integral representation}

The following lemma is a basic consequence of the Biot-Savart law for radially symmetric functions (recall $u(r) = (2\pi)^{-1} r^{-2} \int_{0}^r \Omega(s) s \dd s$). 
\begin{lemma}[Basic properties of the vortex] \label{lem:BasicVort}
Every $\Omega(r)$ satisfying conditions \ref{V1}-\ref{V3} satisfies the following: 
\begin{itemize}
\item $\beta(r) > 0$ for all $r \geq 0$, $u'(r) < 0$ for all $r \in (0,\infty)$, $u(r) > 0$ for all $r \geq 0$, and $u'(0) = 0$;
\item $\abs{(r\partial_r)^j \beta(r)} \lesssim_j \brak{r}^{-8}$ for all $j \geq 0$;
\item $\abs{(r \partial_r)^j u(r)} \lesssim_j \brak{r}^{-2}$ for all $j \geq 0$;
\item for $r \leq 1$ there holds $u'(r) \approx -r$ and for $r \geq 1$ there holds $u'(r) \approx -r^{-3}$; 
\item there holds the identity
\begin{align}\label{eq:magic:vortex}
\beta(r) + u''(r) + \frac{3 u'(r)}{r} = 0, \qquad \forall r\in (0,\infty).
\end{align}
\end{itemize}
\end{lemma} 
We rewrite \eqref{eq:EUfourier1} in terms of the vorticity alone as
\begin{align}\label{eq:EUfourierphi}
\de_t  \omega_k+ ik L_k \omega_k =0,\qquad  L_k:= u(r) +\beta(r) \Delta_k^{-1}.  
\end{align} 
A key observation is the following, which is a straightforward calculation via Schur's lemma. 
\begin{lemma}\label{lem:betaspace}
The operator $L_k:L^2_\beta \rightarrow L^2_\beta$ is bounded and self-adjoint with respect to the inner product 
\begin{align*}
\brak{g_1,g_2}_\beta := \int_0^\infty g_1(r) \overline{g_2(r)} \frac{r}{\beta(r)} \dd r. 
\end{align*}
It follows that the $L^2_\beta$ norm is a conserved quantity: 
\begin{align}
\int_0^\infty \frac{|\omega_k(t,r)|^2}{\beta(r)}r\dd r = \int_0^\infty \frac{|\omega^{in}_{k}(r)|^2}{\beta(r)}r\dd r, \qquad \forall t\in \Real. \label{eq:Conserv}
\end{align}
\end{lemma}
\begin{remark} 
This conserved quantity is the quadratic variation of the Casimir used in the energy-Casimir method of nonlinear stability \cite{Dritschel88}.  
 \end{remark}

The next lemma regards the neutral modes that arises due to translation invariance. 
\begin{lemma}[Neutral modes and orthogonality condition \eqref{eq:ortho}]\label{rmk:steady}
When $k=\pm 1$, we have the conservation law: 
\begin{align}
\frac{d}{dt}\int_0^\infty \omega_{\pm1}(t,r) r^2 \dd r = 0. \label{eq:ConLaw}
\end{align}
Moreover, in view of \eqref{eq:magic:vortex}, it is easy to check that
\begin{align}\label{eq:steady}
\psi_S(r)=r u(r), \qquad \omega_S(r)=r\beta(r)
\end{align}
is a steady state  for \eqref{eq:EUfourier1} with $k=\pm 1$ (this is also pointed out in e.g. \cite{BassomGilbert98}). In particular, $\omega_S(r) = r\beta(r)$ is a neutral eigenmode for $L_{\pm 1}$.  
\end{lemma} 
\begin{proof}[Proof of Lemma~\ref{rmk:steady}] 
Indeed, dropping the time dependence and the indices $k=\pm1$, a straightforward calculation (note that the boundary terms vanish due to $L^2_\beta$ conservation) shows that
\begin{align*}
\int_0^\infty \left[u(r) \omega(r)-\beta(r) \psi(r)\right]r^2\dd r
&=-\int_0^\infty \left[r^2 u(r) \de_{rr} \psi(r)+ru(r)\de_r\psi(r)-u(r)\psi(r) +r^2\beta(r)  \psi(r)\right]\dd r\\
&=-\int_0^\infty \left[ r^2u''(r)\psi(r) + 3 ru'(r)\psi(r) +r^2\beta(r)  \psi(r)\right]\dd r=0,
\end{align*}
where we made use of \eqref{eq:magic:vortex} in the last equality. Hence, from \eqref{eq:EUfourier1} we infer \eqref{eq:ConLaw}.  
That \eqref{eq:steady} is a steady state is also simple consequence of \eqref{eq:magic:vortex}.
\end{proof} 
As a result of the self-adjointness of $L_k$, \eqref{eq:EUfourier1} falls under essentially the same framework as Schr\"odinger operators. 
Hence, for all $\omega^{in}_{k} \in L^2_\beta$, the solution to \eqref{eq:EUfourierphi} can be represented via the formula (see \cite{TaylorPDE2}*{Proposition 1.9})
\begin{align}\label{eq:phirepr}
\omega_k(t,r)=\e^{-ikL_kt}\omega_{k}^{in}(r)&=\lim_{\eps\to 0^+}\frac{1}{2\pi i}\int_{\RR} \e^{-ikct}\left[(c-i\eps-L_k)^{-1}-(c+i\eps-L_k)^{-1}\right]\omega_k^{in}(r) \dd c.
\end{align}
Using \eqref{eq:EUfourierphi}, the function
\begin{align*}
\Delta_k \tPhi_k = (z-L_k)^{-1}\omega_k^{in}(r) \qquad r\in [0,\infty),\, z\in \Complex
\end{align*}
satisfies the so-called \emph{inhomogeneous Rayleigh equation} (explicitly writing out $\Delta_k$): 
\begin{align}\label{eq:inoRay}
\left(\de_{rr} +\frac1r \de_r-\frac{k^2}{r^2}\right)\tPhi_k+\frac{\beta(r)}{u(r)-z} \tPhi_k =\frac{\omega^{in}_k(r)}{u(r)-z}.
\end{align}
Note that for $z = c\pm i \eps$ with $\eps > 0$ \eqref{eq:inoRay} is a smooth perturbation of $\Delta_k$ and hence we will not have qualitative smoothness or local-integrability problems for $\eps > 0$. Such difficulties arise only in the limit $\eps \rightarrow 0$. 
By replacing $\tPhi_k$ with 
\[
\Phi_k=\sqrt{r}\tPhi_k
\]
we obtain
\begin{align}\label{eq:inoRay2}
\Ray_{z}\Phi_k(r,z) =\frac{\omega^{in}_k(r) \sqrt{r}}{u(r) - z}, \qquad \Ray_{z}:=\de_{rr} + \frac{1/4 - k^2}{r^2}+ \frac{\beta(r)}{u(r)- z},
\end{align}
supplemented with the boundary conditions that $\Phi_k$ vanishes as $r \rightarrow 0$ and $\infty$ (asymptotic analysis shows that $\Phi_k(r,z) \sim r^{k+1/2}$ as $r \rightarrow 0$ and $\Phi_k(r,z) \sim r^{1/2-k}$ as $r \rightarrow \infty$ provided that $z \neq u(0)$). 
In what follows, normally we will set $z = c\pm i \eps$ and suppress the $c,\eps$ dependence to write 
\begin{align*}
\Ray_{\pm} := \de_{rr} + \frac{1/4 - k^2}{r^2}+ \frac{\beta(r)}{u(r)- c \mp i \eps}. 
\end{align*}
Finally, from \eqref{eq:phirepr} we deduce that 
\begin{align}
\omega_k(t,r)&=\lim_{\eps\to 0^+}\frac{1}{2\pi i}\int_{\RR} \e^{-ikct}\left[\frac{\omega_k^{in}(r)-\beta(r)\tPhi_k(r,c-i\eps)}{u(r)-c+i\eps} -\frac{\omega_k^{in}(r)-\beta(r)\tPhi_k(r,c+i\eps)}{u(r)-c-i\eps}\right]\dd c\notag\\
& = \e^{-iku(r)t} \omega^{in}_k(r)+  \frac{\beta(r) }{2\pi i\sqrt{r}}\lim_{\eps\to0^+}\int_{\RR} \e^{-ikct}\left[\frac{\Phi_k(r,c+i\eps)}{u(r)-c-i\eps}-\frac{\Phi_k(r,c-i\eps)}{u(r)-c+i\eps} \right]\dd c.\label{eq:integrallimit2}
\end{align}
\subsection{Outline of the proof for \texorpdfstring{$k=1$}{TEXT}}
In this case, the proof of Theorem \ref{thm:Main} is based on the explicit formulas
\begin{align*}
f_{  1;1}(r) = \omega^{in}_{  1}(r) + \frac{\beta(r)}{r^2 u'(r)} \int_0^r \omega^{in}_{  1}(\rho) \rho^2 \dd \rho\,, \qquad
f_{  1;2}(r,t) =  r \beta(r) \int_r^\infty \e^{  i(u(r)-u(\rho))t} f_{  1;1}(\rho) \frac{\dd \rho}{\rho u'(\rho)},
\end{align*}
which can be obtained thanks to the explicit solution of the homogeneous Rayleigh problem for $k=1$:
\begin{align}
\phi(r,z) = r^{3/2}(u-z); \label{eq:keq1phi}
\end{align}
 (c.f. \eqref{eq:steady} and see Section \ref{sub:kis11}). 
The vorticity
depletion effect is encoded in the property that
\begin{align*}
 f_{ 1;1}(r)=O(r^3) \quad \mbox{as} \quad r \to 0
\end{align*}
instead of just an $O(r)$ behavior, while
\begin{align*}
 f_{  1;2}(t,r)=O\left(\frac{r}{\brak{t}}\right) \quad \mbox{as} \quad r \to 0,
\end{align*}
which vanishes slower at the origin, but decays in time. The rigorous arguments needed to complete the proof of Theorem \ref{thm:Main} for $k = 1$
are carried out in Section \ref{sub:kis12}.

\subsection{Outline of the proof for \texorpdfstring{$k\geq2$}{TEXT}}
The case when $k\geq 2$ presents fundamental differences compared to the case when $k=1$.
To simplify notation, we omit the dependence on $k$ of the functions involved except when it is relevant.  

\subsubsection{Depletion trick and contour integral decomposition}
The first step of the proof is to isolate the asymptotic expansion at the origin from the rest of the profile.  
For this, we will apply the following trick, in which we remove a harmonic function (with a smooth cutoff) from $\Phi$. 
Besides the function $F$ in \eqref{def:omegak0F}, we define
\begin{align}
Y(r,z) & := \Phi(r,z) - \frac{\chi(r) r^{k+1/2}}{\beta(0)}\omega_{k,0}^{in}, \label{eq:teisfd}\\
F_\ast(r) & :=   -\frac{1}{\beta(0)}\left( (2k+1)\chi'(r) r^{k-1/2} + \chi''(r) r^{k+1/2}\right)\omega_{k,0}^{in}, \label{def:Fast}
\end{align}
where we recall that $\chi$ is a smooth, non-negative function which satisfies $\chi(r) = 1$ for $\abs{r} < 1/2$ and $0$ for $\abs{r}> 3/4$.
From \eqref{eq:inoRay2}, we deduce that
\begin{align}\label{eq:hphiray}
\Ray_z Y = \frac{F(r)}{u(r)-z}+ F_\ast(r).
\end{align}
Going back to \eqref{eq:teisfd}, we have from \eqref{eq:integrallimit2} that the  \emph{profile} 
\begin{align*}
f(t,r)=\e^{iku(r) t} \omega(t,r),
\end{align*}
satisfies
\begin{align}
f(t,r)=\frac{F(r)}{\sqrt{r}}&+  \frac{\beta(r) }{2\pi i\sqrt{r}}\lim_{\eps\to0^+}\int_{\RR} \e^{ik(u(r)-c)t}\left[\frac{Y(r,c+i\eps)}{u(r)-c-i\eps}-\frac{Y(r,c-i\eps)}{u(r)-c+i\eps} \right]\dd c.\label{eq:integrallimit40}
\end{align}
Next, we sub-divide the integrals in \eqref{eq:integrallimit40} in several natural pieces. 
First, we isolate the contributions near and far from the spectrum \eqref{eq:integrallimit40} via the smooth cut-off function
\begin{equation}\label{eq:spectrcutoff}
\chi_\sigma (c) = 
\begin{cases}
1 \qquad &-R_\delta/2 \leq c \leq u(0)+1/2 \\ 
0 \qquad &c < -R_\delta \\ 
0 \qquad &c > u(0) + 3/4,
\end{cases}
\end{equation}
where $R_\delta>0$ will be fixed later (depending on $\delta$), so that
\begin{align}
f(t,r)=\frac{F(r)}{\sqrt{r}}&+  \frac{\beta(r) }{2\pi i\sqrt{r}}\lim_{\eps\to0^+}\int_{-R_\delta}^{u(0)+1} \e^{ik(u(r)-c)t}\left[\frac{\hPhi(r,c+i\eps)}{u(r)-c-i\eps}-\frac{\hPhi(r,c-i\eps)}{u(r)-c+i\eps} \right]\chi_\sigma (c)\dd c\notag\\
&+  \frac{\beta(r) }{2\pi i\sqrt{r}}\lim_{\eps\to0^+}\int_{\RR} \e^{ik(u(r)-c)t}\left[\frac{\hPhi(r,c+i\eps)}{u(r)-c-i\eps}-\frac{\hPhi(r,c-i\eps)}{u(r)-c+i\eps} \right](1-\chi_\sigma (c))\dd c.\label{eq:integrallimit4}
\end{align}
Further, define
\begin{align}\label{eq:rcdef}
r_c := u^{-1}(c), \qquad c\in (0,u(0)].
\end{align}
The region $r \sim r_c$ is known in the classical fluid mechanics literature as the \emph{critical layer} \cite{DR81}. 
Near the points $c=0$ and $c=u(0)$ there are a variety of subtleties in the resolvent. This can be expected due to the change in the nature of the singularities in the Rayleigh equation \eqref{eq:inoRay2} at these points. 
Since the first-order singular point at the critical layer merges with the second order singularities at $r =0, \infty$, the influence of $\eps$ in \eqref{eq:inoRay2} will be felt much more, hence, we need to pass to the limit $\eps \rightarrow 0$ in a \emph{non-uniform} way over the spectrum $c \in [0,u(0)]$.  
From Lemma \ref{lem:BasicVort} and considering the points where $(u(0)-c)^2 \approx \eps^2$ and $c^2 \approx \eps^2$, we see that the natural place to divide
the complex plane is along the curves $\eps \approx r_c^2$ for $r_c \ll 1$ and $\eps \approx r_c^{-2}$ for $r_c \gg 1$. 
A small, but crucial, point is that we can afford some flexibility in this boundary. 
Let $\alpha\in (0,\delta)$ be a parameter chosen sufficiently small in the sequel depending only on $\delta$ (from Theorem \ref{thm:Main}).  
 Define the set 
\begin{align}
I_\alpha = \set{z = c \pm i \eps \in \Complex: c \in (0,u(0)), \; k^5 \eps \leq \min( r_c^{2+\alpha}, r_c^{-2-\alpha})} \label{eq:Ialpha}
\end{align}
and the associated smooth cut-off function 
\begin{align*}
\chi_I(r_c) = \chi\left(\frac{r_c}{(k^5\eps)^{-\frac{1}{2+\alpha}}}\right)  \left[1-\chi\left(\frac{r_c}{(k^5\eps)^{\frac{1}{2+\alpha}}}\right)   \right].
\end{align*}
Then, we further divide the contour integral by 
\begin{align}
f(t,r)=\frac{F(r)}{\sqrt{r}}&+  \frac{\beta(r) }{2\pi i\sqrt{r}}\lim_{\eps\to0^+}\int_{0}^{u(0)} \e^{ik(u(r)-c)t}\left[\frac{\hPhi(r,c+i\eps)}{u(r)-c-i\eps}-\frac{\hPhi(r,c-i\eps)}{u(r)-c+i\eps} \right]\chi_I(r_c)\dd c\notag\\
&+  \frac{\beta(r) }{2\pi i\sqrt{r}}\lim_{\eps\to0^+}\int_{-R_\delta}^{u(0)+1} \e^{ik(u(r)-c)t}\left[\frac{\hPhi(r,c+i\eps)}{u(r)-c-i\eps}-\frac{\hPhi(r,c-i\eps)}{u(r)-c+i\eps} \right]\chi_\sigma (c)(1-\chi_I(r_c))\dd c\notag\\
&+  \frac{\beta(r) }{2\pi i\sqrt{r}}\lim_{\eps\to0^+}\int_{\RR} \e^{ik(u(r)-c)t}\left[\frac{\hPhi(r,c+i\eps)}{u(r)-c-i\eps}-\frac{\hPhi(r,c-i\eps)}{u(r)-c+i\eps} \right](1-\chi_\sigma (c))\dd c.\label{eq:integrallimit5}
\end{align}
The first term will be thought of as ``close to the spectrum'', whereas the latter two terms will be considered ``far from the spectrum''. 
Given the singular integrals in the representation formula, the two relevant quantities appearing are 
\begin{subequations} \label{eq:XandA} 
\begin{align}
X(r,c,\eps)& =\hPhi(r,c+i\eps)-\hPhi(r,c-i\eps), \\ 
 A(r,c,\eps)& =\hPhi(r,c+i\eps)+\hPhi(r,c-i\eps) = X + 2 \hPhi(r,c-i\eps)
\end{align}
\end{subequations} 
so that from \eqref{eq:integrallimit5} we write: 
\begin{align}
f(t,r)=\frac{F(r)}{\sqrt{r}}&+  \frac{\beta(r) }{2\pi i\sqrt{r}}\lim_{\eps\to0^+}\int_{0}^{u(0)}\frac{i\eps \e^{ik(u(r)-c)t}}{(u(r)-c)^2+\eps^2}\chi_I(r_c) A(r,c,\eps)\dd c\notag\\
&+  \frac{\beta(r) }{2\pi i\sqrt{r}}\lim_{\eps\to0^+}\int_{0}^{u(0)} \frac{(u(r)-c)\e^{ik(u(r)-c)t}}{(u(r)-c)^2+\eps^2}\chi_I(r_c) X(r,c,\eps)\dd c\notag\\
&+  \frac{\beta(r) }{2\pi i\sqrt{r}}\lim_{\eps\to0^+}\int_{-R_\delta}^{u(0)+1} \e^{ik(u(r)-c)t}\left[\frac{\hPhi(r,c+i\eps)}{u(r)-c-i\eps}-\frac{\hPhi(r,c-i\eps)}{u(r)-c+i\eps} \right]\chi_\sigma (c)(1-\chi_I(r_c))\dd c\notag\\
&+  \frac{\beta(r) }{2\pi i\sqrt{r}}\lim_{\eps\to0^+}\int_{\RR} \e^{ik(u(r)-c)t}\left[\frac{\hPhi(r,c+i\eps)}{u(r)-c-i\eps}-\frac{\hPhi(r,c-i\eps)}{u(r)-c+i\eps} \right](1-\chi_\sigma (c))\dd c.\label{eq:integrallimit6}
\end{align}
See Figure \ref{fig} for a summary of how the limiting procedure in \eqref{eq:integrallimit6} is carried out below. 

\begin{figure}[ht]
        \centering 
        \includegraphics[width= 0.70 \textwidth]{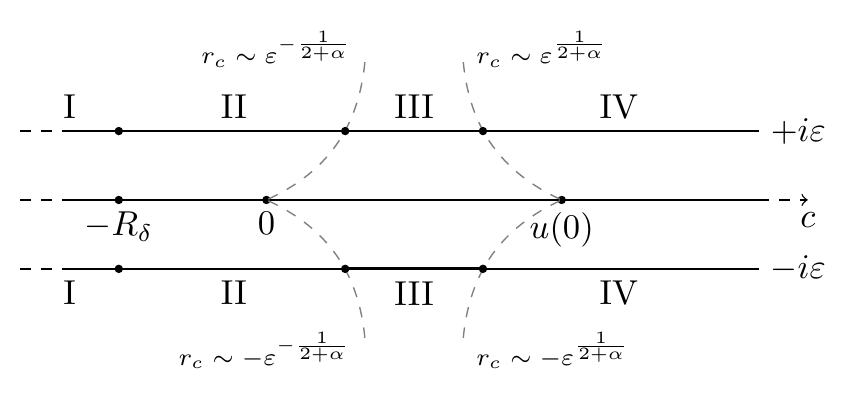}
        \caption{ This figure summarizes the limiting procedure used to treat \eqref{eq:integrallimit6}. Region III represents $I_\alpha$. The contribution from this region is further decomposed into $f_1^\eps$ and $f_2^\eps$ (see \eqref{eq:fIone} and \eqref{eq:fItwo} below) which converge to the decomposition in \eqref{ineq:VortDep}.
 The limiting procedure is done by constructing the Green's function for $\Ray_z$ and making analyses of the resulting integral operators (carried out in \S\ref{sec:HomRay}--\S\ref{sec:repbd} together with Appendix \ref{sec:BndConv}). 
In regions I and IV we apply energy estimates on $\Ray_z$ to prove these contributions vanish (carried out in Appendix \ref{app:vanishingout}). 
Here $\delta > 0$ is traded to gain the flexibility to take $\alpha > 0$. 
 In region II, we apply a compactness-contradiction argument with a second order comparison principle that shows these contributions also must vanish (carried out in Appendix \ref{sec:Ybd4}).} \label{fig} 
\end{figure}

There is one additional decomposition necessary in order to see the vorticity depletion effect -- the decomposition in $f_1$ and $f_2$. 
While $F(r)r^{-1/2}$ has better decay at the origin than $\omega^{in}_k(r)$, it is clear that for, e.g. $r_c=1$, both $A(r,c,\eps) r^{-1/2}$ and $X(r,c,\eps) r^{-1/2}$ can be expected to have the same asymptotic expansion at the origin as the solution of the Laplace equation (specifically, $\sim r^{k}$). 
Hence, at any fixed $t$, clearly $f$ cannot have better decay at the origin (qualitatively speaking) than $\omega^{in}_k(r)$. 
We instead deduce that the leading asymptotic expansion is decaying in time. 
We divide the contribution from $X$ in two pieces, by means of
\begin{align*}
\chi_{1}(r,r_c)  = [1-\chi(r/2) \chi(r/r_c)]\chi_I(r_c), \qquad \chi_{2}(r,r_c) = \chi(r/2) \chi(r/r_c) \chi_I(r_c).
\end{align*}
Note that $ \chi(r/2) \chi(r/r_c)=1$ when $r\leq 1$ and $r_c\geq 2r$. In this way, we can rewrite \eqref{eq:integrallimit6} as
\begin{align}
f(t,r)=\lim_{\eps\to 0}\big[f_1^\eps(t,r)+f_2^\eps(t,r)+f_{S}^\eps(t,r)+f_{E}^\eps(t,r)\big], \label{eq:intlim7}
\end{align}
where
\begin{align}
f_{1}^\eps (t,r)
&= \frac{F(r)}{\sqrt{r}}+  \frac{\beta(r) }{2\pi i\sqrt{r}}\int_{0}^{u(0)}\frac{i\eps \e^{ik(u(r)-c)t}}{(u(r)-c)^2+\eps^2}\chi_I(r_c) A(r,c,\eps)\dd c\notag\\
&\qquad\qquad
+  \frac{\beta(r) }{2\pi i\sqrt{r}}\int_{0}^{u(0)} \frac{(u(r)-c)\e^{ik(u(r)-c)t}}{(u(r)-c)^2+\eps^2}\chi_{1}(r,r_c)  X(r,c,\eps)\dd c,\label{eq:fIone}\\
f_{2}^\eps(t,r)&=\frac{\beta(r) }{2\pi i\sqrt{r}}\int_{0}^{u(0)} \frac{(u(r)-c)\e^{ik(u(r)-c)t}}{(u(r)-c)^2+\eps^2}\chi_{2}(r,r_c) X(r,c,\eps)\dd c,\label{eq:fItwo}\\
f_{S}^\eps(t,r)&= \frac{\beta(r) }{2\pi i\sqrt{r}}\int_{-R_\delta}^{u(0)+1} \e^{ik(u(r)-c)t}\left[\frac{\hPhi(r,c+i\eps)}{u(r)-c-i\eps}-\frac{\hPhi(r,c-i\eps)}{u(r)-c+i\eps} \right]\chi_\sigma (c)(1-\chi_I(r_c))\dd c,\label{eq:fS}\\
f_{E}^\eps(t,r)&=\frac{\beta(r) }{2\pi i\sqrt{r}}\int_{\RR} \e^{ik(u(r)-c)t}\left[\frac{\hPhi(r,c+i\eps)}{u(r)-c-i\eps}-\frac{\hPhi(r,c-i\eps)}{u(r)-c+i\eps} \right](1-\chi_\sigma (c))\dd c.\label{eq:fE}
\end{align}
In order to prove Theorem \ref{thm:Main} we moreover need to express $(r\partial_r)^j f$ for $j \leq k$. 
In what follows, denote $\partial_G$ as: 
\begin{align}
\partial_G := \frac{1}{u'(r)}\partial_r + \partial_c. \label{def:dG} 
\end{align}
The significance of this derivative is the following: formally integrating by parts in $c$ assuming that no boundary terms appear gives: 
\begin{align*}
r\partial_r \left(\int_0^\infty h(u(r) - c) B(r,c) \dd c\right) = \int_0^\infty h(u(r) - c) ru'(r) \partial_G B(r,c) \dd c. 
\end{align*}
The $G$ refers to the terminology of ``good derivative'' in \cite{WeiZhangZhao15}, where an analogous derivative arises for a similar reason.  
Iterating gives, 
\begin{subequations} \label{eq:rdrf1f2} 
\begin{align}
(r\partial_r)^j f_1^\eps(t,r) 
& =  (r\partial_r)^j F(r) + \frac{1}{2\pi i} \int_0^\infty \frac{2i\eps u'(r_c) \e^{itk(u(r)-u(r_c))}}{(u(r)-c)^2 + \eps^2} (ru'(r) \partial_G)^j \left(\chi_{I}(r_c) \frac{\beta(r)}{\sqrt{r}} A(r,c,\eps)\right) \dd r_c \\ 
& \quad + \frac{1}{2\pi i} \int_0^\infty \e^{itk(u(r)-u(r_c))} \frac{(u(r)-c)u'(r_c)}{(u(r)-c)^2 + \eps^2} (ru'(r)\partial_G)^j \left(\chi_{1}(r,r_c) \frac{\beta(r)}{\sqrt{r}} X(r,c,\eps)\right) \dd r_c, \label{eq:rdrf1} \\
(r\partial_r)^j f_2^\eps(t,r) &  = \frac{1}{2\pi i} \int_0^\infty \e^{itk(u(r)-u(r_c))} \frac{(u(r)-c)u'(r_c)}{(u(r)-c)^2 + \eps^2} (ru'(r)\partial_G)^j \left(\chi_{2}(r,r_c) \frac{\beta(r)}{\sqrt{r}} X(r,c,\eps)\right) \dd r_c. 
\end{align}
\end{subequations}
This formula will be used below to obtain higher derivative estimates on $f$ and $f_j^\eps$. 
Finally, in order to complete our characterization of vorticity depletion, we obtain a decay estimate on  $(r\partial_r)^j f_2^\eps$ like $O(t^{-1})$. For this, we will integrate by parts in $r_c$: 
\begin{align}
(r\partial_r)^j f_2^\eps(t,r) &  = -\frac{1}{2\pi kt} \int_0^\infty \e^{itk(u(r)-u(r_c))} \partial_{r_c}\left(\frac{(u(r)-c)}{(u(r)-c)^2 + \eps^2} (ru'(r)\partial_G)^j \left(\chi_{2}(r,r_c) \frac{\beta(r)}{\sqrt{r}} X(r,c,\eps)\right)\right) \dd r_c. \label{eq:f2IBP}
\end{align}
Notice that in the formulas above, near $r=0$, the derivatives landing on $X$ or $A$ will be roughly $O(r^{2\ell}) \partial_G^{\ell}$; we will see that each $\partial_G$ `costs' $r^{-2}$ near the origin. Indeed, we have the following observation regarding $F$ and $F_\ast$: 
\begin{lemma} \label{lem:dGtodr}
There holds 
\begin{align*}
\norm{\partial_G^\ell F}_{L^2_{F,\delta + 2\ell}} & \lesssim_{\ell} \sum_{m=1}^\ell \norm{(r\partial_r)^m F}_{L^2_{F,\delta}}, \qquad 
\norm{\partial_G^\ell F_\ast}_{L^2_{F,\delta + 2\ell}}  \lesssim_{\ell} k^{1+\ell} \abs{\omega_{k,0}^{in}}. 
\end{align*}
\end{lemma}

\subsubsection{Vanishing for $z \notin I_\alpha$ as $\eps \to 0$}
The contributions of $f_S^\eps$ and $f_{E}^\eps$ \eqref{eq:fS} and \eqref{eq:fE} vanish as $\eps\to 0$. 
Hence, the only relevant values of the spectral parameter $c$ are those contained
 in the interval $[0,u(0)]$, the range of $u$, which is the continuous spectrum of the operator $L_k$ in \eqref{eq:EUfourierphi}. 
The precise statement is contained in the following theorem.

\begin{theorem}
Assume $k\geq 2$, and let $j\in \{0,\ldots,k\}$ be a fixed integer, $\alpha < \delta/8$, and $\delta$ sufficiently small. Then for all $t \in \Real$, 
\begin{align*}
\lim_{\eps\to0}\left[\|(r\de_r)^jf_{E}^\eps(ts,\cdot)\|_{L^2_{f,\delta}}+\|(r\de_r)^jf_{S}^\eps(t,\cdot)\|_{L^2_{f,\delta}}\right]=0. 
\end{align*}
\end{theorem}
The proof of this theorem is contained in Appendix \ref{app:vanishingout}. The main ingredient is a set of careful
energy estimates on a slightly more generalized version of the Rayleigh problem \eqref{eq:hphiray}, as stated
in Theorem \ref{thm:boundsonY}. 
These estimates allow to trade some $\delta > 0$ for freedom to choose $\alpha > 0$ in \eqref{eq:Ialpha}. 
The estimates are then used in an iterative way, to bootstrap from the initial 
$L^2$ bound up to the $k$-th derivative.
 Indeed, the functions $X,Y,A$ and their $\de_G$ derivatives satisfy an equation
of the type
\begin{align}\label{eq:hphiray20}
\Ray_{\pm} \de_G^{j+1}Y(\cdot, c\pm i \eps) = \frac{F_{j+1}^{\pm}}{u(r)-c\mp i\eps}+ R_{j+1},
\end{align}
where $R_{j}$ and $F_{j}$ depend on $F,F_\ast$ and lower order derivatives of $Y$. The difficulties 
we face are summarized below.

\medskip

\noindent $\diamond$ The energy estimates depend on the region where $c$ ranges, and, in turn, on the asymptotic
expansion of $u$ near $r=0$ and $r=\infty$. Hence, the bounds are necessarily different and have to be treated 
on a case-by-case basis. 
The non-uniformity in which $\eps\to 0$ discussed in the previous section plays
a key role. Moreover, bounds have to encode the correct integrability in $c$, as the formula \eqref{eq:fS} deals with the
endpoints $c=0,u(0)$, while \eqref{eq:fE} requires integrability as $c\to\pm\infty$. 

\medskip

\noindent $\diamond$ While in most cases a (weighted) energy estimate for \eqref{eq:hphiray20} is obtainable by multiplication
by $Y_\eps$ and integration by parts, the case near $c= 0$ requires a contradiction argument. Due to compactness, 
a failure of the energy estimate would imply the existence of a localized solution to $\Ray_{\pm}\phi = 0$, which is ruled out by a second 
order comparison principle against the homogeneous  solution at $k=1$ associated with \eqref{eq:steady}. 

\medskip

\noindent $\diamond$  In the iterative process, $R_{j}$ contains coefficients (that depend on $u$ and $\beta$)
that are very singular and that require a gain of $r^2$ at $r=0$. This is related to the ``cost'' of taking $\de_G$ derivatives,
discussed above. This loss can be handled up to $k-1$ derivatives, by keeping track of the weight correction allowed
in the energy estimates (the parameter $\gamma$ in Theorem \ref{thm:boundsonY}). The case $j=0$ is carried out in detail in
Sections \ref{sub:j0E}-\ref{sub:j0S}, while the generalization to derivatives up to $k-1$ is handled in Section \ref{sub:jlessk}.

\medskip

\noindent $\diamond$ The $k$-th derivative is handled directly by expanding the Rayleigh operator in \eqref{eq:hphiray20}
and exploiting the elliptic regularity due to the second order derivative present in $\Ray_\pm$ (see Section \ref{sub:jequalk}).

\subsubsection{Green's function for the Rayleigh problem for $z \in I_\alpha$ as $\eps \rightarrow 0$} 

\paragraph{Homogeneous Rayleigh problem}
While for $k=1$ the exact solution \eqref{eq:steady} allows the construction of the Green's function in a fairly direct way (for all
$z\in \Complex$),
the picture in the case $k\geq 2$ is complicated by the lack of explicit formulae for the solution to the homogeneous Rayleigh
problem
$\Ray_\pm \phi=0$.
In Section \ref{sec:HomRay}, we derive the existence of a unique solution $\phi$
of the form
\begin{align*}
\phi(r,z)=P(r,z) (u(r)-z), \qquad z\in I_\alpha,
\end{align*}
which satisfies $P(r_c,z) = 1$, $\partial_r P(r_c,z) = 0$. 
The function $P$ also encodes the behavior of $\phi$ away from the critical layer (essentially, the precise asymptotics as $r\to 0$ and $r\to\infty$). 
Theorem \ref{thm:phik} treats the general case when $z\in I_\alpha$, while Theorem \ref{thm:realphik} focuses
on further properties when $z\in (0,u(0))$ is real-valued. The convergence estimates are stated in Theorem \ref{thm:nonvan}.
The proofs are articulated in different steps. 

\medskip

\noindent $\diamond$ Existence and uniqueness of $P$ is proved through an auxiliary function $\tP$, related to $P$ via
\begin{align*}
P(r,z)= \left(\frac{r}{r_c}\right)^{3/2} \tP(r,z),
\end{align*}
which satisfies a proper integral equation, treatable under the contraction mapping principle in weighted $L^\infty$-based spaces (note this step shares some similarity with \cite{WeiZhangZhao15}). In this way, existence, uniqueness, and the nearly correct behavior at 0 and $\infty$ (up to a small correction) in both $r$ and $r_c$ is obtained at once, along
with suitable bounds (Proposition \ref{prop:tildePfirst}).
Note that the definition of $\tP$ is informed by the exact solution \eqref{eq:keq1phi} in the case $k=1$. 

\medskip

\noindent $\diamond$ When $z\in (0,u(0))$ is real, further monotonicity properties of $P$, are available (see Theorem \ref{thm:realphik}). 
Of extreme importance is the fact that $\phi$ only vanishes at the critical layer, when $r=r_c$,
along with the correct $k$-dependence of the estimates involved.

\medskip

\noindent $\diamond$ In order to prove that the complex solution $\phi(r,z)$ 
only vanishes at the critical layer for every $z\in I_\alpha$, we 
prove convergence estimates for $P(r,c\pm i\eps)-P(r,c)$ and its various derivatives, in the correct $L^\infty$-weighted spaces. 
This is carried out in various steps. In Section \ref{sub:convlame}, we use again the function $\tP$, and we derive sub-optimal 
(in both $r$ and $r_c$) convergence estimates. Similarly, we treat $\de_G\tP$ near the critical layer in Section \ref{sub:convdeG} and $\de_{r_c}\tP$ 
in Section \ref{sub:convderc}.
The optimality in $r$ is then obtained in Section \ref{sub:convrev}. In both cases, 
factors related to $r_c^{2}$ and $r_c^{-2}$ appear in the convergence estimates, due to the nature of the singularities 
at $r,r_c=0,\infty$. This is the primary practical motivation behind the definition of the region $I_\alpha$, and the subsequent non-uniform 
passage to the limit as $\eps\to 0$. 
Indeed, only for $z\in I_\alpha$ are we able to deduce that $\phi(r,z)\neq 0$ for all $r\neq r_c$, a key property (as we see from \eqref{def:H0Hin} below).

\medskip

\paragraph{Inhomogeneous Rayleigh problem}
In order to construct a Green's function for $\Ray_z$ from \eqref{eq:inoRay2}, we again use reduction of order. 
For $z \in I_\alpha \setminus \textup{Ran}\, u$, we define the two homogeneous solutions which each satisfy one of the boundary conditions via: 
\begin{align}
H_0(r,z)=-\phi(r,z)\int_0^r \frac{1}{\phi(s,z)^2}\dd s, \qquad H_\infty(r,z)=\phi(r,z)\int_r^\infty \frac{1}{\phi(s,z)^2}\dd s. \label{def:H0Hin}
\end{align}
Note that $H_0$ and $H_\infty$ are well defined by absolutely convergent integrals for $z \in I_\alpha \setminus \textup{Ran}\, u$ and are solutions to the homogeneous Rayleigh equation \eqref{eq:Rayleigh}, whose Wronskian is
\begin{align}
M(z) := H_0(r,z)\de_rH_\infty(r,z)-H_{\infty}(r,z)\de_r H_0(r,z) = \int_0^\infty \frac{1}{\phi^2(s,z)} \dd s . \label{def:M}
\end{align}
One of the crucial lemmas is Lemma \ref{lem:Wronk}, which provides the following lower bound uniformly for $z \in I_\alpha$, 
\begin{align*}
\abs{M(c \pm i \eps)} \gtrsim k \max(r_c^{-3},r_c^5).
\end{align*}
Note that the singularities at $r_c \rightarrow 0$ and $r_c \rightarrow \infty$ are in fact a gain (which are crucial for obtaining the vorticity depletion).     
Hence,  the Green's function for the Rayleigh operator for $z\in I_\alpha \setminus \textup{Ran} \, u$ is
\begin{align} \label{def:Greens} 
\cG(r,r',z)=\frac{1}{M(z)}
\begin{cases}
H_0(r,z)H_\infty(r',z), \quad & r<r',\\
H_0(r',z)H_\infty(r,z), \quad & r>r'.
\end{cases}
\end{align}
In addition to the lower bound on $M$, precise estimates on $H_0$ and $H_\infty$ follow from our study of $\phi(r,z)$ (see \S\ref{sec:MH}) as well as convergence as $\eps \rightarrow 0$. 
Due to the apparently singular integrals that appear as $\eps \rightarrow 0$ in \eqref{def:H0Hin}, it is not obvious that $H_0(r, c\pm i \eps)$ and $H_\infty(r,c\pm i \eps)$ converge, but cancellations for $\eps > 0$ ensure we have well-defined, log-Lipschitz limiting functions $H_0(r,c\pm i 0)$, $H_\infty(r,c\pm i 0)$ (as expected from shear flows \cite{DR81}). 
See \S\ref{sec:MH} for more details. 

The Green's function gives us representation formulas for $Y^{\pm}:= Y(r,c\pm i \eps) $ and $X:=X(r,c,\eps)$ for $\eps > 0$ (see Lemma \ref{lem:IterInv}), 
\begin{subequations} \label{eq:XYrep}
\begin{align}
Y^{\pm} &= \int_0^\infty \cG(r,s,c\pm i \eps) \frac{u(s)-c}{(u(s)-c)^2 + \eps^2} F(s) \dd s \notag \\ 
& \quad \pm \int_0^\infty \cG(r,s,c\pm i \eps) \frac{i \eps }{(u(s)-c)^2 + \eps^2} F(s) \dd s  + \int_0^\infty \cG(r,s,c\pm i \eps) F_\ast(s)  \dd s \\ 
X&= \int_0^\infty B^{(1)}_{X\delta;\eps}(r,s,c) \frac{2 i\eps}{(u(s)-c)^2 + \eps^2} F(s) \dd s  \notag\\ 
& \quad + \int_0^\infty \left(\int_0^\infty \frac{2i\eps \beta(s_0) }{(u(s_0)-c)^2 + \eps^2}B_{XS;\eps}^{(1)}(r,s_0,c)  B_{XS;\eps}^{(2)}(s_0,s,c) \dd s_0\right) \frac{u(s)-c}{(u(s)-c)^2 + \eps^2} F(s) \dd s \notag\\ 
& \quad + \int_0^\infty \left(\int_0^\infty \frac{2i\eps \beta(s_0) }{(u(s_0)-c)^2 + \eps^2}B_{XG;\eps}^{(1)}(r,s_0,c)  B_{XG;\eps}^{(2)}(s_0,s,c) \dd s_0\right) F_\ast(s) \dd s, 
\end{align}
\end{subequations}
where 
\begin{align*}
&B_{X\delta;\eps}^{(1)}(r,s,c)  = \cG(r,s,c+i\eps) + \int_0^\infty \frac{2i\eps \beta(s_0)}{(u(s_0)-c)^2 + \eps^2} \cG(r,s_0,c+ i\eps) \cG(s_0,s,c-i\eps) \dd s_0,  \\ 
&B_{XG;\eps}^{(1)} = B_{XS;\eps}^{(1)}(r,s_0,c)  = \cG(r,s_0,c+i\eps). 
\end{align*}

\subsubsection{Representation formulas and boundedness for $(r\partial_r)^j f_1$  and $(r\partial_r)^j f_2$} \label{sec:RepBdIntro}

\paragraph{Iteration scheme and representation formulas for $\partial_G^j$ derivatives for $j \leq k-1$} 
From \eqref{eq:rdrf1f2}, we see that a key step in the proof of Theorem \ref{thm:Main} is estimating $\partial_G^jX$and $\partial_G^jA$ for $z \in I_\alpha$. 
The crucial property of $\partial_G$ derivatives is that they vanish on functions of $u-c$, and hence, the commutator $[\partial_G, \Ray_z]$ is not more singular than $\Ray_z$ itself at the critical layer (see \eqref{eq:dGRayComm} below). 
As a result, we are able to use an iteration scheme of the following general form to control higher derivatives; this iteration scheme is one of the insights for obtaining higher regularity of the profile. 

\begin{lemma}[Iteration lemma for $\partial_G^jX$ and $\partial_G^j Y$ for $\eps > 0$] \label{lem:IterSchemeIntro}
For $F_j,R_j,R_j^x,\cE_j$  defined below in Lemma \ref{lem:IterScheme} we have the iteration
\begin{subequations} \label{eq:RaydGjXY}
\begin{align*}
&\Ray_+ \partial_G^{j+1}X  = \frac{2i\eps}{(u-c)^2 + \eps^2 }\left(F_{j+1}^{-} - \beta \partial_G^{j+1}Y^-\right) + \frac{\cE_{j+1}}{u-c-i\eps} + R^x_{j+1} \\
&\Ray_{\pm} \partial_G^{j+1}Y^{\pm}  = \frac{F^{\pm}_{j+1}}{u-c\mp i\eps}  + R^{\pm}_{j+1}.
\end{align*}
The quantities $R_j,\cE_j,R^x_j,$ and $F_j$ depend only on $\partial_G^\ell X$ and $\partial_G^\ell Y$ for $0 \leq \ell \leq j-1$ (as well as $u$, $\beta$, $F$ and $F_\ast$). 
\end{subequations}  
\end{lemma} 

Using the recursion scheme described in \eqref{eq:RaydGjXY} and Fubini's theorem, it is not hard to formally verify the following Proposition by induction, which allows to directly express $\partial_G^jX$ and $\partial_G^jY$ in terms of $F$ and $F_\ast$ in a form essentially the same as \eqref{eq:XYrep} except with much more complicated kernels. 
See Lemmas \ref{lem:XYYRecurse}--\ref{lem:YYYrecurse} below. 
\begin{proposition} \label{prop:repformXY}
For all $j \leq k- 1$, there hold representation formulas of the general form for various kernels $B$ and weights $w$ for $\eps > 0$: 
\begin{align}
\partial_G^j Y & = \sum_{\ell = 0}^j   \int_0^\infty B_{YS;j,\ell}(r,s,c) \frac{u(s)-c}{(u(s)-c)^2 + \eps^2} w_{S;j,\ell}(s)  \partial_G^\ell F(s) \dd s\notag \\ 
& \quad + \sum_{\ell = 0}^j \int_0^\infty B_{Y\delta;j,\ell}(r,s,c) \frac{2 i \eps}{(u(s)-c)^2 + \eps^2} w_{\delta;j,\ell}(s) \partial_G^\ell F(s) \dd s \notag\\
& \quad + \sum_{\ell = 0}^j  \int_0^\infty B_{YG;j,\ell}(r,s,c) w_{G;j,\ell}(s) \partial_G^\ell F_\ast (s) \dd s \notag\\ 
& \quad + \textup{Similar terms with different $B$, $w$}, \label{eq:dGjY}
\end{align}
where by ``Similar terms with different $B$, $w$'' we mean terms with exactly the same formal structure, except with different $B$ kernels and weights $w$ (however, all of the omitted terms will share the same estimates). 
Similarly, for various kernels $B$ and weights $w$ we have a similar representation formula for $\eps > 0$: 
\begin{align}
&\partial_G^j X  = \sum_{\ell = 0}^j \int_0^\infty B^{(1)}_{X\delta1;j,\ell}(r,s,c)
\frac{2i \eps }{(u(s)-c)^2 + \eps^2} w_{X\delta1;j,\ell}(s)  \partial_G^\ell F(s) \dd s \notag\\ 
&  + \sum_{\ell = 0}^j \int_0^\infty \left(\int_0^\infty \frac{2i\eps \beta(s_0)}{(u(s_0)-c)^2 + \eps^2} B^{(1)}_{X\delta;j,\ell}(r,s_0,c) B^{(2)}_{X\delta;j,\ell}(s_0,s,c) \dd s_0\right)  \frac{2 i \eps}{(u(s)-c)^2 + \eps^2}  w_{X\delta2;j,\ell}(s) \partial_G^\ell F(s) \dd s \notag\\ 
&  + \sum_{\ell = 0}^j \int_0^\infty \left(\int_0^\infty \frac{2i\eps \beta(s_0)}{(u(s_0)-c)^2 + \eps^2} B^{(1)}_{XS;j,\ell}(r,s_0,c) B^{(2)}_{XS;j,\ell}(s_0,s,c) \dd s_0 \right)  \frac{(u(s)-c)}{(u(s)-c)^2 + \eps^2}  w_{XS;j,\ell}(s) \partial_G^\ell F(s) \dd s \notag\\ 
&  + \sum_{\ell = 0}^j \int_0^\infty \left(\int_0^\infty \frac{2i\eps \beta(s_0)}{(u(s_0)-c)^2 + \eps^2} B^{(1)}_{XG;j,\ell}(r,s_0,c) B^{(2)}_{XG;j,\ell}(s_0,s,c) \dd s_0 \right) w_{XG;j,\ell}(s) \partial_G^\ell F_\ast(s) \dd s \notag\\ 
&  + \textup{Similar terms with different $B$, $w$}. \label{eq:dGjX}
\end{align}
Furthermore, in each term above, the weights satisfy an estimate of the following form for some $0 \leq \ell' \leq j-\ell$ and all $m \geq 0$ (\emph{different $\ell'$ for each term}), 
\begin{align}
\abs{(s\partial_s)^m w_{\ast;j,\ell}(s)} \lesssim_m \max(s^{-2\ell'},s^{2\ell'}). 
\end{align}
\end{proposition}
Along with \eqref{eq:dGjY} and \eqref{eq:dGjX}, the proof of Proposition \ref{prop:repformXY} derives also an associated recursion for the various $B$ kernels appearing above (see Lemmas \ref{lem:XYYRecurse}--\ref{lem:YYYrecurse}). 
That is, the kernels $B^{(\ast)}_{\ast;j,\ell}$ are determined from $B^{(\ast)}_{\ast;j-1,m}$ via a few canonical integral operators involving the Green's function $\cG$. 
A crucial idea of our method is to use these recursion formulae on the $B$'s, together with precise estimates on $\cG$, to obtain precise estimates on all possible $B$ kernels by induction. 
This method allows us to treat all of $\partial_G^j X$ and $\partial_G^j Y$ for $j \leq k-1$ \emph{simultaneously} (as discussed below, the method only stops due to the $\max(r^{-2}, r^2)$ losses coming from $\partial_G$). 
Define the bounding functions
\begin{subequations}  \label{def:KKcB}
\begin{align}
\KK(r,s,c) & := \mathbf{1}_{r_c > 1} + \mathbf{1}_{r_c \leq 1}\Bigg( \mathbf{1}_{s< r < r_c} + \mathbf{1}_{s<r_c <  r < 1} \frac{r_c^2}{r^2} + \mathbf{1}_{s<r_c < 1 < r} r_c^2+ \mathbf{1}_{r_c < s < r < 1}\frac{s^2}{r^2} + \mathbf{1}_{r_c < s < 1 < r} s^2 + \mathbf{1}_{1 < s < r}
\nonumber  \\ & \quad + \mathbf{1}_{r<s<r_c} + \mathbf{1}_{r < r_c< s < 1}\frac{r_c^2}{s^2} + \mathbf{1}_{r<r_c< 1 < s}r_c^2
 + \mathbf{1}_{r_c < r< s < 1}\frac{r^2}{s^2} + \mathbf{1}_{r_c < r<1 < s} r^2 + \mathbf{1}_{1 < r < s} \Bigg), \\ 
\cB(r,s) & := \left(\mathbf{1}_{s <r}\frac{s^{k-1/2}}{r^{k-1/2}} + \mathbf{1}_{r < s}\frac{r^{k+1/2}}{s^{k+1/2}}\right)\brak{s}^4 \\ 
\cL_{J,\ell}(r,s) & := k^J \max\left(\frac{1}{r^2},\frac{1}{s^2}, r^2,s^2\right)^\ell.  
\end{align}
\end{subequations}
The full properties and the estimates obtained on the kernels are laid out in Definitions \ref{def:SuitableType1} and \ref{def:SuitableType2} below. 
The main result in \S\ref{sec:IIO} is the following. 
\begin{proposition} \label{prop:BBoundIntro}
For $j \leq k-1$, each of the kernels $B^{(1)}_{X\ast;j,\ell}$ appearing in \eqref{eq:dGjY} and \eqref{eq:dGjX} is Suitable $(2\ell'',\ell''+\eta,\gamma)$ of type I for some $\gamma \in (0,1)$, $1 \gg \eta > 0$, and some integer $\ell'' \geq 0$ (difference for each kernel). 
For $j \leq k-1$, each of the kernels $B^{(2)}_{X\ast;j,\ell}$ and $B_{Y\ast;j,\ell}$ appearing in \eqref{eq:dGjY} and \eqref{eq:dGjX} is Suitable $(2\ell'',\ell''+\eta,\gamma)$ of type II for some $\gamma \in (0,1)$, $1 \gg \eta > 0$, and some integer $\ell'' \geq 0$ (difference for each kernel). 
In particular, each satisfies the uniform-in-$\eps$ boundedness: 
\begin{align*}
\abs{B^{(1)}_{\ast;j,\ell}(r,s,c)} &\lesssim_\eta \abs{u'(s)}  \KK(r,s,c) \cB(r,s) \cL_{2\ell'',\ell''_1 + \eta}(r,s) \\ 
\abs{B^{(2)}_{\ast;j,\ell}(r,s,c)} & \lesssim_\eta \abs{u'(s)}  \KK(r,s,c) \cB(r,s) \cL_{2\ell'',\ell''_2 + \eta}(r,s), 
\end{align*} 
and each of the kernels satisfies analogous log-Lipschitz regularity estimates and convergence estimates as $\eps \rightarrow 0$ (see Lemmas \ref{lem:Gsuitable}--\ref{lem:InductionIIO}). 
Finally, all terms in Proposition \ref{prop:repformXY}  satisfy the additional constraint
\begin{align} \label{ineq:wtconstr} 
\ell +  \ell'+ \ell'' & \leq j, \qquad \ell + \ell' + \ell''_1 + \ell''_2 \leq j.  
\end{align}
\end{proposition}
\begin{remark} 
The gains encoded by $\KK$ are what ultimately allows us to deduce the vorticity depletion effect and are inherited from precise estimates on $\cG$. 
The losses encoded by $\cL$ are inherited from the $\max(r^{-2},r^2)$ losses inherent in $\partial_G$ derivatives. It is crucial that the losses in $\cL$ do not depend directly on $r_c$. 
\end{remark}
\begin{remark} 
The constraint \eqref{ineq:wtconstr} arises due to the fact that each application of $\partial_G$ to \eqref{eq:RaydGjXY} can land either on $X$, $Y$, $F$, $F_\ast$, or on the fixed coefficients that depend only on the background vortex. Each application will lose $\max(r^{-2},r^2)$, and hence there are $\max(r^{-2},r^2)^j$ powers to distribute between different factors that the kernels and the weights account for.     
\end{remark} 

\paragraph{Extension to $\partial_G^k$ or $r_c \partial_{r_c} \partial_G^{j}$ with $j\leq k-1$}
It is not clear how to obtain estimates for even a single $r_c$ derivative, e.g.~$\partial_{r_c} \partial_G^j$: the commutator $[\partial_{r_c},\Ray_z]$ is too singular near $r_c$ to use an approach similar to the one we used on $\partial_G$.
Moreover, the arguments of Proposition \ref{prop:BBoundIntro} break down at $j=k$ due to the singularities in the right-hand side of \eqref{eq:RaydGjXY} at zero and infinity (encoded by the constraint \eqref{ineq:wtconstr}). 
In order to overcome this difficulty, first notice that while the commutator $[\partial_r, \Ray_z]$ is too singular to use an approach like what we used on $\partial_G$, we should nevertheless expect to be able to estimate $\partial_r \partial_G^j Y$ by elliptic regularity. 
Indeed, away from  the critical layer, it is a straightforward extension of our methods to show directly that $\partial_r$ derivatives of \eqref{eq:dGjY} and \eqref{eq:dGjX} should not be significantly worse than the $\partial_G^j X$ and $\partial_G^j Y$ themselves. 
The next observation is that, just as $\partial_G$ arises when taking $\partial_r$ derivatives in \eqref{eq:rdrf1f2}, similarly, $\partial_G$ arises when taking $\partial_{r_c}$ (or $\partial_G$) derivatives of \eqref{eq:dGjX} and \eqref{eq:dGjY} (see \S\ref{sec:Vd} for details).   
Therefore, while it seems intractable to build a reasonable iteration scheme for taking multiple $\partial_r$ and $\partial_{r_c}$, it turns out we can take a \emph{single} $\partial_r$, $\partial_{r_c}$ away from the critical layer, or a single additional $\partial_G$ derivative near the critical layer, of \eqref{eq:dGjY} and \eqref{eq:dGjX}. 
For $f_2$, we only need $\partial_{r_c}$ away from $r \sim r_c$ and hence this will be sufficient. 
For $f_1$, away from $r\sim r_c$ we write 
\begin{align*}
ru'(r)\partial_G = r\partial_r + \frac{ru'(r)}{r_c u'(r_c)} r_c \partial_{r_c}, 
\end{align*}
and estimate these two derivatives separately (whereas for $r\sim r_c$ we naturally leave the derivative as is). 
See \S\ref{sec:Vd} for details. 

\paragraph{Convergence and boundedness of $(r\partial_r)^j f^\eps_1$ and $(r\partial_r)^j f^\eps_2$}
Next, our goal is to pass to the limit $\eps \rightarrow 0$ and obtain $L^2$ bounds on $(r\partial_r)^j f_1$ and  $(r\partial_r)^j f_2$.
The first proposition gives convergence in the weaker weighted space $L_{f,\delta}^2$ and boundedness of the limit in the weighted space which is $r^{-2}$ stronger at the origin.   
\begin{proposition}\label{prop:bdf1}
For all $j \leq k$, we have the convergence of $(r\partial_r)^j f_1^\eps(t,r)$ to a limit $(r\partial_r)^j f_1(t,r)$ in the norm: 
\begin{align*}
\lim_{\eps \rightarrow 0} \norm{(r\partial_r)^j\left(f_1^\eps - f_1\right) }_{L^2_{f,\delta}} = 0. 
\end{align*}
Moreover, there holds the uniform-in-$t$ bounds for $n \leq k$ in the stronger weighted space: for all $\eta > 0$, 
\begin{align*}
\norm{\sqrt{r} (r\partial_r)^n f_1}_{L^2_{F,\delta}} \lesssim_{n,\delta,\eta} k^{2n+1} \abs{\omega_{k,0}^{in}} + \sum_{j = 0}^n k^{2(n-j)+\eta} \norm{(r\partial_r)^j F}_{L^2_{F,\delta/4}}. 
\end{align*}
\end{proposition}

Similarly, we prove the requisite decay $O(t^{-1})$ of $f_2$ in the natural $L^2_{f,\delta}$ space. 
\begin{proposition}\label{prop:f2dec} 
For all $j \leq k-1$, we have the convergence of $(r\partial_r)^j f_2^\eps(t,r)$ to a limit $(r\partial_r)^j f_2(t,r)$ in $L^2_{f,\delta}$:
\begin{align*}
\lim_{\eps \rightarrow 0} \norm{(r\partial_r)^j\left(f_2^\eps - f_2\right) }_{L^2_{f,\delta}} = 0. 
\end{align*}
Furthermore, there holds the following decay estimate for $n \leq k-1$: for all $\eta > 0$, 
\begin{align*}
\norm{(r\partial_r)^n f_2}_{L^2_{f,\delta}}  \lesssim_\delta \frac{1}{\brak{kt}}\abs{k}^{2n+4 + \eta}\abs{\omega_{k,0}^{in}} + \frac{1}{\brak{k t}}\sum_{j = 0}^{n+1} k^{2(n-j) + 3 + \eta}\norm{(r\partial_r)^j F}_{L^2_{F,\delta/4}}. 
\end{align*}
\end{proposition}
Propositions \ref{prop:bdf1} and \ref{prop:f2dec} give the vorticity depletion characterization in \eqref{ineq:VortDep}. 
Combining \eqref{eq:rdrf1f2} and Proposition \ref{prop:repformXY}, the proof of Proposition \ref{prop:bdf1} and \ref{prop:f2dec} reduces to passing to the limit as $\eps \rightarrow 0$ in operators of the type  arising in Proposition \ref{prop:repformXY} 
(and obviously bounding the limiting operators) with the $B$'s satisfying a list of properties such as those alluded to in Proposition \ref{prop:BBoundIntro} (the actual list is much longer; see \S\ref{sec:IIO}). 
This involves a number of very technical decompositions over (in general) four variables $r,s_0,r_c,s$ adapted to the various asymptotic behaviors near the origin, infinity, and the critical layer. The various H\"older regularity properties of the kernels becomes important for passing to the limit in the iterated singular integral operators arising. 
The details are carried out in Appendix \ref{sec:BndConv}.

Finally, via the Biot-Savart law, integration by parts in $r$, and the Hilbert-Schmidt lemma, Propositions \ref{prop:bdf1} and \ref{prop:f2dec}, directly imply Theorem \ref{thm:Main} (see Appendix \ref{apx:VortDepl}).
\begin{lemma}[Vorticity depletion implies optimal inviscid damping] \label{lem:VortDepID}
The vorticity depletion estimates \eqref{ineq:VortDep} imply the inviscid damping estimates \eqref{ineq:ID}. 
\end{lemma} 
This completes the proof of Theorem \ref{thm:Main}. 

\section{Dynamics of the $k=1$ mode}
In this section we give the proof of Theorem~\ref{thm:Main} for the mode $k=1$. For this, we derive an equivalent formulation
of \eqref{eq:integrallimit2} as follows. From \eqref{eq:EUfourier1} we obtain that
\begin{align*}
\de_t \left(\e^{ik u(r) t}\omega_k\right)=\e^{ik u(r) t} ik\beta(r) \psi_k.
\end{align*}
An integration over time then yields
\begin{align*}
\e^{iku(r) t} \omega_k(t,r)= \omega^{in}_k(r)+ ik\beta(r) \int_0^t \e^{iku(r) \tau} \psi_k(\tau,r)\dd \tau.
\end{align*}
Now, by writing the evolution equation satisfied by $\psi_k$, we infer that
\begin{align*}
\psi_k(t,r)=\frac{\beta(r)}{2\pi i\sqrt{r}}\lim_{\eps\to 0^+}\int_{\RR} \e^{-ikct}\left[\Phi_k(r,c-i\eps)-\Phi_k(r,c+i\eps)\right]\dd c,
\end{align*}
and therefore
\begin{align}
\e^{iku(r) t} \omega_k(t,r)
&=\omega^{in}_k(r)+\frac{\beta(r) }{2\pi i\sqrt{r}}\lim_{\eps\to 0^+}\int_{\RR} \frac{1-\e^{ik(u(r)-c)t}}{u(r)-c}\left[\Phi_k(r,c+i\eps)-\Phi_k(r,c-i\eps)\right]\dd c.\label{eq:integrallimit3}
\end{align}
We remark that this is roughly the form used in \cite{WeiZhangZhao15} (though the contour integral is set up slightly differently). 

\subsection{An explicit representation of the vorticity profile}\label{sub:kis11}
The case $k=1$ is special, because the homogeneous Rayleigh problem has an explicit solution
\begin{align*}
\phi(r,z) = (u(r) - z) r^{3/2} \quad \mbox{solves} \quad \Ray_z \phi = 0.
\end{align*}
This fact may be verified by a direct computation, in view of \eqref{eq:magic:vortex}.
Additionally, this special solution has the property that
\begin{align*}
\lim_{r\to 0} \phi(r,z) = 0
\end{align*}
and thus it may be used directly in the construction of the Green's function for $\Ray_z$. This fact is yet another special property of $k=1$. Using the reduction of order technique (see e.g. \cite{Miller}), one obtains another independent solution to $\Ray_z = 0$
\begin{align*}
H_\infty(r,z) = \phi(r,z) \int_{r}^{\infty} \frac{\dd s}{\phi(s,z)^2},
\end{align*}
which vanishes as $r \to \infty$. One verifies that the Wronskian of these two solutions is
\begin{align*}
\phi \, \partial_r H_\infty - \partial_r \phi \, H_\infty = -1,
\end{align*}
and thus we may directly combine $\phi$ and $H_{\infty}$ to obtain the Green's function for the $k=1$ Rayleigh operator:
\begin{align}
G_{  1}(r,\rho,z) = 
\begin{cases}
- \phi(r,z) H_{\infty}(\rho,z),& r<\rho\, ,\\
- \phi(\rho,z) H_{\infty}(r,z),& r>\rho \, .
\end{cases}
\label{eq:Greens:k=1}
\end{align}
The upshot of \eqref{eq:Greens:k=1} is that the solution of the inhomogeneous Rayleigh problem for $k=1$, 
\begin{align*}
\Ray_z \Phi_{1} = \frac{\omega^{in}_{1}(r) \sqrt{r}}{u(r) - z}
\end{align*}
for $z\in \Complex$, is given by 
\begin{align}
\Phi_{  1}(r,z) = \int_0^\infty \cG_{  1}(r,\rho,z) \frac{\omega^{in}_{  1}(\rho) \sqrt{\rho}}{u(\rho) - z} \dd\rho \, .
\label{eq:Phi:k=1}
\end{align}
In order to derive a formula for the vorticity profile
\begin{align*}
f_{  1}(t,r) = \e^{  i u(r) t} \omega_{  1}(t,r)
\end{align*}
we appeal to the representation formula \eqref{eq:integrallimit3} which yields
\begin{align}
f_{  1}(t,r) 
=\omega^{in}_{  1}(r)  +  \frac{\beta(r) }{2\pi i\sqrt{r}}
\lim_{\eps\to0^+} \int_{\RR} \frac{1 - \e^{i(u(r)-c)t}}{u(r) - c} \left[\Phi_{  1}(r,c+i\eps) -\Phi_{  1}(r,c-i\eps)  \right]\dd c.
\label{eq:profike:k=1:a}
\end{align}
Further, setting $z= c  \pm i \eps$ in \eqref{eq:Phi:k=1} and using essentially that $\int_0^\infty \omega^{in}_{  1}(r) r^2 \dd r = 0$, we may pass to the $\eps \to 0^{+}$ limit in the contour integral of \eqref{eq:profike:k=1:a}, to obtain 
\begin{subequations}
\label{eq:profile:k=1:def}
\begin{align}
f_{  1}(t,r) &= f_{  1;1}(r) + f_{  1;2}(r,t) \, ,
\label{eq:profile:k=1:b}\\
f_{  1;1}(r) &= \omega^{in}_{  1}(r) + \frac{\beta(r)}{r^2 u'(r)} \int_0^r \omega^{in}_{  1}(\rho) \rho^2 \dd \rho \, ,
\label{eq:profile:k=1:c} \\
f_{  1;2}(r,t) &=  r \beta(r) \int_r^\infty \e^{  i(u(r)-u(\rho))t} f_{  1;1}(\rho) \frac{\dd \rho}{\rho u'(\rho)} \, .
\label{eq:profile:k=1:d}
\end{align}
\end{subequations}
where as usual we used the notation $r_c = u^{-1}(c)$ for $c \in {\rm Ran}(u) = (0,u(0)]$. The proof of the convergence as $\eps\to 0^+$ of the expression in \eqref{eq:profike:k=1:a} to the expression in \eqref{eq:profile:k=1:def} is rather tedious, but direct. We thus omit these details. Alternatively, one may directly verify (by plugging in) that the vorticity 
\begin{align}
\omega_{  1}(t,r) =\e^{- i u(r) t} f_{  1}(t,r)
\label{eq:k=1:f:omega}
\end{align} 
obeys 
\begin{align}\label{eq:linearized:Euler:k=1}
\partial_t \omega_{  1} +  i u(r) \omega_{  1} - i \beta(r) \psi_{  1} &= 0, \qquad- \Delta_{  1} \psi_{  1} = \omega_{  1} \,,
\end{align}
which is what we are after in the first place.

\subsection{Vorticity depletion and inviscid damping}\label{sub:kis12}
A few comments are in order concerning the decomposition \eqref{eq:profile:k=1:def}. Although $f_{  1;1}(r)$ is time independent, and thus it does not decay with $t$, it is unusually small near the center of the vortex, that is 
\begin{align*}
 f_{ 1;1}(r)
 =
 O(r^3) \quad \mbox{as} \quad r \to 0,
\end{align*}
instead of just an $O(r)$ behavior. To see this, using the notation in \eqref{def:omegak0F} for $k=1$, one expands
\begin{align}
\omega_{  1}^{in}(r) = \omega_{1,0}^{in} r + O(r^3)\, , \quad  \mbox{as} \quad r\to 0\, ,
\label{eq:k=1:Taylor:r=0}
\end{align} 
and uses the precise Taylor series for $u$ and $\beta$ near $r=0$ (note that $\beta(0) = -4u''(0)$ from \eqref{eq:magic:vortex}). 
Inserting this in \eqref{eq:profile:k=1:c} shows that the coefficient $\omega_{1,0}^{in}$ of $r$ cancels out, leading to the $O(r^3)$ behavior (recall \eqref{eq:smthExp}).  This is the vorticity depletion  due to the non-locality of the linear equation. 
Moreover, if $\omega_{ 1}^{in}$ is compactly supported away from $r=0$, the same holds for $f_{  1;1}$.  
On the other hand, for the time dependent contribution to $f_{  1}$ we have the asymptotics
\begin{align*}
 f_{  1;2}(t,r)=O\left(\frac{r}{\brak{t}}\right) \quad \mbox{as} \quad r \to 0,
\end{align*}
which vanishes only $O(r)$ as $r \rightarrow 0$, but instead decays in time. 

We now make this intuition rigorous. Using the notation of \eqref{def:omegak0F}, we rewrite
\begin{align}
\omega_1^{in}(r) = r^{-1/2} F_1(r)  + r \chi(r) \frac{\beta(r)}{\beta(0)}  \omega_{1,0}^{in}
\label{eq:w:in:1}
\end{align}
where $\omega_{1,0}^{in} = \lim_{r\to 0} r^{-1} \omega_1^{in}(r)$ (recall \eqref{eq:smthExp}), and by definition we have that 
\begin{align}
F_1(r) \sim 
\begin{cases}
r^{3/2+2}, &\mbox{for } r \leq 1,\\
r^{1/2} \omega_1^{in}(r), &\mbox{for } r>1.
\end{cases}
\label{eq:F=1:asymptotic}
\end{align}
The important observation is that \eqref{eq:F=1:asymptotic} is precisely consistent with the definition of the weight $w_{F,\delta}$ in \eqref{eq:weightF}, when $k=1$. Inserting \eqref{eq:w:in:1} in \eqref{eq:profile:k=1:c} and using \eqref{eq:magic:vortex}, it follows after a short computation that
\begin{align}
f_{1;1}(r) = r^{-1/2} F_1(r) + \frac{\beta(r)}{r^2 u'(r)} \int_0^r \rho^{3/2} F_1(\rho) \dd\rho + \omega_{1,0}^{in} \frac{\beta(r)}{\beta(0) r^2 u'(r)} \int_0^r \rho^3 u'(\rho) \chi'(\rho) \dd\rho.
\label{eq:profile:k=1:c:new}
\end{align}
We note that the last term on the right side of \eqref{eq:profile:k=1:c:new} vanishes identically for $r\leq 1/2$, by the definition of the smooth cut-off function $\chi$, and behaves as $\beta(r) r$ for $r\geq 1$.
It follows from \eqref{eq:F=1:asymptotic}--\eqref{eq:profile:k=1:c:new} and the definition \eqref{eq:weightF} that
\begin{align}
\norm{ \sqrt{r} f_{1;1}}_{L^2_{F,\delta}} \lesssim |\omega_{1,0}^{in}| + \norm{F_1}_{L^2_{F,\delta/2}} .
\end{align}
In a similar way, using that the operator $r \partial_r$ is scale invariant, we may apply derivatives to \eqref{eq:profile:k=1:c:new} and obtain that
\begin{align*}
\norm{ \sqrt{r} (r\partial_r) f_{1;1}}_{L^2_{F,\delta}} \lesssim |\omega_{1;0}^{in}| + \norm{F_1}_{L^2_{F,\delta/2}} + \norm{(r \partial_r) F_1}_{L^2_{F,\delta/2}} 
\end{align*}
and
\begin{align*}
\norm{ \sqrt{r} (r\partial_r)^2 f_{1;1}}_{L^2_{F,\delta}} \lesssim |\omega_{1;0}^{in}| + \norm{F_1}_{L^2_{F,\delta/2}} + \norm{(r \partial_r) F_1}_{L^2_{F,\delta/2}}  + \norm{(r \partial_r)^2 F_1}_{L^2_{F,\delta/2}}
\end{align*}
which completes the proof of the first half of \eqref{ineq:VortDep}.

We obtain similar estimates for $f_{1;2}$ by using \eqref{eq:profile:k=1:d},  the already established bounds for $f_{1;1}$, and integration by parts. For instance, for the weighted $L^2$ estimate on $f_{1;2}$ we have
\begin{align}
it  f_{1;2}(r) 
&= r\beta(r)  \int_{r}^{\infty} \frac{-1}{u'(\rho)} \partial_\rho \left( \e^{i(u(r)-u(\rho))t} \right) f_{1;1}(\rho) \frac{\dd \rho}{\rho u'(\rho)}  \notag\\
&=  \int_{0}^{\infty} \indicator{\rho > r} r\beta(r)  \e^{i(u(r)-u(\rho))t}  \frac{w_{F,\delta/2}(\rho)}{\rho^{5/2} (u'(\rho))^2} \frac{\sqrt{\rho}(\rho \partial_\rho)  f_{1;1}(\rho) }{w_{F,\delta/2}(\rho)}  \dd\rho \notag\\
&\quad - \int_{0}^{\infty} \indicator{\rho > r}  r\beta(r) \e^{i(u(r)-u(\rho))t}   \frac{w_{F,\delta/2}(\rho)(u'(\rho) + 2 \rho u''(\rho))}{\rho^{5/2} (u'(\rho))^3} \frac{\sqrt{\rho} f_{1;1}(\rho) }{w_{F,\delta/2}(\rho)} \dd \rho\notag\\
&\quad +  \sqrt{r}\beta(r)  \frac{w_{F,\delta/2}(r)}{r (u'(r))^2} \frac{\sqrt{r} f_{1;1}(r)}{w_{F,\delta/2}(r)} \,.
\label{eq:profile:k=1:d:new}
\end{align}
Upon dividing by $w_{f,\delta}(r)$ and using that 
\begin{align*}
\norm{\frac{\sqrt{r}\beta(r)}{w_{f,\delta}(r)}   \frac{w_{F,\delta/2}(r)}{r (u'(r))^2}}_{L^\infty_r} 
&\lesssim 1, \\
\norm{\indicator{\rho > r} \frac{r\beta(r)}{ \, w_{f,\delta}(r)}    \frac{w_{F,\delta/2}(\rho)}{\rho^{5/2} (u'(\rho))^2} }_{L^2_r L^2_\rho}
&\lesssim 1, \\
\norm{\indicator{\rho > r} \frac{r\beta(r)}{ \, w_{f,\delta}(r)}   \frac{w_{F,\delta/2}(\rho)(u'(\rho) + 2 \rho u''(\rho))}{\rho^{5/2} (u'(\rho))^3}}_{L^2_r L^2_\rho} &\lesssim 1,
\end{align*}
and, in view of the already established bounds on $f_{1;1}$, we conclude that  
\begin{align*}
t \norm{f_{1;2}}_{L^2_{f,\delta}} 
&\lesssim \norm{\sqrt{r} f_{1;1}}_{L^2_{F,\delta/2}} + \norm{\sqrt{r} (r\partial_r) f_{1;1}}_{L^2_{F,\delta/2}}\lesssim |\omega_{1;0}^{in}| + \norm{F_1}_{L^2_{F,\delta/2}} + \norm{(r \partial_r) F_1}_{L^2_{F,\delta/2}}. 
\end{align*}
 In order to bound the weighted $r\partial_r$ norm of $f_{1;2}$, we apply an $r \partial_r$ derivative (which preserves scale) to \eqref{eq:profile:k=1:d:new} and then divide by $w_{f,\delta}$. Similar $L^2_r L^2_\rho$ bounds hold for the kernel which arises, and the norm of $\sqrt{r} (r\partial_r)^2 f_{1;1}$ enters the calculation (but we have it already bounded in the suitable norm). We omit these computational details. This concludes the proof of the second half of \eqref{ineq:VortDep}.

The proof of inviscid damping for the stream function now directly follows from Lemma~\ref{lem:VortDepID} and the bounds established on $f_{1;1}$ and $f_{1;2}$. We note that for $k=1$ the stream function has a particularly nice formula
\begin{align}
\psi_{  1}(t,r) 
= - \frac{\partial_t   + i u(r)}{i \beta(r)} \left( \e^{- i u(r) t} f_{  1}(t,r) \right)
&= - \frac{\e^{- i u(r) t}}{i \beta(r)} \partial_t f_{  1;2}(t,r) 
\notag\\
&= -  r \int_{r}^{\infty} (u(r) - u(r_c)) \e^{- it u(r_c)}  \frac{f_{  1;1}(r_c)}{r_c u'(r_c)} \dd r_c \,,
\label{eq:stream:k=1}
\end{align}
which could also have been used to obtain inviscid damping, by integrating twice by parts in $r_c$, as was done in \cite{WeiZhangZhao2017}. 

\section{The homogeneous Rayleigh problem for \texorpdfstring{$k\geq2$}{TEXT}} \label{sec:HomRay}
We consider here the homogeneous version of equation \eqref{eq:inoRay2}, namely
\begin{align}
\partial_{rr} \phi + \left( \frac{1/4 - k^2}{r^2} + \frac{\beta(r)}{u(r) - z} \right) \phi = 0, \label{eq:Rayleigh}
\end{align}
for $k \geq 2$. 
In this section we construct a specific solution for $z \in \Complex$ in a neighborhood of the spectrum and study various properties.
Most of the estimates exploit the following weight, defined as
\begin{align}\label{eq:weightHOMORAY}
&w_\phi(r,r_c)=\begin{cases}
  (r/r_c)^{-k+1/2}, \quad & r < 1, \\
  (r/r_c)^{k+1/2}, \quad & r  > 2,
\end{cases}
\qquad
w_{\phi,\gamma}(r,r_c)=\begin{cases}
  (r/r_c)^{-k+1/2-\gamma}, \quad & r < 1, \\
  (r/r_c)^{k+1/2+\gamma}, \quad & r  > 2,
\end{cases}
\end{align} 
using also the notation $L_{\phi,\gamma}^2$ as in \eqref{eq:weightpsi}.
In what follows, we will often use the the following smooth cutoffs: 
\begin{subequations} \label{def:chic} 
\begin{align}
\chi_{c}(r,c) = \chi\left(\frac{k(r-r_c)}{r_c}\right)  
\chi_{\neq }(r,c) = 1-\chi_c.  
\end{align}
\end{subequations} 
Note that $\abs{\partial_G \chi_c} \lesssim \max(r_c^{-2},r_c^2)\abs{\chi'_c(\frac{k(r-r_c)}{r_c})}$ and $\chi'_c$ is supported away from the critical layer; analogous observations hold also for $r\partial_r \chi_c$ and $r_c \partial_{r_c} \chi_c$).

\begin{theorem}[Homogeneous solutions for $k \geq 2$] \label{thm:phik}
Assume $k \geq 2$, let $\eps_0\in(0,1/2)$ and $\varkappa\in (0,1)$, and define
\begin{align}\label{eq:domeps}
D_{\eps_0} = (0,\infty) \times \left((0,u(0)) + i(-\eps_0,\eps_0)\right) \subset (0,\infty) \times \Complex.
\end{align}
There exists a unique solution $\phi:D_{\eps_0}\to \Complex$ to \eqref{eq:Rayleigh} of the form
\begin{align}\label{eq:Pform}
\phi(r,z)=P(r,z) (u(r)-z)
\end{align}
with $P:D_{\eps_0}\to \Complex$ such that $P(r_c,z)=1$ and $\de_r P(r_c,z)=0$ for every $z\in \left((0,u(0)) + i(-\eps_0,\eps_0)\right)$. 
Moreover, $\phi$ is continuously differentiable with respect to $r$ and $r_c$.
\end{theorem}

Let $Q_0,Q_\infty$ be defined
via the identities
\begin{align}\label{eq:Q:def}
\phi(r,z)= \left(\frac{r}{r_c}\right)^{1/2-k} Q_0(r,z)(u(r)-z), \qquad \forall r\in (0,r_c],
\end{align}
and
\begin{align}\label{eq:tQ:def}
\phi(r,z)=\left(\frac{r}{r_c}\right)^{1/2+k}Q_\infty(r,z)(u(r)-z), \qquad \forall r\in (r_c,\infty).
\end{align}
Moreover, the functions
\begin{align}\label{eq:B0stima}
B_0(r,z) =  (k+1) Q_0(r,z) - r \partial_r Q_0(r,z), \qquad r\leq r_c,
\end{align}
and
\begin{align}\label{eq:Binfistima}
B_\infty(r,z) =  (k-1) Q_\infty(r,z) + r \partial_r Q_\infty(r,z), \qquad r\geq r_c,
\end{align}
will play an important role.
When $z$ is real and belongs to the interval $(0,u(0))$, more can be deduced.

\begin{theorem}[Further properties for the real solution] \label{thm:realphik}
Assume that $z=c\in (0,u(0))$, and let $\phi(r,c)$ be the unique solution to the problem posed in \eqref{eq:Pform}. Then
\begin{align}\label{eq:Q0:properties}
Q_0(r,c)>0,\qquad \partial_r Q_0(r,c) >0 , \qquad \forall r\in (0,r_c],
\end{align}
and
\begin{align}\label{eq:tQ0:properties}
Q_\infty(r,c)>0,\qquad \partial_r Q_\infty(r,c) <0, \qquad \forall r\in (r_c,\infty).
\end{align}
Moreover, 
\begin{align}\label{eq:Q0}
Q_0(r,c), Q_\infty(r,c) \approx 1, \qquad \|r\de_rQ_0(\cdot,c)\|_{L^\infty(0,r_c)},\|r\de_rQ_\infty(\cdot,c)\|_{L^\infty(r_c,\infty)}\lesssim k.
\end{align}
Finally, there is a constant $\delta_0$ depending only on $u$ such that for all $r_c \in (0,\infty)$, there holds for all $\abs{r-r_c} < \delta_0 r_c$ (uniformly in $k$ and $r_c$), 
\begin{align}\label{eq:B0up}
k^2 \frac{r^{2k-1}}{r_c^{2k}}|r_c-r|\lesssim B_0(r,c) \leq \frac{k^2-1}{r}|r_c-r|, \qquad \forall r\in (0,r_c],
\end{align}
and
\begin{align}\label{eq:Binfiup}
 k^2\frac{r^{2k-1}}{r_c^{2k}}|r_c-r|\lesssim B_\infty(r,c)\leq \frac{k^2-1}{r_c}|r-r_c|, \qquad \forall r\in [r_c,\infty).
\end{align}
\end{theorem}
As a consequence, the theorem proves that, when the spectral parameter $c\in (0,u(0))$, the only point at which $P$ vanishes
is the critical layer, while otherwise $P$ is bounded away from 0 uniformly, thanks to \eqref{eq:Q0}. Combining
this with suitable convergence estimates deduced below, we deduce non-vanishing properties
of the complex solution as well. However, due to extra factors of $r_c^{2+\varkappa}$ and $r_c^{-2-\varkappa}$ (see
e.g. Lemmas \ref{lem:CONV1}-\ref{lem:CONV6}), this information will only be available in the domain $I_\alpha$ given
in \eqref{eq:Ialpha}.

\begin{theorem}\label{thm:nonvan}
On the domain $(r,c)\in(0,\infty)\times I_\alpha$, we have
\begin{align}
&\norm{(r\de_r)^jP}_{L^\infty_{\phi}}\lesssim_\alpha k^{j}, \qquad j=0,1,\label{eq:Pbestbound}\\
&\norm{(u(r)-z)r^2\de_{rr}P}_{L^\infty_{\phi,\alpha/2}} \lesssim_\alpha k^3,\label{eq:rrPbestbound}\\
&\norm{ r_c\de_{r_c}P}_{L^\infty_\phi} \lesssim  k^3,\label{eq:rcPbestbound}\\
&\norm{(r\de_r)^j(P(r,c \pm i \eps) - P(r,c))}_{L^\infty_\phi} \lesssim_\alpha \eps^{\eta_\alpha}, \qquad j=0,1,\label{eq:Pbestconv}\\
&\norm{r_c\de_{r_c}(P(r,c \pm i \eps) - P(r,c))}_{L^\infty_\phi} \lesssim_\alpha \eps^{\eta_\alpha},\label{eq:r_cPbestconv}\\
&\norm{r\de_rr_c\de_{r_c}(P(r,c \pm i \eps) - P(r,c))\chi_{\neq}}_{L^\infty_\phi} \lesssim_\alpha \eps^{\eta_\alpha},\label{eq:rr_cPbestconv}
\end{align}
where  $\chi_{\neq}$ is defined in \eqref{def:chic}, and we have the (uniform in $\eps$) Lipschitz bound
\begin{align}\label{eq:PLipdr}
\frac{|r\de_rP(r,c\pm i\eps)|}{w_{\phi,\alpha/2}(r,r_c)}\lesssim_\alpha k^3 \frac{|r-r_c|}{r_c},
\end{align}
and convergence estimate
\begin{align}\label{eq:Pbestconv2}
\| (u(r)-c)r^2\de_{rr}(P (r,c\pm i\eps)-P (r,c))\|_{L^\infty_{\phi,\alpha/2}}
&\lesssim_{\alpha}\eps^\frac{2}{2+\alpha} .
\end{align}
Finally, for every $r\in (0,\infty)$ and $r_c\in I_\alpha$ such that
$$
|r-r_c|> \frac{r_c}{k},
$$
there holds (uniformly in $\eps$)
\begin{align}\label{eq:estdrdrcP}
\frac{|r\de_rr_c\de_{r_c}P (r,c\pm i\eps)|}{w_{\phi,\alpha/2}(r,r_c)}\lesssim_\alpha k^4,
\end{align}
while if
$$
|r-r_c|\leq \frac{r_c}{k},
$$
we have the pointwise bounds (uniformly in $\eps$)
\begin{align}\label{eq:estdGP}
\min\{r_c^2,r_c^{-2}\}\frac{|\de_GP(r,c\pm i\eps)|}{w_{\phi,\alpha/2}(r,r_c)}\lesssim_\alpha k^2, \qquad \min\{r_c^2,r_c^{-2}\}\frac{ |r\de_r\de_G P(r,c\pm i\eps)|}{w_{\phi,\alpha/2}(r,r_c)}\lesssim_\alpha k^4\frac{|r-r_c|}{r_c},
\end{align}
and the convergence estimates
\begin{align}\label{eq:convdGP}
\min\{r_c^{2},r_c^{-2}\}\frac{|(r\de_r)^j\de_G(P (r,c\pm i\eps)-P (r,c))|}{w_{\phi,\alpha/2}(r,r_c)}&\lesssim_\alpha \eps^{\eta_\alpha}, \qquad j=0,1.
\end{align}
\end{theorem}

\begin{remark}
The estimates on $P$ of the above Theorem \ref{thm:nonvan} can also be written in terms on $Q_0,Q_\infty$. Of particular importance
are
\begin{align}
&\norm{(r\de_r)^j\left(Q_\bullet(r,c\pm i\eps) - Q_\bullet(r,c)\right)}_{L^\infty} \lesssim_{\alpha}  \eps^{\eta_\alpha}\label{eq:Qconvest1},\\
&\norm{r_c\de_{r_c}\left(Q_\bullet(r,c\pm i\eps) - Q_\bullet(r,c)\right)}_{L^\infty} \lesssim_{\alpha}  \eps^{\eta_\alpha} \label{eq:Qconvest3},\\
&\norm{r\de_rr_c\de_{r_c}\left(Q_\bullet(r,c\pm i\eps) - Q_\bullet(r,c)\right)\chi_{\neq}}_{L^\infty} \lesssim_{\alpha}  \eps^{\eta_\alpha} \label{eq:Qconvest5},
\end{align}
where $\eta_\alpha=\frac{\alpha}{2(2+\alpha)}$ and  $\bullet=0,\infty$. From \eqref{eq:Qconvest1}
combined with \eqref{eq:Q0}, we infer that
\begin{align}\label{eq:Qcomplexbounds}
|Q_\bullet(r,z)|\approx_\alpha 1, \qquad \|r\de_rQ_\bullet(\cdot,z)\|_{L^\infty}\lesssim_\alpha k.
\end{align}
The first estimate is crucially stating that $\phi$ vanishes only at the critical layer, when $r=r_c$.
\end{remark}

The proof of Theorem \ref{thm:nonvan} combines Lemmas \ref{lem:CONV1}-\ref{lem:CONV6} and Remark \ref{rem:sharprc} below
for $\varkappa=\alpha/2$ with the definition of $I_\alpha$.
Once \eqref{eq:Qconvest1} is established, the bounds \eqref{eq:Pbestbound} and \eqref{eq:Qcomplexbounds}
follow from \eqref{eq:Q0} and \eqref{eq:Q:def}-\eqref{eq:tQ:def}, while \eqref{eq:Pbestconv2} is precisely \eqref{eq:FPconvest2}.
The rest are stated in an equivalent way in Propositions \ref{prop:tildePfirst} and \ref{prop:deGder}.

\subsection{Existence and uniqueness of solutions}
Fix $z\in \Complex$ be such that $c=\Re z\in (0,u(0))$. We will denote by $\phi_1=\phi_1(r,z)$ the homogeneous solution for the $k=1$ 
Rayleigh problem devised in \eqref{eq:keq1phi} (see \S\ref{sub:kis11}), namely
\begin{align}\label{eq:solutionk1a}
\phi_1(r,z) = \left(\frac{r}{r_c}\right)^{3/2} (u(r) - z),
\end{align}
appropriately normalized at the critical layer.
We will look for a solution $\phi$ to \eqref{eq:Rayleigh} of the form 
\begin{align}\label{eq:bohboh}
\phi(r,z) = \phi_1(r,z)\tP(r,z),
\end{align}
and set up a contraction mapping argument for $\tP$, which satisfies 
\begin{align}\label{eq:Peqn}
\de_{r} \left(\phi_1^2 \de_r \tP\right) + \frac{1 - k^2}{r^2}  \phi_1^2\tP = 0,
\end{align}
subject to  the boundary conditions
\begin{align}\label{eq:PeqnBC}
&\tP(r_c,z) = 1,\qquad \partial_r \tP(r_c,z) =0.
\end{align}
Integrating \eqref{eq:Peqn} twice and using \eqref{eq:PeqnBC}, we infer that
\begin{align}
\de_r \tP(r,z)
= -\frac{k^2-1}{r^3(u(r)-z)^2} \int_r^{r_c} s(u(s)-z)^2\tP(s,z)\dd s\label{eq:drP}
\end{align}
and
\begin{align}
\tP(r,z)& 
 =1+(k^2-1) \int_r^{r_c} \frac{1}{\rho^3(u(\rho)-z)^2} \int_\rho^{r_c} s (u(s)-z)^2\tP(s,z)\dd s\, \dd\rho
 =: 1 + \mathcal{T}_z[\tP]. \label{eq:integralP}
\end{align} 
The above expression will be useful to set up a proper fixed point scheme to deduce existence and uniqueness of $\tP$.
For further reference, we can take a $\de_{r_c}$ of the above expression, taking into account that $z=u(r_c)\pm i\eps$, obtaining
\begin{align}\label{eq:computerc}
\de_{r_c}\tP (r,z) &= \mathcal{T}_z[\de_{r_c} \tP]+2(k^2-1)\int_r^{r_c} \frac{u'(r_c)}{\rho^3 (u(\rho)-z)^3}\int_\rho^{r_c} s (u(s)-z)^2\tP(s,z)\dd s\, \dd\rho\notag\\
&\quad- 2(k^2-1)  \int_r^{r_c} \frac{u'(r_c)}{\rho^3(u(\rho)-z)^2} \int_\rho^{r_c} s (u(s)-z)\tP(s,z)\dd s\, \dd\rho
-\eps^2(k^2-1)r_c \int_r^{r_c}\frac{1}{\rho^3 (u(\rho)-z)^2}\dd \rho\notag\\
&= \mathcal{T}_z[\de_{r_c} \tP]+\mathcal{I}_1+\mathcal{I}_2+\mathcal{I}_3.
\end{align}
Notice that the above expression is valid for $\eps>0$. 
Using integration by parts, we derive the equivalent formula
\begin{align}
\de_{r_c}\tP (r,z) = \mathcal{T}_z[\de_{r_c} \tP]+(k^2-1)u'(r_c) &\Bigg[  \frac{1}{r^3u'(r)(u(r)-z)^2}\int_r^{r_c} \rho (u(\rho)-z)^2\tP(\rho,z)\dd \rho\notag\ \\
&+\int_r^{r_c} \de_\rho \left(\frac{1}{\rho^3u'(\rho)}\right)\frac{1}{(u(\rho)-z)^2}\int_\rho^{r_c} s (u(s)-z)^2\tP(s,z)\dd s\, \dd\rho\notag\\
&  +\int_r^{r_c} \frac{1}{\rho^3(u(\rho)-z)^2} \int_\rho^{r_c} \de_s\left(\frac{s}{u'(s)}\right) (u(s)-z)^2\tP(s,z)\dd s\, \dd\rho\notag\\\
&  +\int_r^{r_c} \frac{1}{\rho^3(u(\rho)-z)^2} \int_\rho^{r_c} \frac{s}{u'(s)} (u(s)-z)^2\de_s\tP(s,z)\dd s\, \dd\rho\Bigg]\label{eq:computerc2}.
\end{align}
This is the formulation that we use also at $\eps=0$ as no singular integral appears here (see Section \ref{sub:convderc} below).

\subsubsection{An auxiliary weight}
 With the convention adopted in \eqref{eq:rcdef}, we define an auxiliary weight 
$\tw=\tw(r,r_c)$ to solve the ODE
\begin{align}\label{def:Pw}
-\de_{rr}\tw - \frac{3}{r}\de_r\tw +\frac{1}{A^2-1}\frac{1}{r^2}\tw =0,\qquad \tw(r_c,r_c)  = 1, \qquad \de_r\tw(r_c,r_c)  = 0, 
\end{align}
where $A>k$ is a parameter that we leave unspecified at the moment (one should think of $A$ as close to $k$). 
In fact, we can solve the above ODE explicitly, to find
\begin{align}
\tw(r,r_c) = \frac{A+1}{2A} \left(\frac{r}{r_c}\right)^{A-1} + \frac{A-1}{2A}\left(\frac{r}{r_c}\right)^{-A-1}. \label{eq:twexp}
\end{align}
Notice that, given a fixed $r_c>0$, $\tw$ attains its minimum at $r=r_c$ and
\begin{align}\label{eq:twlowerbdd}
\tw(r,r_c)\geq 1, \qquad \forall r>0.
\end{align}
The following properties of $\tw$ will prove useful later.
\begin{lemma}\label{lem:wtilde}
Let $c\in (0,u(0))$, $A>k$ and $\tw$ be given by \eqref{eq:twexp}. Then
\begin{align}\label{eq:exactw}
\int_r^{r'} \frac{1}{\rho^3} \int_\rho^{r_c} s\tw(s,r_c) \dd s\, \dd\rho=\frac{1}{A^2-1}(\tw(r,r_c)-\tw(r',r_c)), \qquad \forall r,r'>0.
\end{align}
Furthermore, for any $b\neq \pm A$ and for $A=k+\varkappa$ with $\varkappa\in (0,1)$, we have
\begin{align}\label{eq:weightest2}
\left|\int_r^{r_c} s^b \tw(s,r_c)\dd s\right|\lesssim_b \frac{r^{1+b}}{k} w(r,r_c), \qquad \forall r>0.
\end{align}
Moreover
\begin{align}\label{eq:weightest3}
\left|\int_r^{r_c} s \tw(s,r_c)\dd s\right|\lesssim \frac{r^{2}}{r_c} |r-r_c| \tw(r,r_c), \qquad \forall r>0,
\end{align}
and
\begin{align}\label{eq:weightest4}
\left|\int_r^{r_c} s^b \tw(s,r_c)\dd s\right|\lesssim_b \max\{r,r_c\}^{b}  |r-r_c| \tw(r,r_c), \qquad \forall r>0.
\end{align}
\end{lemma}

\begin{proof}[Proof of Lemma~\ref{lem:wtilde}]
Equation \eqref{eq:exactw} is obtained by explicitly computing the integrals from formula \eqref{eq:twexp}.  Turning to
 \eqref{eq:weightest2}-\eqref{eq:weightest3}, we first observe that
 \begin{align}\label{eq:compweight1}
\int_r^{r_c} s^b \tw(s,r_c)\dd s
= \frac{A+1}{2A(A+b)}r^{b+1} \left[\left(\frac{r}{r_c}\right)^{-b-1}-\left(\frac{r}{r_c}\right)^{A-1}\right] 
- \frac{A-1}{2A(A-b)} r^{b+1} \left[\left(\frac{r}{r_c}\right)^{-b-1}-\left(\frac{r}{r_c}\right)^{-A-1}\right].
\end{align}
By considering the different cases $r\leq r_c$ and $r>r_c$ and the different ranges of $b\neq \pm A$ with respect to $A$, \eqref{eq:weightest2} follows immediately from the 
definition \eqref{eq:twexp}. Also, \eqref{eq:weightest4} is obvious from \eqref{eq:twlowerbdd}. Regarding \eqref{eq:weightest3},
we use
the fact that if $r\leq r_c$, then applying the mean value theorem 
to the function $(0,1]\ni x\mapsto x^{2A}$ and using \eqref{eq:twexp}, we have
\begin{align}
\left|\int_r^{r_c} s\tw(s,r_c)\dd s\right|&=-\frac{r^2}{2A}\left(\frac{r}{r_c}\right)^{-A-1}\left[\left(\frac{r}{r_c}\right)^{2A}-1\right] \leq \frac{r^2}{2A}\left(\frac{r}{r_c}\right)^{-A-1} 2A  \left|\frac{r}{r_c}-1\right|\notag\\ 
&\leq \frac{2A}{A-1} \frac{r^2}{r_c}  |r-r_c| \tw(r,r_c)\label{eq:compuweight1}.
\end{align}
A similar computation, applied to the function $[1,\infty)\ni x\mapsto x^{-2A}$, also show that if $r> r_c$, then
\begin{align}
\left|\int_r^{r_c} s\tw(s,r_c)\dd s\right|\leq \frac{2A}{A+1} \frac{r^2}{r_c}  |r-r_c| \tw(r,r_c).\label{eq:compuweight2}
\end{align}
The two estimates can be grouped together as in \eqref{eq:weightest2}, concluding the proof.
\end{proof}

\subsubsection{Existence and uniqueness of \texorpdfstring{$\tP$}{TEXT}}
We begin with proving existence, uniqueness and some regularity for \eqref{eq:integralP}. 
It is clear that existence and uniqueness of $\tP$ is equivalent to existence and uniqueness of  $P$ in Theorem \ref{thm:phik}. Morever,
all the properties on $\tP$ translate into properties of $P$, since
\begin{align}\label{eq:relPtildeP}
P(r,z)= \left(\frac{r}{r_c}\right)^{3/2} \tP(r,z).
\end{align}

\begin{proposition}\label{prop:tildePfirst}
Let $\eps_0\in (0,1/2)$ and $D_{\eps_0}$ as in Theorem \ref{thm:phik}, and suppose $\tw$ is given 
by \eqref{eq:twexp} with $A=k+\varkappa$.
Then \eqref{eq:integralP} has a unique
solution $\tP\in L^\infty_{\tw}(D_{\eps_0})$ such that
\begin{align}\label{eq:Linftyesttp}
\|(r\de_r)^j\tP\|_{L^\infty_{\tw}}:=\|\tw^{-1}(r\de_r)^j\tP\|_{L^\infty}\lesssim_\varkappa k^{1+j}, \qquad j=0,1,
\end{align}
and
\begin{align}\label{eq:Winftyesttp2}
\|(u(r)-z)r^2\de_{rr}\tP\|_{L^\infty_{\tw}}  \lesssim_\varkappa k^3.
\end{align}
Moreover
\begin{align}\label{eq:Lipdr}
\frac{|r\de_r\tP(r,z)|}{\tw(r,r_c)}\lesssim_\varkappa k^3 \frac{|r-r_c|}{r_c},
\end{align}
for any $r,r_c>0$.
\end{proposition}

Thus, thanks to \eqref{eq:relPtildeP}, Proposition \ref{prop:tildePfirst} immediately proves the bound
\begin{align}
&\|(r\de_r)^jP\|_{L^\infty_{\phi,\varkappa}}\lesssim_\varkappa k^{1+j}, \qquad j=0,1.\label{eq:FPbound}
\end{align}
Proposition \ref{prop:tildePfirst} is based on the contraction properties of the operator $\mathcal{T}_z$ in \eqref{eq:integralP}.
This approach can be viewed as a refinement of an analogous argument in \cite{WeiZhangZhao15} to the more complicated vortex case. 
\begin{lemma}\label{lem:fasd}
If  $Z\in L^\infty (D_{\eps_0})$, there holds
\begin{align}\label{eq:contractionest}
\|\mathcal{T}_z[\tw Z]\|_{L^\infty_{\tw}} \leq \frac{k^2 - 1}{A^2-1}\|Z\|_{L^\infty},
\end{align}
where $\tw$ is given by \eqref{eq:twexp}. 
\end{lemma} 

\begin{proof}[Proof of Lemma~\ref{lem:fasd}]
By linearity, we can assume that $\| Z\|_{L^\infty}=1$.
In view of \eqref{eq:integralP}, \eqref{eq:solutionk1a}, and the monotonicity of $u$ (hence $\abs{u(s) - z} \leq \abs{u(\rho)-z}$ in the integral), we have the immediate bound   
\begin{align*}
\abs{\tw^{-1}\mathcal{T}_z[\tw Z]} \leq \frac{k^2 - 1}{\tw(r,r_c)}\left|\int_r^{r_c} \frac{1}{\rho^3} \int_\rho^{r_c} s\tw(s,r_c) \dd s\, \dd\rho\right|. 
\end{align*}
From \eqref{eq:exactw}, we readily obtain 
\begin{align}\label{eq:also1}
\abs{\tw^{-1}\mathcal{T}_z[\tw Z]} \leq \frac{k^2 - 1}{A^2-1}\frac{\tw(r,r_c)-1}{\tw(r,r_c)},
\end{align}
which implies \eqref{eq:contractionest}. 
\end{proof}

We can now proceed with the proof of the main result of this section.
\begin{proof}[Proof of Proposition~\ref{prop:tildePfirst}]
The existence and uniqueness of $\tP$ follows from the contraction mapping principle. Indeed, from \eqref{eq:contractionest} 
and the fact that $\tw$ is bounded below (see \eqref{eq:twlowerbdd}), the operator $1+\mathcal{T}_z[\cdot]: L^\infty_{\tw}\to L^\infty_{\tw}$ 
is a contraction whenever $A^2>k^2$. Moreover, from \eqref{eq:integralP} and \eqref{eq:contractionest}, the unique solution $\tP$ satisfies
\begin{align}
|\tw^{-1}\tP(r,z)| \leq \frac{1}{\tw(r,r_c)} + \abs{\tw^{-1}\mathcal{T}_z[\tw \tw^{-1}\tP]} \leq1+  \frac{k^2 - 1}{A^2-1}\| \tP\|_{L^\infty_{\tw}},\end{align}
and \eqref{eq:Linftyesttp} with $j=0$ follows. We now turn to the second part of \eqref{eq:Linftyesttp}. In light of \eqref{eq:drP}, we have that
\begin{align}\label{eq:interm1}
|\tw^{-1}r\de_r\tP(r,z)|\leq \frac{k^2-1}{r^2\tw(r,r_c)} \| \tP\|_{L^\infty_{\tw}}\left|\int_r^{r_c} s \tw(s,r_c)\dd s\right|.
\end{align}
Thus, taking advantage of \eqref{eq:weightest2}, we deduce that
\begin{align}\label{eq:interm2}
|\tw^{-1}r\de_r\tP(r,z)|\lesssim \frac{k^2-1}{k} \|  \tP\|_{L^\infty_{\tw}},
\end{align}
and the claim follows by combining the above estimate with \eqref{eq:Linftyesttp} with $j=0$. 
We now prove \eqref{eq:Winftyesttp2}. Firstly, notice that we can read $\de_{rr}\tP$ directly from \eqref{eq:Peqn}, obtaining
\begin{align}\label{eq:compsecond2}
\de_{rr}\tP(r,z)=-\frac{3}{r} \de_r\tP   -    \frac{2u'(r)}{u(r)-z}  \de_r\tP      -(k^2-1)\frac{\tP(r,z)}{r^2}.
\end{align}
Hence,
\begin{align*}
\|(u(r)-z)r^2\de_{rr}\tP(r,z)\|_{L^\infty_{\tw}}\lesssim\|(u(r)-z)r\de_r\tP\|_{L^\infty_{\tw}}+(k^2-1)\|\tP\|_{L^\infty_{\tw}},
\end{align*}
and \eqref{eq:Winftyesttp2} follows from \eqref{eq:Linftyesttp}.
Finally, \eqref{eq:Lipdr} is a consequence of
\eqref{eq:weightest3} and \eqref{eq:interm1}.
The proof is concluded.
\end{proof}

\subsection{Convergence of \texorpdfstring{$(r\de_r)^j\tP$}{TEXT}}\label{sub:convlame}
As a first step towards the proof of the convergence estimates in Theorem \ref{thm:nonvan}, we deduce a convergence estimate on $\tP$. 
These estimates are relevant near the critical layer, whereas near $r=0,\infty$ improve the estimates further below. 
\begin{proposition}\label{prop:conv1}
Under the assumptions in Theorem \ref{thm:phik}, and with $\tw$  given 
by \eqref{eq:twexp} with $A>k\geq 2$, there holds
\begin{align}\label{eq:convest1}
\|\min\{r_c^{2},r_c^{-2}\}(r\de_r)^j\left(\tP (r,c\pm i\eps)-\tP (r,c)\right) \|_{L^\infty_{\tw}}&\lesssim\eps k^{2+j},\qquad j=0,1,
\end{align}
and
\begin{align}\label{eq:convest2}
\|\min\{r_c^{2},r_c^{-2}\}(u(r)-c)r^2\de_{rr}(\tP (r,c\pm i\eps)-\tP (r,c))\|_{L^\infty_{\tw}}
&\lesssim\eps k^4,
\end{align}
for every $\eps\in (0,\eps_0)$.
\end{proposition}

\begin{proof}[Proof of Proposition~\ref{prop:conv1}]
We begin by showing the case $j=0$. By setting $z=c\pm i \eps$, we use \eqref{eq:integralP} to deduce that
\begin{align}
\tP(r,c)-\tP(r,z)
& = \mp i\eps (k^2-1)\int_r^{r_c} \frac{1}{\rho^3}\frac{2(u(\rho)-c)\mp i\eps }{(u(\rho)-z)^2(u(\rho)-c)^2}\int_\rho^{r_c} s (u(s)-c)^2\tP(s,c)\dd s\, \dd\rho\label{eq:z1}\\
& \quad \pm i\eps (k^2-1) \int_r^{r_c} \frac{1}{\rho^3(u(\rho)-z)^2} \int_\rho^{r_c} s \left[2(u(s)-c) \mp i\eps\right]\tP(s,c)\dd s\, \dd\rho\label{eq:z2}\\
& \quad +\mathcal{T}_z[\tP(r,c)-\tP(r,z)]. \label{eq:z3}
\end{align} 
We now use Lemma \ref{lem:wtilde} and \eqref{eq:Linftyesttp} several times. 
Let us consider the case when $r<r_c$ only, since the other case is analogous.
Define $b_k=1-1/k>0$. Then, 
\begin{align}
|\eqref{eq:z1}+\eqref{eq:z2}|
&\lesssim  \eps k^2\|\tP\|_{L^\infty_{\tw}}\int_{b_kr_c}^{r_c} \frac{1}{\rho^3}\frac{1}{|u(\rho)-z|}\int_\rho^{r_c} s \tw(s,r_c)\dd s\, \dd\rho\notag\\
&\quad+  \eps k^2\|\tP\|_{L^\infty_{\tw}}\int_r^{b_kr_c} \frac{1}{\rho^3}\frac{1}{|u(\rho)-z|}\int_\rho^{r_c} s \tw(s,r_c)\dd s\, \dd\rho\notag\\
&\lesssim  \eps  k^2\|\tP\|_{L^\infty_{\tw}}\frac{1}{|u'(r_c)|r_c}\int_{b_kr_c}^{r_c} \frac{1}{\rho}\tw(\rho,r_c)\dd\rho+  \eps k \|\tP\|_{L^\infty_{\tw}}\frac{1}{|u(b_kr_c)-u(r_c)|}\int_r^{b_kr_c} \frac{1}{\rho} \tw(\rho,r_c) \dd\rho\notag\\
&\lesssim  \eps k^2  \frac{\tw(r,r_c)}{\min\{r_c^2,r_c^{-2}\}}.
\end{align}
Hence, using \eqref{eq:contractionest} to control \eqref{eq:z3}, we arrive at \eqref{eq:convest1} with $j=0$.
We now deal with the similar convergence estimate for the derivative of $\tP$, namely the case $j=1$. Use \eqref{eq:drP} to get
\begin{align}
\frac{r\de_r(\tP(r,c)-\tP(r,z))}{k^2-1}
& = \pm i\eps \frac{2(u(r)-c)\mp i\eps}{r^2(u(r)-z)^2(u(r)-c)^2}\int_r^{r_c} s (u(s)-c)^2\tP(s,c)\dd s\label{eq:dez1}\\
& \quad \mp i\eps  \frac{1}{r^2(u(r)-z)^2} \int_r^{r_c} s \left[2(u(s)-c)\mp i\eps\right]\tP(s,c)\dd s\label{eq:dez2}\\
& \quad - \frac{1}{r^2(u(r)-z)^2} \int_r^{r_c} s (u(s)-z)^2\left[\tP(s,c)-\tP(s,z)\right]\dd s.\label{eq:dez3}
\end{align} 
In the same way as above, we find
\begin{align*}
|\eqref{eq:dez1}+\eqref{eq:dez2}|\lesssim \eps \|\tP\|_{L^\infty_{\tw}} \frac{1}{r^2|u(r)-z|}\left|\int_{r}^{r_c} s \tw(s,r_c)\dd s\right|\lesssim \eps k \frac{w(r,r_c)}{\min\{r_c^2,r_c^{-2}\}},
\end{align*}
while appealing to \eqref{eq:convest1} with $j=0$ we also infer that
\begin{align}
|\eqref{eq:dez3}|\lesssim \|\min\{r_c^{2},r_c^{-2}\}(\tP (r,c\pm i\eps)-\tP (r,c)) \|_{L^\infty_{\tw}}\frac{1}{r^2\min\{r_c^2,r_c^{-2}\}}\left|\int_{r}^{r_c} s \tw(s,r_c)\dd s\right|\lesssim \eps k \frac{w(r,r_c)}{\min\{r_c^2,r_c^{-2}\}},
\end{align}
and \eqref{eq:convest1} follows. 
Concerning \eqref{eq:convest2}, note that we can read $\de_{rr}\tP$ directly from \eqref{eq:Peqn}, obtaining
\begin{align*}
\de_{rr}(\tP(r,z)-\tP(r,c))=&-\frac{3}{r} \de_r(\tP(r,z)-\tP(r,c))   \pm i\eps    \frac{2u'(r)}{(u(r)-z)(u(r)-c)}  \de_r\tP(r,z) \notag\\
&-\frac{2u'(r)}{u(r)-z}  \de_r(\tP(r,z)-\tP(r,c))      +(k^2-1)\frac{\tP(r,z)-\tP(r,c)}{r^2}.
\end{align*}
and hence using the boundedness of $u$ we obtain
\begin{align}\label{eq:nomoreee}
\|\min\{r_c^{2},r_c^{-2}\}(u(r)-c)r^2\de_{rr}(\tP (r,c\pm i\eps)-\tP (r,c))\|_{L^\infty_{\tw}}
&\lesssim\|\min\{r_c^{2},r_c^{-2}\}r \de_r(\tP(r,z)-\tP(r,c))    \|_{L^\infty_{\tw}}\notag\\
&\ +\eps\norm{\min\{r_c^{2},r_c^{-2}\}r^2   \frac{ u'(r)}{u(r)-z}  \de_r\tP(r,z)}_{L^\infty_{\tw}}\notag\\
&\ +k^2\|\min\{r_c^{2},r_c^{-2}\}(\tP(r,z)-\tP(r,c) )\|_{L^\infty_{\tw}}. 
\end{align}
Given \eqref{eq:convest1}, we only need to treat the second term above. 
Using \eqref{eq:Lipdr}, if $|r-r_c|<r_c/2$, we obtain
\begin{align}
\frac{ |u'(r)|r}{|u(r)-z|}  |\tw^{-1}r\de_r\tP(r,z)|
\lesssim  k^3\frac{ |u'(r)|r}{|u(r)-z|}  \frac{|r-r_c|}{r_c}
\lesssim k^3,
\end{align}
while if $|r-r_c|\geq r_c/2$, since $|u(r)-u(r_c)|\gtrsim \min\{r_c^2,r_c^{-2}\}$, we have
\begin{align*}
\frac{ |u'(r)|r}{|u(r)-z|}  |\tw^{-1}r\de_r\tP(r,z)|
\lesssim\frac{ |u'(r)|r}{\min\{r_c^2,r_c^{-2}\}}  \|r\de_r\tP\|_{L^\infty_{\tw}}
\lesssim\frac{ k^2}{\min\{r_c^2,r_c^{-2}\}}. 
\end{align*}
Thus, we can plug in the above estimates in \eqref{eq:nomoreee} and use \eqref{eq:convest1}  to conclude the proof.
\end{proof}
As a consequence of \eqref{eq:relPtildeP}, we deduce 
\begin{align}
\|\min\{r_c^{2},r_c^{-2}\}(r\de_r)^j(P (r,c\pm i\eps)-P (r,c)) \|_{L^\infty_{\phi,\varkappa}}\lesssim_\varkappa\eps k^{2+j}, \qquad j=0,1\label{eq:FPconvest1},
\end{align}
and
\begin{align}\label{eq:FPconvest2}
\|\min\{r_c^{2},r_c^{-2}\}(u(r)-c)r^2\de_{rr}(P (r,c\pm i\eps)-P (r,c))\|_{L^\infty_{\phi,\varkappa}}
&\lesssim\eps k^4.
\end{align}
Note that the last estimate is exactly \eqref{eq:Pbestconv2}. We now deal with  $\de_G$ and $\de_{r_c}$ derivatives.

\subsection{Analysis of \texorpdfstring{$\de_G \tP$}{TEXT}}\label{sub:convdeG}
In this section, we analyze more carefully the behavior of $\tP$ near the critical layer, by making using of the differential
operator $\de_G$ defined in \eqref{def:dG}. 

\begin{proposition}\label{prop:deGder}
Let  $\tw$ be given  by \eqref{eq:twexp} with $A=k+\varkappa$, 
and let us fix $r_c\in (0,\infty)$. For every $r\in (0,\infty)$ such that
$$
|r-r_c|\leq \frac{r_c}{k},
$$
we have the pointwise bounds
\begin{subequations} \label{eq:estdGPtil} 
\begin{align}
\min\{r_c^2,r_c^{-2}\}\frac{|\de_G\tP(r,z)|}{\tw(r,r_c)} & \lesssim_\varkappa k^2, \label{eq:estdGPtila} \\ 
\min\{r_c^2,r_c^{-2}\}\frac{ |r\de_r\de_G \tP(r,z)|}{\tw(r,r_c)} & \lesssim_\varkappa k^4\frac{|r-r_c|}{r_c}, \label{eq:estdGPtilb}
\end{align}
\end{subequations} 
and the convergence estimates
\begin{align}\label{eq:convdGPtil}
(\min\{r_c^{2},r_c^{-2}\})^2\frac{|(r\de_r)^j\de_G(\tP (r,c\pm i\eps)-\tP (r,c))|}{\tw(r,r_c)}&\lesssim_\varkappa k^{3+j}\eps, \qquad j=0,1.
\end{align}
\end{proposition}

\begin{proof}[Proof of Proposition~\ref{prop:deGder}]
A direct computation from \eqref{eq:drP} yields
\begin{align}
\de_r\de_G \tP(r,z)&=\frac{u''(r)}{(u'(r))^2}\de_r\tP(r,z)+ \frac{3(k^2-1)}{r^4u'(r)(u(r)-z)^2}\int_r^{r_c}s(u(s)-z)^2\tP(s,z)\dd s\notag\\
&\quad-\frac{k^2-1}{r^3(u(r)-z)^2}\int_r^{r_c}\de_s\left(\frac{s}{u'(s)}\right)(u(s)-z)^2\tP(s,z)\dd s\notag\\
&\quad-\frac{k^2-1}{r^3(u(r)-z)^2}\int_r^{r_c}s(u(s)-z)^2\de_G\tP(s,z)\dd s \label{eq:drdGdersimple}.
\end{align}
Thus, integrating on $(r,r_c)$ and noting from \eqref{eq:PeqnBC} and \eqref{eq:computerc2} that $\de_G\tP(r_c,z)=0$, we find
\begin{align}
\de_G \tP(r,z)&=-\int_r^{r_c}\frac{u''(\rho)}{(u'(\rho))^2}\de_r\tP(\rho,z)\dd\rho
- \int_r^{r_c}\frac{3(k^2-1)}{\rho^4u'(\rho)(u(\rho)-z)^2}\int_\rho^{r_c}s(u(s)-z)^2\tP(s,z)\dd s\, \dd\rho\notag\\
&\quad+\int_r^{r_c}\frac{k^2-1}{\rho^3(u(\rho)-z)^2}\int_\rho^{r_c}\de_s\left(\frac{s}{u'(s)}\right)(u(s)-z)^2\tP(s,z)\dd s\,\dd\rho
+\mathcal{T}_z[\de_G\tP] \label{eq:dGdersimple}.
\end{align}
Assuming $|r-r_c|<r_c/k$, we now bound each term on the right-hand side above, tacitly using Lemma \ref{lem:wtilde} and the fact that
 $|u'(r)|\approx \min\{r,r^{-3}\}$.
For the first term, we exploit \eqref{eq:Linftyesttp} to get
\begin{align}\label{eq:argue1}
\left|\int_r^{r_c}\frac{u''(\rho)}{(u'(\rho))^2}\de_r\tP(\rho,z)\dd\rho\right|
\lesssim \|r\de_r\tP\|_{L^\infty_{\tw}}\left|\int_r^{r_c}\frac{u''(\rho)}{(u'(\rho))^2\rho}\tw(\rho,r_c)\dd\rho\right|
\lesssim k\frac{\tw(r,r_c)}{\min\{r_c^2,r_c^{-2}\}}.
\end{align}
Similarly
\begin{align*}
\left|\int_r^{r_c}\frac{3(k^2-1)}{\rho^4u'(\rho)(u(\rho)-z)^2}\int_\rho^{r_c}s(u(s)-z)^2\tP(s,z)\dd s\, \dd\rho\right|
&\lesssim k\frac{\tw(r,r_c)}{\min\{r_c^2,r_c^{-2}\}}, 
\end{align*}
and
\begin{align*}
\left|\int_r^{r_c}\frac{k^2-1}{\rho^3(u(\rho)-z)^2}\int_\rho^{r_c}\de_s\left(\frac{s}{u'(s)}\right)(u(s)-z)^2\tP(s,z)\dd s\,\dd\rho\right|
\lesssim k\frac{\tw(r,r_c)}{\min\{r_c^2,r_c^{-2}\}}.
\end{align*}
For the last term, we use \eqref{eq:exactw} to deduce that
\begin{align}
|\mathcal{T}_z[\de_G\tP]|
=\left|\int_r^{r_c}\frac{k^2-1}{\rho^3(u(\rho)-z)^2}\int_\rho^{r_c}s(u(s)-z)^2\de_G\tP(s,z)\dd s\,\dd\rho\right|\leq\frac{k^2-1}{A^2-1}\tw(r,r_c)\sup_{r} \frac{|\de_G\tP(r,z)|}{\tw(r,r_c)}. \label{eq:hyrtdfb}
\end{align}
The fact that the right-hand side above is finite for all $\eps>0$ follows from general ODE theory, since $\de_G \tP$ satisfies essentially a perturbation of Laplace's equation.
Collecting the above estimate, we obtain
\begin{align*}
\min\{r_c^2,r_c^{-2}\}\left[\frac{\de_G \tP(r,z)}{\tw(r,r_c)}-\frac{k^2-1}{A^2-1}\sup_{r} \frac{|\de_G\tP(r,z)|}{\tw(r,r_c)}\right]\lesssim
k,
\end{align*}
and we easily arrive at \eqref{eq:estdGPtila}. We now turn to \eqref{eq:convdGPtil} with $j=0$. From \eqref{eq:dGdersimple},
we infer that
\begin{align*}
\de_{G}\tP (r,c)-\de_{G}\tP (r,z) = \mathcal{T}_c[\de_{G} \tP(r,c)]-\mathcal{T}_z[\de_{G} \tP(r,z)]+\mathcal{V}_1+(k^2-1) &\left[ 
\mathcal{V}_2+\mathcal{V}_3\right],
\end{align*}
where
\begin{align*}
\mathcal{V}_1&=\int_r^{r_c}\frac{u''(\rho)}{(u'(\rho))^2}\de_r\left[\tP(\rho,z)-\tP(\rho,c)\right]\dd\rho,
\end{align*}
\begin{align}
\mathcal{V}_2
&= \pm 3i \eps \int_r^{r_c}  \frac{2(u(\rho)-c)\mp i\eps}{\rho^4u'(\rho)(u(\rho)-z)^2(u(\rho)-c)^2}\int_\rho^{r_c} s (u(s)-c)^2\tP(s,c)\dd s\, \dd\rho\notag\\
&\quad\mp 3i\eps \int_r^{r_c}  \frac{1}{\rho^4u'(\rho)(u(\rho)-z)^2}\int_\rho^{r_c} s  \left[2(u(s)-c)\mp i\eps\right]\tP(s,c)\dd s\, \dd\rho\notag\\
&\quad-3\int_r^{r_c}  \frac{1}{\rho^4u'(\rho)(u(\rho)-z)^2}\int_\rho^{r_c} s (u(s)-z)^2\left[\tP(s,c)-\tP(s,z)\right]\dd s\, \dd\rho,
\end{align}
\begin{align*}
\mathcal{V}_3
&=\pm i \eps\int_r^{r_c}\frac{2(u(\rho)-c)\mp i\eps}{\rho^3(u(\rho)-c)^2u(\rho)-z)^2}\int_\rho^{r_c}\de_s\left(\frac{s}{u'(s)}\right)(u(s)-c)^2\tP(s,c)\dd s\,\dd\rho\notag\\
&\quad\mp i\eps\int_r^{r_c}\frac{1}{\rho^3(u(\rho)-c)^2}\int_\rho^{r_c}\de_s\left(\frac{s}{u'(s)}\right)\left[2(u(s)-c)\mp i\eps\right]\tP(s,c)\dd s\,\dd\rho\notag\\
&\quad+\int_r^{r_c}\frac{1}{\rho^3(u(\rho)-c)^2}\int_\rho^{r_c}\de_s\left(\frac{s}{u'(s)}\right)(u(s)-c)^2\left[\tP(s,c)-\tP(s,z)\right]\dd s\,\dd\rho.
\end{align*}
Arguing as in \eqref{eq:argue1} and appealing to \eqref{eq:convest1}, we find that
\begin{align*}
\min\{r_c^2,r_c^{-2}\}|\mathcal{V}_1|
&\lesssim \frac{1}{k}\|\min\{r_c^2,r_c^{-2}\}r\de_r(\tP(r,z)-\tP(r,c))\|_{L^\infty_{\tw}}\frac{\tw(r,r_c)}{\min\{r_c^2,r_c^{-2}\}}\lesssim \eps k^2\frac{\tw(r,r_c)}{\min\{r_c^2,r_c^{-2}\}}.
\end{align*}
Turning to $\mathcal{V}_2$, we find
\begin{align}
\min\{r_c^2,r_c^{-2}\}|\mathcal{V}_2|&\lesssim
\eps \min\{r_c^2,r_c^{-2}\} \|\tP\|_{L^\infty_{\tw}}\int_r^{r_c}  \frac{1}{\rho^4|u'(\rho)||u(\rho)-z|}\int_\rho^{r_c} s \tw(s,r_c)\dd s\, \dd\rho\notag\\
&\quad+\| \min\{r_c^2,r_c^{-2}\}(\tP(\rho,c)-\tP(\rho,z))\|_{L^\infty_{\tw}} \int_r^{r_c}  \frac{1}{\rho^4|u'(\rho)|}\int_\rho^{r_c} s\tw(s,r_c)\dd s\, \dd\rho\notag\\
&\lesssim
 \frac{ \eps}{k} \frac{\|\tP\|_{L^\infty_{\tw}}}{\min\{r_c^2,r_c^{-2}\} }\tw(r,r_c)
 +\frac{1}{k^2}\frac{\| \min\{r_c^2,r_c^{-2}\}(\tP(\rho,c)-\tP(\rho,z))\|_{L^\infty_{\tw}}}{\min\{r_c^2,r_c^{-2}\}}\tw(r,r_c)\notag\\
  &\lesssim \eps \frac{\tw(r,r_c)}{\min\{r_c^2,r_c^{-2}\}}.\label{ineq:V2bd}
\end{align}
The $\mathcal{V}_3$ contribution is estimated as in $\mathcal{V}_2$ with the same bound. 
Lastly,
\begin{align}
\frac{\mathcal{T}_c[\de_{G} \tP(r,c)]-\mathcal{T}_z[\de_{G} \tP(r,z)]}{k^2-1}
&=\mp i\eps\int_r^{r_c}\frac{2(u(\rho)-c)\mp i\eps}{\rho^3(u(\rho)-c)^2(u(\rho)-z)^2}\int_\rho^{r_c}s(u(s)-c)^2\de_G\tP(s,c)\dd s\,\dd\rho\notag\\
&\quad\pm i\eps\int_r^{r_c}\frac{1}{\rho^3(u(\rho)-z)^2}\int_\rho^{r_c}s\left[2(u(s)-c)\mp i\eps\right]\de_G\tP(s,c)\dd s\,\dd\rho\notag\\
&\quad+\frac{\mathcal{T}_z[\de_{G} \tP(r,z)-\de_G\tP(r,c)]}{k^2-1}
\end{align}
The first two terms are analogous to $\mathcal{V}_2$, estimated in \eqref{ineq:V2bd},  while we deal with the last term 
as in \eqref{eq:hyrtdfb}. Hence, appealing to \eqref{eq:estdGPtila} we find
\begin{align*}
\frac{\mathcal{T}_c[\de_{G} \tP](r,c)-\mathcal{T}_z[\de_{G} \tP](r,z)}{k^2-1} -\frac{1}{A^2-1}\tw(r,r_c)\sup_{s} \frac{|\de_G\tP(s,c)-\de_{G}\tP(s,z)|}{\tw(s,r_c)}\lesssim  \eps  k^2\frac{ \tw(r,r_c) }{\min\left\{r_c^2,r_c^{-2}\right\}}
\end{align*}
Collecting all of the above
\begin{align*}
&(\min\{r_c^2,r_c^{-2}\})^2\left[\frac{|\de_{G}\tP (r,c)-\de_{G}\tP (r,z)|}{\tw(r,r_c)}-\frac{k^2-1}{A^2-1}\sup_{s} \frac{|\de_G\tP(s,c)-\de_G\tP(s,z)|}{\tw(s,r_c)}\right]\lesssim_\varkappa \eps k^2,
\end{align*}
from which \eqref{eq:convdGPtil} with $j=0$ easily follows. 
We now go back \eqref{eq:drdGdersimple}, and use \eqref{eq:weightest4}, \eqref{eq:Linftyesttp}, \eqref{eq:Lipdr} and \eqref{eq:estdGPtila}  
to deduce \eqref{eq:estdGPtilb} by the methods above.
From \eqref{eq:drdGdersimple},
\begin{align}
r\de_r\de_G \tP(r,c)-r\de_r\de_G \tP(r,z)=\mathcal{W}_1+(k^2-1)\left[\mathcal{W}_2+\mathcal{W}_3+\mathcal{W}_4\right]
\end{align}
where
\begin{align*}
\mathcal{W}_1&=\frac{u''(r)}{(u'(r))^2}r\de_r\left[\tP(r,c)-\tP(r,z)\right]
\end{align*}
\begin{align*}
\mathcal{W}_2
&=\mp 3i\eps \frac{2(u(r)-c)\mp i\eps}{r^3u'(r)(u(r)-c)^2(u(r)-z)^2}\int_r^{r_c}s(u(s)-c)^2\tP(s,c)\dd s\notag\\
&\quad\pm 3i\eps\frac{1}{r^3u'(r)(u(r)-z)^2}\int_r^{r_c}s\left[2(u(s)-c)\mp i\eps\right]\tP(s,c)\dd s\notag\\
&\quad+\frac{3}{r^3u'(r)(u(r)-z)^2}\int_r^{r_c}s(u(s)-z)^2\left[\tP(s,c)-\tP(s,z)\right]\dd s
\end{align*}
\begin{align*}
\mathcal{W}_3
&=\pm i\eps \frac{2(u(r)-c)\mp i\eps}{r^2(u(r)-c)^2(u(r)-z)^2}\int_r^{r_c}\de_s\left(\frac{s}{u'(s)}\right)(u(s)-c)^2\tP(s,c)\dd s\notag\\
&\quad\mp i \eps\frac{1}{r^2(u(r)-z)^2}\int_r^{r_c}\de_s\left(\frac{s}{u'(s)}\right)\left[2(u(s)-c)\mp i\eps\right]\tP(s,c)\dd s\notag\\
&\quad-\frac{1}{r^2(u(r)-z)^2}\int_r^{r_c}\de_s\left(\frac{s}{u'(s)}\right)(u(s)-z)^2\left[\tP(s,c)-\tP(s,z)\right]\dd s
\end{align*}
\begin{align*}
\mathcal{W}_4
&=\pm i\eps \frac{2(u(r)-c)\mp i\eps}{r^2(u(r)-c)^2(u(r)-z)^2}\int_r^{r_c}s(u(s)-c)^2\de_G\tP(s,c)\dd s \notag\\
&\quad\mp i \eps\frac{1}{r^2(u(r)-z)^2}\int_r^{r_c}s\left[2(u(s)-c)\mp i\eps\right]\de_G\tP(s,c)\dd s \notag\\
&\quad-\frac{1}{r^2(u(r)-z)^2}\int_r^{r_c}s(u(s)-z)^2\de_G\left[\tP(s,c)-\tP(s,z)\right]\dd s 
\end{align*}
Arguing as above, we see that $\mathcal{W}_1$ is analogous to $\mathcal{V}_1$, without a gain in $k$ due to the absence of
the integral, so that
\begin{align*}
\min\{r_c^2,r_c^{-2}\}|\mathcal{W}_1|
\lesssim \eps k^3\frac{\tw(r,r_c)}{\min\{r_c^2,r_c^{-2}\}}.
\end{align*}
Similarly, $\mathcal{W}_2,\mathcal{W}_3$ and $\mathcal{W}_4$ resemble $\mathcal{V}_2$, provided we take into account the
bounds \eqref{eq:Linftyesttp}, \eqref{eq:convest1},  \eqref{eq:estdGPtil}, and \eqref{eq:convdGPtil} with $j=0$, so that
\begin{align*}
\min\{r_c^2,r_c^{-2}\}|\mathcal{W}_2+\mathcal{W}_3+\mathcal{W}_4|
&\lesssim_\varkappa\eps k^2\frac{\tw(r,r_c)}{\min\{r_c^2,r_c^{-2}\}}.
\end{align*}
The proof of Proposition \ref{prop:deGder} is now complete.
\end{proof}

\subsection{Analysis of \texorpdfstring{$\de_{r_c} \tP$}{TEXT}}\label{sub:convderc}
We proceed with the analysis of $\de_{r_c}\tP$, relying on \eqref{eq:computerc} and \eqref{eq:computerc2}.
The main result of this section reads as follows.
\begin{proposition}[The $\de_{r_c}$ derivative]\label{prop:dercder}
Let  $\tw$ be given  by \eqref{eq:twexp} with $A=k+\varkappa$.
Then,
\begin{align}\label{eq:rcderi}
\| r_c\de_{r_c}\tP\|_{L^\infty_{\tw}} \lesssim_{\varkappa}  k^3,
\end{align}
and
\begin{align}\label{eq:FPconvestrc}
\|\min\{r_c^{2},r_c^{-2}\}r_c\de_{r_c}(\tP (r,c\pm i\eps)-\tP (r,c))\|_{L^\infty_{\tw}}&\lesssim_\varkappa k^4\eps.
\end{align}
Moreover, for every $r_c>0$ and every $r>0$ such that
$$
|r-r_c|> \frac{r_c}{k},
$$
there holds
\begin{align}\label{eq:rdr_rcdrctPbdd}
\frac{|r\de_rr_c\de_{r_c}\tP (r,z)|}{\tw(r,r_c)}&\lesssim_\varkappa k^4,
\end{align}
and
\begin{align}\label{eq:rdr_rcdrctPconv}
\min\{r_c^{2},r_c^{-2}\}\frac{|r\de_rr_c\de_{r_c}\tP (r,c\pm i\eps)-\tP (r,c))|}{\tw(r,r_c)}&\lesssim_\varkappa k^5\eps.
\end{align}
\end{proposition}

\begin{proof}[Proof of Proposition~\ref{prop:dercder}]
The proof heavily relies on Lemma \ref{lem:wtilde} and Propositions \ref{prop:tildePfirst}, \ref{prop:conv1} and \ref{prop:deGder},
but the arguments are very similar to those used earlier. Moreover, we shall only deal with the case
\begin{align}\label{eq:rangeofr}
|r-r_c|> \frac{r_c}{k},
\end{align}
even for \eqref{eq:rcderi} and \eqref{eq:FPconvestrc}, since in the other case 
\begin{align*}
r_c\de_{r_c} =u'(r_c)r_c\de_G-\frac{u'(r_c)r_c}{u'(r)}\de_r \approx \min\{r_c^2,r_c^{-2}\}\de_G+r\de_r,
\end{align*}
and hence the result follows from the respective bounds and convergence estimates on $\de_G$ and $r\de_r$.
For the sake of brevity, we consider only the case $r<r_c$, which from \eqref{eq:rangeofr} implies that
$r<b_kr_c$ with $b_k=1-1/k>0$.
To prove \eqref{eq:rcderi}, we need to show that
\begin{align*}
\frac{|r_c\de_{r_c}\tP(r,z)|}{\tw(r,r_c)}\lesssim  k^3, \qquad \forall r,r_c>0, \quad|r-r_c|> \frac{r_c}{k}.
\end{align*}
We use \eqref{eq:computerc}, multiply by $r_c$ and bound each term. As in \eqref{eq:hyrtdfb},
\begin{align*}
|\mathcal{T}_z[r_c\de_{r_c} \tP]|\leq\frac{k^2-1}{A^2-1}\tw(r,r_c)\sup_{s} \frac{|\de_{r_c}\tP(s,z)|}{\tw(s,r_c)}.
\end{align*}
The fact that the right-hand side above is finite for all $\eps>0$ follows from general ODE theory, since $\de_{r_c} \tP$ satisfies an equation
that is essentially a perturbation of Laplace's equation.
Moreover,
\begin{align}
r_c|\mathcal{I}_1+\mathcal{I}_2|
&\lesssim k^2\| \tP\|_{L^\infty_{\tw}}|u'(r_c)|r_c\int_{b_kr_c}^{r_c} \frac{1}{\rho^3 |u(\rho)-z|}\int_\rho^{r_c} s \tw(s,r_c)\dd s\, \dd\rho\notag\\
&\quad+k^2\| \tP\|_{L^\infty_{\tw}}|u'(r_c)|r_c\int_r^{b_kr_c} \frac{1}{\rho^3 |u(\rho)-z|}\int_\rho^{r_c} s \tw(s,r_c)\dd s\, \dd\rho\notag\\
&\lesssim k^3\int_{r}^{r_c} \frac{\tw(\rho,r_c)}{\rho} \dd\rho\lesssim k^2\tw(r,r_c),
\end{align}
and
\begin{align*}
r_c|\mathcal{I}_3|\lesssim k^2r_c^2 \int_r^{r_c}\frac{1}{\rho^3}\dd \rho\lesssim k^2\frac{r_c^2}{r^2}\lesssim k^2\tw(r,r_c).
\end{align*}
Hence, \eqref{eq:rcderi} is a consequence of the above estimates.
Combining \eqref{eq:convest1} and  \eqref{eq:convdGPtil}, we obtain \eqref{eq:FPconvestrc}
when we restrict to the domain $|r-r_c|\leq r_c/k$. On the same region, to prove \eqref{eq:FPconvestrc} we aim to show: 
\begin{align}\label{eq:FPconvestrc_point}
\min\{r_c^{2},r_c^{-2}\}\frac{|r_c\de_{r_c}(\tP (r,c\pm i\eps)-\tP (r,c))|}{\tw(r,r_c)}&\lesssim_\varkappa k^4\eps.
\end{align}
Again, just consider the case when $0<r<b_k r_c$. A combination of \eqref{eq:computerc} and \eqref{eq:computerc2} allows us to write
\begin{align*}
\de_{r_c}\tP (r,z) = &\mathcal{T}[\de_{r_c} \tP]+(k^2-1)u'(r_c) \Bigg[ 
2\int_r^{b_kr_c} \frac{1}{\rho^3 (u(\rho)-z)^3}\int_\rho^{r_c} s (u(s)-z)^2\tP(s,z)\dd s\, \dd\rho\notag\\
&\qquad \qquad- 2  \int_r^{b_kr_c} \frac{1}{\rho^3(u(\rho)-z)^2} \int_\rho^{r_c} s (u(s)-z)\tP(s,z)\dd s\, \dd\rho\notag\\
&\qquad \qquad-\eps^2\frac{r_c}{u'(r_c)} \int_r^{b_kr_c}\frac{1}{\rho^3 (u(\rho)-z)^2}\dd \rho\notag \\
&\qquad \qquad +\frac{1}{(b_kr_c)^3u'(b_kr_c)(u(b_kr_c)-z)^2}\int_{b_kr_c}^{r_c} \rho (u(\rho)-z)^2\tP(\rho,z)\dd \rho\notag\ \\
&\qquad\qquad+\int_{b_kr_c}^{r_c} \de_\rho \left(\frac{1}{\rho^3u'(\rho)}\right)\frac{1}{(u(\rho)-z)^2}\int_\rho^{r_c} s (u(s)-z)^2\tP(s,z)\dd s\, \dd\rho\notag\\
& \qquad \qquad+\int_{b_kr_c}^{r_c} \frac{1}{\rho^3(u(\rho)-z)^2} \int_\rho^{r_c} \de_s\left(\frac{s}{u'(s)}\right) (u(s)-z)^2\tP(s,z)\dd s\, \dd\rho\notag\\\
& \qquad \qquad+\int_{b_kr_c}^{r_c} \frac{1}{\rho^3(u(\rho)-z)^2} \int_\rho^{r_c} \frac{s}{u'(s)} (u(s)-z)^2\de_s\tP(s,z)\dd s\, \dd\rho\Bigg].
\end{align*}
From this, we write
\begin{align*}
\de_{r_c}\tP (r,c) -\de_{r_c}\tP (r,z)
&=\mathcal{T}[\de_{r_c} \tP(r,c)]-\mathcal{T}[\de_{r_c} \tP(r,z)]+(k^2-1)u'(r_c)\sum_{i=1}^7\mathcal{U}_i,
\end{align*}
where
\begin{align*}
\mathcal{U}_1
&= \mp 2i \eps\int_r^{b_kr_c} \frac{3(u(\rho)-c)^2\mp3i\eps(u(\rho)-c)-\eps^2}{\rho^3 (u(\rho)-c)^3 (u(\rho)-z)^3}\int_\rho^{r_c} s (u(s)-c)^2\tP(s,c)\dd s\, \dd\rho\notag\\
&\quad\pm2i\eps\int_r^{b_kr_c} \frac{1}{\rho^3 (u(\rho)-z)^3}\int_\rho^{r_c} s [2(u(s)-c)\mp i\eps]\tP(s,c)\dd s\, \dd\rho\notag\\
&\quad+2\int_r^{b_kr_c} \frac{1}{\rho^3 (u(\rho)-z)^3}\int_\rho^{r_c} s (u(s)-z)^2\left[\tP(s,c)-\tP(s,z)\right]\dd s\, \dd\rho,
\end{align*}
\begin{align*}
\mathcal{U}_2
&=\pm 2i\eps  \int_r^{b_kr_c} \frac{2(u(\rho)-c)\mp i\eps}{\rho^3(u(\rho)-c)^2(u(\rho)-z)^2} \int_\rho^{r_c} s (u(s)-c)\tP(s,c)\dd s\, \dd\rho\notag\\
&\quad \mp 2i\eps  \int_r^{b_kr_c} \frac{1}{\rho^3(u(\rho)-z)^2} \int_\rho^{r_c} s \tP(s,c)\dd s\, \dd\rho\notag\\
&\quad - 2  \int_r^{b_kr_c} \frac{1}{\rho^3(u(\rho)-z)^2} \int_\rho^{r_c} s (u(s)-c)\left[\tP(s,c)-\tP(z,c)\right]\dd s\, \dd\rho,
\end{align*}
\begin{align*}
\mathcal{U}_3
&=\eps^2\frac{r_c}{u'(r_c)} \int_r^{b_kr_c}\frac{1}{\rho^3 (u(\rho)-z)^2}\dd \rho,
\end{align*}
\begin{align*}
\mathcal{U}_4
&=\mp i\eps\frac{2(u(b_kr_c)-c)\mp i\eps}{(b_kr_c)^3u'(b_kr_c)(u(b_kr_c)-z)^2(u(b_kr_c)-c)^2}\int_{b_kr_c}^{r_c} \rho (u(\rho)-c)^2\tP(\rho,c)\dd \rho\notag\\
&\quad\pm i\eps\frac{1}{(b_kr_c)^3u'(b_kr_c)(u(b_kr_c)-z)^2}\int_{b_kr_c}^{r_c} \rho \left[2(u(\rho)-c)\mp i\eps\right]\tP(\rho,c)\dd \rho\notag\\
&\quad+\frac{1}{(b_kr_c)^3u'(b_kr_c)(u(b_kr_c)-z)^2}\int_{b_kr_c}^{r_c} \rho (u(\rho)-z)^2\left[\tP(\rho,c)-\tP(\rho,z)\right]\dd \rho,
\end{align*}
\begin{align*}
\mathcal{U}_5
&=\mp i\eps\int_{b_kr_c}^{r_c} \de_\rho \left(\frac{1}{\rho^3u'(\rho)}\right)\frac{2(u(\rho)-c)\mp i\eps}{(u(\rho)-z)^2(u(\rho)-c)^2}\int_\rho^{r_c} s (u(s)-c)^2\tP(s,c)\dd s\, \dd\rho\notag\\
&\quad\pm i \eps \int_{b_kr_c}^{r_c} \de_\rho \left(\frac{1}{\rho^3u'(\rho)}\right)\frac{1}{(u(\rho)-z)^2}\int_\rho^{r_c} s \left[2(u(s)-c)\mp i\eps\right]\tP(s,c)\dd s\, \dd\rho\notag\\
&\quad+\int_{b_kr_c}^{r_c} \de_\rho \left(\frac{1}{\rho^3u'(\rho)}\right)\frac{1}{(u(\rho)-z)^2}\int_\rho^{r_c} s (u(s)-z)^2\left[\tP(s,c)-\tP(s,z)\right]\dd s\, \dd\rho,
\end{align*}
\begin{align*}
\mathcal{U}_6
&=\mp i\eps\int_{b_kr_c}^{r_c} \frac{2(u(\rho)-c)\mp i\eps}{\rho^3(u(\rho)-z)^2(u(\rho)-c)^2} \int_\rho^{r_c} \de_s\left(\frac{s}{u'(s)}\right) (u(s)-c)^2\tP(s,c)\dd s\, \dd\rho\notag\\
&\quad\pm i \eps\int_{b_kr_c}^{r_c} \frac{1}{\rho^3(u(\rho)-z)^2} \int_\rho^{r_c} \de_s\left(\frac{s}{u'(s)}\right) \left[2(u(s)-c)\mp i\eps\right]\tP(s,c)\dd s\, \dd\rho\notag\\
&\quad+\int_{b_kr_c}^{r_c} \frac{1}{\rho^3(u(\rho)-z)^2} \int_\rho^{r_c} \de_s\left(\frac{s}{u'(s)}\right) (u(s)-z)^2\left[\tP(s,c)-\tP(s,z)\right]\dd s\, \dd\rho,
\end{align*}
\begin{align*}
\mathcal{U}_7
&=\mp i\eps\int_{b_kr_c}^{r_c} \frac{2(u(\rho)-c)\mp i\eps}{\rho^3(u(\rho)-z)^2(u(\rho)-c)^2} \int_\rho^{r_c} \frac{s}{u'(s)} (u(s)-c)^2\de_s\tP(s,c)\dd s\, \dd\rho\notag\\
&\quad\pm i\eps \int_{b_kr_c}^{r_c} \frac{1}{\rho^3(u(\rho)-z)^2} \int_\rho^{r_c} \frac{s}{u'(s)} \left[2(u(s)-c)\mp i\eps\right]\de_s\tP(s,c)\dd s\, \dd\rho\notag\\
&\quad+\int_{b_kr_c}^{r_c} \frac{1}{\rho^3(u(\rho)-z)^2} \int_\rho^{r_c} \frac{s}{u'(s)} (u(s)-z)^2\left[\de_s\tP(s,c)-\de_s\tP(s,z)\right]\dd s\, \dd\rho. 
\end{align*}
Bounding these terms essentially relies repeatedly on Lemma \ref{lem:wtilde}, Proposition \ref{prop:tildePfirst} and 
Proposition \ref{prop:conv1}. Note that
\begin{align}
\frac{\mathcal{T}_c[\de_{r_c} \tP(r,c)]-\mathcal{T}_z[\de_{r_c} \tP(r,z)]}{k^2-1}
&=\mp i\eps\int_r^{r_c}\frac{2(u(\rho)-c)\mp i\eps}{\rho^3(u(\rho)-c)^2(u(\rho)-z)^2}\int_\rho^{r_c}s(u(s)-c)^2\de_{r_c}\tP(s,c)\dd s\,\dd\rho\notag\\
&\quad\pm i\eps\int_r^{r_c}\frac{1}{\rho^3(u(\rho)-z)^2}\int_\rho^{r_c}s\left[2(u(s)-c)\mp i\eps\right]\de_{r_c}\tP(s,c)\dd s\,\dd\rho\notag\\
&\quad+\frac{\mathcal{T}_z[\de_{r_c} \tP(r,c)-\de_{r_c} \tP(r,z)]}{k^2-1}. \label{eq:fuckthis}
\end{align}
For the first two terms in \eqref{eq:fuckthis} we obtain the bound
\begin{align*}
&\eps \frac{\|r_c\de_{r_c}\tP\|_{L^\infty_{\tw}}}{r_c}\int_{b_kr_c}^{r_c}\frac{1}{\rho^3|u(\rho)-z|}\int_\rho^{r_c}s\tw(s,r_c)\dd s\,\dd\rho
+\eps \frac{\|r_c\de_{r_c}\tP\|_{L^\infty_{\tw}}}{r_c}\int_{r}^{b_kr_c}\frac{1}{\rho^3|u(\rho)-z|}\int_\rho^{r_c}s\tw(s,r_c)\dd s\,\dd\rho\notag\\
&\qquad\lesssim\eps \frac{k^3}{r_c^2}\int_{b_kr_c}^{r_c}\frac{|\rho-r_c|}{\rho|u(\rho)-z|}\tw(\rho,r_c)\dd\rho
+\eps  \frac{k^2}{r_c}\frac{1}{|u(b_kr_c)-u(r_c)|}\int_{r}^{b_kr_c}\frac{\tw(\rho,r_c)}{\rho}\dd\rho\notag\\
&\qquad\lesssim\eps \frac{k^3}{r_c^2|u'(r_c)|}\int_{b_kr_c}^{r_c}\frac{\tw(\rho,r_c)\dd\rho}{\rho}\dd\rho
+\eps \frac{k^2}{r_c}\frac{1}{|u(b_kr_c)-u(r_c)|}\int_{r}^{b_kr_c}\frac{\tw(\rho,r_c)}{\rho}\dd\rho\notag\\
&\qquad\lesssim\eps\frac{k^2}{r_c}\frac{\tw(r,r_c)}{|u'(r_c)|r_c}
\lesssim k^2\eps \frac{\tw(r,r_c)}{r_c\min\{r_c^2,r_c^{-2}\}},
\end{align*}
while for the last term we use \eqref{eq:also1} to obtain
\begin{align*}
\frac{\mathcal{T}_z[\de_{r_c} \tP(r,c)-\de_{r_c} \tP(r,z)]}{k^2-1}
&\leq \sup_r|\de_{r_c}\tP(r,z)-\de_{r_c}\tP(r,c))|\int_r^{r_c}\frac{1}{\rho^3}\int_\rho^{r_c}s\tw(s,r_c)\dd s\,\dd\rho\notag\\
&\leq \frac{1}{A^2-1}\frac{\tw(r,r_c)}{r_c} \sup_r\frac{|r_c\de_{r_c}\tP(r,z)-r_c\de_{r_c}\tP(r,c))|}{\tw(r,r_c)}.
\end{align*}
Concerning the $\mathcal{U}_i$'s, for the first three we have
\begin{align*}
|\mathcal{U}_1+\mathcal{U}_2|
&\lesssim\eps\|\tP\|_{L^\infty_{\tw}}\int_r^{b_kr_c} \frac{1}{\rho^3 |u(\rho)-z|^2}\int_\rho^{r_c} s \tw(s,r_c)\dd s\, \dd\rho\notag\\
&\quad+\|\min\{r_c^2,r_c^{-2}\}(\tP(s,c)-\tP(s,z))\|_{L^\infty_{\tw}}\frac{1}{\min\{r_c^2,r_c^{-2}\}}\int_r^{b_kr_c} \frac{1}{\rho^3 |u(\rho)-z|}\int_\rho^{r_c} s \tw(s,r_c)\dd s\, \dd\rho\notag\\
&\lesssim\eps\frac{1}{ |u(b_kr_c)-u(r_c)|^2}\int_r^{b_kr_c} \frac{\tw(\rho,r_c)}{\rho}\dd\rho
+\eps k\frac{1}{\min\{r_c^2,r_c^{-2}\}}\frac{1}{ |u(b_kr_c)-u(r_c)|}\int_r^{b_kr_c} \frac{\tw(\rho,r_c)}{\rho}\dd\rho\notag\\
&\lesssim\eps k\frac{\tw(r,r_c)}{(\min\{r_c^2,r_c^{-2}\})^2},
\end{align*}
and
\begin{align*}
|\mathcal{U}_3|
&\lesssim \eps\frac{r_c}{|u'(r_c)|} \int_r^{b_kr_c}\frac{1}{\rho^3 |u(\rho)-z|}\dd \rho\lesssim \eps\frac{r_c}{|u'(r_c)| |u(b_kr_c)-u(r_c)|} \int_r^{b_kr_c}\frac{1}{\rho^3}\dd \rho\lesssim\eps k\frac{\tw(r,r_c)}{(\min\{r_c^2,r_c^{-2}\})^2}.
\end{align*}
Concerning the others, we only show how to deal with $\mathcal{U}_4$ and $\mathcal{U}_5$, as $\mathcal{U}_6$ and $\mathcal{U}_7$ are treated similarly. 
We have
\begin{align*}
|\mathcal{U}_4|
&\lesssim \eps \|\tP\|_{L^\infty_{\tw}}\frac{1}{|u'(r_c)|r_c^3|u(b_kr_c)-u(r_c)|}\int_{b_kr_c}^{r_c} \rho \tw(\rho,r_c)\dd \rho\notag\\
&\quad+\|\min\{r_c^2,r_c^{-2}\}(\tP(s,c)-\tP(s,z))\|_{L^\infty_{\tw}}\frac{1}{|u'(r_c)|r_c^3\min\{r_c^2,r_c^{-2}\}}\int_{b_kr_c}^{r_c} \rho \tw(\rho,r_c)\dd \rho\notag\\
&\lesssim\eps k\frac{\tw(r,r_c)}{(\min\{r_c^2,r_c^{-2}\})^2},
\end{align*}
and, considering that $|\de_\rho \left(\rho^3u'(\rho)\right)^{-1}|\approx 1/\min\{\rho^{5},\rho\}$, we also deduce that
\begin{align*}
|\mathcal{U}_5|
&\lesssim \eps \|\tP\|_{L^\infty_{\tw}}\int_{b_kr_c}^{r_c} \frac{1}{\min\{\rho^{5},\rho\}}\frac{1}{|u(\rho)-z|}\int_\rho^{r_c} s \tw(s,r_c)\dd s\, \dd\rho\notag\\
&\quad+\|\min\{r_c^2,r_c^{-2}\}(\tP(r,c)-\tP(r,z))\|_{L^\infty_{\tw}}\frac{1}{\min\{r_c^2,r_c^{-2}\}}\int_{b_kr_c}^{r_c}\frac{1}{\min\{\rho^{5},\rho\}}\int_\rho^{r_c} s\tw(s,r_c)\dd s\, \dd\rho\notag\\
&\lesssim \eps k\frac{1}{\min\{r_c^2,r_c^{-2}\}}\int_{b_kr_c}^{r_c}\frac{1}{\min\{\rho^{3},\rho^{-1}\}}\tw(\rho,r_c)\dd\rho
\lesssim\eps k\frac{\tw(r,r_c)}{(\min\{r_c^2,r_c^{-2}\})^2}. 
\end{align*}
Collecting all of the above, \eqref{eq:FPconvestrc_point} follows. Going back to \eqref{eq:drP}, we also infer that
\begin{align*}
\frac{r\de_r r_c\de_{r_c}\tP(r,z)}{k^2-1}&= -2\frac{u'(r_c)r_c}{r^2(u(r)-z)^3} \int_r^{r_c} s(u(s)-z)^2\tP(s,z)\dd s
+\eps^2\frac{r_c^2}{r^2(u(r)-z)^2} \notag\\
&\quad+2\frac{u'(r_c)r_c}{r^2(u(r)-z)^2} \int_r^{r_c} s(u(s)-z)\tP(s,z)\dd s
-\frac{1}{r^2(u(r)-z)^2} \int_r^{r_c} s(u(s)-z)^2r_c\de_{r_c}\tP(s,z)\dd s
\end{align*}
Hence, we argue as above to obtain (for $\abs{r-r_c} \geq r_c/k$), 
\begin{align*}
\left|\frac{rr_c\de_r \de_{r_c}\tP(r,z)}{k^2-1}\right|
&\lesssim \|\tP\|_{L^\infty_{\tw}}\frac{\min\{r_c^2,r_c^{-2}\}}{r^2|u(b_kr_c)-u(r_c)|} \int_r^{r_c} s\tw(s,r_c)\dd s
+\frac{r_c^2}{r^2} +\|r_c\de_{r_c}\tP\|_{L^\infty_{\tw}}\frac{1}{r^2} \int_r^{r_c} s \tw(s,r_c)\dd s\notag\\
&\lesssim k\|\tP\|_{L^\infty_{\tw}}\tw(r,r_c)
+\tw(r,r_c)+\frac1k\|r_c\de_{r_c}\tP\|_{L^\infty_{\tw}}\tw(r,r_c)\lesssim k^2\tw(r,r_c),
\end{align*}
and \eqref{eq:rdr_rcdrctPbdd} follows. Lastly, the proof of \eqref{eq:rdr_rcdrctPconv} is simpler than the
proof of \eqref{eq:FPconvestrc_point}, and is hence omitted. 
\end{proof}

\subsection{Convergence in optimal weights}\label{sub:convrev}

With the convergence estimates of the previous sections at hand, we now aim to show 
the validity of the convergence estimate in Theorem \ref{thm:nonvan}.
The proof utilizes the Green's function of the operator
$$
\de_{rr}+\frac{1/4-k^2}{r^2},
$$
which is explicitly given by
\begin{align}\label{eq:greenlapl}
\mathcal{L}(r,\rho)= -\frac{1}{2k} \min\left(\frac{\rho}{r},\frac{r}{\rho}\right)^{k} \sqrt{r} \sqrt{\rho}, \qquad r,\rho>0,
\end{align}
as we treat the Rayleigh problem as a perturbation of the Laplacian.
We split the proof in different cases.

\subsubsection{The case \texorpdfstring{$r_c\leq 1$}{TEXT} }
All the estimates are pointwise in $r_c$, and hence the norms
and spaces here are involving only the variable $r$. We begin by optimizing the weight near the origin.

\begin{lemma}\label{lem:CONV1}
Let $j=0,1$. There exists a universal constant $\zeta \in (0,1/4)$ such that, for all $r_c \leq 1$ and $\varkappa\in (0,1)$, there hold
\begin{align}\label{eq:BOUND1}
\norm{r^{k-1/2}(r\de_r)^jP(r,c\pm i\eps)}_{L^\infty(0,\zeta r_c)} \lesssim_{\varkappa,\zeta}   \frac{\eps}{r_c^2} k^{1+j} r_c^{k-1/2},
\end{align}
and
\begin{align}\label{eq:CONV1}
\norm{r^{k-1/2}(r\de_r)^j\left(P(r,c\pm i\eps) - P(r,c)\right)}_{L^\infty(0,\zeta r_c)} \lesssim_{\varkappa,\zeta}   \frac{\eps}{r_c^2} k^{2+j} r_c^{k-1/2}. 
\end{align}
\end{lemma} 
\begin{proof}[Proof of Lemma~\ref{lem:CONV1}]
The proofs of \eqref{eq:BOUND1} and \eqref{eq:CONV1} are very similar, so we focus on the more challenging \eqref{eq:CONV1}.
Note that the norm appearing on the left-hand side of \eqref{eq:CONV1} (and \eqref{eq:BOUND1}) is a priori finite since $\Ray_z$ is a regular perturbation of Laplace's equation near $r \sim 0$. 

Let us first consider the case $j=0$.
If $r\in (0,\zeta r_c)$, then  \eqref{eq:FPbound} and \eqref{eq:FPconvest1} imply that
\begin{align}
&r^{k-1/2} |(r\de_r)^jP(r,c\pm i \eps)| \lesssim_\varkappa k^{1+j} r_c^{k-1/2}  \frac{r_c^\varkappa}{r^\varkappa},\qquad j=0,1,  \label{eq:FIRST1}\\
&r^{k-1/2} |(r\de_r)^j(P(r,c\pm i \eps)-P(r,c))| \lesssim_\varkappa \eps k^{2+j} r_c^{k-1/2}  \frac{1}{r_c^{2-\varkappa}r^\varkappa},\qquad j=0,1.\label{eq:FIRST3}
\end{align}
Let $\chi = \chi(r/a)$ be a smooth cut-off function at some scale $a>0$ to be determined. 
Define 
$$
g_\eps(r,r_c) = \chi(r/a) \left[\phi(r,c\pm i\eps) - \phi(r,c)\right]
$$ 
and compute  
\begin{align*}
\left(\de_{rr}+\frac{1/4-k^2}{r^2}\right)g_\eps & = -\frac{\beta(r)}{u(r) - c \mp i\eps}g_\eps - \frac{\pm i \eps \beta(r)}{(u(r)-c\mp i\eps)(u(r)-c)}\chi(r/a) \phi(r,c) \notag \\
&\quad- \frac{2}{a} \de_r\chi \de_r(\phi(r,c\pm i\eps) - \phi(r,c)) - \frac{1}{a^2}\de_{rr}\chi(\phi(r,c\pm i \eps) - \phi(r,c)). 
\end{align*}
Using \eqref{eq:greenlapl}, we then have
\begin{align}
g_\eps(r,r_c) & = -\int_0^\infty \mathcal{L}(r,\rho) \frac{\beta(\rho)}{u(\rho) - c \mp i\eps}g_\eps(\rho,r_c) \dd\rho - \int_0^\infty \mathcal{L}(r,\rho) \frac{\pm i \eps \beta(\rho)}{(u(\rho)-c\mp i\eps)(u(\rho)-c)}\chi(\rho/a) \phi(\rho,c)\dd\rho\notag \\ 
& \quad - \int_0^\infty \mathcal{L}(r,\rho) \left(\frac{2}{a}  (\de_\rho\chi) (\de_\rho(\phi(\rho,c\pm i\eps)-\phi(\rho,c))) + \frac{1}{a^2}
(\de_{\rho\rho}\chi )(\phi(\rho,c\pm i\eps)-\phi(\rho,c)) \right) \dd\rho\notag \\
& = \sum_{\ell=1}^4 J_\ell \label{eq:LAST1}. 
\end{align}
Since $\beta(\rho)\lesssim 1$ and $|u'(r_c)|\approx r_c$ as $r_c\to 0$, we have for $\zeta \in (0,1/4)$ that
$$
|u(\rho) - c \mp i\eps|\gtrsim r_c^2,\qquad \forall \rho\in (0,\zeta r_c).
$$
Hence, by choosing $a=\frac\zeta2 r_c$, with $\zeta \ll 1$, we obtain
\begin{align}
r^{k-1/2}\abs{J_{1}(r)} & \leq \frac{1}{2k}\int_0^r \rho \abs{\frac{\beta(\rho)}{u(\rho) - c \mp i\eps} \rho^{k-1/2}g_\eps(\rho,r_c)}  \dd\rho + \frac{r^{2k}}{2k}\int_r^{2a}\rho^{-2k+1}\abs{\frac{\beta(\rho)}{u(\rho) - c \mp i\eps} \rho^{k-1/2} g_\eps(\rho,r_c)} \dd\rho\notag \\ 
& \leq \frac12\norm{r^{k-1/2}g_\eps(r,r_c)}_{L^\infty(0,\zeta r_c)}\label{eq:LAST2}.
\end{align}
Turning to $J_2$, we use \eqref{eq:FIRST1} to obtain
\begin{align}
r^{k-1/2}\abs{J_{2}(r)}  &\leq \frac{\eps}{2k}\int_0^r \rho \abs{\frac{\beta(\rho)}{u(\rho)-c\mp i\eps} \rho^{k-1/2} P(\rho,c)\dd\rho}  \dd\rho + \eps\frac{r^{2k}}{2k}\int_r^{2a}\rho^{-2k+1}\abs{\frac{  \beta(\rho)}{u(\rho)-c\mp i\eps} \rho^{k-1/2} P(\rho,c)\dd\rho} \dd\rho\notag \\ 
&\lesssim \frac{\eps}{kr_c^2}\int_0^{2a}\rho \abs{\rho^{k-1/2} P(\rho,c)\dd\rho}  \dd\rho 
\lesssim \frac{\eps}{r_c^{2-\varkappa}}r_c^{k-1/2} \int_0^{2a}\rho^{1-\varkappa}    \dd\rho\lesssim_{\varkappa,\zeta} 
\eps r_c^{k-1/2} \label{eq:LAST3}.
\end{align}
For $J_3$, from \eqref{eq:FIRST1}-\eqref{eq:FIRST3} we obtain the pointwise bound for $\rho\in (0,\zeta r_c)$
\begin{align*}
|\rho\de_\rho(\phi(\rho,c\pm i\eps)-\phi(\rho,c))|&\lesssim_\zeta  r_c^2|P(\rho,c\pm i\eps)-P(\rho,c)|+ r_c^2|\rho\de_\rho(P(\rho,c\pm i\eps)-P(\rho,c))|+\eps |\rho\de_\rho P(\rho,c\pm i\eps)|\\
&\lesssim_{\varkappa,\zeta}  \eps k^3 \frac{r_c^{k-1/2}}{r^{k-1/2}}  \frac{r_c^\varkappa}{r^\varkappa}.
\end{align*}
Hence, arguing as above,
\begin{align}
r^{k-1/2}\abs{J_{3}(r)}  &\leq \frac{1}{a k}\int_0^{2a}   |\rho\de_\rho(\phi(\rho,c\pm i\eps)-\phi(\rho,c))|  \rho^{k-1/2}  \dd\rho
\lesssim_{\varkappa,\zeta} \frac{\eps k^2}{a}r_c^{k-1/2} r_c^\varkappa\int_0^{2a}\rho^{-\varkappa}    \dd\rho \lesssim_{\varkappa,\zeta} 
\eps k^2 r_c^{k-1/2} \label{eq:LAST4}.
\end{align}
Finally, 
by \eqref{eq:FIRST1} and \eqref{eq:FIRST3} we obtain the pointwise bounds
$$
|\phi(\rho,c\pm i\eps)-\phi(\rho,c)|\rho^{k-1/2}\lesssim_{\varkappa,\zeta}
\eps  k^2 r_c^{k-1/2}  \frac{r_c^\varkappa}{r^\varkappa}
$$
and therefore
\begin{align}
r^{k-1/2}\abs{J_{4}(r)}  &\leq \frac{1}{a^2 k}\int_0^{2a}   |\rho(\phi(\rho,c\pm i\eps)-\phi(\rho,c))|  \rho^{k-1/2}  \dd\rho \notag \\ &  \lesssim_{\varkappa,\zeta} \frac{\eps k}{a^2}r_c^{k-1/2} r_c^\varkappa\int_0^{2a}\rho^{1-\varkappa}    \dd\rho \lesssim_{\varkappa,\zeta} 
\eps k r_c^{k-1/2} \label{eq:LAST5},
\end{align}
Hence, collecting \eqref{eq:LAST1}-\eqref{eq:LAST5}, we arrive at
\begin{align*}
\norm{r^{k-1/2}g_\eps(r,r_c)}_{L^\infty(0,\zeta r_c)} \lesssim_{\varkappa,\zeta}   \eps k^2 r_c^{k-1/2}.
\end{align*}
Note that $g_\eps$ and $\phi(r,c\pm i \eps)-\phi(r,c)$ coincide in this region, and \eqref{eq:CONV1} is recovered from 
the definition of $P$, together with the inequality $|u(r) - c|\gtrsim r_c^2$ and a further application of \eqref{eq:FIRST1}.
Finally, the case $j=1$ follows immediately. Indeed, taking and $r\de_r$ derivative of \eqref{eq:LAST1}, we simply notice that
\begin{align}\label{eq:derofgreen}
r\de_r \mathcal{L}(r,\rho)=c_k \mathcal{L}(r,\rho), \qquad c_k=\begin{cases}
-(k-1/2), \quad &\rho\leq r,\\
k+1/2, \quad &\rho>r.
\end{cases}
\end{align}
Therefore, the result follows in the exact same way, by using the estimates on $g_\eps$ derived above. 
\end{proof} 

The interval $(\zeta r_c, R)$, for $R\gg 1$, independent of $r_c\leq 1$, is treated already by
\eqref{eq:FPconvest1}, which implies that for $r\in (\zeta r_c,r_c)$ there holds, 
\begin{align}\label{eq:FIRST5}
r^{k-1/2} |(r\de_r)^j(P(r,c\pm i \eps)-P(r,c))| \lesssim_{\varkappa,\zeta} \eps k^{2+j} r_c^{k-1/2}  \frac{1}{r_c^{2}},
\end{align}
while if $r\in (r_c,R)$ there holds
\begin{align}\label{eq:FIRST6}
r^{-k-1/2} |(r\de_r)^j(P(r,c\pm i \eps)-P(r,c))| \lesssim_{\varkappa,R} \eps k^{2+j} r_c^{-k-1/2}  \frac{1}{r_c^{2+\varkappa}}.
\end{align}
Finally, we need to correct the weight at infinity. 

\begin{lemma} \label{lem:CONV3}
Let $j=0,1$. There exists a universal constant $R>2$ such that, for all $r_c \leq 1$ and  any $\varkappa\in(0,1)$,  there hold
\begin{align*}
\norm{r^{-k-1/2}(r\de_r)^jP(r,c\pm i\eps) }_{L^\infty(R,\infty)}  \lesssim_{\varkappa,R} \frac{\eps}{r_c^{2+\varkappa}} k^{1+j} r_c^{-k-1/2}  
\end{align*}
and
\begin{align*}
\norm{r^{-k-1/2}(r\de_r)^j\left(P(r,c\pm i\eps) - P(r,c)\right)}_{L^\infty(R,\infty)}  \lesssim_{\varkappa,R} \frac{\eps}{r_c^{2+\varkappa}} k^{2+j} r_c^{-k-1/2} .  
\end{align*}
\end{lemma}

\begin{proof}[Proof of Lemma~\ref{lem:CONV3}]
Again, we only treat the case $j=0$. From  \eqref{eq:FPbound} and \eqref{eq:FPconvest1}, if $r\in (R,\infty)$ we have that
\begin{align}
&r^{-k-1/2} |(r\de_r)^jP(r,c\pm i \eps)| \lesssim_\varkappa k^{1+j} r_c^{-k-1/2}  \frac{r^\varkappa}{r_c^\varkappa}, \qquad j=0,1,\label{eq:FIRST7}\\
&r^{-k-1/2} |(r\de_r)^j(P(r,c\pm i \eps)-P(r,c))| \lesssim_\varkappa \eps k^{2+j} r_c^{-k-1/2}  \frac{r^\varkappa}{r_c^{2+\varkappa}},\qquad j=0,1.\label{eq:FIRST9}
\end{align}
Let $\chi = \chi(r/R)$ be a smooth cut-off function, with $R>0$ to be determined, and 
define 
$$
g_\eps(r,r_c) = \chi(r/R) \left[\phi(r,c\pm i\eps) - \phi(r,c)\right].
$$ 
As in the proof of Lemma \ref{lem:CONV1},
\begin{align}
g_\eps(r,r_c) & = -\int_0^\infty \mathcal{L}(r,\rho) \frac{\beta(\rho)}{u(\rho) - c \mp i\eps}g_\eps(\rho,r_c) \dd\rho - \int_0^\infty \mathcal{L}(r,\rho) \frac{\pm i \eps \beta(\rho)}{(u(\rho)-c\mp i\eps)(u(\rho)-c)}\chi(\rho/R) \phi(\rho,c)\dd\rho\notag \\ 
& \quad - \int_0^\infty \mathcal{L}(r,\rho) \left(\frac{2}{R} \de_\rho\chi(\rho/R) \de_\rho(\phi(\rho,c\pm i\eps) - \phi(\rho,c)) - \frac{1}{R^2}\de_{\rho\rho}\chi(\rho/R)(\phi(\rho,c\pm i \eps) - \phi(\rho,c))\right) \dd\rho\notag\\
& = \sum_{\ell=1}^4 J_\ell \label{eq:noidea1}.  
\end{align}
For the first term, we use Lemma \ref{lem:BasicVort} that $|\beta(\rho)|\lesssim \brak{\rho}^{-6}$ and that $|u(\rho)-c\mp i\eps|\gtrsim 1$ (thanks to our choice of $R>2$ and $r_c\leq 1$) to obtain
\begin{align}
r^{-k-1/2}\abs{J_{1}(r)} &\leq r^{-k-1/2}\left|\int_0^\infty \mathcal{L}(r,\rho) \frac{\beta(\rho)}{u(\rho) - c \mp i\eps}g_\eps(\rho,r_c) \dd\rho\right|\notag\\ 
&\lesssim \frac{1}{k} \int_R^\infty \frac{\rho \beta(\rho)}{|u(\rho)-u(r_c)\mp i\eps|} \rho^{-k-1/2} |g_\eps(\rho,r_c)|\dd \rho\notag\\
&\lesssim  R^{-6}\norm{r^{-k-1/2}g_\eps(r,r_c)}_{L^\infty(R,\infty)}\leq  \frac12\norm{r^{-k-1/2}g_\eps(r,r_c)}_{L^\infty(R,\infty)},\label{eq:noidea2}
\end{align}
provided $R\gg 1$ is big enough.
Regarding $J_2$, from \eqref{eq:FIRST7} we infer that
\begin{align}
r^{-k-1/2}|J_2(r)|&\leq r^{-k-1/2}\left|\int_0^\infty \mathcal{L}(r,\rho) \frac{\pm i \eps \beta(\rho)}{(u(\rho)-c\mp i\eps)(u(\rho)-c)}\chi(\rho/R) \phi(\rho,c)\dd\rho\right|\notag\\
&\lesssim \frac{\eps}{k} \int_R^\infty \rho \beta(\rho) \rho^{-k-1/2} |P(\rho,c)|\dd \rho\notag\\
&\lesssim_{\varkappa}\frac{\eps r_c^{-k-1/2}}{r_c^\varkappa} \int_R^\infty \rho^{1+\varkappa} \rho^{-6}\dd \rho\lesssim_{R,\varkappa}\frac{\eps}{r_c^\varkappa}r_c^{-k-1/2}\label{eq:noidea3}.
\end{align}
The terms involving $\de_\rho\chi$ and $\de_{\rho\rho}\chi$ are estimated similarly as in Lemma \ref{lem:CONV1}, 
except for the weight $r^{-k-1/2}$, and the fact that $\de_\rho\chi$ and $\de_{\rho\rho}\chi$
are supported in interval $[R,2R]$. Using that $u$ is bounded, $|u'(\rho)|\approx \rho^{-3}$   and \eqref{eq:FIRST7}-\eqref{eq:FIRST9}, we obtain for $r_c<1$ that
\begin{align*}
|\rho\de_\rho(\phi(\rho,c\pm i\eps)-\phi(\rho,c))|&\lesssim  \frac{1}{\rho^2}|P(\rho,c\pm i\eps)-P(\rho,c)|+ |\rho\de_\rho(P(\rho,c\pm i\eps)-P(\rho,c))|+\eps |\rho\de_\rho P(\rho,c\pm i\eps)|\\
&\lesssim_\varkappa\eps   k^3 \frac{r_c^{-k-1/2}}{\rho^{-k-1/2}}  \frac{\rho^\varkappa}{r_c^{2+\varkappa}},
\end{align*}
and
\begin{align*}
|\phi(\rho,c\pm i\eps)-\phi(\rho,c)|   \lesssim |P(\rho,c\pm i\eps)-P(\rho,c)| +\eps |P(\rho,c\pm i\eps)|\lesssim_\varkappa  \eps k^2 \frac{r_c^{-k-1/2}}{\rho^{-k-1/2}} \frac{\rho^\varkappa}{r_c^{2+\varkappa}}.
\end{align*}
Hence, we arrive at
\begin{align}
r^{-k-1/2}\abs{J_{3}(r)}  &\leq \frac{1}{R k}\int_R^{2R}   |\rho\de_\rho(\phi(\rho,c\pm i\eps)-\phi(\rho,c))|  \rho^{-k-1/2}  \dd\rho
 \lesssim_{\varkappa,R} \frac{\eps}{r_c^{2+\varkappa}} k^2 r_c^{-k-1/2} \label{eq:noidea4},
\end{align}
and
\begin{align}
r^{-k-1/2}\abs{J_{4}(r)}&\leq \frac{1}{R^2 k}\int_R^{2R}   \rho|\phi(\rho,c\pm i\eps)-\phi(\rho,c)|  \rho^{-k-1/2}  \dd\rho
 \lesssim_{\varkappa,R} \frac{\eps}{r_c^{2+\varkappa}} k r_c^{-k-1/2} \label{eq:noidea5}.
\end{align}
We then collect \eqref{eq:noidea1}-\eqref{eq:noidea5} to deduce that
\begin{align*}
\norm{r^{-k-1/2}g_\eps(r,r_c)}_{L^\infty(R,\infty)} \lesssim_{\varkappa,R} \frac{\eps}{r_c^{2+\varkappa}} k^2 r_c^{-k-1/2},
\end{align*}
since $r_c\leq 1$. Since, in this region, $g_\eps$ and $P(r,c\pm i\eps)-P(r,c)$ satisfy the same estimates, the proof for $j=0$ is over,
while the case $j=1$ follows again as in the previous lemma.
\end{proof} 

\subsubsection{The case \texorpdfstring{$r_c> 1$}{TEXT} }
We now turn our attention to the case $r_c>1$. Again, we will split in different cases. The proofs are similar as in the previous
section, so we will only highlight the main differences. Since $r_c>1$,  \eqref{eq:FPbound} and \eqref{eq:FPconvest1}  entail the following estimates in the case $r_c> 1$ (notice that the splitting of the interval $(0,\infty)$ slightly differs from the case $r_c\leq 1$):
\medskip 

\noindent $\diamond$ $r\in (0,\zeta)$: 

\begin{align}
&r^{k-1/2} |(r\de_r)^jP(r,c\pm i \eps)| \lesssim_\varkappa k^{1+j} r_c^{k-1/2}  \frac{r_c^\varkappa}{r^\varkappa},\qquad j=0,1,  \label{eq:SECOND1}\\
&r^{k-1/2} |(r\de_r)^j(P(r,c\pm i \eps)-P(r,c))| \lesssim_\varkappa \eps k^{2+j} r_c^{k-1/2}  \frac{r_c^{2+\varkappa}}{r^\varkappa},\qquad j=0,1.\label{eq:SECOND3}
\end{align}

\medskip

\noindent $\diamond$ $r\in (\zeta,r_c)$:

\begin{align}\label{eq:SECOND5}
r^{k-1/2} |(r\de_r)^j(P(r,c\pm i \eps)-P(r,c))| \lesssim_{\varkappa,\zeta} \eps k^{2+j} r_c^{k-1/2}  r_c^{2+\varkappa}, \qquad j=0,1.
\end{align}

\medskip

\noindent $\diamond$ $r\in (r_c,R r_c)$:

\begin{align}\label{eq:SECOND6}
r^{-k-1/2} |(r\de_r)^j(P(r,c\pm i \eps)-P(r,c))| \lesssim_{\varkappa,R} \eps k^{2+j} r_c^{-k-1/2}   r_c^{2}, \qquad j=0,1.
\end{align}

\medskip

\noindent $\diamond$ $r\in (R r_c,\infty)$: 

\begin{align}
&r^{-k-1/2} |(r\de_r)^jP(r,c\pm i \eps)| \lesssim_\varkappa k^{1+j} r_c^{-k-1/2}  \frac{r^\varkappa}{r_c^\varkappa}, \qquad j=0,1\label{eq:SECOND7}\\
&r^{-k-1/2} |(r\de_r)^j(P(r,c\pm i \eps)-P(r,c))| \lesssim_\varkappa \eps k^{2+j} r_c^{-k-1/2}  r^\varkappa r_c^{2-\varkappa}, \qquad j=0,1. \label{eq:SECOND9}
\end{align}

We begin with the case $r\in (0,\zeta)$.
\begin{lemma}\label{lem:CONV4}
Let $j=0,1$. There exists a universal constant $\zeta \in (0,1/4)$ such that, for all $r_c > 1$ and $\varkappa\in (0,1)$, there hold
\begin{align*}
\norm{r^{k-1/2}(r\de_r)^jP(r,c\pm i\eps)}_{L^\infty(0,\zeta)} \lesssim_{\varkappa,\zeta}   \eps r_c^{2+\varkappa} k^{1+j} r_c^{k-1/2}.
\end{align*}
and
\begin{align}\label{eq:CONV4}
\norm{r^{k-1/2}(r\de_r)^j\left(P(r,c\pm i\eps) - P(r,c)\right)}_{L^\infty(0,\zeta)} \lesssim_{\varkappa,\zeta}   \eps r_c^{2+\varkappa} k^{2+j} r_c^{k-1/2}.
\end{align}
\end{lemma} 

\begin{proof}[Proof of Lemma~\ref{lem:CONV4}]
The proof is similar to that of Lemma \ref{lem:CONV1}, so we only treat the case $j=0$. 
As in Lemma \ref{lem:CONV1}, the norms appearing in this lemma are a priori finite. 
For a cut-off function $\chi = \chi(r/\zeta)$, the function
$$g_\eps(r,r_c) = \chi(r/\zeta) \left[\phi(r,c\pm i\eps) - \phi(r,c)\right]$$ 
can be written as
\begin{align}
g_\eps(r,r_c) & = -\int_0^\infty \mathcal{L}(r,\rho) \frac{\beta(\rho)}{u(\rho) - c \mp i\eps}g_\eps(\rho,r_c) \dd\rho - \int_0^\infty \mathcal{L}(r,\rho) \frac{\pm i \eps \beta(\rho)}{(u(\rho)-c\mp i\eps)(u(\rho)-c)}\chi(\rho/\zeta) \phi(\rho,c)\dd\rho\notag \\ 
& \quad - \int_0^\infty \mathcal{L}(r,\rho) \left(\frac{2}{\zeta}  (\de_\rho\chi) (\de_\rho(\phi(\rho,c\pm i\eps)-\phi(\rho,c))) + \frac{1}{\zeta^2}
(\de_{\rho\rho}\chi)(\phi(\rho,c\pm i\eps)-\phi(\rho,c)) \right) \dd\rho\notag \\
& = \sum_{\ell=1}^4 J_\ell \label{eq:4LAST1}. 
\end{align}
Since $\beta(\rho)\lesssim 1$ and $|u(\rho) - c \mp i\eps|\gtrsim 1$ in this regime, we choose $\zeta\ll 1$ to deduce that
\begin{align}
r^{k-1/2}\abs{J_{1}(r)} & \leq \frac12\norm{r^{k-1/2}g_\eps(r,r_c)}_{L^\infty(0,\zeta)}\label{eq:4LAST2}.
\end{align}
For to $J_2$, we use \eqref{eq:SECOND1} to obtain
\begin{align}
r^{k-1/2}\abs{J_{2}(r)}   
&\lesssim \frac{\eps}{k}\int_0^{\zeta}\rho \abs{\rho^{k-1/2} P(\rho,c)\dd\rho}  \dd\rho 
\lesssim_{\varkappa,\zeta}\eps r_c^\varkappa r_c^{k-1/2} \label{eq:4LAST3},
\end{align}
For $J_3$, we preliminary note that
$$
\de_\rho(\phi(\rho,c\pm i\eps)-\phi(\rho,c))= u'(\rho)(P(\rho,c\pm i\eps)-P(\rho,c))+(u(\rho)-u(r_c))\de_\rho(P(\rho,c\pm i\eps)-P(\rho,c))\mp i \eps \de_\rho P(\rho,c\pm i\eps),
$$
so that from \eqref{eq:SECOND1}-\eqref{eq:SECOND3} we obtain the pointwise bound for $\rho\in (0,\zeta)$
\begin{align*}
|\rho\de_\rho(\phi(\rho,c\pm i\eps)-\phi(\rho,c))|\lesssim_{\varkappa,\zeta}  \eps k^3 \frac{r_c^{k-1/2}}{\rho^{k-1/2}}  \frac{r_c^{2+\varkappa}}{\rho^\varkappa}.
\end{align*}
Hence, arguing as above,
\begin{align}
r^{k-1/2}\abs{J_{3}(r)}  &\leq \frac{1}{\zeta k}\int_0^{\zeta}   |\rho\de_\rho(\phi(\rho,c\pm i\eps)-\phi(\rho,c))|  \rho^{k-1/2}  \dd\rho\lesssim_{\varkappa,\zeta} 
\eps r_c^{2+\varkappa} k^2 r_c^{k-1/2} \label{eq:4LAST4}.
\end{align}
On the other hand, due to \eqref{eq:SECOND1} and \eqref{eq:SECOND3} we obtain the pointwise bounds
$$
|\phi(\rho,c\pm i\eps)-\phi(\rho,c)|\rho^{k-1/2}\lesssim_{\varkappa,\zeta}
\eps k^2 r_c^{k-1/2}  \frac{r_c^{2+\varkappa}}{\rho^\varkappa},
$$
and therefore
\begin{align}
r^{k-1/2}\abs{J_{4}(r)}  &\leq \frac{1}{\zeta^2 k}\int_0^{\zeta}   |\rho(\phi(\rho,c\pm i\eps)-\phi(\rho,c))|  \rho^{k-1/2}  \dd\rho
\lesssim_{\varkappa,\zeta} 
\eps r_c^{2+\varkappa} k r_c^{k-1/2} \label{eq:4LAST5},
\end{align}
Hence, collecting \eqref{eq:4LAST1}-\eqref{eq:4LAST5}, we arrive at
\begin{align*}
\norm{r^{k-1/2}g_\eps(r,r_c)}_{L^\infty(0,\zeta r_c)} \lesssim_{\varkappa,\zeta}   \eps r_c^{2+\varkappa} k^2 r_c^{k-1/2}.
\end{align*}
Since, in this region, $g_\eps$ and $P(r,c\pm i\eps)-P(r,c)$ satisfy the same estimates, the proof is over.
\end{proof} 
Also in this case, the regime $r\in (\zeta, Rr_c)$ for any $R\geq1$ is already contained in \eqref{eq:SECOND5} and \eqref{eq:SECOND6}.
For $r\in (Rr_c,\infty)$, we follow the ideas in Lemma \ref{lem:CONV3}.

\begin{lemma} \label{lem:CONV6}
Let $j=0,1$. There exists a universal constant $R>2$ such that, for all $r_c > 1$ and  any $\varkappa\in(0,1)$,  there hold
\begin{align*}
\norm{r^{-k-1/2}(r\de_r)^jP(r,c\pm i\eps)}_{L^\infty(Rr_c,\infty)}  \lesssim_{\varkappa,R}\eps r_c^2 k^{1+j} r_c^{-k-1/2}  .  
\end{align*}
and 
\begin{align*}
\norm{r^{-k-1/2}(r\de_r)^j\left(P(r,c\pm i\eps) - P(r,c)\right)}_{L^\infty(Rr_c,\infty)}  \lesssim_{\varkappa,R}\eps r_c^2 k^{2+j} r_c^{-k-1/2}  .  
\end{align*}
\end{lemma}

\begin{proof}[Proof of Lemma~\ref{lem:CONV6}]
We use a cut-off of the form $\chi = \chi(r/Rr_c)$,
with $R>0$ to be determined, and define 
$$
g_\eps(r,r_c) = \chi(r/Rr_c) \left[\phi(r,c\pm i\eps) - \phi(r,c)\right].
$$ 
Hence,
\begin{align}
g_\eps(r,r_c) & = -\int_0^\infty \mathcal{L}(r,\rho) \frac{\beta(\rho)}{u(\rho) - c \mp i\eps}g_\eps(\rho,r_c) \dd\rho - \int_0^\infty \mathcal{L}(r,\rho) \frac{\pm i \eps \beta(\rho)}{(u(\rho)-c\mp i\eps)(u(\rho)-c)}\chi(\rho/Rr_c) \phi(\rho,c)\dd\rho\notag \\ 
& \quad - \int_0^\infty \mathcal{L}(r,\rho) \left(\frac{2}{Rr_c} \de_\rho\chi\de_\rho(\phi(\rho,c\pm i\eps) - \phi(\rho,c)) - \frac{1}{R^2r_c^2}\de_{\rho\rho}\chi(\phi(\rho,c\pm i \eps) - \phi(\rho,c))\right) \dd\rho\notag\\
& = \sum_{\ell=1}^4 J_\ell \label{eq:4noidea1}.  
\end{align}
Since $|u'(\rho)|\approx \rho^{-3}$ as $\rho\to \infty$, in this region we have that $|u(R r_c)- u(r_c)|\gtrsim R r_c^{-2}$.
Hence, since $r_c> 1$, we can choose $R\gg 1$, independent of $r_c$, to have
\begin{align}
r^{-k-1/2}\abs{J_{1}(r)} & \lesssim \frac{1}{k} \int_{Rr_c}^\infty \frac{\rho \beta(\rho)}{|u(\rho)-u(r_c)\mp i\eps|} \rho^{-k-1/2} |g_\eps(\rho,r_c)|\dd \rho\notag\\
&\lesssim  \frac{r_c^2\brak{Rr_c}^{-6}}{R}\norm{r^{-k-1/2}g_\eps(r,r_c)}_{L^\infty(Rr_c,\infty)}\leq  \frac12\norm{r^{-k-1/2}g_\eps(r,r_c)}_{L^\infty(R r_c,\infty)}\label{eq:4noidea2}.
\end{align}
For $J_2$, from \eqref{eq:SECOND7} we infer that, 
\begin{align}
r^{-k-1/2}|J_2(r)|&
\lesssim \frac{\eps}{k} \int_{Rr_c}^\infty \rho \beta(\rho) \rho^{-k-1/2} |P(\rho,c)|\dd \rho\lesssim_{R,\varkappa}\eps r_c^{-k-1/2} \label{eq:4noidea3}.
\end{align}
The terms involving $\de_\rho\chi$ and $\de_{\rho\rho}\chi$ are estimated similarly as in Lemma \ref{lem:CONV1}, 
except for the weight $r^{-k-1/2}$, and the fact that $\de_\rho\chi$ and $\de_{\rho\rho}\chi$
are supported in interval $[Rr_c,2Rr_c]$. 

Regarding $J_3$ and $J_4$,  using that $|u'(\rho)|\approx \rho^{-3}$  and 
\eqref{eq:SECOND7}-\eqref{eq:SECOND9} for $\rho\in (Rr_c,2Rr_c)$ we obtain
\begin{align*}
|\rho\de_\rho(\phi(\rho,c\pm i\eps)-\phi(\rho,c))|\lesssim_{\varkappa,R}\eps   k^3 \frac{r_c^{-k-1/2}}{\rho^{-k-1/2}} ,
\end{align*}
and
\begin{align*}
|\phi(\rho,c\pm i\eps)-\phi(\rho,c)|    
|\lesssim_{\varkappa,R}  \eps k^2 \frac{r_c^{-k-1/2}}{\rho^{-k-1/2}},
\end{align*}
so that
\begin{align}
r^{-k-1/2}\abs{J_{3}(r)}  &\leq \frac{1}{Rr_c k}\int_{Rr_c}^{2Rr_c}   |\rho\de_\rho(\phi(\rho,c\pm i\eps)-\phi(\rho,c))|  \rho^{-k-1/2}  \dd\rho
 \lesssim_{\varkappa,R} \eps k^2 r_c^{-k-1/2} \label{eq:4noidea4},
\end{align}
and
\begin{align}
r^{-k-1/2}\abs{J_{4}(r)}&\leq \frac{1}{R^2r_c^2 k}\int_{Rr_c}^{2Rr_c}   \rho|\phi(\rho,c\pm i\eps)-\phi(\rho,c)|  \rho^{-k-1/2}  \dd\rho
 \lesssim_{\varkappa,R} \eps k r_c^{-k-1/2} \label{eq:4noidea5}.
\end{align}
We then collect \eqref{eq:4noidea1}-\eqref{eq:4noidea5} to deduce that
\begin{align*}
\norm{r^{-k-1/2}g_\eps(r,r_c)}_{L^\infty(Rr_c,\infty)} \lesssim_{\varkappa,R}\eps k^2 r_c^{-k-1/2},
\end{align*}
since $r_c> 1$. Since, in this region, $g_\eps$ and $P(r,c\pm i\eps)-P(r,c)$ differ by a factor proportional to $r_c^2$, the proof is over.
\end{proof} 

\begin{remark}\label{rem:sharprc}
The proof carries over to optimize the weights for $r_c\de_{r_c}\tP$, for $j=0,1$. Indeed,  from \eqref{eq:Rayleigh}
we see that
\begin{align*}
\Ray_z r_c\de_{r_c}\phi=\frac{\beta(r)u'(r_c)r_c}{u(r)-z}\phi.
\end{align*}
The extra singularity on the right-hand side is in fact innocuous. Indeed, Lemmas  \ref{lem:CONV1},  
\ref{lem:CONV3}, \ref{lem:CONV4} and \ref{lem:CONV6} are
relevant in regions that are far from the critical layer, in which $|u(r)-z|\gtrsim \min\{r_c^2,r_c^{-2}\}$. 
Since $|u'(r_c)|r_c\approx\min\{r_c^2,r_c^{-2}\}$ as well, the contribution of the singular denominator cancels. Note that
for $r\de_r r_c\de_{r_c}\tP$ we can simply differentiate the Green's function $\mathcal{L}$. However, we can only 
deduce information away from the critical layer, due to \eqref{eq:rdr_rcdrctPconv}.
We therefore arrive at the conclusion of the proof of Theorem \ref{thm:phik}.
\end{remark}

\subsection{Further properties for the real solution}

In this section, we prove Theorem \ref{thm:realphik}. 

\subsubsection{The function \texorpdfstring{$Q_0$}{TEXT} and its properties}
With $Q_0$ defined as in \eqref{eq:Q:def}.
From \eqref{eq:Rayleigh} we have that  $Q_0$ obeys the second order equation (note that $Q_0$ is $C^2$ for $r < r_c$), 
\begin{align}
r   (u(r)-c) \partial_{rr} Q_0 + \Big( 2 r   u'(r) + (1- 2k)   (u(r)-c) \Big) \partial_r Q_0 =   2 (k+1)u'(r)   Q_0,
\label{eq:Q:alpha:correct}
\end{align}
or, similarly, 
\begin{align}
\partial_r (r^2  (u(r)-c) \partial_{r} Q_0) + \Big(   r^2 u'(r)  -(1+ 2k) r (u(r)-c) \Big) \partial_r Q_0 
=   2(k+1) r u'(r)  Q_0.
\label{eq:Q:alpha0:correct}
\end{align}
Note that we may rewrite \eqref{eq:Q:alpha:correct} as
\begin{align}
(u(r)-c) \Big( r \partial_{rr} Q_0 + (1-2k) \partial_r Q_0 \Big) = - 2 u'(r) \Big( r\partial_r Q_0 - (k+1) Q_0 \Big).
\label{eq:Q:alpha:correct:2}
\end{align}
Also,  according to  \eqref{eq:solutionk1a}, \eqref{eq:bohboh} and \eqref{eq:PeqnBC}, we have that 
\begin{align}
Q_0(r_c,c) = 1 \qquad \mbox{and} \qquad \partial_r Q_0(r_c,c) = \left(k+1\right) r_c^{-1}.
\label{eq:good:Q:at:rc} 
\end{align}
Observe also that
 that \eqref{eq:Q:alpha:correct} and \eqref{eq:good:Q:at:rc} imply
\begin{align}
\partial_{rr} Q_0(r_c,c) = \frac{(4k-1) \partial_r Q(r_c,c)}{3 r_c} = \frac{(4 k-1)(k+1)}{3 r_c^{2}}.
\label{eq:good:Q:rr:at:rc}
\end{align}
We begin by proving \eqref{eq:Q0:properties}. The fact that $Q_0(r,c)>0$ for all $r\in (0,r_c]$ is clear from \eqref{eq:integralP} and 
continuity in $r$, 
so we only need to 
prove monotonicity. From \eqref{eq:good:Q:at:rc} it follows that  $\partial_r Q_0(r_c,c)  > 0$. Assume for the sake of contradiction 
that there exists a {\em first} point $r_* \in (0,r_c)$ (meaning closest to $r_c$) such that $\partial_r Q_0(r_*,c) = 0$. 

Note that on $(r_*,r_c]$ we have $\partial_r Q_0 > 0$, and thus also $r^2 (u(r)-c) \partial_r Q_0(r,c) > 0$ on $(r_*,r_c)$. 
By the minimality of $r_*$, we have that $r^2 (u(r)-c) \partial_r Q_0(r,c)$ attains the value $0$ for the first time 
(from $r_c$ towards 0), and thus we must have
\begin{align}
\partial_r \big( r^2 (u(r)-c) \partial_r Q_0(r,c) \big) \Big|_{r=r_*} \geq 0. 
\label{eq:dr:Q:minimality}
\end{align}
On the other hand, evaluating \eqref{eq:Q:alpha0:correct} at $r=r_*$, and using that $\partial_r Q_0(r_*,c) = 0$ we obtain that 
\begin{align*}
\partial_r \big( r^2 (u(r)-c) \partial_r Q_0(r,c) \big) \Big|_{r=r_*} = \big(2( k+1) r_* u'(r_*) \big) Q_0(r_*,c).
\end{align*}
Since $u$ is monotone decreasing, we immediately arrive at a contradiction with \eqref{eq:dr:Q:minimality}, 
and \eqref{eq:Q0:properties} follows.

To establish \eqref{eq:Q0}, we rely on the following lemma.

\begin{lemma}\label{lem:B:r}
For $k\geq 2$, define the function let $B_0$ be defined as in \eqref{eq:B0stima}. Then we have
$B_0(r,c) \geq 0$
for all $r \in (0,r_c]$. As a consequence,
\begin{align}
r \partial_{rr} Q_0 + (1-2k) \partial_r Q_0 < 0
\label{eq:B:positive:corollary}
\end{align} 
on $(0,r_c]$. Moreover $B_0(r_c,c)=0$ and the upper bound stated in \eqref{eq:B0up} holds.
\end{lemma}
\begin{proof}[Proof of Lemma~\ref{lem:B:r}]
Note that according to \eqref{eq:good:Q:at:rc} we have that $B_0(r_c,c) = 0$. Then, upon differentiating, 
appealing to \eqref{eq:Q:alpha:correct:2}, we obtain that
\begin{align*}
(k-1) \partial_r Q_0 + \partial_r B_0 = (2 k -1) \partial_r Q_0 - r \partial_{rr} Q_0 = \frac{- 2 u'(r)}{u(r)-c} B_0.
\end{align*}
holds on $(0,r_c)$. Switching the order of the terms, we arrive at 
\begin{align*}
- \partial_r \Big( (u(r)-c)^2 B_0 \Big) =   (k-1) (u(r)-c)^2 \partial_r Q_0
\end{align*}
which we may integrate from $r$ to $r_c$ and obtain
\begin{align}
(u(r)-c)^2 B_0(r,c) = (k-1) \int_r^{r_c} (u(s)-c)^2 \partial_r Q_0(s,c) \dd s \geq 0,
\label{eq:B:integral}
\end{align}
appealing to that fact that $\partial_r Q_0(r,c)>0$. Also, \eqref{eq:B:positive:corollary} follows immediately for $r\in(0,r_c)$, while
for $r=r_c$ it requires \eqref{eq:good:Q:at:rc} and \eqref{eq:good:Q:rr:at:rc}. Finally, the fact that $B_0(r_c,c)=0$ is 
a simple computation using \eqref{eq:good:Q:at:rc}. Now, from \eqref{eq:B0stima}, the monotonicity of $Q_0$ and the positivity
of $B_0$, we have the upper bound in \eqref{eq:Q0} for $r\de_rQ_0$ as
\begin{align*}
0\leq \partial_r Q_0(r,c)\leq (k+1) \frac{Q_0(r,c)}{r}\leq (k+1) \frac{Q_0(r_c,c)}{r}=\frac{k+1}{r}
\end{align*}
 while from \eqref{eq:B:integral} and the fact that $u$ is decreasing 
we deduce that
\begin{align}
(u(r)-c)^2 B_0(r,c) \leq (k-1) (u(r)-c)^2 \int_r^{r_c} \partial_r Q_0(s,c) \dd s \leq  (k^2-1) (u(r)-c)^2\log(r_c/r),
\end{align}
and the proof is complete.
\end{proof}

Now, from \eqref{eq:B:positive:corollary}, we obtain that 
\begin{align*}
 \partial_r \log( \partial_r Q_0) \leq  \partial_r (\log r^{2k-1}) 
\end{align*}
holds on $(0,r_c]$. Here we used that on $(0,r_c]$ we have $\partial_r Q_0(r,c)>0$.
Integrating the above from $r$ to $r_c$, we infer
\begin{align}
\partial_r Q_0(r,c) \geq \partial_r Q_0(r_c,c) \frac{r^{2k-1}}{r_c^{2k-1}}.
\label{eq:dr:Q:lower:bound}
\end{align}
Inserting the above information in \eqref{eq:B:integral}, which holds for all $r>0$, we infer that
\begin{align}
(k+1) Q_0(r,c) - r \partial_r Q_0(r,c) =  B_0(r,c) &\geq \frac{(k-1)\partial_r Q_0(r_c,c)}{r_c^{2k-1} (u(r)-c)^2} \int_r^{r_c} (u(s) -c)^2   s^{2k-1} \dd s. \label{ineq:B0lwbd}
\end{align}
In the above inequality, we drop the term $- r \partial_r Q_0(r,c)<0$, and moreover (again by monotonicity) we 
have that $Q_0(r,c)$ converges, as $r\to 0$, to some limit $Q_0(0,c)$. 
Therefore, we have a lower bound on $Q_0(0,c)$ which is
\begin{align}
Q_0(0,c)  \geq \frac{(k-1)\partial_r Q_0(r_c,c)}{r_c^{2k-1} (k+1) (u(0)-c)^2} \int_0^{r_c} (u(s) -c)^2   s^{2k-1} \dd s.
\label{eq:Q:0:lower}
\end{align}
We then arrive at the lower bound in \eqref{eq:Q0} through a further use of 
the boundary condition \eqref{eq:good:Q:at:rc}. Regarding the upper bound, going back to \eqref{eq:dr:Q:lower:bound} and re-arranging, we have an explicit lower bound on $\de_r Q_0$, namely
\begin{align*}
\de_rQ_0(r_c,c)r^{2k-1} < r_c^{2k-1} \de_rQ_0 (r,c).  
\end{align*}
By the fundamental theorem of calculus and using this lower bound,  for any $s\in (0,r_c)$ we have
\begin{align*}
Q_0(r_c,c) - Q_0(s,c) = \int_s^{r_c} \de_r Q(r,c) \dd r > \de_r Q_0(r_c,c) r_c^{1-2k} \int_s^{r_c} r^{2k-1} \dd r 
= \frac{\de_r Q_0(r_c,c)}{2k r_c^{2k-1}}\left[r_c^{2k}-s^{2k}\right].
\end{align*}
Hence, passing to the limit as $s\to 0$, we get the upper bound 
\begin{align*}
Q_0(r_c,c) - \frac{\de_r Q_0(r_c,c)r_c}{2k} > Q_0(0,c). 
\end{align*}
The left-hand side however, using \eqref{eq:good:Q:at:rc} is
\begin{align*}
Q_0(r_c,c) - \frac{\de_r Q_0(r_c,c)r_c}{2k} = 1 - \frac{k+1}{2k} =  \frac{k-1}{2k},
\end{align*}
from which the upper bound in \eqref{eq:Q0} follows. The lower bound in \eqref{eq:B0up} is proven in the next lemma.
\begin{lemma}\label{lem:b0sdfas}
There is a constant $\delta_0$ depending only on $u$ such that for all $r_c \in (0,\infty)$, there holds for all $\abs{r-r_c} < \delta_0 r_c$ (uniformly in $k$ and $r_c$), 
\begin{align*}
B_0(r,c) \gtrsim k^2 \frac{r^{2k-1}}{r_c^{2k}}(r_c-r), \quad\quad \textup{ for } r \leq r_c. 
\end{align*}
\end{lemma}
\begin{proof}[Proof of Lemma~\ref{lem:b0sdfas}]
Since we are assuming that $\abs{r-r_c} < \delta_0 r_c$, for $\delta_0$ sufficiently small (depending only on $u''$), we have from Taylor's theorem that $u(r) - u(r_c) \approx u'(r_c)(r-r_c)$, and hence from  \eqref{ineq:B0lwbd}  there holds 
\begin{align*}
B_0(r,c) 
\gtrsim \frac{k^2 - 1}{r_c^{2k} \abs{r-r_c}^2} \int_r^{r_c} (s-r_c)^2 s^{2k-1} \dd s. 
\end{align*}
This integral is explicitly computed via integration by parts: 
\begin{align*}
\int_r^{r_c} (s-r_c)^2 s^{2k-1} \dd s 
 & = -\frac{1}{2k}\abs{r-r_c}^2r^{2k} - \frac{2}{2k}\int_r^{r_c} (s-r_c) s^{2k} \dd s \\ 
& = -\frac{1}{2k}\abs{r-r_c}^2r^{2k} + \frac{2}{2k(2k+1)} (r-r_c) r^{2k+1} + \frac{2}{2k(2k+1)(2k+2)}\left(r_c^{2k+2} - r^{2k+2}\right). 
\end{align*}
By Taylor's theorem, there exists some $\zeta \in (r,r_c)$ such that 
\begin{align*}
r_c^{2k+2} - r^{2k+2} = (2k+2)r^{2k+1}(r_c-r) + \frac{1}{2}(2k+2)(2k+1)r^{2k}(r_c-r)^2 + \frac{1}{3}(2k+2)(2k+1)2k \zeta^{2k-1}(r_c - r)^3, 
\end{align*}
from which the conclusion follows since $\zeta \geq r$.
\end{proof} 

\subsubsection{The function \texorpdfstring{$Q_\infty$}{TEXT} and its properties}
Analogous to \eqref{eq:Q:alpha:correct},  $Q_\infty$ obeys the second order equation
\begin{align}
r   (u(r)-c) \partial_{rr} Q_\infty + \Big( 2 r   u'(r) + (1+2k)   (u(r)-c) \Big) \partial_r Q_\infty =   -2(k-1)  u'(r)  Q_\infty,
\label{eq:tQ:alpha:correct}
\end{align}
or
\begin{align}
\partial_r (r^2  (u(r)-c) \partial_{r} Q_\infty) + \Big(   r^2 u'(r) + (2k-1) r (u(r)-c) \Big) \partial_r Q_\infty 
=   -2(k-1) r u'(r)  Q_\infty.
\label{eq:tQ:alpha0:correct}
\end{align}
Note that we may rewrite \eqref{eq:tQ:alpha:correct} as
\begin{align}
(u(r)-c) \Big( r \partial_{rr} Q_\infty + (1+2k) \partial_r Q_\infty \Big) = - 2 u'(r) \Big( r\partial_r Q_\infty + (k-1) Q_\infty \Big).
\label{eq:tQ:alpha:correct:2}
\end{align}
As conditions at the critical layer we have
\begin{align}
Q_\infty(r_c,c) = 1, \qquad \partial_r Q_\infty(r_c,c) =- \left(k-1\right) r_c^{-1},\qquad 
\partial_{rr} Q_\infty(r_c,c) = \frac{(4 k+1)(k-1)}{3r_c^{2}}.
\label{eq:good:tQ:at:rc} 
\end{align}
Proceeding as in the previous section, it is not hard to verify \eqref{eq:tQ0:properties}.
Turning to \eqref{eq:Q0}, we have a lemma similar to Lemma \ref{lem:B:r} which we state without proof.

\begin{lemma}\label{lem:tB:r}
For $k\geq 2$, define the function let $B_\infty$ be defined as in \eqref{eq:Binfistima}. Then we have
$B_\infty(r,c) \geq 0$
for all $r\in[r_c,\infty)$. As a consequence, 
\begin{align}
r \partial_{rr} Q_\infty + (1+2k) \partial_r Q_\infty < 0
\label{eq:tB:positive:corollary}
\end{align} 
on $[r_c,\infty)$.  Moreover $B_\infty(r_c,c)=0$ and there holds the bound \eqref{eq:Binfiup}.
\end{lemma}
From \eqref{eq:tB:positive:corollary} and arguing as in the previous section, we infer that
\begin{align*}
(k-1) Q_\infty(r,c) + r \partial_r Q_\infty(r,c) =  B_\infty(r,c) &\geq- \frac{(k+1)\partial_r Q_\infty(r_c,c)r_c^{2k+1} }{(u(r)-c)^2} \int_{r_c}^{r} \frac{(u(s) -c)^2}{s^{2k+1}} \dd s.
\end{align*}
which implies the lower bound in \eqref{eq:Q0}. The upper bound follows similarly, while the lower bound for $B_\infty$ in
\eqref{eq:Binfiup} is similar to that of $B_0$. 

\section{The inhomogeneous Rayleigh problem for $k \geq 2$} \label{sec:MH}
Recall from \eqref{def:H0Hin}, \eqref{def:M}, and \eqref{def:Greens} how the reduction of order technique is used to derive two 
linearly independent homogeneous solutions (each satisfying one of the boundary conditions) $H_0$, $H_\infty$, and their Wronksian $M$, and hence the Green's function for 
\eqref{eq:inoRay2}. 
In this section we will lay out a few technical results regarding $M$, $H_0$, $H_\infty$, and of course by extension, $\cG$.
The properties of $P$ (equivalently $Q_0$ and $Q_\infty$) deduced in \S\ref{sec:HomRay} are crucial. 

\begin{lemma}[Complex integral expansion] \label{lem:Complex} 
For $\eps > 0$ and $c \pm i \eps = z\in I_\alpha$, define the following quantities,
\begin{subequations} 
\begin{align}
E(r,z) & = \frac{1}{u'(r)} \de_r\left(\frac{1}{u'(r)P(r,z)^2}\right) \\ 
R_{a,b}^\eps(z) & = \int_a^b \frac{u'(r)(u(r)-c)}{(u(r)-c)^2 + \eps^2} E(r,z) \dd r \\ 
E^{\eps}_{a,b}(z) & = \int_{a}^b \frac{\eps u'(r)}{(u(r)-c)^2 + \eps^2} E(r,z) \dd r.
\end{align}
\end{subequations}
Then, we have, for $z = c \pm i \eps$,
\begin{subequations}\label{eq:H0HinfExp}
\begin{align}
H_0(r,z) & = -\phi(s,z) \int_0^r \frac{1}{\phi^2(s,z)} \dd s = \frac{1}{u'(r)P(r,z)} - \phi(r,z)\left(R_{0,r}^\eps(z) \pm iE^\eps_{0,r}(z) \right) \\ 
H_\infty(r,z) & = \phi(s,z) \int_r^\infty \frac{1}{\phi^2(s,z)} \dd s = \frac{1}{u'(r)P(r,z)} + \phi(r,z)\left(R_{r,\infty}^\eps(z) \pm iE^\eps_{r,\infty}(z)\right). 
\end{align}
\end{subequations}
Similarly, there holds
\begin{align*}
M(z) & = \int_0^\infty \frac{1}{\phi^2(s,z)} \dd s =  R_{0,\infty}^\eps(z) \pm i E^\eps_{0,\infty}(z). 
\end{align*}
\end{lemma}
\begin{proof}[Proof of Lemma~\ref{lem:Complex}]
The lemma follows by integration by parts in the complex integral:
\begin{align*} 
\int_a^b \frac{1}{\phi(r,z)^2} \dd r & = \int_a^b \frac{1}{(u-z)^2 P(r,z)^2} dz 
 = \left.-\frac{1}{(u(r)-z) u'(r)P^2(r,z)}\right\vert_{a}^b + \int_a^b \frac{u'(r)(u-z)}{(u(r)-c)^2 + \eps^2}  E(r,z) \dd r. 
\end{align*}
Note that boundary terms vanish when $a=0$ or $b = \infty$ by the asymptotic behavior of $P$ (Theorem \ref{thm:nonvan}). 
\end{proof}
Denote the formal limits of the above quantities: 
\begin{align*}
E_{0,r}(c) & := \pi E(r_c,c) \mathbf{1}_{r_c < r}, \qquad E_{r,\infty}(c)  := \pi E(r_c,c) \mathbf{1}_{r_c > r}, \\ 
R_{0,r}(c) & := p.v. \int_0^r \frac{u'}{u-c} E(s,c) \dd s, \qquad R_{r,\infty}(c) := p.v. \int_r^\infty \frac{u'}{u-c} E(s,c) \dd s. 
\end{align*}
Here we are defining (and in the remainder of the paper as well)
\begin{align}
p.v. \int_0^\infty \frac{u'(s)}{u(s) - c}g(s,c) \dd s := \lim_{\eps' \rightarrow 0} \int_{\abs{u(s) - c} \geq \eps'} \frac{u'(s)}{u(s) - c} g(s,c) \dd s. \label{def:pv}
\end{align}
For $r<r_c$, by \eqref{eq:magic:vortex} (recall the definition of $B_0$ from Lemma \ref{lem:B:r})
\begin{align}
E(r,z) 
& = \frac{\beta(r)Q_0(r,z) + 2\frac{u'(r)}{r}\left((1+k)Q_0(r,z) - r\partial_r Q_0(r,z) \right)}{(u'(r))^3 Q_0^3(r,z) (r_c/r)^{2k-1}} = \frac{\beta(r)Q_0(r,z) + 2\frac{u'(r)}{r}B_0(r,z)}{(u'(r))^3 Q_0^3(r,z) (r_c/r)^{2k-1}}, \label{eq:Erleqrc}
\end{align}
whereas for $r > r_c$ there holds (recall Lemma \ref{lem:tB:r})
\begin{align}
E(r,z) & = \frac{\beta Q_\infty(r,z) - 2\frac{u'}{r} \left((k-1)Q_\infty(r,z) + r\partial_r Q_\infty(r,z)\right)}{(u'(r))^3 Q_{\infty}^3(r,z) \left(r/r_c\right)^{1+2k}} = \frac{\beta Q_\infty(r,z) - 2\frac{u'}{r} B_\infty(r,z)}{(u'(r))^3 Q_{\infty}^3(r,z) \left(r/r_c\right)^{1+2k}}. \label{eq:Ergeqrc}
\end{align}
Recall from Lemma \ref{lem:tB:r}  $B_\infty \geq 0$ and that $B_\infty(r_c,c) = 0$, and hence $E(r,c) < 0$ for $r \geq r_c$.
However, from Lemma \ref{lem:B:r}  $B_0 \geq  0$ and $B_0(r_c,c)$, and hence $E$ is \emph{not} sign definite for $r < r_c$. 
Finally, we note that there holds 
\begin{align*} 
E(r_c,c) = \frac{\beta(r_c)}{(u'(r_c))^3}. 
\end{align*}

\subsection{Estimates on the Wronksian} 

First, we must deduce lower bounds on the Wronskian $M(c\pm i \eps)$. Moreover, the gain as $r_c \rightarrow 0$ obtained in this lemma is crucial for deducing vorticity depletion. 
This proof requires a delicate use of the monotonicity properties deduced on $Q_0$ and $Q_\infty$ in \S\ref{sec:HomRay}. 
In particular, the lack of sign-definiteness of $E(r,c)$ on each side of the critical layer presents a complication for ruling out cancellations in the singular integrals which define $M(z)$. 

\begin{lemma}[Wronskian lower bounds] \label{lem:Wronk}
For all $z \in I_\alpha$ with $z = c \pm i \eps$ with $\eps$ sufficiently small (depending only on $\alpha$ and $k$), there holds the following lower bound, 
\begin{align}
\abs{M(c\pm i \eps)} \gtrsim_\alpha k\max(r_c^{-3},r_c^5), \label{ineq:WronkLow} 
\end{align}
and if we write $M(c \pm i 0) = R_{0,\infty}(c) \pm i E_{0,\infty}(c)$, then the following uniform convergence holds for some sufficiently small $\eta > 0$  
\begin{align*}
\abs{M(c\pm i\eps) - M(c \pm i0)} \lesssim \eps^\eta \max(r_c^{-3},r_c^5).  
\end{align*}
\end{lemma} 
\begin{proof}[Proof of Lemma~\ref{lem:Wronk}]
For $\eps > 0$ we have, 
\begin{align}
\abs{M(c\pm i\eps)}^2
& = \abs{R^\eps_{0,\infty}(c\pm i \eps)}^2 + \abs{E^{\eps}_{0,\infty}(c\pm i\eps)}^2 \mp 2 \Im E^{\eps}_{0,\infty} \overline{R_{0,\infty}}(c\pm i \eps). \label{eq:M2}
\end{align}
We first study $E_{0,\infty}^\eps(r,z)$ for $z = c \pm i \eps$. 
First, for all $\eps$ sufficiently small and all sufficiently small $\eta > 0$, we prove  
\begin{align}
E^\eps_{0,\infty}(z) 
& = \frac{\pi \beta(r_c)}{(u'(r_c))^3} + \cO \left( \eps^\eta \max(r_c^5,r_c^{-3})  \right). \label{ineq:Eupper}
\end{align}
Indeed,  write 
\begin{align}
E^\eps_{0,\infty}(z) = \int_0^\infty \frac{\eps u'(s)}{(u(s) - c)^2 + \eps^2} E(s,c) \dd s  +  \int_0^\infty \frac{\eps u'(s)}{(u(s) - c)^2 + \eps^2} \left(E(s,z)-E(s,c)\right) \dd s. \label{ineq:EupperDecomp}
\end{align}
Further decompose: 
\begin{align*}
\int_0^\infty \frac{\eps u'(s)}{(u(s) - c)^2 + \eps^2} E(s,c) \dd s & = \int_0^\infty \frac{\eps u'(s)}{(u(s) - c)^2 + \eps^2} \left(\chi_c + \chi_{\neq}\right) E(s,c) \dd s = E^{\eps;0c}_{0,\infty} + E^{\eps;0\neq}_{0,\infty}. 
\end{align*}
Further expand: 
\begin{align*}
E^{\eps;0c}_{0,\infty} & = \frac{\beta(r_c)}{(u'(r_c))^3}\int_0^\infty \frac{\eps u'(s)}{(u(s) - c)^2 + \eps^2} \chi_c \dd s +  \int_0^\infty \frac{\eps u'(s)}{(u(s) - c)^2 + \eps^2} \chi_c \left(\frac{\beta(r_c)}{(u'(r_c))^3} - E(s,c)  \right)  \dd s. 
\end{align*}
By Theorem \ref{thm:nonvan} (and \eqref{eq:Erleqrc}, \eqref{eq:Ergeqrc}, and Lemma \ref{lem:Trivrrc}),
\begin{align}
\chi_c\abs{\frac{\beta(r_c)}{(u'(r_c))^3} - E(s,c)} \lesssim \chi_c k^2\frac{\abs{r-r_c}}{r^2 \abs{u'(r)}^2} \approx \chi_c k^2 \abs{r-r_c} \max(\frac{1}{r^4_c},r^4_c) \label{ineq:betaEvan}
\end{align}
and hence by \eqref{ineq:deltaCon}, for all $\eta > 0$ sufficiently small,   
\begin{align*}
E^{\eps;0c}_{0,\infty} & = \frac{\beta(r_c)}{(u'(r_c))^3} + \cO(\eps^\eta \max(r_c^{-3}, r_c^5)) .
\end{align*} 
Whereas, away from the critical layer there holds by \eqref{ineq:uNCL}, 
\begin{align}
\abs{E^{\eps;0\neq}_{0,\infty}} & \lesssim \eps k \max(\frac{1}{r_c^2},r_c^2) \int_0^\infty \frac{\abs{u'(s)}}{\abs{u-c}} \chi_{\neq} \max(\frac{1}{s^3},s^5) \min\left(\frac{s^{2k-1}}{r_c^{2k-1}},\frac{r_c^{2k+1}}{s^{2k+1}}\right) \dd s \\ & \lesssim  \eps k \max(\frac{1}{r_c^2},r_c^2) \max(r_c^{-3},r_c^5),  \label{ineq:E0neq}
\end{align}
which goes into the error by the definition of $I_\alpha$ \eqref{eq:Ialpha}. 
This completes the first term in \eqref{ineq:EupperDecomp}. 
To control the latter term, we first decompose 
\begin{align}
\int_0^\infty \frac{\eps u'(s)}{(u(s) - c)^2 + \eps^2} \left(E(s,z)-E(s,c)\right) \dd s & = \\
& \hspace{-6cm} \int_0^\infty \frac{\eps u'(s)}{(u(s) - c)^2 + \eps^2} \frac{\beta(s)}{(u'(s))^3} \left(\frac{1}{P^2(s,z)} - \frac{1}{P^2(s,c)}  \right) \dd s 
\\ &  \hspace{-6cm} \quad + \int_0^\infty \frac{\eps u'(s)}{(u(s) - c)^2 + \eps^2}  \left( \left(E(s,z)-\frac{\beta(s)}{(u'(s))^3 P^2(s,z)}\right) - \left(E(s,c)-\frac{\beta(s)}{(u'(s))^3 P^2(s,c)}\right) \right) \dd s \\ 
& = E_{0,\infty}^{\eps;10} + E_{0,\infty}^{\eps;11}.  \label{ineq:Econv}
\end{align}
Then, by Theorem \ref{thm:nonvan}, for all $z \in I_\alpha$ and all sufficiently small $\eta > 0$ (using the same argument as \eqref{ineq:E0neq}) 
\begin{align*}
\abs{E_{0,\infty}^{\eps;10}}   & \lesssim  \eps^\eta \int_0^\infty \frac{\eps \abs{u'(s)}}{(u(s) - c)^2 + \eps^2} \max(\frac{1}{s^3},s^3) \min\left(\frac{s^{2k-1}}{r_c^{2k-1}},\frac{r_c^{2k+1}}{s^{2k+1}}\right) \dd s \lesssim \eps^\eta \max(r_c^3, r_c^{-3}). 
\end{align*}
The vanishing of $E_{0,\infty}^{\eps;11}$ follows in a similar manner (due to the lack of convergence of \eqref{ineq:betaEvan} as $\eps \rightarrow 0$, we use a trick to combine convergence and uniform H\"older regularity from \eqref{ineq:betaEvan}; see e.g. the proofs of Lemmas \ref{lem:SIOConv} or \ref{lem:SIOConv} for details). 
This completes the proof of \eqref{ineq:Eupper}, which gives us a lower bound on $\abs{E_{0,\infty}^\eps(z)}$ as well as an upper bound on its contribution to the error in \eqref{eq:M2}. 

It is clear from \eqref{ineq:Eupper} that we will need to make a detailed analysis on $R_{0,\infty}(z)$ to obtain \eqref{ineq:WronkLow}. 
Indeed, $\beta(r_c) (u'(r_c))^{-3}$ potentially goes to zero very rapidly as $r_c \rightarrow \infty$ unless we impose stringent lower bounds on $\beta$
as $r_c \rightarrow 0$ (note we wish to include the case $\beta(r) \sim \e^{-r^2}$). Moreover, we would like to gain the power of $k$ as well.   

As above, we divide $R_{0,\infty}^\eps$ based on the decompositions of $E_{0,\infty}^\eps$ in \eqref{eq:Erleqrc} and \eqref{eq:Ergeqrc}: 
\begin{align}
R_{0,\infty}(c\pm i \eps) & = \int_0^\infty \frac{(u-c) u'}{(u-c)^2 + \eps^2} \frac{\beta(s)}{(u'(s))^3 P^2(s,c\pm i \eps)} \dd s \nonumber \\ 
& \quad + \int_0^\infty \frac{(u-c) u'}{(u-c)^2 + \eps^2} \left(E(s,c\pm i \eps) - \frac{\beta(s)}{(u'(s))^3 P^2(s,c\pm i \eps)}\right) \dd s \nonumber \\ 
& = R_{0,\infty}^{(1)}(c\pm i \eps) + R_{0,\infty}^{(2)}(c\pm i \eps). \label{eq:R0infty_split}
\end{align}
Denote the formal limit of the first term as 
\begin{align*}
R_{0,\infty}^{(1)}(c) = p.v.\int_0^\infty \frac{u'}{u-c} \frac{\beta(s)}{(u'(s))^3 P^2(s,c)} \dd s.
\end{align*}
Sub-divide via the critical layer: 
\begin{align*}
R_{0,\infty}^{(1)}(c\pm i \eps) & = \int_0^\infty \frac{(u-c) u'}{(u-c)^2 + \eps^2} \left(\chi_c + \chi_{\neq}\right) \frac{\beta(s)}{(u'(s))^3 P^2(s,c\pm i \eps)} \dd s = R_{0,\infty}^{(1,c)} + R_{0,\infty}^{(1,\neq)}; 
\end{align*} 
with analogous definitions also for the $\eps = 0$ limit. 
For $R_{0,\infty}^{(1,\neq)}$ it is straightforward to verify the following from Theorem \ref{thm:nonvan}, \eqref{ineq:uSIONCL}, Lemma \ref{lem:SIOConv}, and  $\beta(r) \lesssim \brak{r}^{-6}$ from Lemma \ref{lem:BasicVort}, for all $z \in I_\alpha$ an all $\eta > 0$ sufficiently small (note that the integration gains $k^{-1}$) 
\begin{subequations} \label{ineq:R1Convneq} 
\begin{align}
\abs{R_{0,\infty}^{(1,\neq)}(c\pm i \eps)} & \lesssim \max(r_c^{-3},r_c^3) \\ 
\abs{R_{0,\infty}^{(1,\neq)}(c\pm i \eps) - R_{0,\infty}^{(1,\neq)}(c)} & \lesssim \eps^\eta \max(r_c^{-3},r_c^3). 
\end{align} 
\end{subequations} 
Similarly, Theorem \ref{thm:nonvan} and Lemma \ref{lem:SIOConv} imply that for $\eta$ sufficiently small there holds: 
\begin{subequations} \label{ineq:R1ConvCL}  
\begin{align}
\abs{R_{0,\infty}^{(1,c)}(c\pm i \eps)} & \lesssim \max(r_c^{-3},r_c^3) \\ 
\abs{R_{0,\infty}^{(1,c)}(c\pm i \eps) - R_{0,\infty}^{(1,c)}(c)} & \lesssim  \eps^\eta \max(r_c^{-3},r_c^3). 
\end{align}
\end{subequations} 
Next, we consider the convergence and uniform estimates of the second term in \eqref{eq:R0infty_split}, which we write as (in particular, note that since $B_0$ and $B_\infty$ are Lipschitz continuous and vanish at the critical layer by Theorem \ref{thm:nonvan}), 
\begin{align*}
R_{0,\infty}^{(2)}(c\pm i \eps) & =  \int_0^{r_c} \frac{(u-c) u'}{(u-c)^2 + \eps^2} \left(2 \frac{B_0(s,z)}{(u'(s))^2 s P^3(s,z)} \right) \dd s - \int_{r_c}^\infty \frac{(u-c) u'}{(u-c)^2 + \eps^2} \left(2 \frac{B_\infty(s,z)}{(u'(s))^2 s P^3(s,z)} \right) \dd s. 
\end{align*}
First, by following the same argument as that applied to $R_{0,\infty}^{(1)}(c\pm i \eps)$ we deduce the analogous properties 
\begin{subequations} \label{ineq:R2convBd} 
\begin{align}
\abs{R_{0,\infty}^{(2)}(c\pm i \eps)} & \lesssim k \max(r_c^{-3},r_c^5) \\ 
\abs{R_{0,\infty}^{(2)}(c\pm i \eps) - R_{0,\infty}^{(2)}(c)} & \lesssim  \eps^\eta \max(r_c^{-3},r_c^5). 
\end{align} 
\end{subequations} 
Next, we establish the lower bound on $R_{0,\infty}^{(2)}$. 
From Theorem \ref{thm:realphik} we make the crucial observation that $R_{0,\infty}^{(2)}(c)$ is \emph{strictly negative}. 
Further, the quantitative lower bounds on $B_0$ and $B_\infty$ near the critical layer will allow us to deduce lower bounds. 
Using the negative definite signs and Theorem \ref{thm:realphik} (and Lemma \ref{lem:BasicVort}): 
\begin{align*}
R_{0,\infty}^{(2)}(c)& \lesssim k^2\int_{(1-\delta_0)r_c }^{r_c} \frac{u'}{u-c} \frac{1}{(u'(s))^2 s} \frac{s^{4k-2}}{r_c^{4k-1}} (r_c-s) \dd s \lesssim -k^2 \int_{(1-\delta_0) r_c}^{r_c} \frac{1}{(u'(s))^2} \frac{s^{4k-3}}{r_c^{4k-1}} \dd s  \lesssim -k \max(r_c^{-3},r_c^5),
\end{align*}
and hence by \eqref{ineq:R2convBd}, we have for all $\eps \geq 0$ sufficiently small,
\begin{subequations} \label{ineq:R2convEtc} 
\begin{align}
R_{0,\infty}^{(2)}(c) & \approx - k \max(r_c^{-3},r_c^5) \\ 
\abs{R_{0,\infty}^{(2)}(c \pm i \eps)} & \approx k \max(r_c^{-3},r_c^5). 
\end{align}
\end{subequations}

Putting together  \eqref{ineq:R2convEtc}, \eqref{ineq:R1Convneq}, \eqref{ineq:R1ConvCL}, and \eqref{ineq:Eupper}, for all small $\eta > 0$ we have, 
\begin{align*}
\abs{2\Im E^{\eps}_{0,\infty} \overline{R_{0,\infty}}(c\pm i \eps)} \lesssim \eps^\eta \max(r_c^{-3},r_c^5)^2. 
\end{align*}
Next, we use \eqref{eq:M2} to deduce \eqref{ineq:WronkLow}. For $r_c \gtrsim 1$, \eqref{ineq:Eupper} is not useful, however, \eqref{ineq:R2convEtc} together with \eqref{ineq:R1Convneq} and \eqref{ineq:R1ConvCL} imply \eqref{ineq:WronkLow} for $r_c > R$ for $R$ large enough depending only on $u$ and universal constants. 
For $r_c < R$, \eqref{ineq:R2convEtc} together with   \eqref{ineq:R1Convneq} and \eqref{ineq:R1ConvCL} imply \eqref{ineq:WronkLow} for $k > k_0$ for some 
$k_0$ depending only on $u$ and universal constants. Whereas, for $r_c < R$ and $k < k_0$, the lower bound \eqref{ineq:WronkLow} follows from \eqref{ineq:Eupper}.  
\end{proof}

\begin{lemma} \label{lem:Wronkdrc}
For $\eps \ll 1$ and all $z \in I_\alpha$ there holds the following
\begin{align*}
\abs{r_c \partial_{r_c} M(z)} \lesssim k^3 \max(r_c^{-3},r_c^{5}), 
\end{align*}
and we have that for all $z \in I_\alpha$ and $\eta > 0$ sufficiently small, 
\begin{align*}
\abs{r_c \partial_{r_c} M(c \pm i\eps) - r_c \partial_{r_c} M(c \pm i 0)} \lesssim  \eps^\eta k^3 \max(r_c^{-3},r_c^{5}), 
\end{align*}
\end{lemma} 
\begin{proof}[Proof of Lemma~\ref{lem:Wronkdrc}]
By integration by parts
\begin{align*}
\frac{1}{u'(r_c)}\partial_{r_c}M(z) & =  \int_0^\infty \frac{u'(s)}{(u-z)^2} \partial_G\frac{\chi_{c}}{u' P^2} \dd s +  \frac{1}{u'(r_c)} \int_0^\infty \frac{1}{(u-z)^2} \partial_{r_c}\frac{\chi_{\neq}}{P^2} \dd s + \int_0^\infty \frac{1}{(u-z)^3} \frac{\chi_{\neq}}{P^2} \dd s. 
\end{align*}
From Theorem \ref{thm:nonvan}, we see that $r_c \partial_{r_c} P$ essentially satisfies the same upper bounds as $P$ (up to powers of $k$), and hence the terms containing $\chi_{\neq}$ can be estimated directly using \eqref{ineq:uNCL} and Lemma \ref{lem:BasicVort}. Consider for example, the latter term (estimating as in \eqref{ineq:uNCL}):
\begin{align}
\abs{\int_0^\infty \frac{1}{(u-z)^3} \frac{\chi_{\neq}}{P^2} \dd s} & \lesssim \int_0^\infty \chi_{\neq} \min\left(\frac{s^{2k-1}}{r_c^{2k-1}},\frac{r_c^{2k+1}}{s^{2k+1}}\right)  \notag \\ 
& \hspace{-3cm} \times \left(\mathbf{1}_{r_c \leq 1} \left(\mathbf{1}_{r \leq 2}\min\left(\frac{k^3}{r_c^6},  \frac{k^3}{r^6}\right) + \mathbf{1}_{r \geq 2}\right) + \mathbf{1}_{r_c \geq 1} \left(\mathbf{1}_{r \leq 1/2} + \mathbf{1}_{r \geq 1/2} \min\left(k^3 r_c^6, k^3 r^6 \right)\right) \right) \dd s \notag \\ 
& \lesssim k^2 \max(r_c^{-5},r_c^7), \label{ineq:nCLdrcM}
\end{align}
which is consistent with the desired estimate (the additional power of $k$ is lost in the term involving $\partial_{r_c}P$; see Theorem \ref{thm:nonvan}). 
Turn next to the term involving $\chi_c$. Here, we integrate by parts again as in Lemma \ref{lem:Complex}: 
\begin{align}
 \int_0^\infty \frac{u'}{(u-z)^2} \partial_G\frac{\chi_{c}}{u' P^2} \dd r = \int_0^\infty \frac{u'}{u-z} \frac{1}{u'}\partial_r \left(\partial_G\frac{\chi_{c}}{u' P^2}\right) \dd r. \label{eq:drMexp} 
\end{align}
Note, 
\begin{align}
\frac{1}{u'}\partial_r \left(\partial_G\frac{\chi_{c}}{u' P^2}\right) & = \frac{1}{u'}\partial_r \left(\frac{\partial_G \chi_{c}}{u' P^2} - \frac{u'' \chi_c}{(u')^3 P^2}\right) -  2\frac{1}{u'}\left(\partial_r \frac{\chi_c}{u'}\right) \frac{\partial_G P}{P^3} - 2\frac{\chi_c \partial_r\partial_G P}{u' P^3} + 6 \frac{\chi_c \partial_rP \partial_G P}{u' P^4}. \label{eq:dGE}
\end{align}
Obtaining estimates on the contributions of the first two terms in \eqref{eq:drMexp} is essentially the same as the estimates made in Lemma \ref{lem:Wronk}. 
Hence, this is omitted for the sake of brevity. 
Next, consider the contribution of the term containing $\partial_r \partial_G P$. 
For this, from \eqref{eq:estdGP} we have, 
\begin{align*}
\abs{r_c u'(r_c)} \abs{\int_0^\infty \frac{1}{u-z} \left(\frac{ \chi_c \partial_r\partial_G P}{P^3} \right)  \dd r} & \lesssim k^3 \int_0^\infty \frac{\chi_c}{r_c^2 \abs{u'(r_c)}} \dd r \lesssim \frac{k^2}{\abs{r_c u'(r_c)}}.
\end{align*}
For the term in \eqref{eq:dGE} involving $r \partial_r$ we similarly use \eqref{eq:PLipdr}. 

Next, consider the problem of convergence as $\eps \rightarrow 0$. 
We define the expected limit as: 
\begin{align*}
\frac{1}{u'(r_c)} \partial_{r_c} M(c \pm i 0) & = p.v.\int_0^\infty \frac{u'}{u-c} \frac{1}{u'} \partial_r \left(\partial_G\frac{\chi_{c}}{u' P^2}\right) \dd r \mp i \pi \left(\frac{1}{u'} \partial_r \left(\partial_G\frac{1}{u'P^2}\right)\right)(r_c,c) \\ 
& \quad + \frac{1}{u'(r_c)} \int_0^\infty \frac{1}{(u-c)^2} \partial_{r_c}\frac{\chi_{\neq}}{P^2} \dd s +  \frac{1}{u'(r_c)} \int_0^\infty \frac{u'(r_c)}{(u-c)^3} \frac{\chi_{\neq}}{P^2} \dd s.
\end{align*}
Convergence of the terms involving $\partial_G$ follows from Lemma \ref{lem:SIOConv} and Theorem \ref{thm:nonvan}. 
Convergence away from the diagonal follows from Theorem \ref{thm:nonvan} along with 
\begin{align*}
\chi_{\neq}\abs{\frac{1}{(u-c \mp i \eps)^2} - \frac{1}{(u-c)^2}} & =  \chi_{\neq}\abs{ \frac{\mp 2i\eps(u-c) \pm \eps^2}{(u-c \mp i \eps)^2 (u-c)^2}} \lesssim \eps \max(r_c^{-2},r_c^2) \frac{\chi_{\neq}}{\abs{u-c}^2}, 
\end{align*}
and the the analogous 
\begin{align*}
\chi_{\neq}\abs{\frac{1}{(u-c \mp i \eps)^3} - \frac{1}{(u-c)^3}} \lesssim \eps \max(r_c^{-2},r_c^2) \frac{\chi_{\neq}}{\abs{u-c}^3}. 
\end{align*}
From there, the estimate follows as in \eqref{ineq:nCLdrcM} and the assumption that $z \in I_\alpha$. 
\end{proof} 
 
\subsection{Estimates on $H_0$ and $H_\infty$}

Next, we outline the basic estimates available on $H_0$ and $H_\infty$; see Appendix \ref{sec:H0HinfBds} for proof sketches, as most are minor refinements of 
ideas appearing in the proofs of Lemmas \ref{lem:Wronk} and \ref{lem:Wronkdrc}. 
 
\begin{lemma}[Estimates on $H_0$ and $H_\infty$] \label{lem:Hests}
Let $z = c \pm i \eps \in I_\alpha$. 
\begin{itemize} 
\item[(a)] (explicit expansions near the critical layer) In the region $\abs{r-r_c} < r_c/k$, there holds: 
\begin{subequations} \label{eq:H0HinfExp} 
\begin{align}
H_0(r,c \pm i \eps) & = \frac{1}{u'(r)P(r,c\pm i \eps)} - \phi(r,c \pm i \eps)\left(R_{0,r}(c \pm i \eps) \mp i E^\eps_{0,r}(c \pm i \eps)\right) \\ 
H_\infty(r,c \pm i \eps) & = \frac{1}{u'(r)P(r,c\pm i \eps)} + \phi(r,c \pm i \eps)\left(R_{r,\infty}(c \pm i \eps) \mp i E^\eps_{r,\infty}(c \pm i \eps)\right), 
\end{align}
\end{subequations}
and we record the following estimates: in the region $\abs{r-r_c} < r_c/k$, we write 
\begin{align*}
\frac{1}{u'(r) P(r,c\pm i \eps)} \approx \frac{1}{u'(r_c)} \approx -\max(r_c^{-1},r_c^3), 
\end{align*}
and 
\begin{subequations} \label{ineq:phiRE}
\begin{align}
\abs{\phi(r,c \pm i \eps)\left(R^\eps_{0,r}(c \pm i \eps) \mp i E^\eps_{0,r}(c \pm i \eps)\right)} & \lesssim \mathbf{1}_{r_c \leq 1}   \frac{\abs{r-r_c}}{r_c^2} \left(k + \abs{\log k\abs{\frac{r-r_c}{r_c}}} \right) \\ 
& \quad + \mathbf{1}_{r_c>1 } r_c^3 \frac{\abs{r-r_c}}{r_c}  \left(k+ \abs{\log k\abs{\frac{r-r_c}{r_c}}} \right), \label{ineq:phiREvan1} \\
\abs{\phi(r,c \pm i \eps)\left(R^\eps_{r,\infty}(c \pm i \eps) \mp i E^\eps_{r,\infty}(c \pm i \eps)\right)} & \lesssim \mathbf{1}_{r_c \leq 1}   \frac{\abs{r-r_c}}{r_c^2} \left(k+ \abs{\log k\abs{\frac{r-r_c}{r_c}}} \right) \\ 
& \quad + \mathbf{1}_{r_c>1 } r_c^3 \frac{\abs{r-r_c}}{r_c}  \left(k+ \abs{\log k\abs{\frac{r-r_c}{r_c}}} \right). \label{ineq:phiREvan2}
\end{align}
\end{subequations}
In particular, for $\abs{r-r_c} \leq r_c/k$, 
\begin{align}
\abs{H_0(r,c\pm i\eps)} + \abs{H_0(r,c\pm i\eps)} \lesssim \max(r_c^{-1},r_c^3). \label{ineq:H0HinfCLbd}
\end{align}
\item[(b)] (bounds away from the critical layer) In the region $\abs{r-r_c} \geq r_c/k$, there holds
\begin{subequations}  
\begin{align}
\abs{H_0(r,c\pm i \eps)} & \lesssim  \mathbf{1}_{r_c \leq 1} \left( \mathbf{1}_{r < r_c} \frac{r^{k+1/2}}{r_c^{k+3/2}} + \mathbf{1}_{r_c < r < 1} k  \frac{r^{k+2+1/2}}{r_c^{k+2+3/2}} + \mathbf{1}_{r>1}k \frac{r^{k+1/2}}{r_c^{k+3+1/2}}\right) \nonumber \\ 
& \quad +  \mathbf{1}_{r_c > 1} \left(\mathbf{1}_{r < 1} \frac{r^{k+1/2}}{r_c^{k-1/2}} + \mathbf{1}_{1 < r < r_c} \frac{r^{k+5/2}}{r_c^{k-1/2}} + \mathbf{1}_{r>r_c} k \frac{r^{k+1/2}}{r_c^{k-5/2}}\right), \label{ineq:H0bd} \\ 
\abs{H_\infty(r,c\pm i \eps)} & \lesssim \mathbf{1}_{r_c \leq 1} \left( \mathbf{1}_{r<r_c} k \frac{r_c^{k-3/2}}{r^{k-1/2}} + \mathbf{1}_{r_c <r < 1} \frac{r_c^{k+1/2}}{r^{k+3/2}} + \mathbf{1}_{r>1} \frac{r_c^{k+1/2}}{r^{k-1/2}} \right) \nonumber \\ 
& \quad +  \mathbf{1}_{r_c > 1} \left(\mathbf{1}_{r<1} k \frac{r_c^{k+5-1/2}}{r^{k-1/2}} + \mathbf{1}_{1 \leq r < r_c} k \frac{r_c^{k+5-1/2}}{r^{k+3/2}} + \mathbf{1}_{r_c < r} \frac{r_c^{k+5/2}}{r^{k-1/2}}\right). \label{ineq:Hinfbd}
\end{align}
\end{subequations}
\item[(c)] (convergence) Furthermore, if we define
\begin{subequations} 
\begin{align*}
H_0(r,c\pm i 0) & = \frac{1}{u'(r)P(r,c)} - \phi(r,c)\left(R_{0,r}(c) \pm i\pi E(r_c)\mathbf{1}_{r_c < r}     \right) \\ 
H_\infty(r,c\pm i 0) & = \frac{1}{u'(r)P(r,c)} + \phi(r,c)\left(R_{r,\infty}(c) \pm i\pi E(r_c)\mathbf{1}_{r_c > r} \right), 
\end{align*}
\end{subequations}
then we have that $H_0(r,c\pm i0)$ and $H_\infty(r,c\pm i 0)$ satisfy the same pointwise estimates as the $\eps > 0$ counterparts and there exists an $\eta > 0$ such that for $z \in I_\alpha$, the difference $\eps^{-\eta }\left(H_0(r,c\pm i \eps) - H_0(r,c\pm i0)\right)$ satisfies \eqref{ineq:H0HinfCLbd} and \eqref{ineq:H0bd} (and analogously for $H_\infty$). 
\end{itemize} 
\end{lemma} 

The next lemma details analogous properties on the derivatives of $H_0$ and $H_\infty$. 

\begin{lemma} \label{lem:dH}
Let $z = c\pm i \eps \in I_\alpha$. 
\begin{itemize} 
\item[(a)] (explicit expansions near critical layer) for $\partial_r$ derivatives there holds for $\abs{r-r_c} \leq r_c/k$: 
\begin{subequations} \label{eq:HderivCL} 
\begin{align}
r\partial_r H_0(r,c \pm i \eps) & = r\partial_r \left(\frac{1}{u'P(r,c\pm i \eps)}\right)  - r\partial_r\phi\left(R_{0,r}^\eps(z) \mp i E^\eps_{0,r}\right) - r u'(r) P E(r,c\pm i \eps), \\
r\partial_r H_\infty(r,c \pm i \eps) & = r\partial_r \left(\frac{1}{u'P(r,c\pm i \eps)}\right)  + r\partial_r\phi\left(R^\eps_{r,\infty}(z) \mp i E^\eps_{r,\infty}\right) - ru'(r) P E(r,c\pm i \eps);  
\end{align}
\end{subequations} 
for $\partial_G$ derivatives there holds (for $\abs{r-r_c} \leq r_c/2k$):
\begin{subequations} \label{eq:dGHs}
\begin{align} 
H_0(r,z) & =  \partial_G \left(\frac{1}{u'P(r,z)}\right)  - \left(u-z\right) (\partial_GP )\left(R_{0,r}^\eps(z) \mp i E^\eps_{0,r}\right) \notag \\ 
& \quad   - \phi \int_0^r \frac{u'(s)}{(u-z)} \partial_{G} \left(\chi_c E\right) \dd s \notag \\ 
& \quad + \frac{1}{u'(r_c)} \phi \int_0^r  \chi_{\neq}\frac{u'(s) u'(r_c)}{(u-z)^2} E(s,z) - \frac{u'(s)}{(u-z)}  \partial_{r_c}\left(\chi_{\neq} E(s,z)\right) \dd s. \label{eq:dGH0} \\ 
H_\infty(r,z) & = \partial_G \left(\frac{1}{u'P(r,c\pm i \eps)}\right)  + \left(u-z\right) (\partial_GP )\left(R_{r,\infty}^\eps(z) \mp i E^\eps_{r,\infty}\right) \notag \\ 
& \quad   + \phi \int_r^\infty \frac{u'(s)}{(u-z)} \partial_{G} \left(\chi_c E\right) \dd s \notag \\ 
& \quad   - \frac{1}{u'(r_c)} \phi \int_r^\infty  \chi_{\neq}\frac{u'(s) u'(r_c)}{(u-z)^2} E(s,z) - \frac{u'(s)}{(u-z)} \partial_{r_c} \left(\chi_{\neq} E(s,z)\right) \dd s. \label{eq:dGHinfty}
\end{align} 
\end{subequations} 
Furthermore, for $\abs{r-r_c} < r_c/k$, there holds the estimates:  
\begin{subequations} \label{ineq:dHnCL}
\begin{align}
\hspace{-2cm}\left(\abs{r\partial_r H_\infty(r,z)} + \abs{r\partial_r H_0(r,z)}\right)\mathbf{1}_{\abs{r-r_c} < r_c/k} & \lesssim \max(r_c^{-1},r_c^3)\left(k + \abs{\log \frac{k\abs{r-r_c}}{r_c}}\right) \label{ineq:drHCL} \\ 
\hspace{-2cm}  \left(\abs{\partial_G H_\infty(r,z)} + \abs{\partial_G H_0(r,z)}\right)\mathbf{1}_{\abs{r-r_c} < r_c/k} & \lesssim k\frac{\max(r_c^{-1},r_c^3)}{\abs{r_c u'(r_c)}} \label{ineq:dGHCL} \\ 
\hspace{-2cm}  \left(\abs{r \partial_r \partial_G H_\infty(r,z)} + \abs{r\partial_r\partial_G H_0(r,z)}\right)\mathbf{1}_{\abs{r-r_c} < r_c/k} & \lesssim k\frac{\max(r_c^{-1},r_c^3)}{\abs{r_c u'(r_c)}}\left(k + \abs{\log \frac{k\abs{r-r_c}}{r_c}}\right). \label{ineq:drdGHCL}
\end{align}
\end{subequations}
\item[(b)] (derivative bounds away from critical layer) For $\abs{r-r_c} \geq r_c/k$, we have that $\abs{r-r_c} \geq r_c/k$, $k^{-1} r\partial_r H_0$, $k^{-3} r_c \partial_{r_c} H_0$, $k^{-4}r\partial_r r_c \partial_{r_c} H_0$ satisfy the estimate \eqref{ineq:H0bd} whereas $k^{-1} r\partial_r H_\infty$,  $k^{-3} r_c\partial_{r_c} H_\infty$, and $k^{-4}r\partial_r r_c \partial_{r_c} H_\infty$ satisfy the estimate \eqref{ineq:Hinfbd}.  

\item[(c)] (convergence) For some $\eta > 0$, for $z \in I_\alpha$ we have the following convergence: $\eps^{-\eta}\left(H_\infty(r,c\pm i\eps) - H_\infty(r,c\pm i 0)\right)$ and $\eps^{-\eta}\left(H_0(r,c\pm i\eps) - H_0(r,c\pm i 0)\right)$ satisfy \eqref{ineq:dHnCL} and the assertions in part (b) for $\eps$ sufficiently small. 
\end{itemize} 
\end{lemma}

\section{Representation formulas and estimates on $(r\partial_r)^jf_1$, and $(r\partial_r)^j f_2$} \label{sec:repbd} 
\subsection{Recursion relations for derivatives of $f_1$ and $f_2$}
Next, our goal is to derive formulas for the derivatives of $f_1$ and $f_2$. 
As discussed in \S\ref{sec:RepBdIntro}, the first step is the iteration scheme outlined in Lemma \ref{lem:IterSchemeIntro}. 
\begin{lemma}[Iteration lemma for $\partial_G^jX$ and $\partial_G^j Y$] \label{lem:IterScheme}
Set 
\begin{align}
\mathcal{E}_0 &= 0, \quad F_0  = F(r),\quad R_0  = F_\ast(r),\quad R_0^x  = 0.  
\end{align} 
For $j \geq 0$ if we define 
\begin{align*}
F^{\pm}_{j+1} & = \partial_G F^{\pm}_j - \frac{2u''}{(u')^2}(F_j-\beta \partial_G^j Y^{\pm}) - \frac{\beta'}{u'}\partial_G^j Y^{\pm} \\
R_{j+1}^{\pm} & = \partial_G R_{j} - \frac{2u''}{(u')^2}R_{j} + 2\left(\frac{1}{4}-k^2\right)\left(\frac{u' + ru''}{r(u')^2}\right)\frac{\partial_G^j Y^{\pm}}{r^2} + \partial_{rr}\left(\frac{1}{u'}\right) \partial_r \partial_G^j Y^{\pm} \\ 
\cE_{j+1} & = \partial_G \cE_j - \frac{2u''}{(u')^2}(\cE_j - \beta \partial_G^j X) - \frac{\beta'}{u'}\partial_G^j X \\ 
R_{j+1}^{x} & = \partial_G R_{j}^x - \frac{2u''}{(u')^2}R_{j}^x + 2\left(\frac{1}{4}-k^2\right)\left(\frac{u' + ru''}{r(u')^2}\right)\frac{\partial_G^j X}{r^2} + \partial_{rr}\left(\frac{1}{u'}\right) \partial_r \partial_G^j X_j, 
\end{align*}
then \eqref{eq:RaydGjXY} holds for all $j$. 
\end{lemma}
\begin{proof}[Proof of Lemma~\ref{lem:IterScheme}]
Recall that $\partial_G$ commutes with functions of $u-c$. 
Hence, the lemma follows from the following relation: 
\begin{align} 
[\partial_G, \Ray_{\pm}] Z = - \partial_{rr}\left(\frac{1}{u'}\right) \partial_r Z + \frac{2u''}{(u')^2}Ray_{\pm}Z  
 - 2\left(\frac{1}{4}-k^2\right)\left(\frac{u'+ ru''}{r(u')^2}\right)\frac{Z}{r^2} + \left(\frac{\beta'}{u'} - \frac{2u'' \beta}{(u')^2}\right) \frac{Z}{u-z}. \label{eq:dGRayComm}
\end{align}
\end{proof}  

\begin{lemma} \label{lem:RFErecurse}
For coefficients $e_{j,\ell}, h_{j,\ell}, q_{j,\ell}$, $p_{j,\ell}$, and $r_{j,\ell}$ we have
\begin{subequations}   \label{eq:Rjp1}
\begin{align}
R_{j+1} & = \sum_{\ell=0}^{j+1} e_{j+1,\ell} \partial_G^{\ell} F_\ast + \sum_{\ell = 0}^{j} q_{j+1,\ell} \partial_r \partial_G^\ell Y^{\pm} + p_{j+1,\ell} \partial_G^\ell Y^{\pm}, \\ 
R^x_{j+1} & = \sum_{\ell = 0}^{j} q_{j+1,\ell} \partial_r \partial_G^\ell X^{\pm} + p_{j+1,\ell} \partial_G^\ell X^{\pm}, \\ 
F_{j+1}^{\pm} & = \sum_{\ell = 0}^{j+1} h_{j+1,\ell} \partial_G^\ell F + \sum_{\ell = 0}^j r_{j+1,\ell} \partial_G^\ell Y^{\pm}, \label{eq:Fjrec}\\ 
\cE_{j+1} & =  \sum_{\ell = 0}^{j} r_{j+1,\ell} \partial_G^\ell X,
\end{align}
\end{subequations} 
where the coefficients are given by the following recursion formulas: for $\ell \leq j$ (with the convention that $e_{j,-1} = h_{j,-1}= 0$), 
\begin{align*}
e_{j+1,j+1} & = 1 \\ 
e_{j+1,\ell} & = \frac{1}{u'}\partial_r e_{j,\ell} + e_{j,\ell-1} - \frac{2u''}{(u')^2} e_{j,\ell} \\ 
h_{j+1,j+1}(r) & = 1 \\ 
h_{j+1,\ell}(r) & = \frac{1}{u'} \partial_r h_{j,\ell} + h_{j,\ell-1} - \frac{2u''}{(u')^2}h_{j,\ell}, 
\end{align*} 
and for $\ell \leq j - 1$, (with the convention that $p_{j,-1} = q_{j,-1}, r_{j,-1} = 0$), 
\begin{align*} 
p_{j+1,j} & = p_{j,j-1} + 2\left(\frac{1}{4} - k^2\right)\left( \frac{u' + ru''}{r(u')^2}\right)\frac{1}{r^2} \\
p_{j+1,\ell} & = \frac{1}{u'}\partial_r p_{j,\ell} + p_{j,\ell-1} - \frac{2u''}{(u')^2} p_{j,\ell} \\
q_{j+1,j} & = q_{j,j-1} + \partial_{rr}\frac{1}{u'} \\
q_{j+1,\ell} & = \frac{1}{u'}\partial_r q_{j,\ell} + q_{j,\ell-1} - \frac{u''}{(u')^2}q_{j,\ell} \\ 
r_{j+1,j}(r) & =  r_{j,j-1} + \frac{2u''}{(u')^2}\beta - \frac{\beta'}{u'} \\
r_{j+1,\ell}(r) & = \frac{1}{u'}\partial_r r_{j,\ell} + r_{j,\ell-1} - \frac{2u''}{(u')^2} r_{j,\ell}. 
\end{align*}
Finally, the above coefficients satisfy the following estimates: for $\ell \leq j+1$ and $n \geq 0$, 
\begin{align*}
\abs{(r\partial_r)^n e_{j+1,\ell}(r)} & \lesssim_n \max(r^{-2},r^2)^{j+1-\ell} \\ 
\abs{(r\partial_r)^n h_{j+1,\ell}(r)} &\lesssim_n \max(r^{-2},r^2)^{j+1-\ell}, 
\end{align*}
and for $\ell \leq j$ and $n \geq 0$,
\begin{align*}
\abs{(r\partial_r)^n q_{j+1,\ell}(r)} &\lesssim_n \max(r^{-3},r) \max(r^{-2},r^2)^{j-\ell} \\ 
\abs{(r\partial_r)^n p_{j+1,\ell}(r)} &\lesssim_n k^2 \max(r^{-4},1) \max(r^{-2},r^2)^{j-\ell} \\
\abs{(r\partial_r)^n r_{j+1,\ell}(r)} &\lesssim_n \frac{1}{r^2 \brak{r}^4}\max(r^{-2},r^2)^{j-\ell}. 
\end{align*}
\end{lemma}
\begin{proof}[Proof of Lemma~\ref{lem:RFErecurse}]
Note
\begin{align*}
\partial_G \left(q(r) \partial_r Z\right) & = \left(\frac{1}{u'}\partial_r q\right) \partial_r Z + q(r) \partial_G \partial_rZ  = \left(\frac{1}{u'}\partial_r q - q(r) \partial_r\left(\frac{1}{u'}\right)\right) \partial_r Z + q(r) \partial_r \partial_G Z. 
\end{align*}
The rest follows from Lemma \ref{lem:BasicVort}. 
\end{proof}

\subsection{Integral operators appearing in the derivatives of $X$ and $Y$} 
Define the following operators which appear in the expression for $\partial_G^j Y^{\pm}$, 
\begin{subequations} \label{def:ZY}
\begin{align}
Z_{Y\delta}^{\pm \eps}[g] & := \pm \int_0^\infty \cG(r,s,c\pm i \eps) \frac{2 i \eps }{(u-c)^2 + \eps^2} g(s,c) \dd s \\
Z_{YS}^{\pm \eps}[g]     & :=  \int_0^\infty \cG(r,s,c\pm i \eps) \frac{(u-c)}{(u-c)^2 + \eps^2} g(s,c) \dd s \\
Z_{YG}^{\pm \eps}[ g] & :=  \int_0^\infty \cG(r,s,c\pm i \eps) g(s,c)  \dd s. 
\end{align}
\end{subequations} 
In an abuse of notation below, we use the definition 
\begin{align}
Z_{YrG}^{\pm \eps} \circ w Z_{Ya}^{\pm \eps}[g] & := \int_0^\infty \cG(r,s,c\pm i \eps) w(s) \partial_s Z_{Ya}^{\pm \eps}[g](s)  \dd s. \label{def:YrGcomp}
\end{align}
A more complicated set of operators arises in the formula for $X$ and its derivatives (using a similar abuse of notation): 
\begin{subequations} \label{def:ZX} 
\begin{align} 
Z_{X\delta}^{\eps}[g] & := \int_0^\infty B_{X\delta;\eps}^{(1)}(r,s,c) \frac{2 i\eps}{(u-c)^2 + \eps^2} g(s) \dd s \notag \\
Z_{XS}^{\eps}[g] & := \int_0^\infty \left(\int_0^\infty \frac{2i\eps \beta(s_0) }{(u-c)^2 + \eps^2}B_{XS;\eps}^{(1)}(r,s_0,c)  B_{XS;\eps}^{(2)}(s_0,s,c) \dd s_0\right) \frac{(u-c)}{(u-c)^2 + \eps^2} g(s) \dd s \notag  \\ 
Z_{XG}^{\eps}[g] & := \int_0^\infty \left(\int_0^\infty \frac{2i\eps \beta(s_0) }{(u-c)^2 + \eps^2}B_{XG;\eps}^{(1)}(r,s_0,c)  B_{XG;\eps}^{(2)}(s_0,s,c) \dd s_0\right) g(s) \dd s  \notag  \\ 
Z_{XrG}^{\eps} \circ w Z_{Ya}^{\pm \eps}[g] & := \int_0^\infty \left(\int_0^\infty \frac{2i\eps \beta(s_0) }{(u-c)^2 + \eps^2}B_{XrG;\eps}^{(1)}(r,s_0,c) B_{XrG;\eps}^{(2)}(s_0,s,c) \dd s_0\right) w(s) \partial_s Z_{Ya}^{\pm \eps}[g](s) \dd s. \label{def:XrGcomp}
\end{align}
\end{subequations}
where the kernels are defined via: 
\begin{subequations} \label{def:Bkers}
\begin{align}
B_{X\delta;\eps}^{(1)}(r,s,c) & = \cG(r,s,c+i\eps) \nonumber \\ & \quad + \int_0^\infty \frac{2i\eps \beta(s_0)}{(u-c)^2 + \eps^2} \cG(r,s_0,c+ i\eps) \cG(s_0,s,c-i\eps) \dd s_0,  \\ 
B_{XrG;\eps}^{(1)}  = B_{XG;\eps}^{(1)} = B_{XS;\eps}^{(1)}(r,s_0,c) & = \cG(r,s_0,c+i\eps) \\
B_{XrG;\eps}^{(2)}  = B_{XG;\eps}^{(2)} = B_{XS;\eps}^{(2)}(s_0,s,c) & = \cG(s_0,s,c-i\eps). 
\end{align}
\end{subequations}

The next lemma verifies the relevance of the above operators. 

\begin{lemma}[Representation formulas] \label{lem:IterInv}
Let $\eps > 0$ and $c \pm i \eps \in I_\alpha$.  
There holds, 
\begin{subequations} \label{eq:XYinv}
\begin{align} 
Y^{\pm \eps}(r,c \pm i \eps) & = \left(Z_{Y\delta}^{\pm \eps} + Z_{YS}^{\pm \eps}\right)[F] + Z_{YG}^{\pm \eps}[F_\ast], \label{def:Yinv} \\ 
X(r,c,\eps) & = \left(Z^{\eps}_{X\delta} + Z^{\eps}_{XS}\right)[F] + Z_{XG}^\eps[F_\ast], \label{def:Xinv}
\end{align}
\end{subequations}
and for $j > 0$ (assuming the integrands are integrable), 
\begin{align}
\partial_G^j Y^{\pm} & = \left(Z_{Y\delta}^{\pm \eps} + Z_{YS}^{\pm \eps}\right)[F_j] + Z_{YG}^{\pm\eps}\left[ \sum_{\ell=0}^{j} e_{j,\ell} \partial_G^{\ell} F_\ast +  \sum_{\ell = 0}^{j-1}p_{j,\ell} \partial_G^\ell Y^{\pm}\right] - Z_{YG}^{\pm \eps}\left[ \sum_{\ell = 0}^{j-1} q_{j,\ell} \partial_r\partial_G^\ell Y^{\pm}\right] \\ 
\partial_G^j X  & = \left(Z_{XS}^\eps + Z_{X\delta}^\eps \right)[F_{j}^-] + Z_{XG}^\eps \left[ \sum_{\ell=0}^{j} e_{j,\ell} \partial_G^{\ell} F_\ast +  \sum_{\ell = 0}^{j-1} p_{j,\ell} \partial_G^\ell Y^{-}\right] - Z_{XG}^\eps\left[ \sum_{\ell = 0}^{j-1} q_{j,\ell} \partial_r\partial_G^\ell Y^{-}\right] \nonumber \\ 
& \quad + \left(Z_{YS}^{+\eps} +  Z_{Y\delta}^{+\eps} \right)[\cE_{j}] +  Z_{YG}^{+\eps}\left[\sum_{\ell=0}^{j-1} p_{j,\ell}  \partial_G^\ell X \right] -  Z^{+\eps}_{YG}\left[ \sum_{\ell = 0}^j q_{j+1,\ell} \partial_r\partial_G^\ell X \right].  \label{eq:Xjp1}
\end{align}
\end{lemma}
\begin{proof}[Proof of Lemma~\ref{lem:IterInv}]
First, observe 
\begin{align*}
\frac{2i\eps}{(u-c)^2 + \eps^2}\left(F - \beta Y\right) & = \frac{2i\eps}{(u-c)^2 + \eps^2}\left(F - \beta \Ray^{-1}_-\left[\frac{F}{u-c+i\eps}\right] \right) - \frac{2i\eps \beta}{(u-c)^2 + \eps^2} Z_{YG}^{-\eps}[F_\ast],
\end{align*}
and by definition 
\begin{align*}
Z_{YG}^{+\eps}\left[ \frac{2i\eps \beta}{(u-c)^2 + \eps^2} Z_{YG}^{-}[F_\ast] \right] = Z_{XG}^\eps [F_\ast]. 
\end{align*}
This completes the proof of \eqref{eq:XYinv}; the cases $j > 0$ follow similarly. 
\end{proof}

\subsection{Iterated integral operators} 
\subsubsection{Recursion scheme for integral operators} \label{sec:IIO}
In this section, we will analyze the recursion algorithm derived above.
By Lemma \ref{lem:RFErecurse}, the $F_j$ and $\cE_j$ terms appearing on the RHS of Lemma \ref{lem:IterInv} can also be expanded only in terms of $X_\ell$ and $Y_\ell$ with $\ell < j$.

We will write all possible operators appearing in the iteration scheme as variations of the original operators with kernels derived from a recursive algorithm. 
The operators appearing in this algorithm are of the following form: 
\begin{subequations} 
\begin{align}
\cO_{S;\eps}^{(1)}[wK](r,s_0,c) & = \int_0^\infty K(s,s_0,c)  \frac{(u-c)}{(u-c)^2 + \eps^2} w(s) \cG(r,s,c+i\eps)  \dd s \\
\cO_{S;\eps}^{(2)}[wK](s_0,r,c) & = \int_0^\infty K(s_0,s,c)  \frac{(u-c)}{(u-c)^2 + \eps^2} w(s) \cG(r,s,c-i\eps)  \dd s \\
\cO_{\delta;\eps}^{(1)}[wK](r,s_0,c) & = \int_0^\infty K(s,s_0,c)  \frac{2i\eps}{(u-c)^2 + \eps^2} w(s) \cG(r,s,c+i\eps) \dd s \\
\cO_{\delta;\eps}^{(2)}[wK](s_0,r,c) & = \int_0^\infty K(s_0,s,c)  \frac{2i\eps}{(u-c)^2 + \eps^2} w(s) \cG(r,s,c-i\eps) \dd s \\
\cO_{G;\eps}^{(1)}[wK](r,s_0,c) & = \int_0^\infty K(s,s_0,c) w(s) \cG(r,s,c+i\eps) \dd s \\
\cO_{G;\eps}^{(2)}[wK](s_0,r,c) & = \int_0^\infty K(s_0,s,c) w(s) \cG(r,s,c-i\eps) \dd s \\
\cO_{r;\eps}^{(1)}[wK](r,s_0,c) & = \int_0^\infty K(s,s_0,c) w(s) \partial_{s} \cG(r,s,c+i\eps)  \dd s \\
\cO_{r;\eps}^{(2)}[wK](s_0,r,c) & = \int_0^\infty K(s_0,s,c) w(s) \partial_{s} \cG(r,s,c-i\eps)  \dd s.
\end{align}
\end{subequations} 

\begin{lemma}\label{lem:XYYRecurse}
Let $a \in \set{S,G,rG}$. 
For each set of weights $\set{w_j}$ and ordering of operators $b_j \in \set{S,\delta,G,rG}$ there exists kernels $B^{(1)}_{...,Xa;\eps}, B_{XS,...,Yb_J;\eps}^{(2)}$ such that (recall abuse of notations \eqref{def:YrGcomp}, \eqref{def:XrGcomp})
\begin{align*}
Z_{Xa} \circ w_1 Z_{Yb_1} \circ \cdots w_J Z_{YS}[g](r,r_c,\eps)   & = \\ & \hspace{-11cm}
\int_0^\infty \left( \int_0^\infty \frac{2i\eps u'(s_0)}{(u-c)^2 + \eps^2} B_{Xa;\eps}^{(1)}(r,s_0,c) B_{Xa,..,YS;\eps}^{(2)}(s_0,s,c) \dd s_0 \right) \frac{(u-c)}{(u-c)^2 + \eps^2} g(s) \dd s  \\
Z_{Yb_1} \circ \cdots w_{J-1} Z_{Yb_J} \circ w_{J} Z_{Xa} \circ w_{J+1} Z_{Yb_{J+1}} \circ \cdots w_{J'} Z_{YS}[g](r,r_c,\eps)  & = \\ & \hspace{-11cm}
\int_0^\infty \left( \int_0^\infty \frac{2i\eps u'(s_0)}{(u-c)^2 + \eps^2} B_{Yb_1,...,Yb_{J},Xa;\eps}^{(1)}(r,s_0,c) B_{Xa,Yb_{J+1},...,Y S;\eps}^{(2)}(s_0,s,c) \dd s_0\right) \frac{(u-c)}{(u-c)^2 + \eps^2} g(s)  \dd s. 
\end{align*}
moreover if the last $YS$ operator is replaced by the allowable alternatives, the expression changes via: 
\begin{align*}
YS & \mapsto YG \Rightarrow \frac{(u-c)}{(u-c)^2 + \eps^2} \mapsto 1 \\
YS & \mapsto Y\delta  \Rightarrow \frac{(u-c)}{(u-c)^2 + \eps^2} \mapsto \frac{2i\eps }{(u-c)^2 + \eps^2},
\end{align*}
(notice that the final operator cannot be $YrG$; see Lemma \ref{lem:IterScheme}). 
The operators are obtained by the following recursion formulas:
\begin{enumerate}
\item To define $B^{(2)}$ we use the recursion scheme: 
\begin{align*}
B^{(2)}_{Xa,Yb_1,...,Y b_{J-1}, Y b_J;\eps}(s_0,r,c) & = \cO_{b_{J-1};\eps}^{(2)}[w_J B^{(2)}_{Xa,Yb_1,...,Y b_{J-1};\eps}](s_0,r,c),
\end{align*}
and if $J=1$ then $Xa$ plays the role of $b_{J-1}$. 
\item To define $B^{(1)}$ we use the recursion scheme 
\begin{align*}
B^{(1)}_{Yb_1,Yb_2,...,Xa;\eps}(r,s_0,c) & = \cO_{b_1;\eps}^{(1)}[w_1B^{(1)}_{Yb_2,...,Xa;\eps}](r,s_0,c). 
\end{align*}
\end{enumerate}
\end{lemma} 

\begin{lemma} \label{lem:XdYYrecurse}
For $a=\delta$, for each set of weights $\set{w_j}$, and choice of operators $b_j \in \set{S,\delta,G,rG}$ there exists kernels (depending on the weights) such that
\begin{align*}
& Z_{Yb_1} \circ \cdots \circ w_J Z_{Y b_{J-1}} \circ w_{J} Z_{X \delta}[g] \notag\\
&\qquad   = \int_0^\infty \frac{2i\eps}{(u-c)^2 + \eps^2} B_{Yb_1,...,Yb_J,Xa;\eps}^{(1)}(r,s,c) g(s) \dd s 
\\
& Z_{Y b_1} \circ \cdots \circ w_{J} Z_{Y b_{J-1}} \circ w_{J} Z_{Xa} \circ w_{J+1} Z_{Yb_{J+1}} \circ \cdots w_{J+1} Z_{YS}[g]
\notag \\ 
&\qquad  =  \int_0^\infty \left(\int_0^\infty \frac{2i\eps \beta(s_0)}{(u-c)^2 + \eps^2} B_{Yb_1,..,Yb_J,X a; \eps}^{(1)}(r,s_0,c) B_{Xa,Yb_{J+1},..,Yb_{J'};\eps}^{(2)}(s_0,s,c) \dd s_0 \right) \frac{(u-c)}{(u-c)^2 + \eps^2} g(s) \dd s,
\end{align*}
where the last operator is changed for $YS \mapsto YG, Y\delta$ then as above 
\begin{align*}
YS & \mapsto YG \Rightarrow \frac{(u-c)}{(u-c)^2 + \eps^2} \mapsto 1 \\
YS & \mapsto Y\delta  \Rightarrow \frac{(u-c)}{(u-c)^2 + \eps^2} \mapsto \frac{2i\eps }{(u-c)^2 + \eps^2}. 
\end{align*}
The kernels are constructed via the following recursion. 
\begin{enumerate}
\item To construct the $B^{(1)}$ kernels we use the iteration scheme
\begin{align*}
B^{(1)}_{Y b_1,Yb_2,...,X\delta;\eps}(r,s_0,c) & = \cO_{b_1;\eps}^{(1)}[w_1B^{(1)}_{Yb_2,...,X\delta;\eps}](r,s_0,c). 
\end{align*}
\item The first steps of $B^{(2)}$ are given by: 
\begin{align*}
B_{X\delta,Ya;\eps}^{(2)}  =  \cG(s_0,r,c-i\eps),  
\end{align*}
and to construct further kernels $B^{(2)}$ we use the recursion scheme 
\begin{align*}
B^{(2)}_{X\delta,Yb_1,...,Y b_{J-1}, Y b_J;\eps}(s_0,s,c) & = \cO_{b_{J-1};\eps}[w_J B^{(2)}_{X\delta,Yb_1,...,Y b_{J-1}}](s_0,r,c).
\end{align*}
\end{enumerate}
\end{lemma} 

\begin{lemma} \label{lem:YYYrecurse}
For each set of weights $\set{w_j}$ and choice of operators $b_j \in \set{S,\delta,G,rG}$ there exist kernels $B_{XS,Yb_1,...,Yb_J;\eps}$ (depending on the weights) such that
\begin{align*}
w_1 Z_{Y b_1} \circ \cdots \circ w_J Z_{Y S}[g]= \int_0^\infty B_{Yb_1,Yb_2,...,Yb_S;\eps}(r,s,c) \frac{(u-c)}{(u-c)^2 + \eps^2} g(s) \dd s, 
\end{align*}
with the requisite change for $YS \mapsto YG,Y\delta$ as above (as usual, the final operator cannot be $YrG$). 
To construct the kernels we use the following scheme as above:  
\begin{enumerate}
\item All of the basic kernels are given by
\begin{align*}
B_{Ya;\eps} = \cG(r,s,c-i\eps). 
\end{align*}
\item Further kernels are constructed via
\begin{align*}
B_{Yb_1,...,Y b_{J-1}, Y b_J;\eps}(s_0,r,c) & = \cO^{(2)}_{b_{J-1};\eps}[w_J B_{Yb_1,...,Yb_{J-1};\eps}]. 
\end{align*}
\end{enumerate}
\end{lemma}

\subsubsection{Estimates on iterated integral kernels} 
The next step is to use induction to deduce the requisite estimates on the $B$ kernels appearing in Lemmas \ref{lem:XYYRecurse} -- \ref{lem:YYYrecurse}. 
In order to effectively pass to the limit $\eps \rightarrow 0$ in \eqref{eq:rdrf1f2}, we need to prove that the kernels satisfy a variety of regularity properties. 
Moreover, in order to close the induction argument, a variety of additional regularity properties are required as well. 
These properties are outlined in Definitions \ref{def:SuitableType1} and \ref{def:SuitableType2} below. 

Recall the bounding functions defined in \eqref{def:KKcB}. 
Further, define the following variants of the $\partial_G$ derivatives suitable for functions of three variables: 
\begin{subequations} \label{def:dGrs}
\begin{align}
\partial_G = \partial_{G}^{(r)} & = \frac{1}{u'(r)} \partial_r + \frac{1}{u'(r_c)}\partial_{r_c} \\ 
\partial_{G}^{(r,s)} & = \frac{1}{u'(r)} \partial_r + \frac{1}{u'(s)} \partial_s  + \frac{1}{u'(r_c)}\partial_{r_c}. 
\end{align}
\end{subequations} 

\begin{definition}[Suitable $(J,\ell,\gamma)$ kernel of type I] \label{def:SuitableType1}
We say $K^{\eps}(r,s,c)$ is a \emph{Suitable$(J,\ell,\gamma)$ kernel of type I} if the following properties hold for $c \pm i \eps \in I_\alpha$ with all constants independent of $\eps$:
\begin{itemize}
\item[(a)] uniform boundedness and regularity away from the critical layer: 
\begin{subequations} \label{ineq:KbdsTypeI} 
\begin{align}
\abs{K^\eps(r,s,c)}  &\lesssim \abs{u'(s)} \KK(r,s,c) \cB(r,s) \cL_{J,\ell}(r,s) \label{ineq:KbdTypeI} \\
 \abs{r \partial_r K^\eps(r,s,c)}  &\lesssim \notag \\ & \hspace{-2cm} \abs{u'(s)} \KK(r,s,c) \cB(r,s) \cL_{J,\ell}(r,s)\left(k + \mathbf{1}_{\abs{r-r_c} < r_c/k}\abs{\log\frac{k\abs{r-r_c}}{r_c}}\right), \label{ineq:drKbdTypeI} \\ 
 \abs{r_c \partial_{r_c} K^\eps(r,s,c)}\mathbf{1}_{\abs{s-r_c} > r_c/k}  &\lesssim \notag \\ & \hspace{-2cm}  k^3\abs{u'(s)} \KK(r,s,c) \cB(r,s) \cL_{J,\ell}(r,s)\left(k + \mathbf{1}_{\abs{r-r_c} < r_c/k}\abs{\log\frac{k\abs{r-r_c}}{r_c}}\right); \label{ineq:drcKTypeI}
\end{align}
\end{subequations} 
\item[(b)] regularity near the critical layer: 
\begin{subequations} \label{ineq:logLipKbdsTypeI}  
\begin{align} 
\abs{\partial_G^{(s)} K^\eps(r,s,c)}\mathbf{1}_{\abs{s-r_c} < r_c/k} \mathbf{1}_{\abs{r-r_c} \geq r_c/k} & \lesssim \notag 
\\ & \hspace{-4cm} \frac{1}{\abs{r_c u'(r_c)}}k\abs{u'(s)} \KK(r,s,c) \cB(r,s) \cL_{J,\ell}(r,s,c) \label{ineq:dGsKTI}\\ 
\abs{\partial_G^{(r)}K^\eps(r,s,c)}\mathbf{1}_{\abs{r-r_c} < r_c/k} \mathbf{1}_{\abs{s-r_c} \geq r_c/k} & \lesssim \notag 
\\ & \hspace{-4cm} \frac{1}{\abs{r_c u'(r_c)}}k\abs{u'(s)} \KK(r,s,c) \cB(r,s) \cL_{J,\ell}(r,s,c) \label{ineq:dGrKTI} \\ 
\abs{r\partial_r \partial_G^{(r)} K^\eps(r,s,c)}\mathbf{1}_{\abs{r-r_c} < r_c/k} \mathbf{1}_{\abs{s-r_c} \geq r_c/k} & \lesssim \notag  \\ & \hspace{-4cm} \frac{1}{\abs{r_c u'(r_c)}}\abs{u'(s)} \KK(r,s,c) \cB(r,s) \cL_{J,\ell}(r,s,c) \left(k^2 + k\abs{\log\frac{k\abs{r-r_c}}{r_c}}\right), \label{ineq:drdGrKTI} \\ 
\abs{\partial_G^{(r,s)}  K^\eps(r,s,c)}\mathbf{1}_{\abs{r-r_c} < r_c/k} \mathbf{1}_{\abs{s-r_c} < r_c/k} & \lesssim \notag 
\\ & \hspace{-4cm} \frac{1}{\abs{r_c u'(r_c)}}k\abs{u'(s)} \KK(r,s,c) \cB(r,s) \cL_{J,\ell}(r,s,c) \label{ineq:dGrsKTI} \\ 
\abs{r\partial_r \partial_G^{(r,s)} K^\eps(r,s,c)}\mathbf{1}_{\abs{r-r_c} < r_c/k} \mathbf{1}_{\abs{s-r_c} < r_c/k} & \lesssim \notag  \\ & \hspace{-4cm} \frac{1}{\abs{r_c u'(r_c)}}\abs{u'(s)} \KK(r,s,c) \cB(r,s) \cL_{J,\ell}(r,s,c) \left(k^2 + k\abs{\log\frac{k\abs{r-r_c}}{r_c}}\right); \label{ineq:drdGrsKTI} 
\end{align}
\end{subequations}
\item[(c)] H\"older regularity in $s$ near the critical layer: for $\abs{s-r_c} < r_c/k$ there holds:
\begin{subequations} \label{ineq:KTypeI_Holder}
\begin{align} 
\abs{K^\eps(r,s,c) - K^\eps(r,r_c,c)} & \lesssim \notag \\ & \hspace{-3cm} \abs{u'(s)} \KK(r,s,c) \cB(r,s) \cL_{J,\ell}(r,s,c) \left(\frac{k\abs{s-r_c}}{r_c}\right)^{\gamma} \label{ineq:KtypeI_Holder1} \\ 
  \abs{\partial_G^{(s)} \left(K^\eps(r,s,c) - K^\eps(r,r_c,c)\right)} \mathbf{1}_{\abs{r-r_c} \geq r_c/k} & \lesssim\notag
\\ & \hspace{-5cm} \frac{1}{\abs{r_c u'(r_c)}}k\abs{u'(s)} \KK(r,s,c) \cB(r,s) \cL_{J,\ell}(r,s,c)\left(\frac{k\abs{s-r_c}}{r_c}\right)^{\gamma} \\ 
\abs{\partial_G^{(r,s)} \left(K^\eps(r,s,c) - K^\eps(r,r_c,c)\right)}\mathbf{1}_{\abs{r-r_c} < r_c/k} & \lesssim
\\ & \hspace{-5cm} \frac{1}{\abs{r_c u'(r_c)}}k\abs{u'(s)} \KK(r,s,c) \cB(r,s) \cL_{J,\ell}(r,s,c)\left(\frac{k\abs{s-r_c}}{r_c}\right)^{\gamma} 
\\ \abs{r \partial_{r} \left(K^\eps(r,s,c) - K^\eps(r,r_c,c)\right)}\mathbf{1}_{\abs{r-r_c} \geq r_c/k} & \lesssim \notag 
\\ & \hspace{-5cm} k^2\abs{u'(s)} \KK(r,s,c) \cB(r,s) \cL_{J,\ell}(r,s,c)\left(\frac{k\abs{s-r_c}}{r_c}\right)^{\gamma}
\\ \abs{r_c \partial_{r_c} \left(K^\eps(r,s,c) - K^\eps(r,r_c,c)\right)}\mathbf{1}_{\abs{r-r_c} \geq r_c/k} & \lesssim \notag
\\ & \hspace{-5cm} k^3\abs{u'(s)} \KK(r,s,c) \cB(r,s) \cL_{J,\ell}(r,s,c)\left(\frac{k\abs{s-r_c}}{r_c}\right)^{\gamma}.
\end{align}
\end{subequations} 
\item[(c)] convergence: there exists an $\eta > 0$ such that $\eps^{-\eta}\left(K^\eps - K^0\right)$ satisfies \eqref{ineq:KbdsTypeI},s \eqref{ineq:logLipKbdsTypeI}, and \eqref{ineq:KTypeI_Holder} for $z \in I_\alpha$ and $\eps$ sufficiently small.   
\end{itemize}
\end{definition}

\begin{definition}[Suitable $(J,\ell)$ kernel of type II] \label{def:SuitableType2}
We say $K^{\eps}(r,s,c)$ is a \emph{Suitable $(J,\ell,\gamma)$ kernel of type II} if the following properties hold for $c \pm i \eps \in I_\alpha$ with all constants independent of $\eps$:
\begin{itemize}
\item[(a)] uniform boundedness and regularity away from the critical layer: 
\begin{subequations} \label{ineq:KbdsTypeII} 
\begin{align}
\abs{K^\eps(r,s,c)} & \lesssim \abs{u'(s)} \KK(r,s,c) \cB(r,s) \cL_{J,\ell}(r,s,c) \label{ineq:KbdTypeII} \\
\abs{s \partial_s K^\eps(r,s,c)} & \lesssim \\ & \hspace{-2cm}  k\abs{u'(s)} \KK(r,s,c) \cB(r,s) \cL_{J,\ell}(r,s,c)\left(k + \mathbf{1}_{\abs{s-r_c} < r_c/k}\abs{\log\frac{k\abs{s-r_c}}{r_c}}\right), \label{ineq:drKbdTypeII} \\ 
\abs{r_c \partial_{r_c} K^\eps(r,s,c)}\mathbf{1}_{\abs{r-r_c} > r_c/k}   & \lesssim \notag \\ & \hspace{-2cm} k^3\abs{u'(s)} \KK(r,s,c) \cB(r,s) \cL_{J,\ell}(r,s,c)\left(k + \mathbf{1}_{\abs{s-r_c} < r_c/k}\abs{\log\frac{k\abs{s-r_c}}{r_c}}\right); \label{ineq:drcKII}
\end{align}
\end{subequations} 
\item[(b)] regularity near the critical layer: 
\begin{subequations} \label{ineq:logLipKbdsTypeII}  
\begin{align} 
\abs{ \partial_G^{(s)} K^\eps(r,s,c)}\mathbf{1}_{\abs{s-r_c} < r_c/k} \mathbf{1}_{\abs{r-r_c} \geq r_c/k} & \lesssim \notag
\\ & \hspace{-4cm} \frac{1}{\abs{r_c u'(r_c)}}k\abs{u'(s)} \KK(r,s,c) \cB(r,s) \cL_{J,\ell}(r,s,c) \label{ineq:dGsKII}\\ 
\abs{\partial_G^{(r)} K^\eps(r,s,c)}\mathbf{1}_{\abs{r-r_c} < r_c/k} \mathbf{1}_{\abs{s-r_c} \geq r_c/k} & \lesssim
\\ & \hspace{-4cm} \frac{1}{\abs{r_c u'(r_c)}}k\abs{u'(s)} \KK(r,s,c) \cB(r,s) \cL_{J,\ell}(r,s,c) \label{ineq:dGrKII} \\ 
\abs{s\partial_s \partial_G^{(s)}  K^\eps(r,s,c)}\mathbf{1}_{\abs{r-r_c} < r_c/k} \mathbf{1}_{\abs{s-r_c} \geq r_c/k} & \lesssim \\ & \hspace{-4cm} \frac{1}{\abs{r_c u'(r_c)}}\abs{u'(s)} \KK(r,s,c) \cB(r,s) \cL_{J,\ell}(r,s,c) \left(k^2 + k\abs{\log\frac{k\abs{s-r_c}}{r_c}}\right), \label{ineq:dsdGsKII} \\ 
\abs{ \partial_{G}^{(r,s)} K^\eps(r,s,c)}\mathbf{1}_{\abs{r-r_c} < r_c/k} \mathbf{1}_{\abs{s-r_c} < r_c/k} & \lesssim
\\ & \hspace{-4cm} \frac{1}{\abs{r_c u'(r_c)}}k\abs{u'(s)} \KK(r,s,c) \cB(r,s) \cL_{J,\ell}(r,s,c) \label{ineq:dGrsKII} \\ 
\abs{s\partial_s \partial_G^{(r,s)} K^\eps(r,s,c)}\mathbf{1}_{\abs{r-r_c} < r_c/k} \mathbf{1}_{\abs{s-r_c} < r_c/k} & \lesssim \\ & \hspace{-4cm} \frac{1}{\abs{r_c u'(r_c)}}\abs{u'(s)} \KK(r,s,c) \cB(r,s) \cL_{J,\ell}(r,s,c) \left(k^2 + k\abs{\log\frac{k\abs{s-r_c}}{r_c}}\right); \label{ineq:dsdGrsKII}
\end{align}
\end{subequations}
\item[(c)] H\"older regularity in $r$ near critical layer: for $\abs{r-r_c} < r_c/k$ there holds: 
\begin{subequations} \label{ineq:KTypeII_Holder}
\begin{align} 
\abs{K^\eps(r,s,c) - K^\eps(r_c,s,c)} & \lesssim \notag \\ & \hspace{-3cm} \abs{u'(s)} \KK(r,s,c) \cB(r,s) \cL_{J,\ell}(r,s,c) \left(\frac{k\abs{r-r_c}}{r_c}\right)^{\gamma} \label{ineq:bdTII_H}\\ 
  \abs{\partial_G^{(r)} \left(K^\eps(r,s,c) - K^\eps(r_c,s,c)\right)} \mathbf{1}_{\abs{s-r_c} \geq r_c/k} & \lesssim \notag
\\ & \hspace{-5cm} \frac{1}{\abs{r_c u'(r_c)}}k\abs{u'(s)} \KK(r,s,c) \cB(r,s) \cL_{J,\ell}(r,s,c)\left(\frac{k\abs{r-r_c}}{r_c}\right)^{\gamma} \label{ineq:dGrKII_H} \\ 
\abs{\partial_G^{(r,s)}  \left(K^\eps(r,s,c) - K^\eps(r_c,s,c)\right)}\mathbf{1}_{\abs{s-r_c} < r_c/k} & \lesssim \notag
\\ & \hspace{-5cm} \frac{1}{\abs{r_c u'(r_c)}}k\abs{u'(s)} \KK(r,s,c) \cB(r,s) \cL_{J,\ell}(r,s,c)\left(\frac{k\abs{r-r_c}}{r_c}\right)^{\gamma} \\
\abs{ s \partial_{s}  \left(K^\eps(r,s,c) - K^\eps(r_c,s,c)\right)}\mathbf{1}_{\abs{s-r_c} \geq r_c/k} & \lesssim \notag
\\ & \hspace{-5cm} k^2\abs{u'(s)} \KK(r,s,c) \cB(r,s) \cL_{J,\ell}(r,s,c)\left(\frac{k\abs{r-r_c}}{r_c}\right)^{\gamma}
\\ \abs{ r_c \partial_{r_c}  \left(K^\eps(r,s,c) - K^\eps(r_c,s,c)\right)}\mathbf{1}_{\abs{s-r_c} \geq r_c/k} & \lesssim \notag
\\ & \hspace{-5cm} k^3\abs{u'(s)} \KK(r,s,c) \cB(r,s) \cL_{J,\ell}(r,s,c)\left(\frac{k\abs{r-r_c}}{r_c}\right)^{\gamma}. 
\end{align}
\end{subequations} 
\item[(c)] convergence: there exists an $\eta > 0$ such that $\eps^{-\eta}\left(K^\eps - K^0\right)$ satisfies \eqref{ineq:KbdsTypeII}, \eqref{ineq:logLipKbdsTypeII}, and \eqref{ineq:KTypeII_Holder} for $z \in I_\alpha$ and $\eps$ sufficiently small (depending only on $k$).
\end{itemize}
\end{definition}

\begin{remark} 
Note that the main differences between Definitions  \ref{def:SuitableType1} and \ref{def:SuitableType2} is in the regularity requirements. 
\end{remark}

First, we prove that the Green's function is the prototypical suitable kernel of both types. 

\begin{lemma} \label{lem:Gsuitable}
The Green's function $\cG(r,s,c\pm i \eps)$ is both a suitable $(0,0,\gamma)$ kernel of type I and a suitable $(0,0,\gamma)$ kernel of type II for all $\gamma \in (0,1)$. 
\end{lemma}
\begin{proof}[Proof of Lemma~\ref{lem:Gsuitable}]
Consider just the case statement of Type I; Type II is exactly analogous.

\paragraph*{Step 1: Proof of \eqref{ineq:KbdsTypeI}}
First, consider the boundedness estimate \eqref{ineq:KbdTypeI}. From Lemmas \ref{lem:Wronk} and \ref{lem:Hests} (and  \ref{lem:BasicVort}), 
\begin{align*}
\abs{\frac{\cG(r,s,c-i\eps)}{u'(s)} }  & \lesssim  \max(s^{-1},s^3) \min(\frac{r_c^3}{k}, \frac{1}{k r_c^5}) \\  & \quad \times \Bigg( \mathbf{1}_{r_c \leq 1} \mathbf{1}_{s<r}\left( \mathbf{1}_{s < r_c} \frac{s^{k+1/2}}{r_c^{k+3/2}} + \mathbf{1}_{r_c < s < 1} k  \frac{s^{k+2+1/2}}{r_c^{k+2+3/2}} + \mathbf{1}_{s>1}k \frac{s^{k+1/2}}{r_c^{k+3+1/2}}\right) \\ 
& \quad\quad \times \left( \mathbf{1}_{r<r_c} k \frac{r_c^{k-3/2}}{r^{k-1/2}} + \mathbf{1}_{r_c <r < 1} \frac{r_c^{k+1/2}}{r^{k+3/2}} + \mathbf{1}_{r>1} \frac{r_c^{k+1/2}}{r^{k-1/2}} \right) \\ 
& \quad + \mathbf{1}_{r_c \leq 1} \mathbf{1}_{r<s}\left( \mathbf{1}_{r < r_c} \frac{r^{k+1/2}}{r_c^{k+3/2}} + \mathbf{1}_{r_c < r < 1} k  \frac{r^{k+2+1/2}}{r_c^{k+2+3/2}} + \mathbf{1}_{r>1}k \frac{r^{k+1/2}}{r_c^{k+3+1/2}}\right) \\ 
& \quad\quad \times  \left( \mathbf{1}_{s<r_c} k \frac{r_c^{k-3/2}}{s^{k-1/2}} + \mathbf{1}_{r_c <s < 1} \frac{r_c^{k+1/2}}{s^{k+3/2}} + \mathbf{1}_{s>1} \frac{r_c^{k+1/2}}{s^{k-1/2}} \right)  \\ 
& \quad + \mathbf{1}_{r_c > 1} \mathbf{1}_{s < r} \left(\mathbf{1}_{s < 1} \frac{s^{k+1/2}}{r_c^{k-1/2}} + \mathbf{1}_{1 < s < r_c} \frac{s^{k+5/2}}{r_c^{k-1/2}} + \mathbf{1}_{s >r_c} k \frac{s^{k+1/2}}{r_c^{k-5/2}}\right) \\ 
& \quad\quad \times \left(\mathbf{1}_{r<1} k \frac{r_c^{k+5-1/2}}{r^{k-1/2}} + \mathbf{1}_{1 \leq r < r_c} k \frac{r_c^{k+5-1/2}}{r^{k+3/2}} + \mathbf{1}_{r_c < r} \frac{r_c^{k+5/2}}{r^{k-1/2}}\right) \\ 
& \quad + \mathbf{1}_{r_c > 1} \mathbf{1}_{s > r} \left(\mathbf{1}_{r < 1} \frac{r^{k+1/2}}{r_c^{k-1/2}} + \mathbf{1}_{1 < r < r_c} \frac{r^{k+5/2}}{r_c^{k-1/2}} + \mathbf{1}_{r >r_c} k \frac{r^{k+1/2}}{r_c^{k-5/2}}\right) \\ 
& \quad\quad \times \left(\mathbf{1}_{s<1} k \frac{r_c^{k+5-1/2}}{s^{k-1/2}} + \mathbf{1}_{1 \leq s < r_c} k \frac{r_c^{k+5-1/2}}{s^{k+3/2}} + \mathbf{1}_{r_c < s} \frac{r_c^{k+5/2}}{s^{k-1/2}}\right) \Bigg) \\ 
& \lesssim \mathbf{1}_{r_c \leq 1} \Bigg( \mathbf{1}_{s< r < r_c} \frac{s^{k-1/2}}{r^{k-1/2}} + \mathbf{1}_{s < r_c < r < 1}\frac{s^{k-1/2} }{r^{k-1/2}}\frac{r_c^2}{r^2} + \mathbf{1}_{s < r_c < 1 < r}\frac{s^{k-1/2}}{r^{k-1/2}} r_c^2 \\ & \quad\quad + \mathbf{1}_{r_c < s < r < 1} \frac{s^{k+3/2}}{r^{k+3/2}} + \mathbf{1}_{r_c < s < 1 < r} \frac{s^{k+3/2}}{r^{k-1/2}} + \mathbf{1}_{1 < s < r} \frac{s^{k+3+1/2}}{r^{k-1/2}}\Bigg)  \\ 
& \quad + \mathbf{1}_{r_c \leq 1} \Bigg( \mathbf{1}_{r<s<r_c} \frac{r^{k+1/2}}{s^{k+1/2}} + \mathbf{1}_{r<r_c < s < 1}\frac{r^{k+1/2}}{s^{k+1/2}}\frac{r_c^2}{s^2}  + \mathbf{1}_{r < r_c < 1 < s} \frac{r^{k+1/2}}{s^{k+1/2}} s^4 r_c^2 \\ & \quad\quad + \mathbf{1}_{r_c < r < s < 1}\frac{r^{k+2+1/2}}{s^{k+2+1/2}} + \mathbf{1}_{r_c < r < 1 < s}\frac{r^{k+2+1/2}}{s^{k-1/2}}s^3 + \mathbf{1}_{1 < r < s}\frac{r^{k+1/2}}{s^{k-1/2}}s^3\Bigg) \\ 
& \quad  + \mathbf{1}_{r_c > 1} \Bigg( \mathbf{1}_{s < r < 1}\frac{s^{k-1/2}}{r^{k-1/2}} + \mathbf{1}_{s < 1 < r < r_c}\frac{s^{k-1/2}}{r^{k+3/2}} + \mathbf{1}_{s < 1 < r_c < r}\frac{s^{k-1/2}}{r^{k-1/2}} \frac{1}{r_c^2}  \\ 
& \quad\quad + \mathbf{1}_{1 < s < r < r_c } \frac{s^{k-1/2}}{r^{k+3/2}} + \mathbf{1}_{1 < s < r_c < r}\frac{s^{k+3+5/2}}{r^{k-1/2}}\frac{1}{r_c^2} + \mathbf{1}_{1 < r_c < s < r}\frac{s^{k+1/2+3}}{r^{k-1/2}}  \Bigg) \\
& \quad + \mathbf{1}_{r_c > 1} \Bigg( \mathbf{1}_{r< s < 1}\frac{r^{k+1/2}}{s^{k+1/2}} + \mathbf{1}_{r < 1 < s < r_c}\frac{r^{k+1/2}}{s^{k+1/2}}s^2  + \mathbf{1}_{r < 1 < r_c < s}\frac{r^{k+1/2}}{s^{k+1/2}}s^4 r_c^{-2} \\ & \quad\quad   + \mathbf{1}_{1 < r < s < r_c}\frac{r^{k+1/2}}{s^{k+1/2}} r^2 s^2  + \mathbf{1}_{1 < r< r_c < s}\frac{r^{k+1/2}}{s^{k+1/2}} s^4 r^2 r_c^{-2}  + \mathbf{1}_{1 < r_c < r < s}\frac{r^{k+1/2}}{s^{k+1/2}} s^4 \Bigg) \\ 
& \lesssim \KK(r,s,c) \cB(r,s),
\end{align*}
which is \eqref{ineq:KbdTypeI}. 
The estimate \eqref{ineq:drKbdTypeI} follows from a similar argument together with Lemma \ref{lem:dH}. 
To see the $r_c \partial_{r_c}$ control, first note that
\begin{align*}
r_c \partial_{r_c} \left( \frac{\cG(r,s,c)}{u'(s)} \right) & =
\frac{1}{u'(s)M(z)}
\begin{cases} 
\frac{-r_c\partial_{r_c}M(z)}{M(z)} H_0(r,c) H_\infty(s,c) \quad & r < s \\
\frac{-r_c\partial_{r_c}M(z)}{M(z)} H_0(r,c) H_\infty(s,c) \quad & s > r \\
\end{cases} 
\\
& \quad + \frac{1}{u'(s)M(z)} 
\begin{cases} 
r_c\partial_{r_c}H_0(r,c) H_\infty(s,c) + H_0(r,c) r_c \partial_{r_c} H_\infty(s,c)  \quad & r < s \\ 
r_c\partial_{r_c}H_0(s,c) H_\infty(r,c) + H_0(s,c) r_c \partial_{r_c} H_\infty(r,c)  \quad & r > s. 
\end{cases} 
\end{align*} 
Hence, all the lemmas in \S\ref{sec:MH} together imply \eqref{ineq:drcKTypeI} as in \eqref{ineq:KbdTypeI}. 

\paragraph*{Step 2: Proof of \eqref{ineq:logLipKbdsTypeI}} 
Consider the case $\abs{r-r_c} \geq r_c/k$ and $\abs{s-r_c} < r_c/k$ as in \eqref{ineq:dGsKTI}. 
Here we have, 
\begin{align*}
\partial_G^{(s)} \left( \frac{\cG(r,s,c)}{u'(s)} \right) & = 
\frac{1}{u'(s)M(z)}
\begin{cases} 
r_c\partial_{r_c}H_0(r,c) H_\infty(s,c) + H_0(r,c) \partial_G H_\infty(s,c)  \quad & r < s \\ 
\partial_G H_0(s,c) H_\infty(r,c) + H_0(s,c) r_c \partial_{r_c} H_\infty(r,c)  \quad & r > s 
\end{cases} 
\\ 
&  \quad - \left(\frac{u''(s)}{(u'(s))^3M(z)} + \frac{r_c\partial_{r_c}M(z)}{M^2(z)} \right)
\begin{cases} 
H_0(r,c) H_\infty(s,c) \quad & r < s \\ 
H_0(r,c) H_\infty(s,c)  \quad & r > s. 
\end{cases} 
\end{align*} 
By the lemmas in \S\ref{sec:MH}, we deduce \eqref{ineq:dGsKTI}. 
Note further that from the lemmas in \S\ref{sec:MH}, we may deduce a logarithmically singular upper bound on $s \partial_s \partial_G^{(s)} \left( (u'(s))^{-1} \cG(r,s,c)\right)$. 
This is important both to verify that the kernel is suitable of type II but also to prove the H\"older regularity \eqref{ineq:KTypeI_Holder} below. 
The estimates \eqref{ineq:dGrKTI} and \eqref{ineq:drdGrKTI} are analogous and omitted for the sake of brevity.  
Consider the estimates \eqref{ineq:dGrsKTI} and \eqref{ineq:drdGrsKTI} next. 
We have 
\begin{align}
\partial_G^{(r,s) }\left( \frac{\cG(r,s,c)}{u'(s)} \right) & = 
\frac{1}{u'(s)M(z)}
\begin{cases} 
\partial_G H_0(r,c) H_\infty(s,c) + H_0(r,c) \partial_G H_\infty(s,c)  \quad & r < s \\ 
\partial_G H_0(s,c) H_\infty(r,c) + H_0(s,c) \partial_G H_\infty(r,c)  \quad & r > s 
\end{cases} 
\notag \\ 
& \quad - \left(\frac{u''(s)}{(u'(s))^3M(z)} + \frac{r_c\partial_{r_c}M(z)}{M(z)} \right) 
\begin{cases} 
H_0(r,c) H_\infty(s,c) \quad & r < s \\ 
H_0(r,c) H_\infty(s,c)  \quad & r > s. 
\end{cases}
\label{eq:dGrsG}
\end{align}
Hence, this satisfies the desired estimates by the lemmas in \S\ref{sec:MH}. Similarly, we can obtain $s\partial_s$ estimates as well. 

\paragraph*{Step 3: Proof of \eqref{ineq:KTypeI_Holder}}

The inequalities \eqref{ineq:KtypeI_Holder1} are a consequence of the log-Lipschitz regularity of $\cG$ in both variables (from Lemma \ref{lem:dH}); we omit the details as they are straightforward. 
Finally, we note that the convergence stated in Definition \ref{def:SuitableType1} follows from the lemmas in \S\ref{sec:MH}.
\end{proof}

Next, we prove that all possible integral operators arising in the iteration scheme are suitable. 

\begin{lemma}[Iterated integral operators $\cO_{\delta;\eps}^{(j)}$ and $\cO_{S;\eps}^{(j)}$] \label{lem:IIOSdelta}
Let $K^{(1)}_\eps$ be a Suitable$(J,\ell',\gamma)$ kernel of type I and $K^{(2)}_\eps$ a Suitable$(J,\ell',\gamma)$ kernel of type II.
Further, suppose that 
\begin{align*}
\abs{w(r)} + \abs{r\partial_r w(r)} \lesssim \max\left(\frac{1}{r^2},r^2\right)^{\ell+1} \frac{1}{\brak{r}^6}.
\end{align*}
Further, suppose $\ell + \ell' < k - 1/2$. Then, for all $\eta > 0$ and $\gamma' \in (0,\gamma)$, 
\begin{itemize} 
\item $\cO_{S;\eps}^{(1)}[wK^{(1)}_\eps]$ is suitable $(J+1,\ell' + \ell + 1 + \eta,\gamma')$ of type I, $\cO_{\delta;\eps}^{(1)}[wK^{(1)}_\eps]$ is suitable $(J,\ell' + \ell + 1 + \eta,\gamma')$ of type I;
\item $\cO_{S;\eps}^{(2)}[wK^{(2)}_\eps]$ is suitable $(J+1,\ell' + \ell + 1 + \eta,\gamma')$ of type II and $\cO_{\delta;\eps}^{(2)}[wK^{(2)}_\eps]$ is suitable $(J,\ell' + \ell + 1 + \eta, \gamma')$ kernel of type II.
\end{itemize} 
\end{lemma} 

\begin{lemma}[Iterated integral operators $\cO_{G}^{(j)}$] \label{lem:IIOG}
Suppose that $K^{(1)}_\eps$ is a suitable $(J,\ell',\gamma)$ kernel of type I and $K^{(2)}_\eps$ is suitable $(J,\ell',\gamma)$ kernel of type II. 
Further, suppose that
\begin{align*}
\abs{w(r)} + \abs{r\partial_r w(r)} \lesssim k^2 \max\left(\frac{1}{r^2},r^2\right)^\ell \max(\frac{1}{r^4},1),
\end{align*}
and that $\ell + \ell' < k-3/2$. Then, for all $\eta > 0$, and $\gamma' \in (0,\gamma)$, $\cO_{G;\eps}^{(1)}[wK^{(1)}_\eps]$ is suitable $(J+2,\ell'+\ell+1 + \eta,\gamma')$ of type I and $\cO_{G;\eps}^{(2)}[wK^{(2)}_\eps]$ is suitable $(J+2,\ell' + \ell + 1 + \eta,\gamma')$ of type II.  
\end{lemma} 

\begin{lemma}[Iterated integral operator $\cO_{r;\eps}^{(j)}$] \label{lem:IIOr}
Suppose that 
\begin{align*}
\abs{w(r)} + \abs{r\partial_r w(r)} \lesssim \max\left(\frac{1}{r^2},r^2\right)^\ell \max(r^{-4},1)
\end{align*}
and that $\ell + \ell' < k- 3/2$.
Suppose that the kernel $K_\eps^{(1)}$ is a suitable $(J,\ell',\gamma)$ kernel of type I and that $K_\eps^{(2)}$ is a suitable $(J,\ell',\gamma)$ kernel of type II. 
Then, for all $\eta> 0$ and $\gamma' \in (0,\gamma)$, $\cO_{r;\eps}^{(1)}[wK_\eps^{(1)}]$ is a suitable $(J+1,\ell' + \ell+ 1+\eta,\gamma')$ kernel of type I, and $\cO_{r;\eps}^{(2)}[wK_\eps^{(2)}]$ is a suitable $(J+1,\ell' + \ell+ 1+\eta,\gamma')$ kernel of type II. 
\end{lemma} 

We now prove Lemmas \ref{lem:IIOSdelta} -- \ref{lem:IIOr}. 

\begin{proof}[Proof of Lemma~\ref{lem:IIOSdelta}]
The cases of $j=1$ and $j=2$ are essentially the same, we will focus on $j=1$ here. 
The treatment of $\cO_{S}$ is similar to, but slightly harder than, the case $\cO_{\delta}$, so we focus on the former. 
Define: 
\begin{align*}
\tilde{K}(r,s_0,c) & :=  \int_0^\infty \frac{(u-c)}{(u-c)^2 + \eps^2} w(s)\cG(r,s,c+i\eps) K(s,s_0,c) \dd s.
\end{align*}

\paragraph*{Boundedness Estimate \eqref{ineq:KbdTypeI}}
Most of the non-trivial methods involved in the proof of Lemma \ref{lem:IIOSdelta} appear in some form in the proof of \eqref{ineq:KbdTypeI}.   
Recall \eqref{def:chic}  and split the integral based on proximity to the critical layer: 
\begin{align}
\tilde{K}(r,s_0,c) 
& = \int_0^\infty \frac{(u-c) u'(s) }{(u-c)^2 + \eps^2} \left(\chi_c + \chi_{\neq}\right) w(s)  \frac{\cG(r,s,c+i\eps)}{u'(s)}  K(s,s_0,c) \dd s  =: \tilde{K}_c + \tilde{K}_{\neq}. \label{eq:tKdiv}
\end{align}
First consider the problem of estimating $\tilde{K}_{\neq}$ for $r_c \leq 1$. 
In the case of $\tilde{K}_{\neq}$, we apply Lemma \ref{lem:Gsuitable} and \eqref{ineq:uNCL}, 
\begin{align*}
\mathbf{1}_{r_c \leq 1} \abs{\frac{1}{u'(s_0)}\tilde{K}_{\neq}(r,s_0,c)} & \lesssim \mathbf{1}_{r_c \leq 1} \int_0^\infty \frac{\abs{u'(s)}}{\abs{u-c}} \chi_{\neq} \abs{w(s)} \KK(r,s,c) \cB(r,s) \KK(s,s_0,c) \cB(s,s_0) \cL_{J,\ell'}(s,s_0) \dd s \\ 
& \lesssim \mathbf{1}_{r_c \leq 1} \int_0^\infty \left(\frac{k}{s^{1+2\ell} \max(s^2,r_c^2)} \mathbf{1}_{s \leq 1} + s^{2\ell-7} \mathbf{1}_{s \geq 1}\right) \\ & \quad \quad \times  \KK(r,s,c) \cB(r,s) \KK(s,s_0,c) \cB(s,s_0) \cL_{J,\ell'}(s,s_0) \dd s. 
\end{align*}
This integral is estimated by a tedious, but straightforward, calculation. 
Note that the requirement $\ell + \ell' < k-1/2$ is necessary to ensure the resulting integrands are integrable at zero and infinity.  
The calculation is summarized via: 
\begin{align*}
\mathbf{1}_{r_c \leq 1} \abs{\frac{1}{u'(s_0)}\tilde{K}_{0,\neq}(r,s_0,c)} & \lesssim \\ 
& \hspace{-4cm}\quad  \mathbf{1}_{r \leq s_0 \leq r_c \leq 1}\frac{r^{k+1/2}}{s_0^{k+1/2}}\left(\frac{k}{r^{2+2(\ell+\ell')}} \right) + \mathbf{1}_{s_0 \leq r \leq r_c \leq 1}  \frac{s_0^{k-1/2}}{r^{k-1/2}}\left(\frac{k}{s_0^{2+2(\ell+\ell')}}\right) + \mathbf{1}_{s_0 \leq r_c \leq r \leq 1} \frac{r_c^2 s_0^{k-1/2}}{r^2 r^{k-1/2}}\left(\frac{k}{s_0^{2(\ell+\ell')+2}} \right) \\
& \hspace{-4cm} \quad + \mathbf{1}_{s_0 \leq r_c \leq 1 \leq r} \frac{r_c^2 s_0^{k-1/2}}{r^{k-1/2}}\left( \frac{k}{s_0^{2(\ell+\ell') + 2}} + \max(s_0^{-2\ell'},r^{2\ell-2} s_0^{-2\ell'}, r^{2(\ell+\ell') - 2})\brak{\log r}  \right) \\ 
& \hspace{-4cm} \quad + \mathbf{1}_{0 \leq r \leq  r_c \leq s_0 \leq 1} \frac{r_c^2 r^{k+1/2}}{s_0^2 s_0^{k+1/2}}\left(\frac{k}{r^{2(\ell+\ell')+2}} \right) + \mathbf{1}_{r_c \leq r \leq s_0 \leq 1} \frac{r^2 r^{k+1/2}}{s_0^2 s_0^{k+1/2}} \left(\frac{r_c^{2+2k-2\ell-2\ell'}}{r^{2k+4}} + \frac{k}{r^{2+2\ell+2\ell'}}  \right)  \\ 
&  \hspace{-4cm} \quad + \mathbf{1}_{r_c \leq s_0 \leq r \leq 1} \frac{s_0^{3/2+k}}{r^{3/2+k}} \left(\frac{r_c^{2+2k-2(\ell+\ell')}}{s_0^{4 + 2k}} + \frac{k}{s_0^{2(\ell + \ell')+2}}\right) 
\\ & \hspace{-4cm} \quad + \mathbf{1}_{r_c \leq s_0 \leq 1 \leq r} \frac{s_0^{k+3/2}}{r^{k-1/2}} \left(\frac{r_c^{2+2k-2(\ell + \ell')}}{s_0^{2k+4}} + \frac{k}{s_0^{2+2(\ell + \ell')}} + \max(s_0^{-2\ell'},r^{2(\ell+\ell')-2}, r^{2\ell-2} s_0^{-2\ell'}) \brak{\log r} \right) 
\\ & \hspace{-4cm}  \quad + \mathbf{1}_{r \leq r_c \leq 1 \leq s_0} \frac{r^{k+1/2}r_c^2}{s_0^{k+1/2}} \brak{s_0}^4 \left(\frac{k}{r^{2+2\ell}}\max(\frac{1}{r^{2\ell'}},s_0^{2\ell'}) + \max(1,s_0^{2(\ell+\ell')-2}) \brak{\log s_0} \right) 
\\ & \hspace{-4cm} \quad +  \mathbf{1}_{r_c \leq r \leq 1 \leq s_0}\frac{r^2 r^{k+1/2}}{s_0^{k+1/2}} \brak{s_0}^4 \left(\frac{r_c^{2+2k-2\ell}}{r^{4+2k}} \max(r_c^{-2\ell'},s_0^{2\ell'}) + \frac{k}{r^{2+2\ell}}\max(r^{-2\ell'},s_0^{2\ell'}) + \max(1,s_0^{2(\ell+\ell') - 2}) \brak{\log s_0}\right) 
\\ & \hspace{-4cm} \quad + \mathbf{1}_{r_c \leq 1 \leq r \leq s_0} \frac{r^{k+1/2}}{s_0^{k+1/2}} \brak{s_0}^4 \left( \frac{r_c^{2+2k-2\ell}}{r^{2k}}\max(r_c^{-2\ell'},s_0^{2\ell'}) + \max(1,s_0^{2(\ell + \ell')-2})\brak{\log s_0}\right) 
\\ & \hspace{-4cm}  \quad + \mathbf{1}_{r_c \leq 1 \leq s_0 \leq r} \frac{s_0^{k-1/2}}{r^{k-1/2}} \brak{s_0}^4 \left(\frac{r_c^{2+2k-2\ell}}{s_0^{2k}} \max(r_c^{-2\ell'},s_0^{2\ell'}) + \brak{ \log r}\max(1,r^{2(\ell + \ell')-2}) \right). 
\end{align*}
A key constraint is to not lose powers of $r_c^{-1}$ while still gaining the improvement encoded in $\KK$, that is, we specifically want $\cL$ to be independent of $r_c$, 
so when e.g. $r_c \ll r,s_0$, we still need  to get good estimates (which one observes is indeed the case using that $\ell + \ell' \leq k$). 
After simplification, we therefore have the following estimate for all $\eta > 0$,  
\begin{align}
\mathbf{1}_{r_c \leq 1} \abs{\frac{1}{u'(s_0)}\tilde{K}_{0,\neq}(r,s_0,c)} & \lesssim \KK(r,s_0,c) \cB(r,s_0) \cL_{J+1,\ell+\ell' + 1 + \eta}(r,s_0), \label{ineq:rcl1Kneq} 
\end{align}
which is the desired estimate. In the case $r_c \geq 1$, we similarly have (omitting the tedious intermediate steps): 
\begin{align}
\mathbf{1}_{r_c > 1} \abs{\frac{1}{u'(s_0)}\tilde{K}_{0,\neq}(r,s_0,c)} 
& \lesssim \mathbf{1}_{r_c > 1} \int_0^\infty \left(\frac{1}{s^{1+2\ell}} \mathbf{1}_{s \leq 1} + \max(r_c^2,s^2) s^{2\ell-7} \mathbf{1}_{s \geq 1}\right) \notag \\ 
& \quad \times  \KK(r,s,c) \cB(r,s) \KK(s,s_0,c) \cB(s,s_0) \cL_{J,\ell'}(s,s_0) \dd s \notag \\
& \lesssim \cB(r,s_0) \cL_{J+1,\ell+\ell' + 1 + \eta}(r,s_0). \label{ineq:rcg1Kneq} 
\end{align}
This completes the treatment of the $\tilde{K}_{\neq}$.

Consider next the contributions of $\tilde{K}_c$.
First consider the case that $\abs{r-r_c} < r_c/k$.   
Near the critical layer we write 
\begin{align*}
\frac{1}{u'(s_0)}\tilde{K}_{c}(r,s_0,c) & = w(r_c) \frac{\cG(r,r_c,c)}{u'(r_c)} \frac{K(r_c,s_0,c)}{u'(s_0)} \int_0^\infty \chi_c \frac{(u-c)u'(s)}{(u-c)^2 + \eps^2} \dd s  \\ 
& \quad +  \int_0^\infty \chi_c \frac{(u-c)u'(s)}{(u-c)^2 + \eps^2} \left(w(s) \frac{\cG(r,s,c)}{u'(s)} \frac{K(s,s_0,c)}{u'(s_0)} -  w(r_c) \frac{\cG(r,r_c,c)}{u'(r_c)} \frac{K(r_c,s_0,c)}{u'(s_0)}  \right)  \dd s \\ 
& = \tilde{K}_{c1} + \tilde{K}_{c2}. 
\end{align*}
For $\tilde{K}_{c1}$, by Definition \ref{def:SuitableType1}, Lemma \ref{lem:Gsuitable}, Lemma \ref{lem:Trivrrc}, Lemma \ref{lem:pvchic}, followed by arguments similar to those to deduce \eqref{ineq:rcl1Kneq} and \eqref{ineq:rcg1Kneq} without losing powers of $r_c^{\pm 1}$ we have the following estimates: 
\begin{align*}
\abs{\tilde{K}_{c1}(r,s_0,c)} & \lesssim \max(r_c^2,r_c^{-2})^{\ell+1} \brak{r_c}^{-6} \KK(r,r_c,c)\cB(r,r_c) \KK(r_c,s_0,c) \cB(r_c,s_0) \cL_{J,\ell'}(r_c,s_0) \\
& \lesssim \KK(r,s_0,c) \cB(r,s_0) \cL_{J,\ell'+\ell + 1}(r,s_0).
\end{align*}
Next, turn to $\tilde{K}_{c2}$. 
By the logarithmic regularity in Definition \ref{def:SuitableType1}, we have (using $\abs{r_c - s} < r_c/k$ and Lemma \ref{lem:Trivrrc}), 
\begin{align*}
\abs{K(s,s_0,c) - K(r_c,s_0,c)} & \lesssim \KK(r_c,s_0,c) \cB(r_c,s_0) \cL_{J,\ell'}(r_c,s_0) \frac{\abs{s-r_c}}{r_c} \int_0^1 r_c \abs{\partial_s K(r_c + \theta(s-r_c),s_0,c)} \dd \theta \\ 
 & \lesssim \KK(r_c,s_0,c) \cB(r_c,s_0) \frac{\abs{s-r_c}}{r_c} \int_0^1 \abs{k + \abs{\log \frac{k\theta\abs{s-r_c}}{r_c}} } \dd \theta \\ 
& \lesssim  \KK(r_c,s_0,c) \cB(r_c,s_0) \cL_{J,\ell'}(r_c,s_0) \frac{\abs{s-r_c}}{r_c} \left(k +   \abs{\log \frac{k\abs{s-r_c}}{r_c}}\right).  
\end{align*}
Analogous estimates also hold for $\cG$ due to Lemma \ref{lem:Gsuitable} (and clearly also $w$). Therefore, we have (again arguing as in \eqref{ineq:rcl1Kneq} and \eqref{ineq:rcg1Kneq} to avoid losing powers of $r_c$),  
\begin{align*}
\abs{\tilde{K}_{c2}} & \lesssim \max(r_c^2,r_c^{-2})^{\ell+1} \brak{r_c}^{-6} \KK(r,r_c,c) \cB(r,r_c)  \KK(r_c,r_c,c) \cB(r_c,s_0) \cL_{J,\ell'}(r_c,s_0) \\ & \quad\quad \times \int_0^\infty \chi_c \frac{(u-c)u'(s)}{(u-c)^2 + \eps^2}  \frac{\abs{s-r_c}}{r_c} \left(k + \abs{\log \frac{k\abs{s-r_c}}{r_c}}\right) \dd s \\ 
& \lesssim \max(r_c^2,r_c^{-2})^{\ell+1} \brak{r_c}^{-6} \KK(r,r_c,c) \cB(r,r_c)  \KK(r_c,s_0,c) \cB(r_c,s_0) \cL_{J+1,\ell'}(r_c,s_0) \\ 
& \lesssim \KK(r,s_0,c) \cB(r,s_0) \cL_{J+1,\ell'+\ell + 1}(r,s_0). 
\end{align*}
This completes the proof of \eqref{ineq:KbdTypeI}. 

\paragraph*{$r\partial_r$ estimates \eqref{ineq:drKbdTypeI}} 
Next, we estimate $r\partial_r \tilde{K}$. 
By continuity of $\cG$ and $K$, 
\begin{align*}
r\partial_r \frac{\tilde{K}(r,s_0,c)}{u'(s_0)} 
 & =  r\partial_r H_\infty(r,c-i\eps)  \int_{0}^r \frac{(u-c)u'}{(u-c)^2 + \eps^2} \frac{K(s,s_0,c)}{u'(s_0)} \frac{wH_0(s,c-i\eps)}{M(c - i \eps)} \dd s \\ 
& \quad + r\partial_r H_0(r,c-i\eps) \int_r^\infty \frac{(u-c)u'}{(u-c)^2 + \eps^2}  \frac{K(s,s_0,c)}{u'(s_0)} \frac{w H_\infty(s,c-i\eps)}{M(c - i \eps) } \dd s. 
\end{align*}
From the arguments used to deduce \eqref{ineq:KbdTypeI}, for $\abs{r-r_c} \geq r_c/k$ and the lemmas of \S\ref{sec:MH}, we can derive \eqref{ineq:drKbdTypeI} in the same manner as \eqref{ineq:KbdTypeI}. The details are omitted for brevity. 
Turn to the case $\abs{r-r_c} < r_c/k$. In this region, our goal is to deduce the logarithmic upper bound; in order to avoid losing an additional logarithm we will need to extract a cancellation. Write, 
\begin{align*}
r\partial_r \frac{\tilde{K}(r,s_0,c)}{u'(s_0)}  & = ru' R_{r,\infty}^\eps(z) \int_{0}^r \frac{(u-c)u'}{(u-c)^2 + \eps^2} K(s,s_0,c) \frac{wH_0(s,c-i\eps)}{M(c - i \eps)} \dd s \\ 
& \quad -  ru' R_{0,r}^\eps(z) \int_r^\infty \frac{(u-c)u'}{(u-c)^2 + \eps^2}  K(s,s_0,c) \frac{w H_\infty(s,c-i\eps)}{M(c - i \eps) } \dd s \\ 
& \quad + \left(r\partial_r H_\infty(r,z) - ru' R_{r,\infty}^\eps(z)\right) \int_{0}^r \frac{(u-c)u'}{(u-c)^2 + \eps^2} K(s,s_0,c) \frac{wH_0(s,c-i\eps)}{M(c - i \eps)} \dd s \\ 
& \quad + \left(r\partial_r H_0(r,c-i\eps) + ru' R_{0,r}^\eps(z)\right) \int_r^\infty \frac{(u-c)u'}{(u-c)^2 + \eps^2}  K(s,s_0,c) \frac{w H_\infty(s,c-i\eps)}{M(c - i \eps) } \dd s \\ 
& = \sum_{j=1}^4 T_j. 
\end{align*}
From Lemmas \ref{lem:Hests} and \ref{lem:dH}, we deduce $r \partial_r H_\infty - ru' R_{r,\infty}^\eps$ is bounded near $r \approx r_c$, and hence $T_3$ and $T_4$ are treated using techniques used above in the proof of  \eqref{ineq:KbdTypeI}. Due to the inability to extract an additional cancellation, these terms are logarithmically unbounded (since $\int_0^r \frac{(u-c) u'}{(u-c)^2 + \eps^2} g(s) \dd s$ is singular near $r \sim r_c$ for any smooth $g$; see Lemma \ref{lem:pvchic}):  
\begin{align*}
\left(\abs{T_3} + \abs{T_4}\right)\mathbf{1}_{\abs{r-r_c} < r_c/k} \lesssim \abs{u'(s_0)} \KK(r,s_0,c) \cB(r,s_0) \cL_{J+1,\ell'+\ell + 1 + \eta}(r,s_0)\left(k + \abs{\log\frac{k\abs{r-r_c}}{r_c}}\right), 
\end{align*}
which is consistent with \eqref{ineq:drKbdTypeI}.  
For $T_1$ and $T_2$, first divide via
\begin{align*}
T_1 + T_2 & = ru' R_{r,\infty}^\eps(z) \int_{0}^r \frac{(u-c)u'}{(u-c)^2 + \eps^2} \chi_c K(s,s_0,c) \frac{wH_0(s,c-i\eps)}{M(c - i \eps)} \dd s \\ 
& \quad -  ru' R_{0,r}^\eps(z) \int_r^\infty \frac{(u-c)u'}{(u-c)^2 + \eps^2} \chi_c K(s,s_0,c) \frac{w H_\infty(s,c-i\eps)}{M(c - i \eps) } \dd s \\ 
& \quad +   ru' R_{r,\infty}^\eps(z) \int_{0}^r \frac{(u-c)u'}{(u-c)^2 + \eps^2} \chi_{\neq} K(s,s_0,c) \frac{wH_0(s,c-i\eps)}{M(c - i \eps)} \dd s \\ 
& \quad -   ru' R_{0,r}^\eps(z) \int_r^\infty \frac{(u-c)u'}{(u-c)^2 + \eps^2} \chi_{\neq} K(s,s_0,c) \frac{w H_\infty(s,c-i\eps)}{M(c - i \eps) } \dd s \\ 
& = T_{1;c} + T_{2;c} + T_{1;\neq} + T_{2;\neq}. 
\end{align*}
The terms $T_{1;\neq}$ and $T_{2;\neq}$ are also treated as in \eqref{ineq:KbdTypeI} and are hence omitted for the sake of brevity (in particular, $R_{0,r}^\eps$ and $R_{r,\infty}^\eps$ contain a logarithmic singularity by Lemma \ref{lem:asympsRE} but the integrals involving $K$ do not due to $\chi_{\neq}$). 
For the remaining terms we use the following cancellation: 
\begin{align*}
T_{1;c} + T_{2;c} & = ru' \left( R_{r,\infty}^\eps(z) - E(r_c,z) \int_r^\infty \frac{(u-c)u'}{(u-c)^2 + \eps^2} \chi_c \dd s\right) \int_{0}^r \frac{(u-c)u'}{(u-c)^2 + \eps^2} \chi_c K(s,s_0,c) \frac{wH_0(s,c-i\eps)}{M(c - i \eps)} \dd s \\ 
& \quad -  ru' \left(R_{0,r}^\eps(z)  - E(r_c,z) \int_0^r \frac{(u-c)u'}{(u-c)^2 + \eps^2} \chi_c \dd s \right)  \int_r^\infty \frac{(u-c)u'}{(u-c)^2 + \eps^2} \chi_c K(s,s_0,c) \frac{w H_\infty(s,c-i\eps)}{M(c - i \eps) } \dd s \\ 
& \quad + ru' \left(E(r_c,z) \int_r^\infty \frac{(u-c)u'}{(u-c)^2 + \eps^2} \chi_c \dd s\right)  \int_{0}^r \frac{(u-c)u'}{(u-c)^2 + \eps^2} \chi_c \\ & \quad\quad \times \left(K(s,s_0,c) \frac{wH_0(s,c-i\eps)}{M(c - i \eps)} - K(r_c,s_0,c) \frac{w(r_c)}{u'(r_c) M(c - i\eps)} \right) \dd s \\ 
& \quad - r u' \left(E(r_c,z) \int_0^r \frac{(u-c)u'}{(u-c)^2 + \eps^2} \chi_c \dd s \right)  \int_r^\infty \frac{(u-c)u'}{(u-c)^2 + \eps^2} \chi_c \\ & \quad\quad \times \left(K(s,s_0,c) \frac{w H_\infty(s,c-i\eps)}{M(c - i \eps) }  - K(r_c,s_0,c) \frac{w(r_c)}{u'(r_c) M(c - i\eps)}  \right)  \dd s. 
\end{align*}
Note that e.g. 
\begin{align*}
R_{r,\infty}^\eps(z) - E(r_c,z) \int_r^\infty \frac{(u-c)u'}{(u-c)^2 + \eps^2} \chi_c \dd s = \int_r^\infty \frac{(u-c)u'}{(u-c)^2 + \eps^2}\left(E(s,z) - E(r_c,z) \right) \chi_c \dd s, 
\end{align*}
does not have any logarithmic singularities by arguments used in the proof Lemma \ref{lem:Wronk}. 
From here, the above terms are estimated in manners analogous to the arguments in the proof of \eqref{ineq:KbdTypeI} and are hence omitted for the sake of brevity. 

\paragraph*{Derivatives involving $\partial_{r_c}$} 
First we prove \eqref{ineq:drcKTypeI}. Taking a $\partial_{r_c}$ derivative directly yields (integrating by parts in $s$),  
\begin{align*}
r_c \partial_{r_c} \tilde{K}  
& = \int_0^\infty  r_c \partial_{r_c} \left(\frac{(u-c) u'(s) }{(u-c)^2 + \eps^2}\right) w(s)  \frac{\cG(r,s,c+i\eps)}{u'(s)}  K(s,s_0,c) \dd s \\ 
& \quad + \int_0^\infty  \chi_{\neq} \left(\frac{(u-c) u'(s) }{(u-c)^2 + \eps^2}\right) w(s)  r_c \partial_{r_c}\left( \chi_{\neq} \frac{\cG(r,s,c+i\eps)}{u'(s)}  K(s,s_0,c)\right) \dd s \\ 
& \quad + \int_0^\infty  u'(r_c) r_c  \left(\frac{(u-c) u'(s) }{(u-c)^2 + \eps^2}\right) w(s) \partial_G^{(s)} \left( \chi_c \frac{\cG(r,s,c+i\eps)}{u'(s)}  K(s,s_0,c)\right) \dd s \\ 
& = \sum_{j=1}^3 \tilde{K}_j. 
\end{align*}
Due to Definition \ref{def:SuitableType1} and Lemma \ref{lem:Gsuitable}, we can apply the methods used above to prove \eqref{ineq:KbdTypeI} to prove that \eqref{ineq:drcKTypeI} holds for $\tilde{K}_1$ and $\tilde{K}_2$.
For $\tilde{K}_3$, we need to argue that the presence of $\partial_{r_c}$ derivatives on $\cG$ does not stop us from finding a similar cancellation as we used  in the proof of \eqref{ineq:drKbdTypeI} above. 
To that end, note that for $\abs{r-r_c} < r_c/k$: 
\begin{align*}
r_c \partial_{r_c} H_0(r,z) = r_c u'(r_c) \partial_G^{(r)} H_0(r,z) - \frac{r_c u'(r_c)}{r u'(r)} r \partial_r H_0(r,z), 
\end{align*}
The former term is bounded from Lemma \ref{lem:dH} and hence in the neighborhood of $r \approx r_c$, we can extract the \emph{same} cancellation in $\tilde{K}_3$ as we did in \eqref{ineq:drKbdTypeI} so that we deduce only one power of logarithm is lost. We omit the repetitive details for brevity, which concludes the proof of \eqref{ineq:drcKTypeI}. 

Next, consider the proof of \eqref{ineq:dGsKTI} (which is relevant for the region where $r \not\approx r_c$ and $s_0 \approx r_c$). 
In this case, 
\begin{align*}
\partial_G^{(s_0)} \tilde{K}(r,s_0,c) & 
  = \int_0^\infty \frac{1}{u'(r_c)}\partial_{r_c} \left(\frac{(u-c) u'(s) }{(u-c)^2 + \eps^2}\right) w(s) \chi_{\neq}  \frac{\cG(r,s,c+i\eps)}{u'(s)}  K(s,s_0,c) \dd s \\ 
&  \quad + \int_0^\infty   \left(\frac{(u-c) u'(s) }{(u-c)^2 + \eps^2}\right)  \partial_G^{(s_0)}\left( w(s)\chi_{\neq}\frac{\cG(r,s,c+i\eps)}{u'(s)}  K(s,s_0,c)\right) \dd s \\ 
& \quad + \int_0^\infty  \left(\frac{(u-c) u'(s) }{(u-c)^2 + \eps^2}\right)  \partial_G^{(s,s_0)} \left(w(s)  \left(\frac{\cG(r,s,c+i\eps)}{u'(s)}\right) \chi_c K(s,s_0,c)\right) \dd s, 
\end{align*}
Notice that  since $r \not \approx r_c$, we do not need to obtain additional regularity in $r$ in order to satisfy Definition \ref{def:SuitableType1}. 
From here, we may again apply the methods of \eqref{ineq:KbdTypeI} to deduce the desired estimates. 

Next consider \eqref{ineq:dGrKTI} (which holds in $r \approx r_c$ but $s_0 \not\approx r_c$) Here, 
\begin{align}
\partial_G^{(r)}\tilde{K} 
&  =  \int_0^\infty \frac{\chi_{\neq}}{u'(r_c)}\partial_{r_c} \left(\frac{(u-c) u'(s) }{(u-c)^2 + \eps^2}\right) w(s)  \frac{\cG(r,s,c+i\eps)}{u'(s)}  K(s,s_0,c) \dd s \notag \\ 
&  \quad + \int_0^\infty  \left(\frac{(u-c) u'(s) }{(u-c)^2 + \eps^2}\right) w(s)  \partial_G^{(r)} \left( \chi_{\neq} \frac{\cG(r,s,c+i\eps)}{u'(s)}  K(s,s_0,c)\right) \dd s \notag \\ 
& \quad + \int_0^\infty \left(\frac{(u-c) u'(s) }{(u-c)^2 + \eps^2}\right) w(s)  \partial_G^{(r,s)} \left(\chi_{c} \frac{\cG(r,s,c+i\eps)}{u'(s)}  K(s,s_0,c)\right) \dd s. \label{eq:dGrtK}
\end{align} 
Obtaining boundedness estimates is again a straightforward adaptation of the proof of \eqref{ineq:KbdTypeI}. 
Next, consider obtaining \eqref{ineq:drdGrKTI}. For this we apply an $r\partial_r$ to \eqref{eq:dGrtK}, however, some care must be taken due to the jumps in the derivatives of $\cG$: 
\begin{align}
r \partial_r \partial_G^{(r)} \tilde{K}  
& = \int_0^\infty \frac{\chi_{\neq}}{u'(r_c)}\partial_{r_c} \left(\frac{(u-c) u'(s) }{(u-c)^2 + \eps^2}\right) w(s)  r\partial_r \frac{\cG(r,s,c+i\eps)}{u'(s)}  K(s,s_0,c) \dd s \notag \\ 
&  \quad + \int_0^\infty  \left(\frac{(u-c) u'(s) }{(u-c)^2 + \eps^2}\right) w(s)  r\partial_r \partial_G^{(r)} \left( \chi_{\neq} \frac{\cG(r,s,c+i\eps)}{u'(s)}  K(s,s_0,c)\right) \dd s \notag \\ 
& \quad + \int_0^\infty \left(\frac{(u-c) u'(s) }{(u-c)^2 + \eps^2}\right) w(s)  r\partial_r \partial_G^{(r,s)} \left(\chi_{c} \frac{\cG(r,s,c+i\eps)}{u'(s)}  K(s,s_0,c)\right) \dd s \notag \\
& = \sum_{j=1}^3 \tilde{K}_j; \label{eq:drdGrK_}
\end{align}
note that two of the potential boundary terms from $r = s$ vanished due to the presence of $\chi_{\neq}$ and the assumption that $\abs{r-r_c} < r_c/k$ whereas the other two boundary terms coming from $r=s$ vanished due to the symmetric structure of \eqref{eq:dGrsG}.
The $\tilde{K}_1$ and $\tilde{K}_2$ terms are estimated in essentially the same manner as done previously for \eqref{ineq:drKbdTypeI} and are omitted for the sake of brevity. 
To treat $\tilde{K}_3$, again we use a cancellation analogous that used in \eqref{ineq:drKbdTypeI} to avoid losses of higher powers of logarithms. 
This is only an issue if both derivatives land on $\cG$ (otherwise the situation is essentially the same as \eqref{ineq:drKbdTypeI}).
Recall the identities: 
\begin{align*}
\partial_r \partial_G H_0 & = \partial_r \partial_G \left(\frac{1}{u'P(r,c\pm i \eps)}\right)  - u'(\partial_GP )   \left(R_{0,r}^\eps(z) \mp i E^\eps_{0,r}\right) \\ 
& \quad - (u-z) (\partial_r\partial_GP )   \left(R_{0,r}^\eps(z) \mp i E^\eps_{0,r}\right) -  (\partial_GP ) u'(r) E(r,z) \\ 
& \quad + \left(u'P + (u-z) \partial_r P \right) \int_0^r \frac{u'(s)}{(u-z)} \partial_{G} \left(\chi_c E\right) \dd s + P u'(r) \partial_{G} \left(\chi_c E\right) \\ 
& \quad - \frac{1}{u'(r_c)} (\partial_r\phi) \int_0^r  \chi_{\neq}\frac{u'(s) u'(r_c)}{(u-z)^2} E(s,z) - \frac{u'(s)}{(u-z)} \chi_{\neq} \partial_{r_c} E(s,z) \dd s, 
\end{align*}
and
\begin{align*}
\partial_r \partial_G H_\infty & = \partial_r \partial_G \left(\frac{1}{u'P(r,c\pm i \eps)}\right)  + u'(\partial_GP )   \left(R_{r,\infty}^\eps(z) \mp i E^\eps_{r,\infty}\right) \\ 
& \quad + (u-z) (\partial_r\partial_GP )   \left(R_{r,\infty}^\eps(z) \mp i E^\eps_{r,\infty}\right) +  (\partial_GP ) u'(r) E(r,z) \\ 
& \quad - \left(u'P + (u-z) \partial_r P \right) \int_r^\infty \frac{u'(s)}{(u-z)} \partial_{G} \left(\chi_c E\right) \dd s - P u'(r) \partial_{G} \left(\chi_c E\right) \\ 
& \quad + \frac{1}{u'(r_c)} (\partial_r\phi) \int_r^\infty  \chi_{\neq}\frac{u'(s) u'(r_c)}{(u-z)^2} E(s,z) - \frac{u'(s)}{(u-z)} \chi_{\neq} \partial_{r_c} E(s,z) \dd s. 
\end{align*}
We see that the logarithmically singular terms between $\partial_r \partial_G H_\infty$ and $\partial_r \partial_G H_0$ in $\tilde{K}_3$ in \eqref{eq:drdGrK_} have a structure similar to that
of $\partial_r H_\infty$ and $\partial_r H_0$ that was exploited in the proof of \eqref{ineq:drKbdTypeI}. 
Note that singular terms in e.g. $\partial_r \partial_GH_0$ are $u'\partial_G R_{0,r}^\eps$ and $u'P \int_0^r \frac{u'}{u-z} \partial_G (\chi_c E) \dd s$. 
After these observations, the proof follows analogously and is hence omitted for the sake of brevity. This completes the proof of \eqref{ineq:drdGrKTI}. 
The proofs of \eqref{ineq:dGrsKTI} and \eqref{ineq:drdGrsKTI} are straightforward variants of the techniques used on the other inequalities in \eqref{ineq:logLipKbdsTypeI} and are hence omitted for the sake of brevity. 

\paragraph*{H\"older Regularity in $s_0$ near $r_c$ estimates}  
Notice that: 
\begin{align*}
\tilde{K}(r,s_0,c) -\tilde{K}(r,r_c,c) & = \int_0^\infty \left(K(s,s_0,c) - K(s,r_c,c)\right) \frac{(u-c)}{(u-c)^2 + \eps^2} w(s) \cG(r,s,c - i\eps) \dd s. 
\end{align*}
Consider first obtaining \eqref{ineq:KtypeI_Holder1}. 
The only difference between proving \eqref{ineq:KbdTypeI} and proving \eqref{ineq:KtypeI_Holder1} is the lack of an analogous H\"older regularity estimate on $\partial_r \tilde{K}(r,s_0,c) - \partial_r \tilde{K}(r,r_c,c)$ (which was used to control $s \approx r_c$). 
This is dealt with via the following (see the proof of Lemma \ref{lem:SIOConv} for a similar approach): for $\theta \in (0,1)$, 
\begin{align*}
\abs{\tilde{K}(r,s_0,c) - \tilde{K}(r,r_c,c) - \tilde{K}(r_c,s_0,c) + \tilde{K}(r_c,r_c,c)}  &  \\
& \hspace{-7cm} \leq  \left(\abs{\tilde{K}(r,s_0,c) - \tilde{K}(r_c,s_0,c)} + \abs{\tilde{K}(r,r_c,c) - \tilde{K}(r_c,r_c,c)} \right)^{\theta} \\ & \hspace{-7cm} \quad \times \left(\abs{\tilde{K}(r,s_0,c) - \tilde{K}(r,r_c,c)} + \abs{\tilde{K}(r_c,s_0,c) - \tilde{K}(r_c,r_c,c)}\right)^{1-\theta} \\ 
& \hspace{-7cm} \lesssim \left(\frac{\abs{r-r_c}}{r_c}\right)^\theta \left(k + \abs{\log \abs{k\frac{r - r_c}{r_c}}}\right)^{\theta} \left(\frac{k \abs{s_0 - r_c}}{r_c} \right)^{\gamma(1-\theta)},
\end{align*}
and hence we choose $\gamma' = \gamma(1-\theta)$. Other than this slight difference, the proof of \eqref{ineq:KtypeI_Holder1} follows as in \eqref{ineq:KbdTypeI}.
The rest of the estimates in \eqref{ineq:KTypeI_Holder} are adapted in essentially this same way; the details are omitted for the sake of brevity. 

\paragraph*{Convergence as $\eps \rightarrow 0$}
As seen above, verifying the boundedness estimate contains most of the non-trivial work. 
Hence, for the convergence estimate, we will focus on this.
That is, if we define
\begin{align*}
\tilde{K}_{0}(r,s_0,c) = p.v. \int_0^\infty \frac{u'(s)}{u(s) - c} \frac{\cG(r,s,c+i0)}{u'(s)} K_0(s,s_0,c) \dd s,
\end{align*}
we are interested in obtaining an estimate of the following form for some $\eta_1,\eta_2 > 0$: 
\begin{align*}
\abs{\tilde{K}_\eps(r,s_0,c) - \tilde{K}_0(r,s_0,c)} \lesssim_{\eta_1,\eta_2} \eps^{\eta_1} \KK(r,s_0,c)\cB(r,s_0) \cL_{J+1,\ell+\ell'+ 1 + \eta_2}(r,s_0). 
\end{align*}
Write the difference as the following: 
\begin{align*}
\tilde{K}_\eps(r,s_0,c) - \tilde{K}_0(r,s_0,c)  
& =  \int \frac{(u-c)u'}{(u-c)^2 + \eps^2} \frac{\cG(r,s,c+i\eps) - \cG(r,s,c+i0)}{u'(s)} K_\eps(s,s_0,c) \dd s \\ & \quad + \int \frac{(u-c)u'}{(u-c)^2 + \eps^2} \frac{\cG(r,s,c-i0)}{u'(s)} \left(K_\eps(s,s_0,c) - K_0(s,r_c,c)\right) \dd s \\ 
& \quad + p.v.\int \left(\frac{(u-c)u'}{(u-c)^2 + \eps^2} - \frac{u'}{u-c} \right) \frac{\cG(r,s,c+i0)}{u'(s)} K_0(s,s_0,c) \dd s \\
 & = \sum_{j=1}^3 T_j. 
\end{align*}
The terms $T_1$ and $T_2$ are treated in essentially the same way as above (see Lemma \ref{lem:Gsuitable} for control on $\cG(r,s,c+i\eps) - \cG(r,s,c+i0)$).  
Hence, it remains to treat $T_3$. 
Sub-divide via:  
\begin{align*}
T_3 & = p.v.\int \left(\frac{(u-c)u'}{(u-c)^2 + \eps^2} - \frac{u'}{u-c} \right)\left(\chi_c + \chi_{\neq}\right) \frac{\cG(r,s,c+i0)}{u'(s)} K^0(s,s_0,c) \dd s =: T_{3c} + T_{3\neq}. 
\end{align*}
For $T_{3\neq}$ we apply \eqref{ineq:uSIONCL} to obtain some decay and then we argue as in \eqref{ineq:KbdTypeI}. 
For $T_{3c}$, we apply \eqref{lem:SIOConv} (which applies due to the various regularity and convergence estimates satisfied by $K$ and $\cG$). 

As remarked above, the rest of the claimed inequalities are a straightforward adaptation of the above arguments and are hence omitted for the sake of brevity. 

\paragraph*{Adaptation to $\cO^{(2)}_{S;\eps}$ case}  

The case $j=2$ is essentially the same. 
As above, we begin by dividing based on the critical layer: 
\begin{align}
\frac{\tilde{K}^{(2)}(s_0,r,c)}{u'(r)} & :=  \int_0^\infty \frac{(u-c)u'}{(u-c)^2 + \eps^2 }\left(\chi_c + \chi_{\neq}\right)  w(s) \frac{\cG(r,s,c+i\eps)}{u'(r)} \frac{K(s_0,s,c)}{u'(s)} \dd s = \tilde{K}_c + \tilde{K}_{\neq}. \label{eq:tK2div}
\end{align}
The contribution from near the critical layer, $\tilde{K}_{c}$, is treated in essentially the same manner as in the $j=1$ case and is hence omitted for brevity. 
Hence, we need only to check the tedious but simple contributions of $\tilde{K}_{\neq}$; we omit the intermediate steps for the sake of brevity: for all $\eta > 0$, 
\begin{align*} 
\abs{\frac{1}{u'(r)}\tilde{K}_{\neq}(s_0,r,c)} \mathbf{1}_{r_c \leq 1} & \lesssim \mathbf{1}_{r_c \leq 1} \int_0^\infty \left(\frac{k}{\max(s^2,r_c^2)s^{1+2\ell}}\mathbf{1}_{s \leq 1} + s^{2\ell-7} \mathbf{1}_{s > 1} \right) \\ & \quad \quad \times \chi_{\neq} \KK(s,r,c) \cB(s,r) \KK(s_0,s,c) \cB(s_0,s) \dd s \\ 
& \lesssim_\eta \mathbf{1}_{r_c \leq 1} \KK(s_0,r,c) \cB(s_0,r)\cL_{1,\ell+\ell'+1+\eta}(s_0,r). 
\end{align*}
Similarly, we have the contributions from $r_c > 1$: for all $\eta > 0$, 
\begin{align*}
\abs{\frac{1}{u'(r)}\tilde{K}_{\neq}(s_0,r,c)} \mathbf{1}_{r_c > 1} & \lesssim \mathbf{1}_{r_c > 1} \int_0^\infty \left(\frac{1}{s^{1+2\ell}} \mathbf{1}_{s \leq 1} + \max(r_c^2, s^2) s^{2\ell-7} \mathbf{1}_{s> 1}\right) \\ & \quad\quad \times \chi_{\neq} \KK(s,r,c) \cB(s,r) \KK(s_0,s,c) \cB(s_0,s) \dd s \\ 
\abs{\frac{1}{u'(r)}\tilde{K}_{\neq}(s_0,r,c)} \mathbf{1}_{r_c > 1} & \lesssim \cB(s_0,r) \cL_{1,\ell+\ell'+1+\eta}(s_0,r). 
\end{align*} 
This completes the proof of \eqref{ineq:KbdTypeII} in Definition \ref{def:SuitableType2}. As discussed above, this estimate involves most of the non-trivial work necessary to deduce the rest of Definition \ref{def:SuitableType2}, and hence the remaining estimates are omitted for the sake of brevity.  
\end{proof} 

\begin{proof}[Proof of Lemma~\ref{lem:IIOG}]
Consider the case $j=1$ first.
As there are no singular integrals, we do not need to separate the critical layer from the rest, and hence the calculations are a small variant 
of those done to estimate $\tilde{K}_{\neq}$ in the proof of Lemma \ref{lem:IIOSdelta} above. 
As above, this is a tedious, but simple and direct, calculation, and hence we omit the intermediate steps. 
For the case of $r_c \leq 1$, using Definition \ref{def:SuitableType1}, Lemma \ref{lem:Gsuitable}, and the requirement that $\ell + \ell' < k-3/2$ (to retain integrability at zero and infinity), there holds 
\begin{align*}
\abs{\frac{1}{u'(s_0)}\tilde{K}(r,s_0,c)}\mathbf{1}_{r_c \leq 1} & \leq \int_0^\infty \abs{u'(s) w(s)\frac{\cG(r,s,c+i\eps)}{u'(s)}} \abs{\frac{K(s,s_0,c)}{u'(s_0)}} \dd s \\
& \lesssim \int_0^\infty \frac{k^2}{s^3}\max(\frac{1}{s^{2\ell}}, s^{2\ell}) \KK(r,s,c) \cB(r,s) \KK(s,s_0,c) \cB(s,s_0) \cL(s,s_0,c) \dd s \\
& \lesssim \KK(r,s_0,c) \cB(r,s_0) \cL_{J+2,\ell+\ell' + 1 + \eta}(r,s_0), 
\end{align*} 
and similarly for $r_c \geq 1$: 
\begin{align*}
\abs{\frac{1}{u'(s_0)}\tilde{K}(r,s_0,c)}\mathbf{1}_{r_c \geq 1} & \lesssim \int_0^\infty \frac{k^2}{s^3}\max(\frac{1}{s^{2\ell}}, s^{2\ell}) \cB(r,s) \cB(s,s_0) \cL(s,s_0,c) \dd s \lesssim \cB(r,s_0) \cL_{J+2,\ell+\ell' + 1 + \eta}(r,s_0), 
\end{align*}
which completes the proof of \eqref{ineq:KbdTypeI}. 
As in the proof of Lemma \ref{lem:IIOSdelta}, the remaining estimates in Definition \ref{def:SuitableType1} follow from straightforward variants of the proof of \eqref{ineq:KbdTypeI}, and hence we omit these arguments for brevity. 

As above in the proof of Lemma \ref{lem:IIOSdelta}, the case $j=2$ follows in a similar manner with slightly different integrals. 
The repetitive details are omitted for brevity.  
\end{proof} 
    
\begin{proof}[Proof of Lemma~\ref{lem:IIOr}]
We will consider only the $j=2$ case; $j=1$ is the same. 
For $\cO_{r}^{(2)}$ (by symmetry of $\cG$), 
\begin{align*}
\frac{\tilde{K}(s_0,r,c)}{u'(r)} = \int_0^\infty \frac{1}{u'(r)} s\partial_s \cG(s, r, c - i\eps) \frac{w(s)}{s} K(s_0,s,c) \dd s, 
\end{align*}

\paragraph*{Boundedness estimate \eqref{ineq:KbdTypeII}}
From Lemma \ref{lem:Gsuitable} and definition \ref{def:SuitableType2}, the proof of \eqref{ineq:KbdTypeII} is essentially the same as the corresponding estimate made on $\cO_{G}^{(2)}$. 

\paragraph*{Regularity estimates}
First consider \eqref{ineq:drKbdTypeII}. 
Taking  an $r\partial_r$ derivative gives (using the definition of $M$ \eqref{def:M}), 
\begin{align*}
r\partial_r \tilde{K}(s_0,r,c) & = w(r) K(s_0,r,c) + r\partial_r H_\infty(r,c-i\eps)  \int_{0}^r  K(s_0,s,c) \frac{w \partial_s H_0(s,c-i\eps)}{M(c - i \eps)} \dd s \\ 
& \quad + r\partial_r H_0(r,c-i\eps) \int_r^\infty K(s_0,s,c) \frac{w \partial_s H_\infty(s,c-i\eps)}{M(c - i \eps) } \dd s. 
\end{align*}
This does not present any new challenges and hence \eqref{ineq:drKbdTypeII} is deduced as in Lemma \ref{lem:IIOSdelta} (though significantly easier, as no delicate cancellation is necessary) and is hence omitted for the sake of brevity. 

The more subtle problem is $\partial_{r_c}$ derivatives: 
\begin{align*}
\partial_{r_c} \frac{\tilde{K}(s_0,r,c)}{u'(r)} & = \int_0^\infty \frac{1}{u'(r)} \partial_s \partial_{r_c}  \cG(r,s,c- i\eps) w(s) K(s_0,s,c) \dd s \\
& \quad + \int_0^\infty \frac{1}{u'(r)} \partial_s  \cG(r,s,c - i \eps) w(s) \partial_{r_c}K(s_0,s,c) \dd s \\
& = \tilde{K}_1 + \tilde{K}_2.
\end{align*}
The treatment of $\tilde{K}_2$ is similar to the proof of \eqref{ineq:KbdTypeII} and is hence omitted. 
Consider next $\tilde{K}_1$. The trick is to integrate by parts so that two derivatives never land on the same kernel (note that the boundary terms vanish): 
\begin{align*}
\tilde{K}_1 & = -\int_0^r \frac{1}{u'(r)} \partial_s \left(\partial_{r_c} H_\infty(r,c-i \eps) H_0(s,c-i\eps) + H_\infty(r,c-i \eps) \partial_{r_c} H_0(s,c-i\eps) \right)   \left( w(s) K(s_0,s,c)\right) \dd s \\
& \quad -\int_r^\infty \frac{1}{u'(r)} \partial_s\left(\partial_{r_c} H_0(r,c-i \eps)H_\infty(s,c-i\eps) + H_\infty(s,c-i \eps) \partial_{r_c} H_0(r,c-i\eps) \right)  \left( w(s) K(s_0,s,c)\right) \dd s. \\
& = \int_0^r \frac{1}{u'(r)} \left(\partial_{r_c} H_\infty(r,c-i \eps) H_0(s,c-i\eps) + H_\infty(r,c-i \eps) \partial_{r_c} H_0(s,c-i\eps) \right)   \partial_s \left( w(s) K(s_0,s,c)\right) \dd s \\
& \quad +\int_r^\infty \frac{1}{u'(r)} \left(\partial_{r_c} H_0(r,c-i \eps)H_\infty(s,c-i\eps) + H_\infty(s,c-i \eps) \partial_{r_c} H_0(r,c-i\eps) \right)  \partial_s \left( w(s) K(s_0,s,c)\right) \dd s. 
\end{align*}
The log-boundedness \eqref{ineq:drcKII} follows from similar arguments as in Lemmas \ref{lem:IIOG} and \ref{lem:IIOSdelta}.
Consider \eqref{ineq:dGrKII} (which is relevant for $r \not\approx r_c$ and $s_0 \approx r_c$) for which we apply the same approach: 
\begin{align*}
\partial_G^{(s_0)} \tilde{K}(s_0,r,c) & = -\int_0^\infty \frac{1}{u'(r_c)} \partial_{r_c} \cG(s, r, c - i\eps) \partial_s \left(w(s) K(s_0,s,c)\right) \dd s \\
& \quad + \int_0^\infty  \partial_s \cG(s, r, c - i\eps) w(s) \partial_G^{(s_0)} K(s_0,s,c) \dd s. 
\end{align*}
Due to the restriction $\abs{r-r_c} \geq r_c/k$ we can apply analogous estimates as in \eqref{ineq:drKbdTypeII} to deduce \eqref{ineq:dGrKII}. 
Similar arguments deduce \eqref{ineq:dGsKII} and \eqref{ineq:dGrsKII} which we omit for brevity.

Consider next estimate \eqref{ineq:dsdGsKII}. 
Some care is required due to the jumps in the derivatives of $\cG$ (recall this is only relevant in the case $\abs{r-r_c} < r_c/k$ and $\abs{s_0 - r_c} \geq r_c/k$): 
\begin{align*}
r\partial_r \partial_G^{(r)} \tilde{K}(s_0,r,c) & = r\partial_r \int_0^r \frac{1}{M} \Bigg(\partial_G H_\infty(r,z) \partial_s H_0(s,z) + H_\infty(r,z) \frac{1}{u'(r_c)}\partial_{r_c} \partial_s H_0(s,z) \\ & \quad\quad   - \frac{\partial_{r_c} M}{M} H_\infty(r,z) \partial_s H_0(s,z)\Bigg) w(s) K(s_0,s,c) \dd s \\
&  \quad + r\partial_r \int_r^\infty \frac{1}{M} \Bigg(\partial_G H_0(r,z) \partial_s H_\infty(s,z) + H_0(r,z) \frac{1}{u'(r_c)}\partial_{r_c} \partial_sH_\infty(s,z) \\ & \quad\quad - \frac{\partial_{r_c} M}{M} H_\infty(r,z),\partial_s H_0(s,z)\Bigg) w(s) K(s_0,s,c) \dd s \\
&  \quad +  r\partial_r \int_0^r \frac{1}{M} H_\infty(r,z) \partial_s H_0(s,z) w(s)  \frac{1}{u'(r_c)}\partial_{r_c} K(s_0,s,c) \dd s \\
&  \quad +  r\partial_r \int_r^\infty \frac{1}{M} H_0(r,z) \partial_s H_\infty(s,z) w(s)  \frac{1}{u'(r_c)}\partial_{r_c} K(s_0,s,c) \dd s \\
&  = \frac{r}{M} \Bigg(\partial_G H_\infty(r,z) \partial_r H_0(r,z) + H_\infty(r,z) \frac{1}{u'(r_c)}\partial_{r_c} \partial_r H_0(r,z)  \\ & \quad\quad - \frac{\partial_{r_c} M}{M} H_\infty(r,z) \partial_r H_0(r,z)\Bigg) w(r) K(s_0,r,c) \\
& \quad - \frac{r}{M} \Bigg(\partial_G H_0(r,z) \partial_r H_\infty(r,z) + H_0(r,z) \frac{1}{u'(r_c)}\partial_{r_c} \partial_rH_\infty(r,z) \\ & \quad\quad - \frac{\partial_{r_c} M}{M} H_\infty(r,z),\partial_r H_0(r,z)\Bigg) w(r) K(s_0,r,c) \\
&  \quad +   \frac{r}{M} H_\infty(r,z) \partial_r H_0(r,z) w(r)  \frac{1}{u'(r_c)}\partial_{r_c} K(s_0,r,c) \\
&  \quad -   \frac{r}{M} H_0(r,z) \partial_r H_\infty(r,z) w(r)  \frac{1}{u'(r_c)}\partial_{r_c} K(s_0,r,c)  \\
& \quad +  \int_0^r \frac{1}{M} \Bigg(r\partial_r\partial_G H_\infty(r,z) \partial_s H_0(s,z) + r\partial_r H_\infty(r,z) \frac{1}{u'(r_c)}\partial_{r_c} \partial_s H_0(s,z)  \\ & \quad\quad - \frac{\partial_{r_c} M}{M} r\partial_r H_\infty(r,z) \partial_s H_0(s,z)\Bigg) w(s) K(s_0,s,c) \dd s \\
& \quad +  \int_r^\infty \frac{1}{M} \Bigg(r\partial_r\partial_G H_0(r,z) \partial_s H_\infty(s,z) + r\partial_r H_0(r,z) \frac{1}{u'(r_c)}\partial_{r_c} \partial_sH_\infty(s,z) \\ & \quad\quad - \frac{\partial_{r_c} M}{M} r\partial_rH_\infty(r,z),\partial_s H_0(s,z)\Bigg) w(s) K(s_0,s,c) \dd s \\
&  \quad +  \int_0^r \frac{1}{M} r\partial_r H_\infty(r,z) \partial_s H_0(s,z) w(s)  \frac{1}{u'(r_c)}\partial_{r_c} K(s_0,s,c) \dd s \\
&  \quad +  \int_r^\infty \frac{1}{M} r\partial_rH_0(r,z) \partial_s H_\infty(s,z) w(s)  \frac{1}{u'(r_c)}\partial_{r_c} K(s_0,s,c) \dd s \\
& = \sum_{j=1}^8  \tilde{K}_j. 
\end{align*}
There are several cancellations to observe. First, we observe that (recalling $\abs{r-r_c} < r_c/k$) from Lemma \ref{lem:Hests}, 
\begin{align*}
\partial_r H(r,z) - \partial_r H_\infty(r,z) & =  r\partial_r\phi M(z) \\ 
\abs{H_0(r,z) - H_\infty(r,z)} & \lesssim \mathbf{1}_{r_c \leq 1}   \frac{\abs{r-r_c}}{r_c^2} \left(k + \abs{\log k\abs{\frac{r-r_c}{r_c}}} \right) \\ 
& \quad + \mathbf{1}_{r_c>1 } r_c^3 \frac{\abs{r-r_c}}{r_c}  \left(k+ \abs{\log k\abs{\frac{r-r_c}{r_c}}} \right),
\end{align*}
and hence the terms $\tilde{K}_3 + \tilde{K}_4$ are only logarithmically singular at $r \approx r_c$ due to $\partial_{r_c}K$.
For $\tilde{K}_1$ and $\tilde{K}_2$ we can uncover the cancellations via writing the following for $r \approx r_c$: 
\begin{align*}
H_\infty(r,z) \frac{1}{u'(r_c)}\partial_{r_c} \partial_r H_0(r,z) - H_0(r,z) \frac{1}{u'(r_c)}\partial_{r_c} \partial_rH_\infty(r,z) & \\
& \hspace{-7cm} = H_\infty(r,z) \partial_G \partial_r H_0(r,z) - H_0(r,z) \partial_G \partial_rH_\infty(r,z) \\ & \hspace{-7cm} \quad - \frac{1}{u'(r)}H_\infty(r,z) \partial_{rr}H_0(r,z) + \frac{1}{u'(r)}H_0(r,z) \partial_{rr}H_\infty(r,z) \\
& \hspace{-7cm} = H_\infty(r,z) \partial_G \partial_r H_0(r,z) - H_0(r,z) \partial_G \partial_rH_\infty(r,z) \\ & \hspace{-7cm} \quad - \frac{1}{u'(r)}H_\infty(r,z) \left(\frac{k^2 - 1/4}{r^2} H_0 - \frac{\beta H_0}{u-z}\right)+ \frac{1}{u'(r)}H_0(r,z)\left(\frac{k^2 - 1/4}{r^2} H_\infty - \frac{\beta H_\infty}{u-z}\right) \\
& \hspace{-7cm}  = H_\infty(r,z) \partial_G \partial_r H_0(r,z) - H_0(r,z) \partial_G \partial_rH_\infty(r,z). 
\end{align*}
Note the commutation relation: $\partial_G \partial_r h = \partial_r \partial_Gh  + \frac{u''}{(u')^2} \partial_r h$. 
Hence, $\tilde{K}_1 + \tilde{K}_2$ is again  again only logarithmically singular via Lemmas \ref{lem:Hests} and \ref{lem:dH}. 
Finally the remaining terms $\tilde{K}_5$ through $\tilde{K}_{8}$ are treated using techniques used on previously made estimates in \eqref{ineq:logLipKbdsTypeII} and are hence omitted for the sake of brevity. The treatment of \eqref{ineq:dsdGrsKII} is similar and is hence omitted for brevity. This completes the estimates in \eqref{ineq:logLipKbdsTypeII}.

\paragraph*{H\"older regularity}
Consider next the estimates in \eqref{ineq:KTypeII_Holder}. 
As in Lemma \ref{lem:IIOSdelta}, write 
\begin{align*}
\tilde{K}(s_0,r,c) -\tilde{K}(r_c,r,c) & = \int_0^\infty \left(K(s_0,s,c) - K(r_c,s,c)\right) w(s) \partial_s \cG(r,s,c - i\eps) \dd s. 
\end{align*}
The estimate \eqref{ineq:bdTII_H} hence follows as in the proof of \eqref{ineq:KbdTypeII}. 
Consider next \eqref{ineq:dGrKII_H} (recall $\abs{r-r_c} \geq r_c/k$ in this case):  
\begin{align*}
\partial_G^{(s_0)}\tilde{K}(s_0,r,c) -\partial_G^{(s_0)}\tilde{K}(r_c,r,c) & = -\int_0^\infty \partial_G^{(s,s_0)}\left(w \chi_{c}\left(K(s_0,s,c) - K(r_c,s,c)\right)\right) \cG(r,s,c - i\eps) \dd s  \notag \\ 
& \quad -\int_0^\infty \partial_G^{(s_0)}\left(w \chi_{\neq} \left(K(s_0,s,c) - K(r_c,s,c)\right)\right) \partial_s \cG(r,s,c - i\eps) \dd s  \notag \\ 
& \quad + \int_0^\infty \chi_{\neq} \left(K(s_0,s,c) - K(r_c,s,c)\right) w(s) \frac{1}{u'(r_c)} \partial_{r_c} \partial_s \cG(r,s,c - i\eps) \dd s. 
\end{align*} 
Note that no boundary terms appear (as in the proof of \eqref{ineq:drcKII} above). 
From here, the proof follows as in \eqref{ineq:drcKII} using the hypotheses on $K$. 
the treatment of the other inequalities in \eqref{ineq:KTypeII_Holder} follow via similar reductions and are hence omitted for the sake of brevity. 

\paragraph*{Convergence}
By the Lemmas in \S\ref{sec:MH}, the $r\partial_r$ derivative of $H_0$ and $H_\infty$ satisfy analogous quantitative convergence estimates as $H_0$ and $H_\infty$ themselves and hence the convergence as $\eps \rightarrow 0$ is a straightforward consequence of arguments used previously; the details are omitted for the sake of brevity.  
\end{proof}

Finally, we verify that the original kernels satisfy the estimates necessary to run the iteration scheme. 
\begin{lemma}\label{lem:InductionIIO}
For $a \in S,\delta,G,rG$, $B_{Xa}^{(1)}$ is suitable $(0,0)$ of Type I. For $a \in S,\delta,G,rG$, $B_{Xa}^{(2)}$ and $B_{Ya}$ are suitable $(0,0)$ of Type II.
\end{lemma}
\begin{proof}[Proof of Lemma~\ref{lem:InductionIIO}]
The treatment of $B_{X\delta;\eps}^{(1)}$ is the only case not covered by Lemma \ref{lem:Gsuitable}. This follows Lemma \ref{lem:IIOSdelta} -- indeed: 
\begin{align*}
B_{X\delta;\eps}^{(1)}(r,s,c) = B_{XS;\eps}^{(1)}(r,s,c) + \cO_{\delta}^1[\beta B_{XS}^{(1)}],
\end{align*}
and hence we may apply the lemma if we set $w(r) = \beta(r)$, and $\ell = \ell' = J = 0$. 
\end{proof} 

\begin{proof}[Proof of Propositions~\ref{prop:repformXY} and~\ref{prop:BBoundIntro}] 
From Lemmas \ref{lem:IterScheme}, \ref{lem:RFErecurse}, and \eqref{lem:IterInv}, we can express all $\partial_G^j Y^{\pm}$ and $\partial_G^j X$ in terms of $\partial_G^\ell F$, $\partial_G^\ell F_\ast$, the coefficients derived in Lemma \ref{lem:RFErecurse}, and compositions of the integral operators in \eqref{def:ZY} and \eqref{def:ZX}. 
Moreover, the coefficients are such that if one has $\ell''$ compositions  and $\partial_G^\ell F$ (or $\partial_G^\ell F_\ast$), then the \emph{total}   
of all of the losses from all of the coefficients is $\ell'$ with $\ell + \ell' + \ell'' \leq j$. 
This condition  ensures that the compositions all involve integrable functions (for $\eps > 0$) and hence we iteratively apply Fubini's theorem and prove Lemmas  \ref{lem:XYYRecurse}-\ref{lem:YYYrecurse}. 
This, in turn, implies Proposition \ref{prop:repformXY} and finally Lemmas \ref{lem:Gsuitable}-\ref{lem:InductionIIO} imply Proposition \ref{prop:BBoundIntro}. 
\end{proof}

\subsection{Vorticity decomposition} \label{sec:Vd} 
In this section, we prove Propositions \ref{prop:bdf1} and \ref{prop:f2dec}.
  
\begin{proof}[Proof of Proposition~\ref{prop:bdf1}]
We will first prove the lemma in the case $n=0$, then explain how to extend to $n \leq k-1$, and finally, to extend to $n \leq k$. 

\paragraph*{Case $n=0$}  
Write
\begin{align}
f_1 & = \frac{F}{\sqrt{r}} + \frac{1}{2\pi i} \int_0^\infty \e^{itk(u(r)-u(r_c))} \frac{2i\eps u'(r_c)}{(u-c)^2 + \eps^2} \chi_{I} \frac{\beta(r)}{\sqrt{r}} A(r,c,\eps) \dd r_c \\ & \quad + \frac{1}{2\pi i} \int_0^\infty \e^{itk(u(r)-u(r_c))} \frac{(u-c)u'(r_c)}{(u-c)^2 + \eps^2} \chi_{1} \frac{\beta(r)}{\sqrt{r}} X(r,c,\eps) \dd r_c \\ 
& = \frac{F}{\sqrt{r}} + f_{1;A} + f_{1;X}. \label{eq:f1decomp}
\end{align}
From the expansion for $X$, there holds 
\begin{align}
\sqrt{r} w_{F,\delta}^{-1} f_{1;X} & = \frac{e^{itku(r)}}{2\pi i} \int_0^\infty \frac{(u-c)u'(r_c)}{(u-c)^2 + \eps^2} \chi_{1}(r,r_c) \frac{\beta(r) \e^{-iktc}}{w_{F,\delta}(r)} \int_0^\infty \frac{2i\eps }{(u-c)^2  + \eps^2} B_{X\delta;\eps}(r,s,c) F(s) \dd s \dd r_c\notag   \\
& \quad + \frac{e^{itku(r)}}{2\pi i} \int_0^\infty \frac{(u-c)u'(r_c)}{(u-c)^2 + \eps^2} \chi_{1}(r,r_c) \frac{\beta(r) \e^{-iktc}}{w_{F,\delta}(r)} \int_0^\infty \int_0^\infty \frac{2i\eps \beta(s_0)}{(u(s_0)-c)^2  + \eps^2} w_{F,\delta/2}(s_0)  \notag  \\
& \quad\quad \times B_{XS;\eps}^{(1)}(r,s_0,c) B_{XS;\eps}^{(2)}(s_0,s',c) \frac{w_{F,\delta/4}(s')}{w_{F,\delta/2}(s_0)} \frac{(u-c)}{(u-c)^2 + \eps^2} \frac{F(s')}{w_{F,\delta/2}(s')} \dd s' \dd s_0 \dd r_c \notag  \\ 
& \quad + \frac{e^{itku(r)}}{2\pi i} \int_0^\infty \frac{(u-c)u'(r_c)}{(u-c)^2 + \eps^2} \chi_{1}(r,r_c) \frac{\beta(r) \e^{-iktc}}{w_{F,\delta}(r)} \int_0^\infty \int_0^\infty \frac{2i\eps \beta(s_0)}{(u(s_0)-c)^2  + \eps^2} w_{F,\delta/2}(s_0) \notag \\ 
& \quad\quad \times B_{XG;\eps}^{(1)}(r,s_0,c) B_{XG;\eps}^{(2)}(s_0,s',c) \frac{w_{F,\delta/4}(s')}{w_{F,\delta/2}(s_0)} \frac{F_\ast(s')}{w_{F,\delta/4}(s')} \dd s' \dd s_0 \dd r_c. \label{eq:f1ds6}
\end{align}
To pass to the limit, we apply Theorems \ref{thm:BBB:main} and \ref{thm:other:integral:operators} (together with Theorem \ref{thm:2SIO:1DELTA}). 
Note that the requisite properties on the kernel are obtained by Lemmas \ref{lem:Gsuitable} and \ref{lem:InductionIIO} above.    
Hence, we have the strong $L^2$ limit:  
\begin{align*}
\lim_{\eps \rightarrow 0} \sqrt{r} w_{F,\delta}^{-1} f_{1;X} & = \e^{ikt u(r)} p.v. \int_0^\infty \frac{u'(r_c)}{u-c} \chi_1(r,r_c) \frac{w_{F,\delta/2}(r_c)}{w_{F,\delta}(r)} \beta(r) \frac{B_{X\delta}(r,r_c,c)}{u'(r_c)} \frac{ \e^{-iktc} F(r_c)}{w_{F,\delta/2}(r_c)} \dd r_c \\
& \quad + \frac{e^{itku(r)}}{2\pi i} p.v. \int_0^\infty \frac{u'(r_c)}{(u-c)} \chi_1(r,r_c)  \frac{w_{F,\delta/2}(r_c)}{w_{F,\delta}(r)} \beta(r)  \frac{B_{XS}^{(1)}(r,r_c,c)}{u'(r_c)} \\& \quad\quad \times \e^{-iktc} \left( p.v. \int_0^\infty \beta(r_c) B_{XS}^{(2)}(r_c,s',c) \frac{w_{F,\delta/4}(s')}{w_{F,\delta/2}(r_c)} \frac{1}{(u-c)} \frac{F(s')}{w_{F,\delta/4}(s')} \dd s' \right) \dd r_c \\ 
& \quad + \frac{e^{itku(r)}}{2\pi i} p.v. \int_0^\infty \frac{u'(r_c)}{(u-c)} \chi_1(r,r_c)  \frac{w_{F,\delta/2}(r_c)}{w_{F,\delta}(r)} \beta(r) \frac{B_{XG}^{(1)}(r,r_c,c)}{u'(r_c)} \\& \quad\quad \times \e^{-iktc} \left(\int_0^\infty \beta(r_c)  B_{XG}^{(2)}(r_c,s',c)  \frac{w_{F,\delta/4}(s')}{w_{F,\delta/2}(r_c)}  \frac{F_\ast(s')}{w_{F,\delta/4}(s')} \dd s' \right) \dd r_c. 
\end{align*}
A crucial point to notice is that if $r \leq 1$ then $\chi_1$ implies that $r_c < 2r$. This is what allows to transfer the gain in $r_c$ from $\KK$ to a gain in $r$ in Theorem \ref{thm:BBB:main} so that we may use the stronger $w_{F,\delta}$ as opposed to $w_{f,\delta}$. 
Theorems \ref{thm:BBB:main} and \ref{thm:other:integral:operators} also provide the following estimate: for all $\eta > 0$ (recall \eqref{def:Fast}),  
\begin{align*}
\norm{\lim_{\eps \rightarrow 0}\sqrt{r} w_{F,\delta}^{-1} f_{1;X}}_{L^2} & \lesssim_{\eta,\delta} k^{\eta} \norm{F}_{L^2_{F,\delta/4}} +  \norm{F_\ast}_{L^2_{F,\delta/4}} \lesssim \abs{k}^{\eta} \norm{F}_{L^2_{F,\delta/4}} + \abs{k} \abs{\omega_{k,0}^{in}}. 
\end{align*}
Turn next to $f_{1;A}$ in \eqref{eq:f1decomp} and expand via 
\begin{align*}
f_{1;A} & = \frac{1}{2\pi i} \int_0^\infty \frac{2i\eps u'(r_c) \e^{itk(u(r)-u(r_c))}}{(u-c)^2 + \eps^2} \chi_{I} \frac{\beta(r)}{\sqrt{r}}\left(X(r,c,\eps) + 2Y(r,c-i\eps)\right)\dd r_c = f_{1;AX} + f_{1;AY}. 
\end{align*} 
Further expand $f_{1;AX}$ via $\chi_1$ and $\chi_2$: 
\begin{align*}
f_{1;AX} & = \frac{1}{2\pi i} \int_0^\infty \e^{itk(u(r)-u(r_c))} \frac{2i\eps u'(r_c)}{(u-c)^2 + \eps^2} \chi_{1} \frac{\beta(r)}{\sqrt{r}} X(r,c,\eps) \dd r_c \\ 
& \quad + \frac{1}{2\pi i} \int_0^\infty \e^{itk(u(r)-u(r_c))} \frac{2i\eps u'(r_c)}{(u-c)^2 + \eps^2} \chi_{2} \frac{\beta(r)}{\sqrt{r}} X(r,c,\eps) \dd r_c \\
& = f_{1;AX1}  + f_{1;AX2}.  
\end{align*}
The contribution of $f_{1;AX1}$ is treated in the same way as $f_{1;X}$ (but with different integral operators in Theorem \ref{thm:other:integral:operators}) and is hence omitted. 
Next, we show that $f_{1;AX2}$ vanishes as $\eps \rightarrow 0$. 
Indeed, expanding $X$ as in $f_{1;AX1}$ gives 
\begin{align*}
\sqrt{r} w_{f,\delta}^{-1} f_{1;AX2}(t,r) & = \frac{e^{iktu(r)}}{2\pi i} \int_0^\infty \frac{2i\eps u'(r_c) }{(u-c)^2 + \eps^2} \chi_{2} \beta(r) \e^{-iktu(r_c)} \\ 
& \quad\quad \times  \int_0^\infty \frac{2i\eps}{(u-c)^2  + \eps^2} \frac{w_{F,\delta/4}(s)}{w_{f;\delta}(r)} B_{X\delta;\eps}(r,s,c) \frac{F(s)}{w_{F,\delta/4}(s)} \dd s \dd r_c \\ 
& \quad + \frac{e^{iktu(r)}}{2\pi i} \int_0^\infty \frac{2i\eps u'(r_c) }{(u-c)^2 + \eps^2} \chi_{2}  \e^{-ikt c} \int_0^\infty \int_0^\infty \frac{2i\eps \beta(s_0)}{(u(s_0)-c)^2  + \eps^2} \frac{\beta(r) w_{F,\delta/2}(s_0)}{w_{f;\delta}(r)}  \\
& \quad\quad \times B_{XS;\eps}^{(1)}(r,s_0,c) B_{XS;\eps}^{(2)}(s_0,s',c) \frac{w_{F,\delta/4}(s')}{w_{F,\delta/2}(s_0)} \frac{(u-c)}{(u-c)^2 + \eps^2} \frac{F(s')}{w_{F,\delta/4}(s')} \dd s' \dd s_0 \dd r_c \\ 
& \quad + \frac{e^{iktu(r)}}{2\pi i} \int_0^\infty \frac{2i\eps u'(r_c) }{(u-c)^2 + \eps^2} \chi_{2} \e^{-iktu(r_c)} \int_0^\infty \int_0^\infty \frac{2i\eps \beta(s_0)}{(u(s_0)-c)^2  + \eps^2} \frac{\beta(r) w_{F,\delta/2}(s_0)}{w_{f;\delta}(r)} \\ 
& \quad\quad \times B_{XG;\eps}^{(1)}(r,s_0,c) B_{XG;\eps}^{(2)}(s_0,s',c) \frac{w_{F,\delta/4}(s')}{w_{F,\delta/2}(s_0)} \frac{F_\ast(s')}{w_{F,\delta/4}(s')} \dd s' \dd s_0 \dd r_c. 
\end{align*}
Therefore, Theorems \ref{thm:other:integral:operators} and Theorem \ref{thm:2SIO:1DELTA:a} imply this term vanishes in the limit due to the support of $\chi_2$  (note the weaker space $w_{f,\delta} $). 
Turn next to  $f_{1;AY}$, which we similarly decompose via: 
\begin{align*}
f_{1;AY} & =  -\frac{2}{\pi i} \int_0^\infty \frac{i\eps u'(r_c) \e^{itk(u(r)-u(r_c))}}{(u-c)^2 + \eps^2} \chi_{1} \frac{\beta(r)}{\sqrt{r}}Y(r,c-i\eps) \dd r_c \\ & \quad - \frac{2}{\pi i} \int_0^\infty \frac{i\eps u'(r_c) \e^{itk(u(r)-u(r_c))}}{(u-c)^2 + \eps^2} \chi_{2} \frac{\beta(r)}{\sqrt{r}}Y(r,c-i\eps) \dd r_c \\ 
& = f_{1;AY 1} + f_{1;AY 2}. 
\end{align*}
By Lemmas \ref{lem:Gsuitable} and \ref{lem:InductionIIO}, the kernels satisfy the hypotheses necessary to apply Theorems \ref{thm:other:integral:operators}, \ref{thm:2SIO:1DELTA:a}, and \ref{thm:2SIO:1DELTA:a}. Therefore, we have that again that $\lim_{\eps \rightarrow 0}w_{f,\delta}^{-1}f_{1;AY2} = 0$ in $L^2$,
and that we may pass to the limit $\eps \rightarrow 0$ in $f_{1;AY1}$ and deduce: 
\begin{align*}
\lim_{\eps \rightarrow 0} \sqrt{r} w_{F,\delta}^{-1} f_{1;AY1} & = - p.v. \int_0^\infty \frac{w_{F,\delta/2}(s)}{w_{F,\delta}(r)} \beta(r)  B_{YS}(r,s,u(r)) \frac{1}{u(s)-u(r)} \frac{F(s)}{w_{F,\delta/2}(s)} \dd s \\ 
&  \quad - 2i \pi \frac{1}{w_{F,\delta}(r)} \beta(r)  \frac{B_{Y\delta}(r,r,u(r))}{u'(r)} F(r) \\ 
&  \quad -\int_0^\infty \frac{w_{F,\delta/2}(s)}{w_{F,\delta}(r)} \beta(r) B_{Y G}(r,s,u(r)) \frac{F_\ast(s)}{w_{F,\delta/2}(s)} \dd s. 
\end{align*}
Similarly, Theorem \ref{thm:other:integral:operators} implies that we have the boundedness for all $\eta > 0$,  
\begin{align*}
\norm{\sqrt{r} w_{F,\delta}^{-1} f_{1;AY1}}_{L^2} & \lesssim k^{\eta} \norm{F}_{L^2_{F,\delta/2}} + \norm{F_\ast}_{L^2_{F,\delta/2}} \lesssim k^{\eta} \norm{F}_{L^2_{F,\delta/2}} + \abs{k} \norm{\omega_{k,0}^{in}}_{L^2_{F,\delta/2}}.  
\end{align*}
This completes the case $j=0$. 

\paragraph*{Case $n \leq k-1$}  
Next, turn to $(r\partial_r)^n$ for $n \leq k-1$. 
From \eqref{eq:rdrf1f2}, denote the three contributions of $f_1$ as:  
\begin{align*}
(r\partial_r)^n f_1^\eps & = (r\partial_r)^n F + f^\eps_{1;X} + f^\eps_{1;A}. 
\end{align*}
Consider first $f_{1;X}^\eps$. 
After distributing the $\partial_G$ derivatives there are many terms all of the general form  
\begin{align*}
\int_0^\infty \e^{itk(u(r)-u(r_c))} \frac{(u-c)u'(r_c)}{(u-c)^2 + \eps^2} \frac{1}{w_{F,\delta}(r)} H(r,c,\eps) \partial_G^j X(r,c,\eps) \dd r_c, 
\end{align*}
for some weight $H^\eps$ satisfying the following for all  $m \geq 0$ (depending on $\eps$ through $\chi_I$): 
\begin{align}
\abs{(r\partial_r)^m H^\eps(r,c)} & \lesssim_{n,j,m} \min(r^{2j}, r^{-7-2j})\left( \chi_1  +  \chi_{r \leq 1} \chi_{2r \approx r_c} + \chi_{r\approx 1} \chi_{2r < r_c}\right). \label{ineq:Hbds}
\end{align}
From Proposition \ref{prop:repformXY} we have representations of the following form for a variety of complicated integral kernels:  
\begin{align*}
\int_0^\infty \e^{itk(u(r)-u(r_c))} \frac{(u-c)u'(r_c)}{(u-c)^2 + \eps^2} \frac{H^\eps(r,c)}{w_{F,\delta}(r)} \partial_G^j X(r,c,\eps) \dd r_c & \\ 
& \hspace{-7cm} = \sum_{\ell = 0}^j \frac{e^{itku(r)}}{2\pi i} \int_0^\infty \frac{(u-c)u'(r_c)}{(u-c)^2 + \eps^2} H^\eps(r,c)  \frac{e^{-iktc} w_{F,\delta/4 + 2\ell}(s)}{w_{F,\delta}(r)} \\ & \hspace{-7cm} \quad\quad \times \int_0^\infty \frac{2i\eps}{(u-c)^2  + \eps^2} B_{X\delta;j,\ell}(r,s,c) w_{X\delta1;j,\ell}(s) \frac{\partial_G^\ell F(s)}{w_{F,\delta/4 + 2\ell}(s)} \dd s \dd r_c \\
& \hspace{-7cm} \quad +\sum_{\ell = 0}^j \frac{e^{itku(r)}}{2\pi i} \int_0^\infty \frac{(u-c)u'(r_c)}{(u-c)^2 + \eps^2} H^\eps(r,c) \frac{e^{-iktc}}{w_{F,\delta}(r)} \int_0^\infty \int_0^\infty \frac{2i\eps  \beta(s_0)}{(u(s_0)-c)^2  + \eps^2} B_{X\delta;j,\ell}^{(1)}(r,s_0,c) \\
& \hspace{-7cm} \quad\quad \times B_{X\delta;j,\ell}^{(2)}(s_0,s,c) w_{F,\delta/4 + 2\ell}(s)  w_{X\delta2;j,\ell}(s)  \frac{2i\eps}{(u-c)^2 + \eps^2} \frac{\partial_G^\ell F(s)}{w_{F,\delta/4 + 2\ell}(s)} \dd s \dd s_0 \dd r_c \\ 
& \hspace{-7cm} \quad +\sum_{\ell = 0}^j \frac{e^{itku(r)}}{2\pi i} \int_0^\infty \frac{(u-c)u'(r_c)}{(u-c)^2 + \eps^2} H^\eps(r,c) \frac{e^{-iktc}}{w_{F,\delta}(r)} \int_0^\infty \int_0^\infty \frac{2i\eps  \beta(s_0)}{(u(s_0)-c)^2  + \eps^2} B_{XS;j,\ell}^{(1)}(r,s_0,c) \\
& \hspace{-7cm} \quad\quad \times B_{XS;j,\ell}^{(2)}(s_0,s,c) w_{F,\delta/4 + 2\ell}(s)  w_{XS;j,\ell}(s)  \frac{(u-c)}{(u-c)^2 + \eps^2} \frac{\partial_G^\ell F(s)}{w_{F,\delta/4 + 2\ell}(s)} \dd s \dd s_0 \dd r_c \\ 
& \hspace{-7cm} \quad + \sum_{\ell = 0}^j \frac{e^{itku(r)}}{2\pi i} \int_0^\infty \frac{(u-c)u'(r_c)}{(u-c)^2 + \eps^2} H^\eps(r,c) \frac{e^{-iktc} }{w_{F,\delta}(r)} \int_0^\infty \int_0^\infty \frac{2i\eps \beta(s_0)}{(u(s_0)-c)^2  + \eps^2} B_{XG;j,\ell}^{(1)}(r,s_0,c) \\
& \hspace{-7cm} \quad\quad \times B_{XG;j,\ell}^{(2)}(s_0,s,c) w_{F,\delta/4+ 2\ell}(s) w_{XG;j,\ell}(s) \frac{\partial_G^\ell F_\ast(s)}{w_{F,\delta/4 + 2\ell}(s')} \dd s \dd s_0 \dd r_c \\ 
& \hspace{-7cm} \quad + \textup{Similar terms with different $B$ and $w$}. 
\end{align*}
Recall that in each term there holds for some $\ell'$ (\emph{different $\ell'$ in each term}) $\abs{w_{\ast;j,\ell}(s)} \lesssim \max(s^{-2\ell'},s^{2\ell'})$.  
By Lemmas \ref{lem:Gsuitable}--\ref{lem:IIOr} (using also the recursion scheme laid out in Lemmas \ref{lem:XYYRecurse} and \ref{lem:XdYYrecurse}), for sufficiently small $\eta > 0$ and $\gamma >0$,  $B_{X\delta;j,\ell}$ is suitable $(2\ell_1'',\ell_1'' + \eta,\gamma)$ of type I, $B^{(1)}_{Xa;j,\ell}$ is suitable $(2\ell_1'',\ell_1'' + \eta,\gamma)$ of type I for some $\ell_1''$ and $\gamma > 0$, whereas $B^{(2)}_{Xa;j,\ell}$ is suitable $(2\ell_2'',\ell_2'' +\eta ,\gamma)$ of Type II for some $\ell_2''$ (as above, each term may have a different $\ell_j''$). 
In all terms there holds the inequality (where for the $B_{X\delta;j,\ell}$ terms we take $\ell_2'' = 0$): 
\begin{align*}
\ell + \ell' + \ell_1'' + \ell_2'' \leq j. 
\end{align*}
Note that the total losses matches with \eqref{ineq:Hbds}.  
Therefore, for $\eta$ chosen sufficiently small relative to $\delta$, Theorems \ref{thm:BBB:main} and \ref{thm:other:integral:operators} (together with Theorems \ref{thm:2SIO:1DELTA} and \ref{thm:2SIO:1DELTA:a}), we can pass to the limit $\eps \rightarrow 0$ in the same way as we did for the case $j=0$, giving also the $L^2$ bounds: 
\begin{align*}
\norm{\lim_{\eps \rightarrow 0}\int_0^\infty \e^{itk(u(r)-u(r_c))} \frac{(u-c)u'(r_c)}{(u-c)^2 + \eps^2} \frac{H(r,c,\eps)}{w_{F,\delta}(r)} \partial_G^j X(r,c,\eps) \dd r_c}_{L^2} & \\ 
& \hspace{-6cm} \lesssim \sum_{\ell = 0}^j \abs{k}^{2(j-\ell) + \eta} \norm{\partial_G^\ell F}_{L^2_{F,\delta/4+2\ell}} + \abs{k}^{2(j-\ell) + \eta} \norm{\partial_G^\ell F_\ast}_{L^2_{F,\delta/4+2\ell}} \\ 
& \hspace{-6cm} \lesssim \abs{k}^{2j+\eta + 1} \abs{\omega_{k,0}^{in}} + \sum_{\ell = 0}^j \abs{k}^{2(j-\ell) + \eta} \sum_{m=0}^\ell \norm{(r\partial_r)^m F}_{L^2_{F,\delta/4}}. 
\end{align*}
Analogous to the case $j=0$, the other contributions to $f_{1}^\eps$ are similarly; the details are omitted for brevity. 

\paragraph*{Case $n=k$}  
Recall that, as discussed in Remark \ref{rmk:FixedSob}, we do not really get useful information about the $k$-dependence in the case $n=k$. 
In the case $n=k$, the problematic terms in \eqref{eq:rdrf1} are those that contain $\partial_G^k X$ and $\partial_G^k A$; all other terms are treated as in the case $j \leq n-1$.  
Hence, the terms we must consider are: 
\begin{align}
\tilde{f}_{1;A} & :=  \frac{1}{2\pi i} \int_0^\infty \frac{2i\eps u'(r_c) \e^{itk(u(r)-u(r_c))}}{(u(r)-c)^2 + \eps^2} (ru'(r))^k  \chi_{I}(r_c) \frac{\beta(r)}{\sqrt{r}} \partial_G^k A(r,c,\eps)  \dd r_c \\ 
\tilde{f}_{1;X} & :=  \frac{1}{2\pi i} \int_0^\infty \e^{itk(u(r)-u(r_c))} \frac{(u(r)-c)u'(r_c)}{(u(r)-c)^2 + \eps^2} (ru'(r))^k \chi_{1}(r,r_c) \frac{\beta(r)}{\sqrt{r}} \partial_G^k X(r,c,\eps) \dd r_c. 
\end{align}
The key difficulty is that we cannot use the iteration scheme to compute this derivative in the same manner as above.
Consider $\tilde{f}_{1;X}$ and sub-divide based on the critical layer: 
\begin{align}
\tilde{f}_{1;X} & =  \int_0^\infty \frac{\e^{itk(u(r)-u(r_c))}}{2\pi i} \frac{(u(r)-c)u'(r_c)}{(u(r)-c)^2 + \eps^2} (ru'(r))^k \chi_{1} \frac{\beta(r)}{\sqrt{r}}\left(\chi_c(r,r_c) + \chi_{\neq}(r,r_c)\right)\partial_G^{k} X(r,c,\eps) \dd r_c \\ 
& = \tilde{f}_{1;X c} + \tilde{f}_{1;X\neq}. \label{def:f1XneqXc}
\end{align}
On the support of $\tilde{f}_{1;X\neq}$, we write 
\begin{align}
(ru'(r))^k\partial_G^k X & = \frac{r u'(r)}{r_c u'(r_c)} (ru'(r))^{k-1} r_c\partial_{r_c} \partial_G^{k-1} X  + (ru'(r))^{k-1}  r \partial_{r} \partial_G^{k-1} X, \label{eq:dGkXnCL}
\end{align}
Next, we take $\partial_r$ and $\partial_{r_c}$ derivatives of the representation formula \eqref{eq:dGjX}. 
Let us start with the easier $\partial_{r}$ (note that $\abs{r-r_c} \gtrsim r_c/k$ on the support of $f_{1;X\neq}$). 
Due to Lemmas \ref{lem:Gsuitable}--\ref{lem:IIOr}, these derivatives only land on Type I kernels: 
\begin{align}
 \partial_r \partial_G^j X & = \sum_{\ell = 0}^j \int_0^\infty \partial_r B_{X\delta;j,\ell}(r,s,c)
\frac{2i \eps }{(u-c)^2 + \eps^2} w_{X\delta1;j,\ell}(s)  \partial_G^\ell F(s) \dd s \notag\\ 
& \quad + \sum_{\ell = 0}^j \int_0^\infty \left(\int_0^\infty \frac{2i\eps \beta(s_0)}{(u-c)^2 + \eps^2} \partial_r B^{(1)}_{X\delta;j,\ell}(r,s_0,c) B^{(2)}_{X\delta;j,\ell}(s_0,s,c) \dd s_0\right)  \frac{2 i \eps}{(u-c)^2 + \eps^2}  w_{X\delta2;j,\ell}(s) \partial_G^\ell F \dd s \notag\\ 
& \quad + \sum_{\ell = 0}^j \int_0^\infty \left(\int_0^\infty \frac{2i\eps \beta(s_0)}{(u-c)^2 + \eps^2} \partial_r B^{(1)}_{XS;j,\ell}(r,s_0,c) B^{(2)}_{XS;j,\ell}(s_0,s,c) \dd s_0 \right)  \frac{(u-c)}{(u-c)^2 + \eps^2}  w_{XS;j,\ell}(s) \partial_G^\ell F \dd s \notag\\ 
& \quad + \sum_{\ell = 0}^j \int_0^\infty \left(\int_0^\infty \frac{2i\eps \beta(s_0)}{(u-c)^2 + \eps^2} \partial_r B^{(1)}_{XG;j,\ell}(r,s_0,c) B^{(2)}_{XG;j,\ell}(s_0,s,c) \dd s_0 \right) w_{XG;j,\ell}(s) \partial_G^\ell F_\ast \dd s \notag\\ 
& \quad + \textup{Similar terms with different $B$, $w$}. \label{eq:drdGjX}
\end{align}
By Definition \ref{def:SuitableType1}, $k^{-1} \chi_{\neq}(r,r_c) r\partial_r B_{X\delta;j,\ell}(r,s,c)$ and $k^{-1} \chi_{\neq}(r,r_c) \partial_r B^{(1)}_{Xa;j,\ell}(r,s,c)$ satisfy the conditions necessary to apply Theorems \ref{thm:BBB:main}, \ref{thm:other:integral:operators}, \ref{thm:2SIO:1DELTA}, and \ref{thm:2SIO:1DELTA:a} (note Remark \ref{rmk:NotSuitOK}) with the same parameters as $\chi_{\neq}(r,r_c)B_{X\delta;j,\ell}$ and $\chi_{\neq}(r,r_c) \partial_r B^{(1)}_{Xa;j,\ell}$ respectively. 
Specifically, the $r\partial_r$ derivative does \emph{not} incur a loss on the weights. 
Hence, 
\begin{align*}
\int_0^\infty \frac{\e^{itk(u(r)-u(r_c))}}{2\pi i} \frac{(u(r)-c)u'(r_c)}{(u(r)-c)^2 + \eps^2} \chi_{1} \frac{\beta(r)}{\sqrt{r}} \chi_{\neq}(r,r_c) (ru'(r))^{k-1}  r \partial_{r} \partial_G^{k-1} X(r,c,\eps) \dd r_c,  
\end{align*}
is treated via the same methods used to treat the case $j = k-1$ above. Repetitive details are omitted for brevity. 

Turn next to $\partial_{r_c}$ derivatives, which are more technical. From \eqref{eq:dGjX}, (still for $\abs{r-r_c} \geq r_c/k$), we have
\begin{align}
&\partial_{r_c}\partial_G^j X  = \sum_{\ell = 0}^j \int_0^\infty \partial_{r_c}  B_{X\delta;j,\ell}(r,s,c) \frac{2i \eps }{(u-c)^2 + \eps^2} w_{X\delta1;j,\ell}(s)  \partial_G^\ell F(s) \dd s \notag\\ 
& \quad + \sum_{\ell = 0}^j \int_0^\infty  B_{X\delta;j,\ell}(r,s,c) \partial_{r_c}\left(\frac{2i \eps }{(u-c)^2 + \eps^2} \right) w_{X\delta1;j,\ell}(s)  \partial_G^\ell F(s) \dd s \notag\\ 
& \quad + \sum_{\ell = 0}^j \int_0^\infty \left(\int_0^\infty \partial_{r_c}\left(\frac{2i\eps \beta(s_0)}{(u-c)^2 + \eps^2}\right) B^{(1)}_{X\delta;j,\ell}(r,s_0,c) B^{(2)}_{X\delta;j,\ell}(s_0,s,c) \dd s_0\right)  \frac{2 i \eps}{(u-c)^2 + \eps^2}  w_{X\delta2;j,\ell}(s) \partial_G^\ell F \dd s \notag\\ 
& \quad + \sum_{\ell = 0}^j \int_0^\infty \left(\int_0^\infty \frac{2i\eps \beta(s_0)}{(u-c)^2 + \eps^2} \partial_{r_c}B^{(1)}_{X\delta;j,\ell}(r,s_0,c) B^{(2)}_{X\delta;j,\ell}(s_0,s,c) \dd s_0\right)  \frac{2 i \eps}{(u-c)^2 + \eps^2}  w_{X\delta2;j,\ell}(s) \partial_G^\ell F \dd s \notag\\ 
& \quad + \sum_{\ell = 0}^j \int_0^\infty \left(\int_0^\infty \frac{2i\eps \beta(s_0)}{(u-c)^2 + \eps^2} B^{(1)}_{X\delta;j,\ell}(r,s_0,c) \partial_{r_c} B^{(2)}_{X\delta;j,\ell}(s_0,s,c) \dd s_0\right)  \frac{2 i \eps}{(u-c)^2 + \eps^2}  w_{X\delta2;j,\ell}(s) \partial_G^\ell F \dd s \notag\\ 
& \quad + \sum_{\ell = 0}^j \int_0^\infty \left(\int_0^\infty \frac{2i\eps \beta(s_0)}{(u-c)^2 + \eps^2} B^{(1)}_{X\delta;j,\ell}(r,s_0,c) B^{(2)}_{X\delta;j,\ell}(s_0,s,c) \dd s_0\right)  \partial_{r_c}\left(\frac{2 i \eps}{(u-c)^2 + \eps^2}\right)  w_{X\delta2;j,\ell}(s) \partial_G^\ell F \dd s \notag\\
& \quad + \sum_{\ell = 0}^j \int_0^\infty \left(\int_0^\infty \partial_{r_c}\left(\frac{2i\eps \beta(s_0)}{(u-c)^2 + \eps^2}\right) B^{(1)}_{XS;j,\ell}(r,s_0,c) B^{(2)}_{XS;j,\ell}(s_0,s,c) \dd s_0 \right)  \frac{(u-c)}{(u-c)^2 + \eps^2}  w_{XS;j,\ell}(s) \partial_G^\ell F \dd s \notag\\ 
& \quad + \sum_{\ell = 0}^j \int_0^\infty \left(\int_0^\infty \frac{2i\eps \beta(s_0)}{(u-c)^2 + \eps^2} \partial_{r_c}B^{(1)}_{XS;j,\ell}(r,s_0,c) B^{(2)}_{XS;j,\ell}(s_0,s,c) \dd s_0 \right)  \frac{(u-c)}{(u-c)^2 + \eps^2}  w_{XS;j,\ell}(s) \partial_G^\ell F \dd s \notag\\ 
& \quad + \sum_{\ell = 0}^j \int_0^\infty \left(\int_0^\infty \frac{2i\eps \beta(s_0)}{(u-c)^2 + \eps^2} B^{(1)}_{XS;j,\ell}(r,s_0,c) \partial_{r_c}B^{(2)}_{XS;j,\ell}(s_0,s,c) \dd s_0 \right)  \frac{(u-c)}{(u-c)^2 + \eps^2}  w_{XS;j,\ell}(s) \partial_G^\ell F \dd s \notag\\ 
& \quad + \sum_{\ell = 0}^j \int_0^\infty \left(\int_0^\infty \frac{2i\eps \beta(s_0)}{(u-c)^2 + \eps^2} B^{(1)}_{XS;j,\ell}(r,s_0,c) B^{(2)}_{XS;j,\ell}(s_0,s,c) \dd s_0 \right)  \partial_{r_c}\left(\frac{(u-c)}{(u-c)^2 + \eps^2}\right)  w_{XS;j,\ell}(s) \partial_G^\ell F \dd s \notag\\ 
& \quad + \sum_{\ell = 0}^j \int_0^\infty \left(\int_0^\infty \partial_{r_c}\left(\frac{2i\eps \beta(s_0)}{(u-c)^2 + \eps^2}\right) B^{(1)}_{XG;j,\ell}(r,s_0,c) B^{(2)}_{XG;j,\ell}(s_0,s,c) \dd s_0 \right) w_{XG;j,\ell}(s) \partial_G^\ell F_\ast \dd s \notag\\ 
& \quad + \sum_{\ell = 0}^j \int_0^\infty \left(\int_0^\infty \frac{2i\eps \beta(s_0)}{(u-c)^2 + \eps^2} \partial_{r_c}B^{(1)}_{XG;j,\ell}(r,s_0,c) B^{(2)}_{XG;j,\ell}(s_0,s,c) \dd s_0 \right) w_{XG;j,\ell}(s) \partial_G^\ell F_\ast \dd s \notag\\ 
& \quad + \sum_{\ell = 0}^j \int_0^\infty \left(\int_0^\infty \frac{2i\eps \beta(s_0)}{(u-c)^2 + \eps^2} B^{(1)}_{XG;j,\ell}(r,s_0,c) \partial_{r_c}B^{(2)}_{XG;j,\ell}(s_0,s,c) \dd s_0 \right) w_{XG;j,\ell}(s) \partial_G^\ell F_\ast \dd s \notag\\ 
& \quad + \textup{Similar terms with different $B$, $w$} \notag \\ 
& = \sum_{n=0}^{12} T_n + \textup{Similar terms with different $B$, $w$} \label{eq:drcdGjX}.
\end{align}
Many of the terms permit a similar treatment, hence, let us only consider a few.
As in \eqref{eq:rdrf1f2} and Lemma \ref{lem:Wronkdrc} (and \S\ref{sec:repbd}), we use $\frac{1}{u'(r_c)}\partial_{r_c} h(u(s) - u(r_c)) = -\frac{1}{u'(s)}\partial_{s} h(u(s) - u(r_c))$. 
Hence, we integrate by parts: 
\begin{align}
r_c \left(T_0 + T_1\right) & = \sum_{\ell = 0}^j \int_0^\infty r_c \partial_{r_c}  \left(\chi_{\neq}(s,c) B_{X\delta;j,\ell}(r,s,c)\right) \frac{2i \eps }{(u-c)^2 + \eps^2} w_{X\delta1;j,\ell}(s)  \partial_G^\ell F(s) \dd s \notag\\
& \quad + \sum_{\ell = 0}^j \int_0^\infty  B_{X\delta;j,\ell}(r,s,c)\chi_{\neq}(s,c) r_c \partial_{r_c}\left(\frac{2i \eps }{(u-c)^2 + \eps^2} \right) w_{X\delta1;j,\ell}(s)  \partial_G^\ell F(s) \dd s \notag\\  
& \quad + \sum_{\ell = 0}^j \int_0^\infty r_c u'(r_c) \partial_{G}^{(s)} \left( \chi_c B_{X\delta;j,\ell}(r,s,c)w_{X\delta1;j,\ell}(s)  \partial_G^\ell F(s)\right) \frac{2i \eps }{(u-c)^2 + \eps^2} \dd s, \label{eq:T0T1drc}
\end{align}
and note we are interested in passing to the limit in the singular integral 
\begin{align*}
\int_0^\infty \frac{\e^{itk(u(r)-u(r_c))}}{2\pi i} \frac{(u(r)-c)u'(r_c)}{(u(r)-c)^2 + \eps^2} \chi_{1} \frac{\beta \chi_{\neq}}{w_{F,\delta}(r)\sqrt{r}} \frac{r u'(r)}{r_c u'(r_c)} (ru'(r))^{k-1} r_c \left(T_0 + T_1\right) \dd r_c. 
\end{align*}
Due to the cutoffs in $\chi_c(s,c)$ and $\chi_{\neq}(s,c)$ in \eqref{eq:T0T1drc} and Definitions \ref{def:SuitableType1} and \ref{def:SuitableType2}, we are still in a position to apply Theorems \ref{thm:BBB:main} and \ref{thm:other:integral:operators} and that the $r_c \partial_{r_c}$ derivatives have not changed the weights (as was the case for $r\partial_r$). 
Note that there is the leading ratio $(ru'(r))(r_c u'(r_c))^{-1}$. The numerator of this represents a gain in the weight in $r$ and hence is what allows us to use the stronger weight $w_{F}$whereas the loss of $(r_c u'(r_c))^{-1}$ is balanced by the gains in $\KK$. With these observations, we may hence apply Theorems \ref{thm:BBB:main} and \ref{thm:other:integral:operators} and pass to the limit $\eps \rightarrow 0$, also obtaining the bounds (using Lemma \ref{lem:dGtodr}):  
\begin{align*}
\norm{\lim_{\eps \rightarrow 0}\int_0^\infty \frac{\e^{itk(u(r)-u(r_c))}}{2\pi i} \frac{(u(r)-c)u'(r_c)}{(u(r)-c)^2 + \eps^2} \chi_{1} \frac{\beta(r)}{w_{F,\delta}(r)\sqrt{r}} \chi_{\neq}(r,r_c) \frac{r u'(r)}{r_c u'(r_c)} (ru'(r))^{k-1} r_c \left(T_0 + T_1\right) \dd r_c}_{L^2} & \\ 
& \hspace{-12cm} \lesssim \sum_{\ell = 0}^{k-1} \abs{k}^{2(j-\ell) + \eta + 3} \norm{\partial_G^\ell F}_{L^2_{F,\delta/4+2\ell}} + \abs{k}^{2(j-\ell) + \eta + 3} \norm{\partial_G^\ell F_\ast}_{L^2_{F,\delta/4+2\ell}} \\ 
& \hspace{-12cm} \lesssim_{k} \abs{\omega_{k,0}^{in}} + \sum_{\ell = 0}^{k-1} \norm{(r\partial_r)^m F}_{L^2_{F,\delta/4}}; 
\end{align*}
(recall Remark \ref{rmk:FixedSob}). 
For the compound terms in \eqref{eq:drcdGjX} the picture is a little more complicated as these involve the triple derivatives appearing in Definitions \ref{def:SuitableType1} and \ref{def:SuitableType2}: 
\begin{align}
\frac{1}{u'(r_c)}\sum_{n=2}^5 T_n & =  \sum_{\ell = 0}^j \int_0^\infty \left(\int_0^\infty \frac{\chi_{\neq}(s_0,c)}{u'(r_c)} \partial_{r_c}\left(\frac{2i\eps \beta(s_0)}{(u-c)^2 + \eps^2}\right) B^{(1)}_{X\delta;j,\ell}(r,s_0,c) B^{(2)}_{X\delta;j,\ell}(s_0,s,c) \dd s_0\right) \notag \\ & \quad\quad \times \frac{2 i \eps}{(u-c)^2 + \eps^2}  w_{X\delta2;j,\ell}(s) \partial_G^\ell F \dd s \notag \\
& \quad + \sum_{\ell = 0}^j \int_0^\infty \left(\int_0^\infty \frac{2i\eps \beta(s_0)}{(u-c)^2 + \eps^2} B^{(1)}_{X\delta;j,\ell}(r,s_0,c) B^{(2)}_{X\delta;j,\ell}(s_0,s,c) \dd s_0\right) \notag \\ & \quad\quad \times  \chi_{\neq}(s,c) \frac{1}{u'(r_c)} \partial_{r_c}\left(\frac{2 i \eps }{(u-c)^2 + \eps^2}\right)  w_{X\delta2;j,\ell}(s) \partial_G^\ell F \dd s 
\notag\\ & \quad + \sum_{\ell = 0}^j \int_0^\infty \left(\int_0^\infty \frac{2i\eps \beta(s_0)}{(u-c)^2 + \eps^2} \frac{1}{u'(r_c)} \partial_{r_c} \left( \chi_{\neq}(s_0,c) B^{(1)}_{X\delta;j,\ell}(r,s_0,c) \chi_{\neq}(s,c) B^{(2)}_{X\delta;j,\ell}(s_0,s,c)\right) \dd s_0\right) \notag \\ & \quad\quad \times \frac{2 i \eps}{(u-c)^2 + \eps^2}  w_{X\delta2;j,\ell}(s) \partial_G^\ell F \dd s 
\notag\\ & \quad + \sum_{\ell = 0}^j \int_0^\infty \left(\int_0^\infty \frac{2i\eps \beta(s_0)}{(u-c)^2 + \eps^2} \partial_G^{(s)} \left( \chi_{\neq}(s_0,c) B^{(1)}_{X\delta;j,\ell}(r,s_0,c) \chi_{c}(s,c) B^{(2)}_{X\delta;j,\ell}(s_0,s,c)\right) \dd s_0\right) \notag \\ & \quad\quad \times  \frac{2 i \eps }{(u-c)^2 + \eps^2}  w_{X\delta2;j,\ell}(s) \partial_G^\ell F \dd s 
\notag\\ & \quad + \sum_{\ell = 0}^j \int_0^\infty \left(\int_0^\infty \frac{2i\eps \beta(s_0)}{(u-c)^2 + \eps^2}  \left( B^{(1)}_{X\delta;j,\ell}(r,s_0,c) \chi_{c}(s,c) B^{(2)}_{X\delta;j,\ell}(s_0,s,c)\right) \dd s_0\right) \notag \\ & \quad\quad \times  \frac{2 i \eps }{(u-c)^2 + \eps^2} \frac{1}{u'(s)}\partial_s \left( w_{X\delta2;j,\ell}(s) \partial_G^\ell F\right) \dd s 
\notag\\ & \quad + \sum_{\ell = 0}^j \int_0^\infty \left(\int_0^\infty \frac{2i\eps \beta(s_0)}{(u-c)^2 + \eps^2} \partial_G^{(s,s_0)} \left( \chi_{c}(s_0,c) B^{(1)}_{X\delta;j,\ell}(r,s_0,c) \chi_{c}(s,c) B^{(2)}_{X\delta;j,\ell}(s_0,s,c)\right) \dd s_0\right) \notag \\ & \quad\quad \times  \frac{2 i \eps }{(u-c)^2 + \eps^2}  w_{X\delta2;j,\ell}(s) \partial_G^\ell F \dd s 
\notag \\ & \quad + \sum_{\ell = 0}^j \int_0^\infty \left(\int_0^\infty \frac{2i\eps \beta(s_0)}{(u-c)^2 + \eps^2} \partial_G^{(s_0)} \left( \chi_{c}(s_0,c) B^{(1)}_{X\delta;j,\ell}(r,s_0,c) \chi_{\neq}(s,c) B^{(2)}_{X\delta;j,\ell}(s_0,s,c)\right) \dd s_0\right) \notag \\ & \quad\quad \times  \frac{2 i \eps }{(u-c)^2 + \eps^2}  w_{X\delta2;j,\ell}(s) \partial_G^\ell F \dd s \label{eq:compdrc} 
\end{align}
It is crucial to note the very specific structure in \eqref{eq:compdrc}: whenever $s \approx r_c$ and/or $s_0 \approx r_c$ the derivatives landing on the kernels 
are either $\partial_G^{(s)}$ or $\partial_{G}^{s_0}$ or $\partial_G^{s,s_0}$ so that one never evaluates $\partial_{r_c}$ (or $\partial_r$) of a kernel near the critical layer without the matching $s$ or $s_0$ derivatives.    
A similar structure is seen also in the $\sum_{j=6}^9 T_j$ terms which are omitted for brevity. The last three terms instead have: 
\begin{align}
\frac{1}{u'(r_c)}\sum_{j=10}^{12} T_j & =  \sum_{\ell = 0}^j \int_0^\infty \left(\int_0^\infty \frac{1}{u'(r_c)}\partial_{r_c}\left(\frac{2i\eps \beta(s_0)}{(u-c)^2 + \eps^2}\right) \chi_{\neq}(s_0,c) B^{(1)}_{XG;j,\ell}(r,s_0,c) B^{(2)}_{XG;j,\ell}(s_0,s,c) \dd s_0 \right)\notag \\ & \quad\quad \times w_{XG;j,\ell}(s) \partial_G^\ell F_\ast \dd s \notag\\ 
& \quad + \sum_{\ell = 0}^j \int_0^\infty \left(\int_0^\infty\left(\frac{2i\eps \beta(s_0)}{(u-c)^2 + \eps^2}\right) \frac{1}{r_c}\partial_{r_c} \left(\chi_{\neq}(s_0,c) B^{(1)}_{XG;j,\ell}(r,s_0,c) B^{(2)}_{XG;j,\ell}(s_0,s,c) \right) \dd s_0 \right) \notag \\ & \quad\quad \times w_{XG;j,\ell}(s) \partial_G^\ell F_\ast \dd s \notag\\ 
& \quad + \sum_{\ell = 0}^j \int_0^\infty \left(\int_0^\infty\left(\frac{2i\eps \beta(s_0)}{(u-c)^2 + \eps^2}\right) \partial_{G}^{(s_0)} \left(\chi_{c}(s_0,c) B^{(1)}_{XG;j,\ell}(r,s_0,c) B^{(2)}_{XG;j,\ell}(s_0,s,c) \right) \dd s_0 \right) \notag \\ & \quad\quad \times w_{XG;j,\ell}(s) \partial_G^\ell F_\ast \dd s. \label{eq:Tj10t12}
\end{align}
Notice that near $s_0 \approx r_c$ we are still using $\partial_G^{(s_0)}$ derivatives, despite that $s$ can be close to the critical layer as well. 
Hence, the derivatives are not quite the correct form for directly using that $B^{(2)}_{XG;j,\ell}$ is suitable $(2\ell''_2,\ell''_2 + \eta,\gamma)$ of type II (for some $\ell''_2$ and $\gamma$ and all $\eta > 0$). 
However, 
\begin{align*}
\partial_{G}^{(s_0)} B^{(2)}_{XG;j,\ell}(s_0,s,c) = \partial_{G}^{(s,s_0)} B^{(2)}_{XG;j,\ell}(s_0,s,c) - \frac{1}{u'(s)}\partial_s B^{(2)}_{XG;j,\ell}(s_0,s,c). 
\end{align*}
Note that the former is bounded near the critical layer whereas the latter is logarithmically singular there (see Definition \ref{def:SuitableType2}). 
However, since there are no singular integral operators or approximately $\delta$-functions in $s$ in these terms, it is straightforward to verify that we my still apply Theorem \ref{thm:other:integral:operators}.  

Finally, putting together \eqref{def:f1XneqXc} and \eqref{eq:dGkXnCL} with the associated decompositions of \eqref{eq:drdGjX}, \eqref{eq:drcdGjX}, \eqref{eq:T0T1drc}, \eqref{eq:compdrc} (and the analogous omitted terms), and \eqref{eq:Tj10t12} with Definitions \ref{def:SuitableType1}, \ref{def:SuitableType2} and Theorems \ref{thm:BBB:main}, \ref{thm:other:integral:operators}, \ref{thm:2SIO:1DELTA}, \ref{thm:2SIO:1DELTA:a} (and Remark \ref{rmk:NotSuitOK}), we may pass to the limit in $\eps \rightarrow 0$ as we did in the $n = k-1$ case. 
This gives us a (very complicated) representation formula for $f_{1;X\neq}$, and, in particular, the bound
\begin{align*}
\norm{\lim_{\eps \rightarrow 0} \sqrt{r} w_{F,\delta} f_{1;X\neq}}_{L^2} \lesssim_{k,\delta,\alpha} \abs{\omega_{k,0}^{in}} +  \sum_{\ell = 0}^k \norm{(r\partial_r)^\ell F}_{L^2_{F,\delta/4}}.
\end{align*}

Next, we consider $\tilde{f}_{1;Xc}$. 
We directly take a $\partial_G = \partial_G^{(r)}$ derivative of \eqref{eq:dGjX}
and as above, apply the usual integration by parts when $s_0 \approx r_c$ and/or $s \approx r_c$. 
This yields (on the support of the integrand) 
\begin{align*}
\partial_G \partial_G^j X & = \sum_{\ell = 0}^j \int_0^\infty  B_{X\delta;j,\ell}(r,s,c) \chi_{\neq}(s,c) \frac{1}{u'(r_c)} \partial_{r_c} \left(\frac{2i \eps }{(u-c)^2 + \eps^2}\right) w_{X\delta1;j,\ell}(s)  \partial_G^\ell F(s) \dd s \notag\\ 
& \quad + \sum_{\ell = 0}^j \int_0^\infty  \partial_G^{(r)} \left(B_{X\delta;j,\ell}(r,s,c) \chi_{\neq}(s,c)\right) \left(\frac{2i \eps }{(u-c)^2 + \eps^2}\right) w_{X\delta1;j,\ell}(s)  \partial_G^\ell F(s) \dd s \notag\\ 
& \quad + \sum_{\ell = 0}^j \int_0^\infty  \left(\frac{2i \eps }{(u-c)^2 + \eps^2}\right) \partial_G^{(r,s)} \left(B_{X\delta;j,\ell}(r,s,c) \chi_{c}(s,c) w_{X\delta1;j,\ell}(s)  \partial_G^\ell F(s)\right)  \dd s \notag\\ 
& \quad + \sum_{\ell = 0}^j \int_0^\infty \left(\int_0^\infty \chi_{\neq}(s_0,c) \frac{1}{u'(r_c)} \partial_{r_c}\left(\frac{2i\eps \beta(s_0)}{(u-c)^2 + \eps^2}\right) B^{(1)}_{X\delta;j,\ell}(r,s_0,c) B^{(2)}_{X\delta;j,\ell}(s_0,s,c) \dd s_0\right) \notag \\ & \quad\quad \times \frac{2 i \eps}{(u-c)^2 + \eps^2}  w_{X\delta2;j,\ell}(s) \partial_G^\ell F \dd s \notag\\ 
& \quad + \sum_{\ell = 0}^j \int_0^\infty \left(\int_0^\infty \left(\frac{2i\eps \beta(s_0)}{(u-c)^2 + \eps^2}\right) \partial_G^{(r)}\left(\chi_{\neq}(s_0,c) B^{(1)}_{X\delta;j,\ell}(r,s_0,c) \chi_{\neq}(s,c) B^{(2)}_{X\delta;j,\ell}(s_0,s,c)\right) \dd s_0\right) \notag \\ & \quad\quad \times \frac{2 i \eps}{(u-c)^2 + \eps^2}  w_{X\delta2;j,\ell}(s) \partial_G^\ell F \dd s \notag\\ 
& \quad + \sum_{\ell = 0}^j \int_0^\infty \left(\int_0^\infty \left(\frac{2i\eps \beta(s_0)}{(u-c)^2 + \eps^2}\right) \partial_G^{(r,s)}\left(\chi_{\neq}(s_0,c) B^{(1)}_{X\delta;j,\ell}(r,s_0,c) \chi_{c}(s,c) B^{(2)}_{X\delta;j,\ell}(s_0,s,c)\right) \dd s_0\right) \notag \\ & \quad\quad \times \frac{2 i \eps}{(u-c)^2 + \eps^2}  w_{X\delta2;j,\ell}(s) \partial_G^\ell F \dd s \notag\\ 
& \quad + \sum_{\ell = 0}^j \int_0^\infty \left(\int_0^\infty \left(\frac{2i\eps \beta(s_0)}{(u-c)^2 + \eps^2}\right) \partial_G^{(r,s_0)}\left(\chi_{c}(s_0,c) B^{(1)}_{X\delta;j,\ell}(r,s_0,c) \chi_{\neq}(s,c) B^{(2)}_{X\delta;j,\ell}(s_0,s,c)\right) \dd s_0\right) \notag \\ & \quad\quad \times \frac{2 i \eps}{(u-c)^2 + \eps^2}  w_{X\delta2;j,\ell}(s) \partial_G^\ell F \dd s \notag\\ 
& \quad + \sum_{\ell = 0}^j \int_0^\infty \left(\int_0^\infty \left(\frac{2i\eps \beta(s_0)}{(u-c)^2 + \eps^2}\right) \partial_G^{(r,s_0,s)}\left(\chi_{c}(s_0,c) B^{(1)}_{X\delta;j,\ell}(r,s_0,c) \chi_{c}(s,c) B^{(2)}_{X\delta;j,\ell}(s_0,s,c)\right) \dd s_0\right) \notag \\ & \quad\quad \times \frac{2 i \eps}{(u-c)^2 + \eps^2}  w_{X\delta2;j,\ell}(s) \partial_G^\ell F \dd s \notag\\ 
& \quad + \sum_{\ell = 0}^j \int_0^\infty \left(\int_0^\infty \left(\frac{2i\eps \beta(s_0)}{(u-c)^2 + \eps^2}\right) \left(B^{(1)}_{X\delta;j,\ell}(r,s_0,c) \chi_{\neq}(s,c) B^{(2)}_{X\delta;j,\ell}(s_0,s,c)\right) \dd s_0\right) \notag \\ & \quad\quad \times \frac{1}{u'(r_c)}\partial_{r_c}\left(\frac{2 i \eps}{(u-c)^2 + \eps^2}\right)  w_{X\delta2;j,\ell}(s) \partial_G^\ell F \dd s \notag\\ 
& \quad + \sum_{\ell = 0}^j \int_0^\infty \left(\int_0^\infty \left(\frac{2i\eps \beta(s_0)}{(u-c)^2 + \eps^2}\right) \left(B^{(1)}_{X\delta;j,\ell}(r,s_0,c) \chi_{c}(s,c) B^{(2)}_{X\delta;j,\ell}(s_0,s,c)\right) \dd s_0\right) \notag \\ & \quad\quad \times \left(\frac{2 i \eps}{(u-c)^2 + \eps^2}\right)  \frac{1}{u'(s)} \partial_s \left(w_{X\delta2;j,\ell}(s) \partial_G^\ell F \right) \dd s \notag \\ 
& \quad + \sum_{\ell = 0}^j \int_0^\infty \left(\int_0^\infty \chi_{\neq}(s_0,c) \frac{1}{u'(r_c)} \partial_{r_c}\left(\frac{2i\eps \beta(s_0)}{(u-c)^2 + \eps^2}\right) B^{(1)}_{XS;j,\ell}(r,s_0,c) B^{(2)}_{XS;j,\ell}(s_0,s,c) \dd s_0\right) \notag \\ & \quad\quad \times \frac{(u-c)}{(u-c)^2 + \eps^2}  w_{XS;j,\ell}(s) \partial_G^\ell F \dd s \notag\\ 
& \quad + \sum_{\ell = 0}^j \int_0^\infty \left(\int_0^\infty \left(\frac{2i\eps \beta(s_0)}{(u-c)^2 + \eps^2}\right) \partial_G^{(r)}\left(\chi_{\neq}(s_0,c) B^{(1)}_{XS;j,\ell}(r,s_0,c) \chi_{\neq}(s,c) B^{(2)}_{XS;j,\ell}(s_0,s,c)\right) \dd s_0\right) \notag \\ & \quad\quad \times \frac{(u-c)}{(u-c)^2 + \eps^2}  w_{XS;j,\ell}(s) \partial_G^\ell F \dd s \notag\\ 
& \quad + \sum_{\ell = 0}^j \int_0^\infty \left(\int_0^\infty \left(\frac{2i\eps \beta(s_0)}{(u-c)^2 + \eps^2}\right) \partial_G^{(r,s)}\left(\chi_{\neq}(s_0,c) B^{(1)}_{XS;j,\ell}(r,s_0,c) \chi_{c}(s,c) B^{(2)}_{XS;j,\ell}(s_0,s,c)\right) \dd s_0\right) \notag \\ & \quad\quad \times \frac{(u-c)}{(u-c)^2 + \eps^2}  w_{XS;j,\ell}(s) \partial_G^\ell F \dd s \notag\\ 
& \quad + \sum_{\ell = 0}^j \int_0^\infty \left(\int_0^\infty \left(\frac{2i\eps \beta(s_0)}{(u-c)^2 + \eps^2}\right) \partial_G^{(r,s_0)}\left(\chi_{c}(s_0,c) B^{(1)}_{XS;j,\ell}(r,s_0,c) \chi_{\neq}(s,c) B^{(2)}_{XS;j,\ell}(s_0,s,c)\right) \dd s_0\right) \notag \\ & \quad\quad \times \frac{(u-c)}{(u-c)^2 + \eps^2}  w_{XS;j,\ell}(s) \partial_G^\ell F \dd s \notag\\ 
& \quad + \sum_{\ell = 0}^j \int_0^\infty \left(\int_0^\infty \left(\frac{2i\eps \beta(s_0)}{(u-c)^2 + \eps^2}\right) \partial_G^{(r,s_0,s)}\left(\chi_{c}(s_0,c) B^{(1)}_{XS;j,\ell}(r,s_0,c) \chi_{c}(s,c) B^{(2)}_{XS;j,\ell}(s_0,s,c)\right) \dd s_0\right) \notag \\ & \quad\quad \times \frac{(u-c)}{(u-c)^2 + \eps^2}  w_{XS;j,\ell}(s) \partial_G^\ell F \dd s \notag\\ 
& \quad + \sum_{\ell = 0}^j \int_0^\infty \left(\int_0^\infty \left(\frac{2i\eps \beta(s_0)}{(u-c)^2 + \eps^2}\right) \left( B^{(1)}_{XS;j,\ell}(r,s_0,c) \chi_{\neq}(s,c) B^{(2)}_{XS;j,\ell}(s_0,s,c)\right) \dd s_0\right) \notag \\ & \quad\quad \times \frac{1}{u'(r_c)}\partial_{r_c}\left(\frac{u-c}{(u-c)^2 + \eps^2}\right)  w_{XS;j,\ell}(s) \partial_G^\ell F \dd s \notag\\ 
& \quad + \sum_{\ell = 0}^j \int_0^\infty \left(\int_0^\infty \left(\frac{2i\eps \beta(s_0)}{(u-c)^2 + \eps^2}\right) \left( B^{(1)}_{XS;j,\ell}(r,s_0,c) \chi_{c}(s,c) B^{(2)}_{XS;j,\ell}(s_0,s,c)\right) \dd s_0\right) \notag \\ & \quad\quad \times \left(\frac{(u-c)}{(u-c)^2 + \eps^2}\right)  \frac{1}{u'(s)} \partial_s \left(w_{X\delta2;j,\ell}(s) \partial_G^\ell F \right) \dd s \notag \\ 
& \quad + \sum_{\ell = 0}^j \int_0^\infty \left(\int_0^\infty \chi_{\neq}(s_0,c) \frac{1}{u'(r_c)} \partial_{r_c}\left(\frac{2i\eps \beta(s_0)}{(u-c)^2 + \eps^2}\right) B^{(1)}_{XG;j,\ell}(r,s_0,c) B^{(2)}_{XG;j,\ell}(s_0,s,c) \dd s_0\right) \notag \\ & \quad\quad \times w_{XG;j,\ell}(s) \partial_G^\ell F \dd s \notag\\ 
& \quad + \sum_{\ell = 0}^j \int_0^\infty \left(\int_0^\infty \left(\frac{2i\eps \beta(s_0)}{(u-c)^2 + \eps^2}\right) \partial_G^{(r)}\left(\chi_{\neq}(s_0,c) B^{(1)}_{XG;j,\ell}(r,s_0,c)  B^{(2)}_{XG;j,\ell}(s_0,s,c)\right) \dd s_0\right) \notag \\ & \quad\quad \times  w_{XG;j,\ell}(s) \partial_G^\ell F \dd s \notag\\ 
& \quad + \sum_{\ell = 0}^j \int_0^\infty \left(\int_0^\infty \left(\frac{2i\eps \beta(s_0)}{(u-c)^2 + \eps^2}\right) \partial_G^{(r,s_0)}\left(\chi_{c}(s_0,c) B^{(1)}_{XG;j,\ell}(r,s_0,c) B^{(2)}_{XG;j,\ell}(s_0,s,c)\right) \dd s_0\right) \notag \\ & \quad\quad \times  w_{XG;j,\ell}(s) \partial_G^\ell F \dd s \notag\\ 
& \quad + \textup{Similar terms with different $B$, $w$.}
\end{align*}
We see that, although slightly more technical, the overall structure of which derivatives appear in what contributions of the integrals, is very similar to the 
case of $\partial_{r_c} \partial_G^j X$. 
Hence, the arguments used above apply with no major variations and we may pass to the limit and deduce the estimate:
\begin{align*}
\norm{\lim_{\eps \rightarrow 0} \sqrt{r} w_{F,\delta} f_{1;Xc}}_{L^2} \lesssim_{k,\delta,\alpha} \abs{\omega_{k,0}^{in}} +  \sum_{\ell = 0}^k \norm{(r\partial_r)^\ell F}_{L^2_{F,\delta/4}};
\end{align*} 
we omit the repetitive details for brevity. 
This completes the treatment of $\tilde{f}_{1;X}$. 

Turn next to $\tilde{f}_{1;A}$. As in the cases $n \leq k -1$, we write 
\begin{align*}
\tilde{f}_{1;A} & = \frac{1}{2\pi i} \int_0^\infty \frac{2i\eps u'(r_c) \e^{itk(u(r)-u(r_c))}}{(u(r)-c)^2 + \eps^2} (ru'(r))^k  \chi_{I}(r_c) \frac{\beta(r)}{\sqrt{r}} \partial_G^k \left(X(r,c,\eps) + 2 Y(r,c-i\eps)\right)  \dd r_c \\ 
& = \tilde{f}_{1;AX} + \tilde{f}_{1;AY}. 
\end{align*}
The term $\tilde{f}_{1;AX}$ is treated in essentially the same manner as $\tilde{f}_{1;X}$ and is hence omitted for the sake of brevity. 
Similarly, we see that the treatment of $\tilde{f}_{1;AY}$ is made via a small variant of the method used to treat the first term in \eqref{eq:dGjX}. Hence, this is also omitted for the brevity. 
This completes the proof of Proposition \ref{prop:bdf1}. 
\end{proof} 

\begin{proof}[Proof of Proposition~\ref{prop:f2dec}]
Recall from \eqref{eq:f2IBP},  
\begin{align} 
w_{f,\delta}^{-1}(r\partial_r)^nf_2^\eps & = -\frac{1}{2\pi kt w_{f,\delta}(r)} \int_0^\infty \e^{itk(u(r)-u(r_c))} \partial_{r_c}\left(\frac{(u(r)-c)}{(u(r)-c)^2 + \eps^2}\right) (ru'(r)\partial_G)^n \left(\chi_{2}(r,r_c) \frac{\beta(r)}{\sqrt{r}} X(r,c,\eps)\right) \dd r_c \notag \\ 
& \quad -\frac{1}{2\pi kt w_{f,\delta}(r)} \int_0^\infty \e^{itk(u(r)-u(r_c))} \left(\frac{(u(r)-c)}{(u(r)-c)^2 + \eps^2}\right) \partial_{r_c}(ru'(r)\partial_G)^n \left(\chi_{2}(r,r_c) \frac{\beta(r)}{\sqrt{r}} X(r,c,\eps)\right) \dd r_c \notag \\ 
& = f_{2,a}^\eps + f_{2,b}^\eps. \label{eq:rdrnf2}
\end{align}
Due to the presence of $\chi_2$, the support of these integrands satisfies $r < \min(r_c/2,1)$. 
In particular, the integral in $r_c$ is not converging to a singular integral as $\eps \rightarrow 0$. 
Analogous to the treatment of $f_1^\eps$ in the proof of Proposition \ref{prop:bdf1} above,
we may write $f_{2,a}^\eps$ as the sum of terms of the general form 
\begin{align*}
-\frac{1}{2\pi kt} \int_0^\infty \e^{itk(u(r)-u(r_c))} \frac{1}{w_{f,\delta}(r)} H^\eps(r,c)  \partial_G^j X(r,c;\eps) \dd r_c, 
\end{align*}
for weights $H$ satisfying, 
\begin{align*}
\abs{H^\eps(r,c)} & \lesssim \mathbf{1}_{r < r_c/2} \mathbf{1}_{r \leq 1} r^{-1/2} \min(r^{2j}, r^{-2j-7}) \max(r_c^{-3},1). 
\end{align*}
From Proposition \ref{eq:dGjX} we have an expansion as in the proof Proposition \ref{prop:bdf1} above 
which by lemmas \ref{lem:Gsuitable}--\ref{lem:InductionIIO} (using also the recursion scheme laid out in Lemmas \ref{lem:XYYRecurse} and \ref{lem:XdYYrecurse}) satisfies similar properties. 
The main difference here is that we are using a weaker weight ($w_{f,\delta}$ instead of $w_{F,\delta}$) and we have lost an additional $r_c^{-2}$ from the integration by parts in $r_c$. 
The loss in $r_c$ is balanced by the gains in $\KK$; these were used to recover the strong weight on $f_1$ whereas here the gains are used to gain the $r_c^{-2}$ necessary to allow us to integrate by parts in $r_c$ to deduce. 
After this adjustment, the proof of convergence follows from Theorems \ref{thm:BBB:main} and \ref{thm:other:integral:operators} (together with Theorems \ref{thm:2SIO:1DELTA} and \ref{thm:2SIO:1DELTA:a}) as in the proof of Proposition \ref{prop:bdf1} and is hence omitted for brevity. 
  
Consider next $f_{2,b}^\eps$ in \eqref{eq:rdrnf2}. 
The terms where $\partial_{r_c}$ lands on $\chi_2$ are treated as in $f_{2,a}^\eps$. 
For terms containing $\partial_{r_c}\partial_G^j X$, we apply the same methods as in Proposition \ref{prop:bdf1} when $\partial_{r_c}$ derivatives were computed away from the critical layer as in \eqref{eq:drcdGjX}. Indeed, due to $\chi_2$, the entire integrand in $f_{2,b}^\eps$ is supported away from the critical layer and hence this is the only case we need to consider here. 
Hence, combining ideas in Proposition \ref{prop:bdf1} with those used to treat $f_{2,a}^\eps$ completes the desired bounds; we omit the details for brevity as they are repetitive. 
This completes the proof of Proposition \ref{prop:f2dec}. 
\end{proof} 

\appendix
\section{Preliminary technical lemmas} 
We record a few minor technical observations used several times in the proof. 
\begin{lemma} \label{lem:Trivrrc}
Let $r,r_c \in (0,\infty)$ and $k \geq 2$ such that $\abs{r - r_c} \leq r_c/k$. 
Then 
\begin{itemize} 
\item $\frac{k-1}{k} r_c \leq r \leq \frac{k+1}{k} r_c$; 
\item for all $a \in \Real$, there exists constants $c_a,C_a$ (depending only on $a$) such that $c_a r_c^{ak} \leq r^{ak} \leq C_{a}r_c^{ak}$. 
\end{itemize} 
\end{lemma}
The next lemma contains a few useful inequalities regarding $u$. The proof follows immediately from Lemma \ref{lem:BasicVort}. 
\begin{lemma} \label{lem:uNCL} 
There holds 
\begin{align}
\abs{\frac{\chi_{\neq} u'(r)}{u(r)-c}} & \lesssim \mathbf{1}_{r_c \leq 1} \left(\mathbf{1}_{r \leq 2} \min\left(\frac{kr}{r_c^2}, \frac{kr}{r^2}\right) + \frac{1}{r^3} \mathbf{1}_{r \geq 2} \right) + \mathbf{1}_{r_c \geq 1} \left( \mathbf{1}_{r \leq 2} r  + \mathbf{1}_{r \geq 2} \min\left(\frac{k r_c^2}{r^3}, \frac{k}{r}\right) \right), \label{ineq:uNCL} 
\end{align} 
and for $z \in I_\alpha$ there holds
\begin{align}
\abs{\frac{\eps u'}{(u-c)^2 + \eps^2}}\chi_{\neq} &\lesssim \eps^{\frac{2\alpha}{2+\alpha}} \frac{\abs{u'}\chi_{\neq}}{\abs{u-c}}, \label{ineq:uDeltaNCL} \\ 
\abs{\frac{(u-c)u'}{(u-c)^2 + \eps^2} - \frac{u'}{u-c}}\chi_{\neq} & \lesssim \eps^{\frac{2\alpha}{2+\alpha}} \frac{\abs{u'}\chi_{\neq}}{\abs{u-c}}. \label{ineq:uSIONCL}
\end{align}
\end{lemma} 

\begin{lemma} \label{lem:pvchic}
Let $\chi_c$ be defined as in \eqref{def:chic}. Then, the following holds independent of $c$: 
\begin{align}
\int_0^\infty \frac{(u(s)-c) u'(s)}{(u(s)-c)^2 + \eps^2} \chi_c(s,c) \dd s & \lesssim 1, \label{ineq:pvchic}
\end{align}
and for $\abs{r-r_c} \lesssim r_c/k$ there holds 
\begin{align*}
\int_0^r \frac{(u(s)-c) u'(s)}{(u(s)-c)^2 + \eps^2} \chi_c(s,c) \dd s & \lesssim 1 + \abs{\log \frac{k \abs{r-r_c}}{r_c}} \\ 
\int_r^\infty \frac{(u(s)-c) u'(s)}{(u(s)-c)^2 + \eps^2} \chi_c(s,c) \dd s & \lesssim 1 + \abs{\log \frac{k \abs{r-r_c}}{r_c}}. 
\end{align*}
\end{lemma} 
\begin{proof}[Proof of Lemma~\ref{lem:pvchic}]
Consider just \eqref{ineq:pvchic}; the other estimates follow similarly (and are slightly easier).  
Integration by parts yields the following for any $r$:  
\begin{align*}
\int_0^\infty \frac{(u-c) u'(s)}{(u(s)-c)^2 + \eps^2} \chi_c(s,c) \dd s & = -\int_0^\infty \log\left((u(s)-c)^2 + \eps^2\right) \partial_s\chi_c(s,c) \dd s \\ 
& = -\int_0^\infty \left(\log\left((u(s)-c)^2 + \eps^2\right)  -  \log\left((u(r)-c)^2 + \eps^2\right)\right) \partial_s\chi_c(s,c) \dd s \\ 
& = -\int_0^\infty \left(\log\frac{\left((u(s)-c)^2 + \eps^2\right)}{\left((u(r)-c)^2 + \eps^2\right)}\right) \partial_s\chi_c(s,c) \dd s.  
\end{align*}
Choose $r = (1 + \frac{1}{k})r_c$ and hence, on the support of the integrand,
\begin{align*}
\abs{\log\frac{\left((u(s)-c)^2 + \eps^2\right)}{\left((u(r)-c)^2 + \eps^2\right)}} \lesssim \abs{\frac{(u(s)-c)^2 - (u(r)-c)^2}{\left((u(r)-c)^2 + \eps^2\right)}} \lesssim 1. 
\end{align*}
and hence \eqref{ineq:pvchic} follows. 
\end{proof}

\begin{lemma} \label{lem:SIOConv}
For all $0 < \gamma \leq 1$, for all $\eta$ sufficiently small (depending on $\gamma$ and $\alpha$) and all  $z \in I_\alpha$, there holds 
\begin{align}
\int_0^\infty \frac{\eps \abs{u'(r)}}{(u-c)^2 + \eps^2} \left(\frac{\abs{r-r_c}}{r_c}\right)^\gamma \chi_c \dd r \lesssim \eps^\eta. \label{ineq:deltaCon}
\end{align}
Let $G^\eps(r,c)$ be defined for $z = c\pm i \eps \in I_\alpha$ and $\abs{r-r_c} \lesssim r_c$ and (over the same range of $r,r_c,\eps$) satisfy the following estimates (uniformly in $\eps$) for some exponents $\gamma_i \in (0,1]$: 
\begin{align} 
\abs{G^\eps(r,c) - G^\eps(r_c,c)} & \lesssim \left(\frac{\abs{r-r_c}}{r_c}\right)^{\gamma_0} \\ 
\abs{G^\eps(r,c) - G^0(r,c)} & \lesssim  \eps^{\gamma_1}. 
\end{align}
Then for all $\eta$ sufficiently small (depending on $\alpha, \gamma_i$), 
\begin{align}
\abs{\int_0^\infty \frac{(u-c) u'}{(u-c)^2 + \eps^2} \chi_c G^\eps(r,c) \dd s - p.v. \int_0^\infty \frac{u'}{u-c} \chi_c G^0(r,c) \dd s} \lesssim \eps^\eta. \label{ineq:pvCon}
\end{align} 
\end{lemma} 
\begin{proof}[Proof of Lemma~\ref{lem:SIOConv}]
Consider first \eqref{ineq:deltaCon}. 
For all $0 < p < \gamma$, by Lemmas \ref{lem:Trivrrc} and \ref{lem:BasicVort},   
\begin{align*} 
\int_0^\infty \frac{\eps \abs{u'(r)}}{(u-c)^2 + \eps^2} \left(\frac{\abs{r-r_c}}{r_c}\right)^\gamma \chi_c \dd r & \lesssim \eps^{p} \frac{k^p}{\abs{u'(r_c)}^{p}} \int_0^\infty \frac{1}{\abs{r-r_c}^{1 + p -\gamma}} \dd r  \lesssim \eps^p k^p \max(\frac{1}{r_c^{2p}},r_c^{2p}),  
\end{align*} 
and hence \eqref{ineq:deltaCon} follows from the definition of $I_\alpha$.  
Next, consider \eqref{ineq:pvCon}. We have 
\begin{align*}
\abs{\int_0^\infty \frac{(u-c) u'}{(u-c)^2 + \eps^2} \chi_c G^\eps (r,c) \dd s - p.v.\int_0^\infty \frac{u'}{u-c} \chi_c G^0(r,c) \dd s} & \leq 
\\ & \hspace{-6cm}  \abs{\int_0^\infty \frac{(u-c) u'}{(u-c)^2 + \eps^2} \chi_c \left(G^\eps(r,c) -  G^0(r,c)\right) \dd s  } 
\\ & \hspace{-6cm} \quad +    \abs{p.v.\int_0^\infty \left(\frac{(u-c) u'}{(u-c)^2 + \eps^2} - \frac{1}{u-c}\right) \chi_c G^0(r,c) \dd s} \\ 
& = T_1 + T_2. 
\end{align*}
By the assumptions on $G^\eps$, we have  
\begin{align*}
\chi_c\abs{G^\eps(r,c) - G^\eps(r_c,c) -  G^0(r,c) + G^0(r_c,c)} & \\ 
&  \hspace{-7cm} \leq \chi_c \left(\abs{G^\eps(r,c) - G^\eps(r_c,c)} + \abs{G^0(r,c) - G^0(r_c,c)}\right)^{\gamma}\left( \abs{G^{\eps}(r,c) - G^0(r,c)} + \abs{G^\eps(r_c,c) - G^0(r_c,c)}\right)^{1-\gamma} \\ 
& \hspace{-7cm} \lesssim \eps^{\gamma_1(1-\gamma)} \left(\frac{\abs{r-r_c}}{r_c}\right)^{\gamma \gamma_0}. 
\end{align*} 
Therefore, setting $\eta= \gamma_1(1-\gamma)$ and using Lemma \ref{lem:pvchic} implies
\begin{align*}
T_1 & \leq \abs{G^\eps(r_c,c) -  G^0(r_c,c)} \abs{\int_0^\infty \frac{(u-c) u'}{(u-c)^2 + \eps^2} \chi_c \dd r} \\ 
& \quad + \abs{\int_0^\infty \frac{(u-c) u'}{(u-c)^2 + \eps^2}\chi_c\left(G^\eps(r,c) - G^\eps(r_c,c) -  G^0(r,c) + G^0(r_c,c)\right) \dd r} \\ 
& \lesssim \eps^\eta. 
\end{align*}
The proof of $T_2$ follows from noting: 
\begin{align*}
T_2 & \leq \abs{G^0(r_c,c) p.v.\int_0^\infty \left(\frac{\eps^2 u'}{(u-c)((u-c)^2 + \eps^2)} \right) \chi_c \dd s } \\ 
& \quad + \abs{p.v.\int_0^\infty \left(\frac{\eps^2 u'}{(u-c)((u-c)^2 + \eps^2)} \right)  \chi_c \left(G^0(r,c) - G^0(r_c,c)\right) \dd s}. 
\end{align*}
The latter integral is treated by an easy variant of the treatment of $T_1$ and is hence omitted. 
The former integral is estimated via 
\begin{align*}
\abs{p.v.\int_0^\infty \left(\frac{\eps^2 u'}{(u-c)((u-c)^2 + \eps^2)} \right) \chi_c \dd s } & = \abs{\int_0^\infty \chi_c \frac{1}{2} \partial_s \left(\log((u-c)^2 + \eps^2) - \log (u-c)^2 \right) \dd s} \\ 
& \lesssim \int_0^\infty \abs{\frac{(u-c)^2 + \eps^2}{(u-c)^2} - 1} \abs{\chi_c'} \dd s \\ 
& \lesssim \eps^2 k^2 \max(r_c^{-4},r_c^4),
\end{align*}
which completes the proof by the definition of $I_\alpha$. 
\end{proof}

\section{Vorticity depletion implies optimal inviscid damping} \label{apx:VortDepl}
In this appendix we prove Lemma \ref{lem:VortDepID}: the statement that the vorticity depletion estimates \eqref{ineq:VortDep} imply the inviscid damping estimates \eqref{ineq:ID}. 
We may without loss of generality consider the case $\abs{kt} \geq 1$. 
Denote $G_k(r,\rho)$ to be the Green's function for the Laplacian restricted to the $k$-th angular Fourier mode, i.e.
\begin{align*}
G_k(r,\rho) & = \frac{\rho}{k}
\begin{cases}
\frac{\rho^k}{r^k} \quad\quad \rho \leq r \\
\frac{r^k}{\rho^k} \quad\quad \rho \geq r. 
\end{cases}
\end{align*}
Using the decomposition of $\omega_k$ as in \eqref{eq:vdecomp}, we have
\begin{align*}
\psi_k(t,r) &= \int_0^\infty G_k(r,\rho) \e^{-iktu(\rho)} f_{k;1}(t,\rho) d\rho + \int_0^\infty G_k(r,\rho) \e^{-iktu(\rho)} f_{k;2}(t,\rho) d\rho \notag\\
&:= \psi_{k;1}(t,r) + \psi_{k;2}(t,r).
\end{align*}
It is convenient to denote
\begin{align*}
\sup_{t\geq 0} \sum_{n=0}^2 k^{-n}\norm{\sqrt{r} (r \partial_r)^n f_{k;1}(t)}_{L^2_{F,\delta}} + \sup_{t\geq 0} \brak{kt} \sum_{n=0}^1 k^{-n} \norm{(r \partial_r)^n f_{k;2}(t)}_{L^2_{f,\delta}} = M_0
\end{align*}
where by \eqref{ineq:VortDep} we have that $M_0$ is bounded in terms of the datum $\omega_{k}^{in}$. For the contribution from $\psi_{k;2}$, we integrate by parts and obtain
\begin{align*}
(kt) \brak{kt}\frac{|\psi_{k;2}(t,r)|}{w_{\psi,2\delta}(r)} 
& = \brak{kt} \left| \int_0^\infty \frac{-i \e^{-iktu(\rho)}}{w_{\psi,2\delta}(r)}  \partial_{\rho} \left(\frac{1}{u'(\rho)} G_k(r,\rho) f_{k;2}(t,\rho)\right) d\rho  \right|
\notag\\
&\leq  \int_0^\infty T_2(r,\rho)
\frac{\brak{kt} \left(|f_{k;2}(t,\rho)| + k^{-1}|(\rho \partial_\rho)f_{k;2}(t,\rho)|\right)}{w_{f,\delta}(\rho)} d\rho, 
\end{align*}
where
\begin{align*}
T_2(r,\rho) := \frac{w_{f,\delta}(\rho) G_k(r,\rho)}{\rho |u'(\rho)| w_{\psi,2\delta}(r)}  \left(\frac{|(\rho \partial_\rho) G_k(r,\rho)|}{G_k(r,\rho)} + \frac{k \rho |u''(\rho)|}{|u'(\rho)|} + k \right). 
\end{align*}
Using that (recall the strong decay imposed on $f$ \eqref{eq:weightf} at infinity), 
\begin{align*}
\norm{ T_2(r,\rho) }_{L^2(dr \, d\rho)} \lesssim k \norm{\frac{w_{f,\delta}(\rho) G_k(r,\rho)}{\rho |u'(\rho)| w_{\psi,2\delta}(r)}}_{L^2(dr \, d\rho)} \lesssim \norm{\frac{w_{f,\delta}(\rho) \min\left\{ \frac{\rho^k}{r^k}, \frac{r^k}{\rho^k} \right\}}{|u'(\rho)| w_{\psi,2\delta}(r)}}_{L^2(dr \, d\rho)} \lesssim 1,
\end{align*}
we obtain, 
\begin{align*}
 \norm{\psi_{k;2}(t)}_{L^2_{\psi,2\delta}} \lesssim \frac{M_0}{(kt) \brak{kt}},
\end{align*}
as desired. Similarly, for the $\psi_{k;1}$ contribution, we integrate by parts in $\rho$ twice, keeping track of the boundary terms arising from the second derivative of the Green's function. We arrive at
\begin{align*}
(kt)^2  \psi_{k;1}(t,r) 
&=2 \e^{-iktu(r)} \frac{f_{k;1}(t,r)}{(u'(r))^2} - \int_0^\infty \e^{-iktu(\rho)} \partial_\rho\left( \frac{1}{u'(\rho)} \partial_\rho \left(  \frac{G_k(r,\rho)}{u'(\rho)} f_{k;1}(t,\rho) \right)\right) \dd \rho
\end{align*}
from which we deduce
\begin{align*}
(kt)^2 \frac{|\psi_{k;1}(t,r)|}{w_{\psi,2\delta}(r)}
&\leq \frac{2 w_{F,\delta}(r)}{w_{\psi,2\delta}(r)\sqrt{r} (u'(r))^2} \frac{\sqrt{r} |f_{k;1}(t,r)|}{w_{F,\delta}(r)} \notag\\
&\qquad + \int_0^\infty T_1(r,\rho) \frac{\sqrt{\rho} \left( |f_{k;1}(t,\rho)| + k^{-1}|(\rho \partial_\rho) f_{k;1}(t,\rho)| + k^{-2}|(\rho \partial_\rho)^2 f_{k;1}(t,\rho)|\right)}{w_{F,\delta}(\rho)} \dd \rho
\end{align*}
where 
\begin{align*}
T_1(r,\rho) 
&= \frac{w_{F,\delta}(\rho)G_k(r,\rho)}{\rho^2(u'(\rho))^2\sqrt{\rho} w_{\psi,2\delta}(r)} \Bigg(\frac{|(\rho \partial_{\rho})^2 G_k(r,\rho)|}{G_k(r,\rho)} + k\left( 3 + \frac{3 \rho |u''(\rho)|}{|u'(\rho)|}\right) \frac{|\rho \partial_{\rho} G_k(r,\rho)|}{G_k(r,\rho)} + k^2  \notag\\
&\qquad \qquad\qquad \qquad\qquad\qquad\qquad+ \frac{k^2 \rho^2 |u'''(\rho)|}{|u'(\rho)|} + \frac{3 k^2 \rho^2 (u''(\rho))^2}{(u'(\rho))^2} + \frac{3 k^2 \rho |u''(\rho)| }{ |u'(\rho)|}\Bigg)
\notag\\
&\lesssim k  \frac{w_{F,\delta}(\rho) \min\left\{ \frac{\rho^k}{r^k}, \frac{r^k}{\rho^k}\right\}}{\rho^2(u'(\rho))^2\sqrt{\rho} w_{\psi,2\delta}(r)}.
\end{align*}
Therefore, using that 
\begin{align*}
\norm{\frac{w_{F,\delta}(r)}{w_{\psi,2\delta}(r)\sqrt{r} (u'(r))^2}}_{L^\infty(dr)}
+ \norm{\frac{w_{F,\delta}(\rho) \min\left\{ \frac{\rho^k}{r^k}, \frac{r^k}{\rho^k}\right\}}{\rho^{2+1/2} (u'(\rho))^2 w_{\psi,2\delta}(r)}}_{L^2(dr \, d\rho)} \lesssim 1
\end{align*}
which may be checked directly, we obtain that
\begin{align*}
\norm{\psi_{k;1}(t)}_{L^2_{\psi,2\delta}} \lesssim \frac{M_0}{k^2 t^2}
\end{align*}
which is the desired estimate. The inviscid damping of the velocity field follows in a similar manner from the Biot-Savart law, or by noting that $r(u_k^r,u_k^\theta) = (ik , - r\partial_r)\psi_k$, and we omit these details to avoid redundancy.

\section{Boundedness and convergence of integral operators} \label{sec:BndConv}

\subsection{Two singular integrals and one delta distribution}
Our goal is to prove the convergence as $\eps\to 0$ of the following ``model operator'':
\begin{align}
L_{\eps} [f](r) 
&= \int_{0}^\infty \!\!\! \int_{0}^\infty \!\!\! \int_{0}^\infty \frac{(u(r)-u(r_c)) u'(r_c)}{(u(r)-u(r_c))^2+\eps^2}  \frac{(u(s)-u(r_c)) u'(s)}{(u(s)-u(r_c))^2 + \eps^2}  \frac{\eps u'(s_0)}{(u(s_0)-u(r_c))^2 +\eps^2} \notag\\
&\qquad \qquad \qquad \times {\mathfrak B}_{\eps,1}(r,s_0,r_c) {\mathfrak B}_{\eps,2}(s_0,s,r_c) f(s) \, \dd s  \dd s_0 \dd r_c
\label{eq:L:eps:def}
\end{align}
in $L^2(dr)$ as $\eps\to 0$, under certain assumptions on the weights $B_{\eps,1}$ and $B_{\eps,2}$.

\begin{theorem}
\label{thm:BBB:main}
Let $\delta \in (0, \frac 12 )$ and assume that for some $\gamma \in(0, \frac{\delta}{4})$  we have that the functions $ {\mathfrak B}_{\eps,1}$ and $ {\mathfrak B}_{\eps,2}$ obey the conditions~\eqref{eq:BBB:cond:0}, \eqref{eq:BBB:cond:*}, and either \eqref{eq:BBB:cond:0:*}--\eqref{eq:BBB:cond:0:**} or \eqref{eq:BBB:cond:0:*:dual}--\eqref{eq:BBB:cond:0:**:dual}. Additionally, assume that there exists $\zeta \in (0,\gamma)$ such that conditions \eqref{eq:BBB:cond:nasty:1}--\eqref{eq:BBB:cond:nasty:2} hold, for some limiting weights ${\mathfrak B}_{0,1}$ and ${\mathfrak B}_{0,2}$. Then, if $ f \in L^2$, we have that the operator $L_{\eps}[f]$, defined in \eqref{eq:L:eps:def}, converges  as $\eps\to 0$, in $L^2$ to the operator $L_{0}[f]$, defined by duality via
\begin{align}
\left\langle L_{0} [f] , \varphi \right\rangle 
&= -\pi \int_{0}^\infty   \left(p.v. \int_{0}^\infty  \frac{u'(r)}{u(r)-u(r_c)}    \frac{u'(r_c) {\mathfrak B}_{0,1}(r,r_c,r_c)}{u'(r) }  \varphi(r) \, \dd r \right) \notag\\
&\qquad \qquad \times \left(p.v. \int_{0}^\infty  \frac{u'(s)}{u(s)-u(r_c)}    \frac{{\mathfrak B}_{0,2}(r_c,s,r_c)}{\brak{s}^\delta} (\brak{s}^\delta f(s)) \, \dd s\right)\dd r_c
\label{eq:L:0:def}
\end{align}
and the operator $L_0$ is bounded from $L^2$ to $L^2$, with norm less than $k^\zeta$.
\end{theorem}
The first standard example of pairs of weights ${\mathfrak B}_{\eps,1}$ and ${\mathfrak B}_{\eps,2}$ which obey the conditions of Theorem~\ref{thm:BBB:main} are:
 
\begin{theorem}
\label{thm:2SIO:1DELTA}
Let $0 \leq j \leq k-1$, and $0\leq \ell, \ell_1,\ell_2$ be such that $\ell + \ell_1 + \ell_2 \leq j$. Let $0 < \zeta \leq \frac{\delta}{6}$, and   $0< \eta \leq \frac{\zeta}{4}$. Consider the weights
\begin{align}
{\mathfrak B}_{\eps,1}(r,s_0,r_c) 
&=  \chi_1(r,r_c) \frac{\beta(r)\min(r^2,r^{-2})^j}{w_{F,\delta}(r) u'(s_0)}  B_{\ell,\eps}^{(1)}(r,s_0,r_c) 
\label{eq:B1:2SIO:1DELTA}\\
{\mathfrak B}_{\eps,2}(s_0,s,r_c) 
&=  \frac{\beta(s_0) w_{F,\frac{\delta}{4} + 2\ell}(s)}{u'(s)}   B_{\ell,\eps}^{(2)}(s_0,s,r_c)
\label{eq:B2:2SIO:1DELTA}
\end{align} 
where $B_{\ell,\eps}^{(1)}$ is a suitable $(2\ell_1,\ell_1+\eta/2)$ kernel of type $I$ or $II$, and $B_{\ell,\eps}^{(2)}$ is a suitable $(2\ell_2,\ell_2+\eta/2)$ kernel of type $I$ or $II$.
Then the conditions of Theorem~\ref{thm:BBB:main} are satisfied for the weighs \eqref{eq:B1:2SIO:1DELTA}--\eqref{eq:B2:2SIO:1DELTA}. The corresponding operator $L_\eps$ defined in \eqref{eq:L:eps:def} converges to the corresponding operator $L_0$ defined in \eqref{eq:L:0:def}, which in this case becomes
\begin{align*}
\left\langle L_{0} [f] , \varphi \right\rangle 
&= -\pi \int_{0}^\infty   \left(p.v. \int_{0}^\infty  \frac{\chi_1(r,r_c)}{u(r)-u(r_c)} \frac{\beta(r)\min(r^2,r^{-2})^j}{w_{F,\delta}(r)}  B_{\ell,0}^{(1)}(r,r_c,r_c)  \varphi(r) \, \dd r \right)  \notag\\
&\qquad \qquad \times \left(p.v. \int_{0}^\infty  \frac{1}{u(s)-u(r_c)}     \frac{\beta(r_c) w_{F,\frac{\delta}{4} + 2\ell}(s)}{\brak{s}^\delta}  B_{\ell,0}^{(2)}(r_c,s,r_c)   (\brak{s}^\delta f(s)) \, \dd s\right)\dd r_c,
\end{align*}
as operators from $L^2$ to $L^2$. The operator $L_0$ is bounded on $L^2$   with norm $\lesssim k^{\zeta + 2\ell_1+2\ell_2}$.
\end{theorem}
\begin{proof}[Proof of Theorem~\ref{thm:2SIO:1DELTA}]
The theorem follows from Theorem~\ref{thm:BBB:main}, upon verifying that the weights in \eqref{eq:B1:2SIO:1DELTA}--\eqref{eq:B2:2SIO:1DELTA} obey the needed conditions. This is done in Corollary~\ref{cor:abstract:off:diagonal:1}, Corollary~\ref{cor:abstract:off:diagonal:2}, and Corollary~\ref{cor:abstract:4} below.  
\end{proof}
The second standard example of pairs of weights ${\mathfrak B}_{\eps,1}$ and ${\mathfrak B}_{\eps,2}$ which obey the conditions of Theorem~\ref{thm:BBB:main} are:
\begin{theorem}
\label{thm:2SIO:1DELTA:a}
Let $0 \leq j \leq  {k-1}$, and $0\leq \ell, \ell_1,\ell_2$ be such that $\ell + \ell_1 + \ell_2 \leq j$. Let $0 < \zeta \leq \frac{\delta}{6}$, and   $0< \eta \leq \frac{\zeta}{4}$. Consider the weights
\begin{align}
{\mathfrak B}_{\eps,1}(r,s_0,r_c) 
&=  \chi_2(r,r_c) \frac{\beta(r)\min(r^2,r^{-2})^j}{r^{\frac 12} w_{f,\delta}(r) u'(s_0) r_c u'(r_c)}  B_{\ell,\eps}^{(1)}(r,s_0,r_c) 
\label{eq:B1:2SIO:1DELTA:a}\\
{\mathfrak B}_{\eps,2}(s_0,s,r_c) 
&=  \frac{\beta(s_0) w_{F,\frac{\delta}{4} + 2\ell}(s)}{u'(s)}   B_{\ell,\eps}^{(2)}(s_0,s,r_c)
\label{eq:B2:2SIO:1DELTA:a}
\end{align} 
where $B_{\ell,\eps}^{(1)}$ is a suitable $(2\ell_1,\ell_1+\eta/2)$ kernel of type $I$ or $II$, and $B_{\ell,\eps}^{(2)}$ is a suitable $(2\ell_2,\ell_2+\eta/2)$ kernel of type $I$ or $II$.
Then the conditions of Theorem~\ref{thm:BBB:main} are satisfied for the weighs \eqref{eq:B1:2SIO:1DELTA:a}--\eqref{eq:B2:2SIO:1DELTA:a}. The corresponding operator $L_\eps$ defined in \eqref{eq:L:eps:def} converges to the corresponding operator $L_0$ defined in \eqref{eq:L:0:def}, as operators from $L^2$ to $L^2$, and the limiting operator is bounded on this space, with norm bounded by $k^{\zeta + 2\ell_1 + 2\ell_2}$.
\end{theorem}
\begin{proof}[Proof of Theorem~\ref{thm:2SIO:1DELTA:a}]
We remark that the main difference between \eqref{eq:B1:2SIO:1DELTA} and \eqref{eq:B1:2SIO:1DELTA:a} is a factor proportional to 
\begin{align*}
\frac{r^2 \brak{r_c}^4}{r_c^2} 
\end{align*}
besides the obvious difference of replacing $\chi_1$ with $\chi_2 \approx \indicator{r\leq 1} \indicator{2r\leq r_c}$.
The theorem follows from Theorem~\ref{thm:BBB:main}, upon verifying that the weights in \eqref{eq:B1:2SIO:1DELTA:a}--\eqref{eq:B2:2SIO:1DELTA:a} obey the needed conditions. This is done in Corollary~\ref{cor:abstract:off:diagonal:1:a}, Corollary~\ref{cor:abstract:off:diagonal:2}, and Corollary~\ref{cor:abstract:5} below.
\end{proof}

\begin{remark} \label{rmk:NotSuitOK}
\label{rem:k:th:derivative}
It is clear from the proof of Theorems~\ref{thm:BBB:main}, \ref{thm:2SIO:1DELTA}, and \ref{thm:2SIO:1DELTA:a} that not all properties of a type I or type II kernel are used.  For instance, for a type I kernel, Theorems~\ref{thm:2SIO:1DELTA} and \ref{thm:2SIO:1DELTA:a} only use  \eqref{ineq:KbdTypeI}, \eqref{ineq:drKbdTypeI}, \eqref{ineq:KtypeI_Holder1}, and the convergence as $\eps\to 0$ in these inequalities, with some positive rate $\eps^\zeta$ for some $\zeta>0$. Similarly, for a type II kernel, Theorems~\ref{thm:2SIO:1DELTA} and \ref{thm:2SIO:1DELTA:a} only use the global uniform boundedness, weighted H\"older regularity  in each of the two variables not called $r_c$ near the critical layer, and the convergence as $\eps \to 0$ in these inequalities, at a positive rate.  
\end{remark}

The remainder of this section is dedicated to the proof of Theorem~\ref{thm:BBB:main}, which is decomposed into several steps, detailed in the following subsections. In each subsection, we show that the weights \eqref{eq:B1:2SIO:1DELTA}--\eqref{eq:B2:2SIO:1DELTA} obey the necessary properties, so that the proof of Theorem~\ref{thm:2SIO:1DELTA} is done concomitantly.  Checking that the weights \eqref{eq:B1:2SIO:1DELTA:a}--\eqref{eq:B2:2SIO:1DELTA:a} obey the necessary properties is done at the end of this section, yielding the proof of Theorem~\ref{thm:2SIO:1DELTA:a}.

\subsubsection{Convergence away from the diagonal $s_0=r_c$}
In this section we consider the contribution to the operator $L_\eps$ in \eqref{eq:L:eps:def} due to the set
\begin{align*}
\left\{ |s_0-r_c| \geq \frac{r_c}{k} \right\}.
\end{align*}
We first prove an abstract lemma, and then show that the available conditions on the coefficients ${\mathfrak B}_{\eps,1}$ and ${\mathfrak B}_{\eps,2}$ are sufficient in order to apply this lemma.
Let us denote by $L_{\eps,1}$ the contribution to the operator $L_\eps$ in \eqref{eq:L:eps:def} from  $|s_0-r_c|\geq \frac{r_c}{k}$, i.e. the operator 
\begin{align}
L_{\eps,1} [f](r) 
&= \int_{0}^\infty \!\!\! \int_{0}^\infty \!\!\! \int_{0}^\infty \frac{(u(r)-u(r_c)) u'(r_c)}{(u(r)-u(r_c))^2+\eps^2}  \frac{(u(s)-u(r_c)) u'(s)}{(u(s)-u(r_c))^2 + \eps^2}  \frac{\eps u'(s_0)}{(u(s_0)-u(r_c))^2 +\eps^2} \notag\\
&\qquad \qquad \qquad \times \indicator{|s_0-r_c| \geq \frac{r_c}{k}} {\mathfrak B}_{\eps,1}(r,s_0,r_c) {\mathfrak B}_{\eps,2}(s_0,s,r_c) f(s) \, \dd s  \dd s_0 \dd r_c
\label{eq:L:eps:1:def}.
\end{align}

\begin{lemma}
\label{lem:abstract:off:diagonal:1}
Assume that 
\begin{align}
\left| {\mathfrak B}_{\eps,1}(r,s_0,r_c) {\mathfrak B}_{\eps,2}(s_0,s,r_c) \right| \lesssim {\mathfrak B}_0(r,s_0,s,r_c) 
\label{eq:BBB:cond:0}
\end{align}
holds for some $\eps$-independent function ${\mathfrak B}_0$. In addition, define the cut-off 
\begin{align}
\indicator{\rm strange} = \indicator{\frac{r_c}{2} \leq r \leq 1}  \indicator{r_c\leq 1} \indicator{2 r\leq s}  \indicator{2 r_c \leq s} + \indicator{r\leq 1} \indicator{2r \leq r_c} \label{def:strange} 
\end{align} 
and assume that ${\mathfrak B}_0$ obeys the bound 
\begin{align}
{\mathfrak B}_0(r,s_0,s,r_c)    \lesssim   \left(\indicator{\rm strange} \frac{s^{\frac{1}{2}}}{r^{\frac 12} \brak{s}^{\frac 12} } + (1 - \indicator{\rm strange}) \right)  \frac{  r^{\frac{\delta}{2}}     \brak{s}^{\frac 12 } }{\brak{r}^{\frac 12 + \frac{3\delta}{2}} \brak{s_0}^{3\gamma}} 
\left( \max \left\{ r^2,r^{-2}, s^2, s^{-2}, s_0^2,s_0^{-2} \right\} \right)^\eta
\label{eq:BBB:cond:*}
\end{align}
uniformly in $r,s_0,s$, and $r_c$, 
for some $\eta \in (0,\frac{\gamma}{16})$, $\gamma \in(0, \frac{\delta}{4})$  and $\delta \in (0, \frac 12 )$.
Then, if $ f \in L^2$, we have that $L_{\eps,1}[f] \to 0$ as $\eps\to 0$, in $L^2 $.
\end{lemma}
\begin{proof}[Proof of Lemma~\ref{lem:abstract:off:diagonal:1}]

Let $\varphi \in L^2(\RR_+)$ be arbitrary. We then have by \eqref{eq:BBB:cond:0} that
\begin{align}
\left| \left\langle  L_{\eps,1}[f](r) ,\varphi(r) \right\rangle  \right|
&\leq
\int_{\RR_+^4} \frac{|(u(r)-u(r_c)) u'(r_c)|}{(u(r)-u(r_c))^2+\eps^2}  \frac{|(u(s)-u(r_c)) u'(s)|}{(u(s)-u(r_c))^2 + \eps^2}  \frac{\eps |u'(s_0)|}{(u(s_0)-u(r_c))^2 +\eps^2} \notag\\
&\qquad \qquad \qquad \times \indicator{|s_0-r_c| \geq \frac{r_c}{k}}  {\mathfrak B}_0(r,s_0,s,r_c)   \, |  f(s) \varphi(r) |\, \dd s  \dd s_0 dr_c  \notag\\
&=: \int_{\RR_+^4}  {\mathbb J}(r,s_0,s,r_c) \, | f(s)  \varphi(r) | \, \dd s  \dd s_0 \dd r_c  \dd r. 
\label{eq:Merlot:1}
\end{align}
Our goal is to show that the integrand on the right side of \eqref{eq:Merlot:1} lies in $L^1(dr\, \dd s\, \dd s_0\, d r_c)$, and moreover vanishes as $\eps\to 0$ in this norm.

\paragraph*{Case $r_c\geq 1$}
The proof is based on the following estimate (c.f. \eqref{ineq:uNCL}), 
\begin{align}
\frac{|u'(\rho)|}{|u(\rho) - u(t)|} \lesssim \indicator{|\rho-t| \leq \frac{1}{10}} \frac{1}{|\rho- t|} \left( \indicator{\rho\leq 1} \frac{\rho}{\rho+t} + \indicator{\rho \geq 1}\right)+ \indicator{|\rho-t| \geq \frac{1}{10}} \frac{1}{\brak{\rho}}
\label{eq:u':info:1}
\end{align}
and the asymptotic description
\begin{align}
|u'(\rho)| \approx \frac{\rho}{\brak{\rho}^4}.
\label{eq:u':info:2}
\end{align}
Here we use \eqref{eq:u':info:1}--\eqref{eq:u':info:2} to estimate
\begin{align}
\frac{\indicator{r_c\geq 1} \eps^{\frac{\gamma}{3}} |u(r)-u(r_c)| \, |u'(r_c)|}{(u(r) -u(r_c))^2 + \eps^2}
&\lesssim \indicator{r_c\geq 1}  |u'(r_c)|^{\frac{\gamma}{3}} \left(\frac{|u'(r_c)|}{|u(r_c) - u(r)|} \right)^{1- \frac{\gamma}{3}} \notag\\
&\lesssim  \frac{\indicator{r_c\geq 1} \indicator{|r_c-r|\leq \frac{1}{10}}}{\brak{r_c}^\gamma |r_c-r|^{1-\frac{\gamma}{3}}}+ \frac{\indicator{r_c\geq 1}\indicator{|r_c-r|\geq \frac{1}{10}}}{\brak{r_c}^{1 + \frac{2\gamma}{3}}}
\label{eq:Warm:Merlot:1}
\end{align}
and similarly, using that $|s-r_c| \leq \frac{1}{10} \Rightarrow s \geq r_c - \frac{1}{10} \geq \frac{9}{10}$ for $r_c \geq 1$, we obtain
\begin{align}
\frac{\indicator{r_c\geq 1} \eps^{\frac{\gamma}{3}} |u(s)-u(r_c)| \, |u'(s)|}{(u(s) -u(r_c))^2 + \eps^2}
&\lesssim 
\frac{\indicator{r_c\geq 1} \indicator{|s-r_c|\leq \frac{1}{10}} \indicator{s\geq \frac{9}{10}}}{\brak{s}^\gamma |s-r_c|^{1- \frac{\gamma}{3}}} +    \frac{ s^{\frac{\gamma}{3}} \indicator{r_c\geq 1} \indicator{|s-r_c|\geq \frac{1}{10}}}{\brak{s}^{1 + \gamma}}.
\label{eq:Warm:Merlot:2}
\end{align}
Lastly, we have
\begin{align}
\frac{\indicator{r_c\geq 1} \indicator{|s_0-r_c| \geq \frac{r_c}{k}} \eps^{1- \gamma} |u'(s_0)|}{(u(s_0)-u(r_c))^2 +\eps^2}
&\lesssim \indicator{r_c\geq 1} \indicator{|s_0-r_c| \geq \frac{r_c}{k}}  \frac{1}{|u'(s_0)|^\gamma} \left(\frac{|u'(s_0)|}{|u(s_0)-u(r_c)|}\right)^{1+\gamma} 
\notag\\
&\lesssim  \frac{\indicator{r_c\geq 1} \indicator{|s_0-r_c| \geq \frac{1}{10}}}{s_0^\gamma \brak{s_0}^{1 - 3\gamma}}   + \frac{ \indicator{r_c\geq 1} \indicator{\frac{r_c}{k} \leq |s_0-r_c| \leq \frac{1}{10}}}{ |s_0-r_c|^{1+\gamma}} \left( \indicator{s_0\leq 1} \frac{s_0}{(s_0+r_c)^{1+\gamma}} + \indicator{s_0 \geq 1} s_0^{3\gamma}\right)
\notag\\
&\lesssim \brak{s_0}^{3\gamma} \left( \frac{\indicator{r_c\geq 1} \indicator{|s_0-r_c| \geq \frac{1}{10}}}{s_0^\gamma \brak{s_0}}   
+ k^{2\gamma} \frac{  \indicator{r_c\geq 1} \indicator{ |s_0-r_c| \leq \frac{1}{10}}}{ |s_0-r_c|^{1-\gamma} } \right)
\label{eq:Warm:Merlot:3}
\end{align}
since $r_c\geq 1$.
In the above inequalities, the implicit constants are independent of $\eps$ and $k$, but may depend on $\gamma$.
From the above three estimates we arrive at (recall \eqref{eq:Merlot:1})
\begin{align*}
\indicator{r_c\geq 1} {\mathbb J}(r,s_0,s,r_c) 
&\lesssim \eps^{\frac{\gamma}{3}} \indicator{r_c\geq 1} \indicator{|s_0-r_c|\geq \frac{r_c}{k}}
 \brak{s_0}^{3\gamma} {\mathfrak B}_0(r,s_0,s,r_c)  
\left[\frac{  \indicator{|r_c-r|\leq \frac{1}{10}}}{\brak{r_c}^\gamma |r_c-r|^{1-\frac{\gamma}{3}}}+ \frac{ \indicator{|r_c-r|\geq \frac{1}{10}}}{\brak{r_c}^{1 + \frac{2\gamma}{3}}} \right] 
\notag\\
&\qquad \times \left[ \frac{ \indicator{|s-r_c|\leq \frac{1}{10}} }{\brak{s}^{\gamma} |s-r_c|^{1- \frac{\gamma}{3}}} +    \frac{ s^{\frac{\gamma}{3}}  \indicator{|s-r_c|\geq \frac{1}{10}}}{\brak{s}^{1 + \gamma + \delta}}\right]  \left[ \frac{ \indicator{|s_0-r_c| \geq \frac{1}{10}}}{s_0^\gamma \brak{s_0}}   
+ k^{2\gamma} \frac{   \indicator{ |s_0-r_c| \leq \frac{1}{10}}}{ |s_0-r_c|^{1-\gamma}} 
 \right].
\end{align*}
At this stage we use assumption \eqref{eq:BBB:cond:*}, noting that $\indicator{\rm strange} \indicator{r_c\geq 1} = \indicator{2r \leq r_c}\indicator{r_c\geq 1}   \indicator{r\leq 1}$, (recall \eqref{def:strange}) to obtain
\begin{align}
\frac{\indicator{r_c\geq 1} {\mathbb J}(r,s_0,s,r_c)}{\eps^{\frac{\gamma}{3}}}
&\lesssim   \indicator{r_c\geq 1} \indicator{|s_0-r_c|\geq \frac{r_c}{k}}
\left[\indicator{2r \leq r_c} \indicator{r\leq 1} \frac{s^{\frac 12} }{r^{\frac 12} \brak{s}^{\frac 12}} + (1- \indicator{2r \leq r_c} \indicator{r\leq 1}) \right]\notag\\
&\qquad \times \frac{r^{\frac{\delta}{2}} \brak{s}^{\frac 12 + \delta } }{ \brak{r}^{\frac 12 + \frac{3\delta}{2}}  } 
\left[\frac{  \indicator{|r_c-r|\leq \frac{1}{10}}}{\brak{r_c}^\gamma |r_c-r|^{1-\frac{\gamma}{3}}}+ \frac{ \indicator{|r_c-r|\geq \frac{1}{10}}}{\brak{r_c}^{1 + \frac{2\gamma}{3}}} \right]  \left( \max \left\{ r^2,r^{-2}, s^2, s^{-2}, s_0^2,s_0^{-2} \right\} \right)^\eta
\notag\\
&\qquad \times \left[ \frac{ \indicator{|s-r_c|\leq \frac{1}{10}} }{\brak{s}^{\gamma+\delta} |s-r_c|^{1- \frac{\gamma}{3}}} +    \frac{ s^{\frac{\gamma}{3}}  \indicator{|s-r_c|\geq \frac{1}{10}}}{\brak{s}^{1 + \gamma + \delta}}\right]  \left[ \frac{ \indicator{|s_0-r_c| \geq \frac{1}{10}}}{s_0^\gamma \brak{s_0} }   
+ k^{2\gamma} \frac{ \indicator{ |s_0-r_c| \leq \frac{1}{10}}}{ |s_0-r_c|^{1-\gamma}} 
 \right].
 \label{eq:Merlot:2}
\end{align}
We first note that since $\eta$ is sufficiently small, we have
\begin{align*}
\int_{\RR_+} \indicator{r_c\geq 1} \indicator{|s_0-r_c| \geq \frac{r_c}{k}} (\max\{s_0^2 ,s_0^{-2}\})^{\eta} \left[ \frac{k^{2 \gamma}     \indicator{ |s_0-r_c| \leq \frac{1}{10}}}{  |s_0-r_c|^{1-\gamma}}  +  \frac{ \indicator{|s_0-r_c| \geq \frac{1}{10}}}{s_0^\gamma \brak{s_0}}   
\right] \dd s_0 \lesssim k^{2\gamma},
\end{align*}
so that, after some manipulations, we arrive at
\begin{align}
&\frac{\indicator{r_c\geq 1}}{\eps^{\frac{\gamma}{3}}} \int_{\RR_+}  {\mathbb J}(r,s_0,s,r_c)  \dd s_0 \notag\\
&\lesssim
\indicator{r_c\geq 1}  \left[\indicator{2r \leq r_c} \indicator{r\leq 1} \frac{s^{\frac 12} }{r^{\frac 12} \brak{s}^{\frac 12}} + (1- \indicator{2r \leq r_c} \indicator{r\leq 1}) \right]  \left( \max \left\{ r^2,r^{-2}, s^2, s^{-2}  \right\} \right)^\eta
\notag\\
&\qquad \times
\frac{ r^{\frac{\delta}{2}}  \brak{s}^{\frac 12 + \delta} }{ \brak{r}^{\frac 12 + \frac{3\delta}{2}}  } 
\left[\frac{  \indicator{|r_c-r|\leq \frac{1}{10}}}{\brak{r_c}^\gamma |r_c-r|^{1-\frac{\gamma}{3}}}+ \frac{ \indicator{|r_c-r|\geq \frac{1}{10}}}{\brak{r_c}^{1 + \frac{2\gamma}{3}}} \right] 
 \left[ \frac{ \indicator{|s-r_c|\leq \frac{1}{10}} }{\brak{s}^{\gamma+\delta} |s-r_c|^{1- \frac{\gamma}{3}}} +    \frac{ s^{\frac{\gamma}{3}}  \indicator{|s-r_c|\geq \frac{1}{10}}}{\brak{s}^{1 + \gamma + \delta}}\right]    
\notag\\
&\lesssim 
 \frac{  \indicator{|r_c-r|\leq \frac{1}{10}} }{\brak{r-r_c}^{\gamma-2\eta} |r_c-r|^{1-\frac{\gamma}{3}}} 
 \frac{ \indicator{|s-r_c|\leq \frac{1}{10}} }{\brak{s-r_c}^{\gamma-2\eta} |s-r_c|^{1- \frac{\gamma}{3}}}     
   +
\frac{1}{ \brak{r_c}^{\frac 12 + \delta -2\eta}  } 
 \frac{  \indicator{|r_c-r|\leq \frac{1}{10}}}{\brak{r-r_c}^\gamma |r_c-r|^{1-\frac{\gamma}{3}}}     \frac{1}{\brak{s}^{\frac12 + \frac{2\gamma}{3} -2\eta}}   
\notag\\
&\quad   +
\frac{1}{  \brak{r_c}^{\frac{1}{2} + \frac{2\gamma}{3}-2\eta}} 
  \frac{(\max\{r^2,r^{-2}\})^\eta}{r^{\frac{1-\delta}{2}} \brak{r}^{\frac{3\delta}{2}}  } 
 \frac{ \indicator{|s-r_c|\leq \frac{1}{10}} }{\brak{s-r_c}^{\gamma} |s-r_c|^{1- \frac{\gamma}{3}}}     
    +
\frac{1}{\brak{r_c}^{1 + \frac{2\gamma}{3}}}
  \frac{(\max\{r^2,r^{-2}\})^\eta}{r^{\frac{1-\delta}{2}} \brak{r}^{\frac{3\delta}{2}}  } 
  \frac{(\max\{s^2,s^{-2}\})^\eta}{\brak{s}^{\frac{1}{2} + \frac{2\gamma}{3}}}     
\notag\\
&=: {\mathbb J}_{11}(r_c-r) {\mathbb J}_{12}(r_c-s) + {\mathbb J}_{21}(r_c) {\mathbb J}_{22}(r_c-r) {\mathbb J}_{23}(s)  \notag\\
&\quad + {\mathbb J}_{31}(r_c) {\mathbb J}_{32}(r) {\mathbb J}_{33}(r_c-s)  + {\mathbb J}_{41}(r_c) {\mathbb J}_{42}(r) {\mathbb J}_{43}(s), 
\label{eq:November:1}
\end{align}
where the identification of the ${\mathbb J}_{ij}$ functions, for $1 \leq i \leq 4$, and $1\leq j \leq 3$ is the obvious one.
We then use Young's inequality and H\"older's inequality to deduce
\begin{align*}
\int_{\RR_+^3}{\mathbb J}_{11}(r_c-r) {\mathbb J}_{12}(r_c-s)   |f(s)| \, |\varphi(r)| \, \dd s \dd r \dd r_c
&\leq   \norm{ {\mathbb J}_{11} \ast |\varphi|}_{L^{2}(d r_c)} 
 \norm{ {\mathbb J}_{12} \ast | f|}_{L^{2}(d r_c)} \notag\\
 &\leq \norm{\mathbb J_{11}}_{L^1} \norm{\mathbb J_{12}}_{L^1} \norm{\varphi}_{L^2} \norm{ f}_{L^2} \notag\\
&\lesssim \norm{\varphi}_{L^2} \norm{ f}_{L^2}, 
\end{align*}
and similarly 
\begin{align*}
\int_{\RR_+^3}{\mathbb J}_{21}(r_c) {\mathbb J}_{22}(r_c-r){\mathbb J}_{23}(s)     | f(s)| \, |\varphi(r)| \, \dd s \dd r \dd r_c
&\leq \norm{\mathbb J_{21}}_{L^2} \norm{\mathbb J_{22}}_{L^1} \norm{\mathbb J_{23}}_{L^2} \norm{\varphi}_{L^2} \norm{ f}_{L^2} \notag\\
&\lesssim \norm{\varphi}_{L^2} \norm{  f}_{L^2} \\
\int_{\RR_+^3}{\mathbb J}_{31}(r_c) {\mathbb J}_{32}(r){\mathbb J}_{33}(r_c-s)     | f(s)| \, |\varphi(r)| \, \dd s \dd r \dd r_c
&\leq \norm{\mathbb J_{31}}_{L^2} \norm{\mathbb J_{32}}_{L^2} \norm{\mathbb J_{33}}_{L^1} \norm{\varphi}_{L^2} \norm{ f}_{L^2} \notag\\
&\lesssim \norm{\varphi}_{L^2} \norm{ f}_{L^2} \\
\int_{\RR_+^3} {\mathbb J}_{41}(r_c) {\mathbb J}_{42}(r) {\mathbb J}_{43}(s)     | f(s)| \, |\varphi(r)| \, \dd s \dd r \dd r_c
&\leq \norm{\mathbb J_{41}}_{L^1} \norm{\mathbb J_{42}}_{L^2} \norm{\mathbb J_{43}}_{L^2} \norm{\varphi}_{L^2} \norm{ f}_{L^2} \notag\\
&\lesssim \norm{\varphi}_{L^2} \norm{ f}_{L^2} .
\end{align*}
Summarizing the above estimates, we arrive at
\begin{align*}
\int_{\RR_+^4}  \indicator{r_c\geq 1}{\mathbb J}(r,s_0,s,r_c)  | f(s) \varphi(r)| \, \dd s  \dd s_0 \dd r_c  \dd r
\lesssim \eps^{\frac{\gamma}{3}} \norm{\varphi}_{L^2} \norm{ f}_{L^2} \to 0 \quad \mbox{as} 
\quad \eps \to 0.
\end{align*}

\paragraph*{Case $r_c\leq 1$}
In this case, instead of \eqref{eq:u':info:1}--\eqref{eq:u':info:2}, we also have the improved estimate
\begin{align}
\frac{|u'(r_c)|}{|u(r)-u(r_c)|} \indicator{r_c \leq 1} \indicator{|r_c-r|\geq \frac{1}{10}}   \lesssim r_c, 
\label{eq:u':info:3}
\end{align}
which is useful when $r_c \ll 1$.
Similar to the $r_c \leq 1$ case we obtain the bounds 
\begin{align}
\frac{\indicator{r_c\leq 1} \eps^{\frac{\gamma}{3}} |u(r)-u(r_c)| \, |u'(r_c)|}{(u(r) -u(r_c))^2 + \eps^2}
&\lesssim  \frac{r_c \indicator{r_c\leq 1} \indicator{|r_c-r|\leq \frac{1}{10}}}{ (r_c+r)^{1-\frac{\gamma}{3}} |r_c-r|^{1-\frac{\gamma}{3}}}+ r_c \indicator{r_c\leq 1}\indicator{|r_c-r|\geq \frac{1}{10}} 
\label{eq:Warm:Merlot:4} \\
\frac{\indicator{r_c\leq 1} \eps^{\frac{\gamma}{3}} |u(s)-u(r_c)| \, |u'(s)|}{(u(s) -u(r_c))^2 + \eps^2}
&\lesssim 
\frac{s \indicator{r_c\leq 1} \indicator{|s-r_c|\leq \frac{1}{10}} }{(s+r_c)^{1 - \frac{\gamma}{3}} |s-r_c|^{1- \frac{\gamma}{3}}} +    \frac{ s^{\frac{\gamma}{3}} \indicator{r_c\leq 1} \indicator{|s-r_c|\geq \frac{1}{10}}}{\brak{s}^{1 + \gamma}}
\label{eq:Warm:Merlot:5}
\\
\frac{\indicator{r_c\leq 1} \indicator{|s_0-r_c| \geq \frac{r_c}{k}} \eps^{1- \gamma} |u'(s_0)|}{(u(s_0)-u(r_c))^2 +\eps^2}
&\lesssim   \frac{k^{2\gamma} \brak{s_0}^{3\gamma}  \indicator{r_c\leq 1} \indicator{ |s_0-r_c| \leq \frac{1}{10}}}{r_c^{2 \gamma} |s_0-r_c|^{1-\gamma}} + \frac{\indicator{r_c\leq 1} \indicator{|s_0-r_c| \geq \frac{1}{10}}}{s_0^\gamma \brak{s_0}^{1 - 3\gamma}}. 
\label{eq:Warm:Merlot:6}
\end{align}
From the above three estimates, and by using \eqref{eq:BBB:cond:*} we arrive at
\begin{align}
&\frac{\indicator{r_c\leq 1} {\mathbb J}(r,s_0,s,r_c)}{\eps^{\frac{\gamma}{3}}} \notag\\
&\lesssim  \indicator{r_c\leq 1} \indicator{|s_0-r_c|\geq \frac{r_c}{k}} {\mathfrak B}_0(r,s_0,s,r_c)   \notag\\
&\qquad \times \left[\frac{r_c \indicator{r_c\leq 1} \indicator{|r_c-r|\leq \frac{1}{10}}}{ (r_c+r)^{1-\frac{\gamma}{3}} |r_c-r|^{1-\frac{\gamma}{3}}} + r_c  \indicator{r_c\leq 1}\indicator{|r_c-r|\geq \frac{1}{10}}  \right] \notag\\
&\qquad \times \left[\frac{s \indicator{r_c\leq 1} \indicator{|s-r_c|\leq \frac{1}{10}} }{(s+r_c)^{1 - \frac{\gamma}{3}} |s-r_c|^{1- \frac{\gamma}{3}}} +    \frac{ s^{\frac{\gamma}{3}} \indicator{r_c\leq 1} \indicator{|s-r_c|\geq \frac{1}{10}}}{\brak{s}^{1 + \gamma}} \right] 
\left[   \frac{k^{2\gamma} \brak{s_0}^{3\gamma}  \indicator{r_c\leq 1} \indicator{ |s_0-r_c| \leq \frac{1}{10}}}{  r_c^{2\gamma} |r_c-s_0|^{1-\gamma} } + \frac{\indicator{r_c\leq 1} \indicator{|s_0-r_c| \geq \frac{1}{10}}}{s_0^\gamma \brak{s_0}^{1 - 3\gamma}}  \right]
\notag\\
&\lesssim  
\indicator{r_c\leq 1} \indicator{|s_0-r_c|\geq \frac{r_c}{k}} \left(\indicator{\rm strange} \frac{s^{\frac{1}{2}}}{\brak{s}^{\frac 12} r^{\frac 12}} + (1 - \indicator{\rm strange}) \right) \frac{r^\frac{\delta}{2} \brak{s}^{\frac 12}}{\brak{r}^{\frac 12 + \frac{3\delta}{2}}  } \notag\\
&\qquad \times \left[\frac{r_c \indicator{r_c\leq 1} \indicator{|r_c-r|\leq \frac{1}{10}}}{ (r_c+r)^{1-\frac{\gamma}{3}} |r_c-r|^{1-\frac{\gamma}{3}}} + r_c  \indicator{r_c\leq 1}\indicator{|r_c-r|\geq \frac{1}{10}}  \right]
\left( \max \left\{ r^2,r^{-2}, s^2, s^{-2}, s_0^2,s_0^{-2} \right\} \right)^\eta
\notag\\
&\qquad \times \left[\frac{s \indicator{r_c\leq 1} \indicator{|s-r_c|\leq \frac{1}{10}} }{(s+r_c)^{1 - \frac{\gamma}{3}} |s-r_c|^{1- \frac{\gamma}{3}}} +    \frac{ s^{\frac{\gamma}{3}} \indicator{r_c\leq 1} \indicator{|s-r_c|\geq \frac{1}{10}}}{\brak{s}^{1 + \gamma}} \right] 
\left[   \frac{k^{2\gamma}   \indicator{r_c\leq 1} \indicator{ |s_0-r_c| \leq \frac{1}{10}}}{     r_c^{2\gamma} |r_c-s_0|^{1-\gamma} } + \frac{\indicator{r_c\leq 1} \indicator{|s_0-r_c| \geq \frac{1}{10}}}{s_0^\gamma \brak{s_0} }  \right].
\label{eq:Merlot:3}
\end{align}
Similar to the estimates \eqref{eq:Merlot:2}--\eqref{eq:November:1} for the case $r_c \leq 1$, we first integrate the $s_0$ dependent-part of \eqref{eq:Merlot:3} in $s_0$ to obtain that  
\begin{align*}
&\int_{\RR_+}  \indicator{r_c\leq1}  (\max\{s_0^2,s_0^{-2}\})^\eta \left|   \frac{k^{2\gamma}    \indicator{\frac{r_c}{k} \leq |s_0-r_c| \leq \frac{1}{10}}}{ r_c^{2\gamma} |s_0-r_c|^{1-\gamma}} + \frac{ \indicator{|s_0-r_c| \geq \frac{1}{10}}}{s_0^\gamma \brak{s_0} }  \right| \dd s_0  
\notag\\
&\lesssim \frac{\indicator{r_c\leq1}  k^{2\gamma}}{r_c^{2\gamma}}  \int_{\RR_+} (\max\{s_0^2,s_0^{-2}\})^\eta \left(  \frac{ \indicator{ |s_0-r_c| \leq \frac{1}{10}}}{   |s_0-r_c|^{1-\gamma}} + \frac{ \indicator{|s_0-r_c| \geq \frac{1}{10}}}{s_0^\gamma \brak{s_0} }  \right) \dd s_0 
 \lesssim \frac{k^{2\gamma}}{r_c^{2\gamma}}
\end{align*}
since $\eta$ is sufficiently small.
Then, using that $\gamma \leq \frac{\delta}{4}$, we have the inequality
\begin{align*}
&\indicator{r_c\leq 1} \frac{r^\frac{\delta}{2}}{r_c^{2\gamma}} \left(\indicator{|r_c-r|\leq \frac{1}{10}} \frac{r_c }{(r+r_c)^{1-\frac{\gamma}{3}}|r_c-r|^{1- \frac{\gamma}{3}}} +  r_c \indicator{|r_c-r|\geq \frac{1}{10}}  \right)   \lesssim r^{\frac{\delta}{2}} \indicator{r_c\leq 1}   \left( \frac{ \indicator{|r_c-r|\leq \frac{1}{10}}}{ |r_c-r|^{1- \frac{\gamma}{3}}} +  \indicator{|r_c-r|\geq \frac{1}{10}}  \right),
\end{align*}
and since on the support of $\indicator{r_c\leq 1} \indicator{|s-r_c|\leq \frac{1}{10}}$ we have $s \leq r_c + \frac{1}{10} \leq \frac{11}{10}$,
it remains to consider the integral of $\varphi(r)  f(s)$ multiplied by
\begin{align*}
\int_{\RR_+} \frac{\indicator{r_c\leq 1} {\mathbb J}(r,s_0,s,r_c)}{\eps^{\frac{\gamma}{3}}} \dd s_0 &\lesssim \indicator{r_c\leq 1}   \left(\indicator{\rm strange} \frac{s^{\frac{1}{2}}}{\brak{s}^{\frac 12} r^{\frac 12}} + (1 - \indicator{\rm strange}) \right) \frac{r^{\frac{\delta}{2}}}{\brak{r}^{\frac{\delta}{2}}}  \left( \max \left\{ r^2,r^{-2}  \right\} \right)^\eta \notag\\
&\qquad \qquad \times \left[ \frac{ \indicator{|r_c-r|\leq \frac{1}{10}}}{  |r_c-r|^{1- \frac{\gamma}{3}}} +  \frac{\indicator{|r_c-r|\geq \frac{1}{10}}}{\brak{r}^{\frac 12 + \delta} } \right]  \left[\frac{ \indicator{|s-r_c|\leq \frac{1}{10}} }{ |s-r_c|^{1- \frac{\gamma}{3}}} +    \frac{\indicator{|s-r_c|\geq \frac{1}{10}}}{\brak{s}^{\frac{1}{2} + \frac{2\gamma}{3} - 2\eta}} \right] 
\notag\\
&\lesssim 
\indicator{r_c\leq 1} \frac{1}{r^{\frac 12 - \frac{\gamma}{6}} r_c^{\frac{\gamma}{6}}}   \left[ \frac{ \indicator{|r_c-r|\leq \frac{1}{10}}}{  |r_c-r|^{1- \frac{\gamma}{3}}} +  \frac{1}{\brak{r}^{\frac 12 + \delta} } \right] \left[ \frac{\indicator{|s-r_c|\leq \frac{1}{10}} }{|s-r_c|^{\frac 12- \frac{\gamma}{3}}} +    \frac{1}{\brak{s}^{\frac{1}{2} + \frac{2\gamma}{3}}} \right]   \notag\\
&\quad + \indicator{r_c\leq 1}  \left[ \frac{ \indicator{|r_c-r|\leq \frac{1}{10}}}{  |r_c-r|^{1- \frac{\gamma}{3}}} +  \frac{1}{\brak{r}^{\frac 12 + \delta} } \right]  \left[\frac{ \indicator{|s-r_c|\leq \frac{1}{10}} }{ |s-r_c|^{1- \frac{\gamma}{3}}} +    \frac{1}{\brak{s}^{\frac{1}{2} + \frac{2\gamma}{3}}} \right]  \   .
\end{align*}
Here we have used properties of the support of $\indicator{\rm strange}$ (recall \eqref{def:strange}). 
Similar to the case $r_c \geq 1$, for the second term above (the one coming from $1-\indicator{\rm strange}$) one may use Young's and H\"older's inequality to check that
\begin{align*}
\int_{{\mathbb R}_{+}^3} \indicator{r_c\leq 1}  \left[ \frac{ \indicator{|r_c-r|\leq \frac{1}{10}}}{  |r_c-r|^{1- \frac{\gamma}{3}}} +  \frac{1}{\brak{r}^{\frac 12 + \delta} } \right]  \left[\frac{ \indicator{|s-r_c|\leq \frac{1}{10}} }{ |s-r_c|^{1- \frac{\gamma}{3}}} +    \frac{1}{\brak{s}^{\frac{1}{2} + \frac{2\gamma}{3}}} \right] | \varphi(r)  f(s)| \dd r \dd s \dd r_c \lesssim \norm{\varphi}_{L^2} \norm{ f}_{L^2}.
\end{align*}
For the term first term (due to $\indicator{\rm strange}$), we first note that 
\begin{align*}
\sup_{r_c\leq 1} \int_{\RR_+} | f(s)| \left[ \frac{\indicator{|s-r_c|\leq \frac{1}{10}}}{|s-r_c|^{\frac 12- \frac{\gamma}{3}}} +    \frac{1}{\brak{s}^{\frac{1}{2} + \frac{2\gamma}{3}}} \right] \dd s 
&\lesssim \norm{f}_{L^2} 
\sup_{r_c\leq 1}\norm{ \frac{ \indicator{|s-r_c|\leq \frac{1}{10}} }{|s-r_c|^{\frac 12- \frac{\gamma}{3}}} +    \frac{1}{\brak{s}^{\frac{1}{2} + \frac{2\gamma}{3}}} }_{L^2(ds)}\notag\\
& \lesssim \norm{ f}_{L^2} ,
\end{align*}
so that we only are left to bound
\begin{align*}
&\int_{\RR_+^2} |\varphi(r)|  \frac{\indicator{r\leq \frac{11}{10}}\indicator{r_c\leq 1}}{r^{\frac 12 - \frac{\gamma}{6}} r_c^{\frac{\gamma}{6}}}   \left[ \frac{ \indicator{|r_c-r|\leq \frac{1}{10}}}{  |r_c-r|^{1- \frac{\gamma}{3}}} +  \frac{1}{\brak{r}^{\frac 12 + \delta} } \right]  \dd r \dd r_c \notag\\
&\lesssim 
\norm{\varphi}_{L^2} \norm{\frac{\indicator{r\leq \frac{11}{10}}}{r^{\frac 12 - \frac{\gamma}{6}} }}_{L^2(\dd r)} \sup_{r\leq \frac{11}{10}}
\norm{\frac{\indicator{r_c\leq 1} \indicator{|r-r_c|\leq \frac{1}{10}}}{r_c^{\frac{\gamma}{6}} |r-r_c|^{1-\frac{\gamma}{3}}}}_{L^1(dr_c)}
+
\norm{\varphi}_{L^2} \norm{\frac{1}{r^{\frac 12 - \frac{\gamma}{6}} \brak{r}^{\frac 12 + \delta}}}_{L^2(dr)} 
\norm{\frac{\indicator{r_c\leq 1}}{r_c^{\frac{\gamma}{6}}}}_{L^1(dr_c)}
\notag\\
&\lesssim \norm{\varphi}_{L^2}.
\end{align*}
Combining the above, we arrive at 
\begin{align*}
\int_{\RR_+^4}  \indicator{r_c\leq 1} |{\mathbb J}(r,s_0,s,r_c) (  f(s)) \varphi(r)| \, \dd s  \dd s_0 \dd r_c  \dd r
\lesssim \eps^{\frac{\gamma}{3}} \norm{\varphi}_{L^2} \norm{ f}_{L^2} \to 0 \quad \mbox{as} 
\quad \eps \to 0
\end{align*}
which is the desired estimate, and concludes the proof of the lemma.
\end{proof}

\begin{corollary}
\label{cor:abstract:off:diagonal:1}
Let $0 \leq j \leq k-1$, and $0\leq \ell, \ell_1,\ell_2$ be such that $\ell + \ell_1 + \ell_2 \leq j$. 
Assume that the functions ${\mathfrak B}_{\eps,1}$ and ${\mathfrak B}_{\eps,2}$ in \eqref{eq:L:eps:1:def} are given by \eqref{eq:B1:2SIO:1DELTA}--\eqref{eq:B2:2SIO:1DELTA}, 
where $B_{\ell,\eps}^{(1)}$ is a suitable $(2\ell_1,\ell_1+\eta/2)$ kernel of type $I$ or $II$, and $B_{\ell,\eps}^{(2)}$ is a suitable $(2\ell_2,\ell_2+\eta/2)$ kernel of type $I$ or $II$.
Then the operator $L_{\eps,1}$ defined in \eqref{eq:L:eps:1:def} vanishes in $L^2(dr)$ as $\eps \to 0$.
\end{corollary}
\begin{proof}[Proof of Corollary~\ref{cor:abstract:off:diagonal:1}]
We recall that  the following estimates are available
\begin{align*}
| B_{\ell,\eps}^{(1)}(r,s_0,r_c) | &\lesssim |u'(s_0)| {\mathcal B}(r,s_0) {\mathbb K}(r,s_0,r_c) {\mathcal L}_{2\ell_1,\ell_1+\eta/2}(r,s_0)\\
| B_{\ell,\eps}^{(2)}(s_0,s,r_c) | &\lesssim |u'(s)| {\mathcal B}(s_0,s) {\mathbb K}(s_0,s,r_c) {\mathcal L}_{2\ell_2,\ell_2+\eta/2}(s_0,s)
\end{align*}
on $\RR_+^3$, where  as before we recall the definitions 
\begin{align}
{\mathcal B}(r , s) &= \left( \indicator{s<r} \frac{s^{k-\frac 12}}{r^{k-\frac 12}} + \indicator{s>r} \frac{r^{k+\frac 12}}{s^{k+\frac 12}}  \right) \brak{s}^4
\label{eq:script:B:def} 
\\
{\mathbb K}(r ,s ,r_c) &= \indicator{r_c>1} + \indicator{r_c\leq 1} \Big( \indicator{s<r<r_c} + \indicator{s<r_c<r<1} \frac{r_c^2}{r^2} + \indicator{s<r_c<1<r} r_c^2 +  \indicator{r_c<s<r<1} \frac{s^2}{r^2} +  \indicator{r_c<s<1<r} s^2 + \indicator{1<s<r}  \notag\\
&\qquad \qquad \qquad    +  \indicator{r<s<r_c} +  \indicator{r<r_c<s<1}\frac{r_c^2}{s^2} +  \indicator{r<r_c<1<s} r_c^2  +  \indicator{r_c<r<s<1} \frac{r^2}{s^2} +  \indicator{r_c<r<1<s} r^2 + \indicator{1<r<s}   \Big) \notag \\
{\mathcal L}_{J,\ell}(r,s) &= k^{J} \left(\max\left\{\frac{1}{r^2},r^2, \frac{1}{s^2}, s^2\right\} \right)^{\ell}. \notag
\end{align}
Recalling the definition of the weights $w_{F,\delta}$ we obtain that 
\begin{align}
&\indicator{|s_0-r_c|\geq \frac{r_c}{k}} | B_{\ell,\eps}^{(1)}(r,s_0,r_c) \, B_{\ell,\eps}^{(2)}(s_0,s,r_c) |  \notag\\
&\lesssim \indicator{|s_0-r_c|\geq \frac{r_c}{k}} ( \indicator{r \leq 1} \indicator{r_c \leq 2r} + \indicator{r\geq 1} )  \left[\frac{\indicator{r\leq 1}}{r^{k+3 - 2j -\delta}} + \indicator{r\geq 1} r^{k-1- 2j - \delta} \right] {\mathcal L}_{2\ell_1,\ell_1+\eta/2}(r,s_0) {\mathcal L}_{2\ell_2,\ell_2+\eta/2}(s_0,s)\notag\\
&\quad \times \frac{1}{\brak{s_0}^{6}}
{\mathcal B}(r,s_0) {\mathbb K}(r,s_0,r_c)\left[ \indicator{s\leq1} s^{k+3- 2\ell - \frac{\delta}{4}}+ \frac{\indicator{s\geq 1}}{s^{k+5  - 2\ell - \frac{\delta}{4}}} \right] {\mathcal B}(s_0,s) {\mathbb K}(s_0,s,r_c) \notag\\
&=: {\mathfrak B}_0(r,s_0,s,r_c).
\end{align}
The above defined function ${\mathfrak B}_0$ is explicit, and we need to verify that it obeys condition \eqref{eq:BBB:cond:*}. Note that the terms due to the $\eta$ corrections in ${\mathcal L}$ are already incorporated in the $(\max\{r^2,r^{-2},s^2,s^{-2},s_0^2,s_0^{-2}\})^\eta$ term on the right side of \eqref{eq:BBB:cond:*}, so that we ignore these factors from here on, working as if $\eta=0$. This is done by considering the possible orderings of $r,s_0,s$, and $r_c$.
It is useful to denote by 
\begin{align*}
{\mathbb W}(r,s,s_0)&:= 
\left[\frac{\indicator{r\leq 1}  }{r^{k+\frac 12}} +  \indicator{r\geq 1} r^{k- \frac 12} \right] 
\left[\indicator{s\leq1} s^{k + \frac 12 }+ \frac{\indicator{s\geq 1}}{s^{k - \frac 12}} \right] 
\frac{{\mathcal B}(r,s_0)}{\brak{s_0}^4} \frac{{\mathcal B}(s_0,s)}{\brak{s}^4} \\
{\mathbb L}(r,s,s_0)&:= \left[\indicator{r\leq 1}  r^{2j} + \frac{\indicator{r\geq 1}}{r^{2j}} \right] \left[ \frac{\indicator{s\leq1}} {s^{ 2\ell}}+  \indicator{s\geq 1} s^{2\ell} \right] {\mathcal L}_{2\ell_1,\ell_1}(r,s_0) {\mathcal L}_{2\ell_2,\ell_2}(s_0,s).
\end{align*}
In view of Lemma~\ref{lem:weights:product:2}, we have that 
\begin{align}
{\mathbb W}(r,s,s_0) {\mathbb L}(r,s,s_0) \lesssim 1.
\label{eq:Merlot:*}
\end{align}
Estimate \eqref{eq:Merlot:*} requires some care in proving and we defer the proof to the Subsection~\ref{sec:product:for:weights}.
With this notation, and using estimate \eqref{eq:Merlot:*}, we have that 
\begin{align}
{\mathfrak B}_0(r,s_0,s,r_c) &\lesssim \indicator{|s_0-r_c|\geq \frac{r_c}{k}} ( \indicator{r \leq 1} \indicator{r_c \leq 2r} + \indicator{r\geq 1} )  \left[\frac{\indicator{r\leq 1}}{r^{3 - \frac 12  -\delta}} + \frac{\indicator{r\geq 1}}{r^{\frac 12 + \delta}} \right] \left[ \indicator{s\leq1} s^{3 - \frac 12- \frac{\delta}{4}}+ \frac{\indicator{s\geq 1}}{s^{\frac 32  - \frac{\delta}{4}}} \right]\notag\\
&\quad \times \frac{1}{\brak{s_0}^{2}}
 {\mathbb K}(r,s_0,r_c)  {\mathbb K}(s_0,s,r_c). 
\end{align}
Checking condition \eqref{eq:BBB:cond:*} for the above defined ${\mathfrak B}_0$ thus reduces to verifying the uniform boundedness of
\begin{align}
{\mathbb J}(r,s_0,s,r_c)  &:= \indicator{|s_0-r_c|\geq \frac{r_c}{k}} ( \indicator{r \leq 1} \indicator{r_c \leq 2r} + \indicator{r\geq 1} ) \left(\indicator{\rm strange} \frac{r^{\frac 12} \brak{s}^{\frac 12} }{s^{\frac{1}{2}}} + (1 - \indicator{\rm strange}) \right) \notag\\
&\quad \times \left[\frac{\indicator{r\leq 1}}{r^{\frac 52  - \frac{ \delta}{2}}} +  \indicator{r\geq 1} \right] \left[ \indicator{s\leq1} s^{\frac 52- \frac{\delta}{4}}+ \frac{\indicator{s\geq 1}}{s^{2 - \frac{3\delta}{4}}} \right] 
 {\mathbb K}(r,s_0,r_c)  {\mathbb K}(s_0,s,r_c),
 \label{eq:Merlot:**}
\end{align}
where we have used that $3\gamma \leq 2$, and thus $\brak{s_0}^{-2+3\gamma} \lesssim 1$.

\paragraph*{Case $r_c>1$} Note that here we are working on the support of $\indicator{\rm strange} = 0$, and by definition $ {\mathbb K}(r,s_0,r_c) \lesssim 1$, ${\mathbb K}(s_0,s,r_c)  \lesssim 1$. Thus, condition \eqref{eq:Merlot:**} reduces to proving the uniform boundedness  of  
\begin{align*}
{\mathbb J}_{r_c>1} = \left[\frac{\indicator{\frac 12 < r\leq 1}  }{r^{\frac 52 -\frac{ \delta}{2}}} +   \indicator{r\geq 1}  \right]  \left[\indicator{s\leq1} s^{\frac 52- \frac{\delta}{4}}+ \frac{\indicator{s\geq 1}}{s^{2 - \frac{3\delta}{4}}} \right].
\end{align*}
Since both of the above terms are $\lesssim 1$, so is their product, and thus 
\begin{align*}
{\mathbb J}_{r_c>1} \lesssim 1
\end{align*}
as desired.

\paragraph*{Case $r_c<1$} From condition \eqref{eq:BBB:cond:*} we need to show that 
\begin{align}
{\mathbb J}_{r_c\leq 1}
&:= \indicator{r_c\leq 1} \left(\indicator{\rm strange} \frac{r^{\frac 12} \brak{s}^{\frac 12} }{s^{\frac{1}{2}}} + (1 - \indicator{\rm strange}) \right)  ( \indicator{r \leq 1} \indicator{r_c \leq 2r} + \indicator{r\geq 1} )  \notag\\
&\quad \times \left[\frac{\indicator{\frac{r_c}{2} \leq r\leq 1}}{r^{\frac 52  - \frac{ \delta}{2}}} +  \indicator{r\geq 1}  \right] \left[ \indicator{s\leq1} s^{\frac 52- \frac{\delta}{4}}+ \frac{\indicator{s\geq 1}}{s^{2 - \frac{3 \delta}{4}}} \right] 
 {\mathbb K}(r,s_0,r_c)  {\mathbb K}(s_0,s,r_c)
 \label{eq:Merlot:***}
\end{align}
is uniformly bounded in $r,r_c,s,s_0$. 
As in the case $r_c >1$, since ${\mathbb K}(s_0,s,r_c) \leq 1$, ${\mathbb K}(r,s_0,r_c)\leq 1$, it is clear that proving the uniform boundedness of the term ${\mathbb J}_{r_c\leq 1}$ defined in \eqref{eq:Merlot:***}, resumes to checking the uniform boundedness of ${\mathbb J}_{r_c\leq 1} \indicator{r \leq 1} \indicator{r\leq s}$. Indeed, when $s\leq r$ the quotient $(s/r)^{5/2-\delta/4} \lesssim 1$, and no singularity at $r \ll 1$ arises. This issue is avoided altogether if $r \geq 1$. Thus, we see that our desired estimate ${\mathbb J}_{r_c\leq 1}$ reduces to proving the uniform boundedness of 
\begin{align}
&{\mathbb J}_{r_c\leq 1} \indicator{\frac{r_c}{2} \leq r \leq 1} \indicator{r\leq s}
=  {\mathbb J}_{r_c\leq 1}^{(1)} + {\mathbb J}_{r_c\leq 1}^{(2)} \notag\\
&:= \indicator{r_c\leq 1} \left(\indicator{\rm strange} \frac{r^{\frac 12} \brak{s}^{\frac 12} }{s^{\frac{1}{2}}} + (1 - \indicator{\rm strange}) \right)   \left[\frac{\indicator{\frac{r_c}{2} \leq r\leq 1}}{r^{\frac 52  - \frac{ \delta}{2}}} +  \indicator{r\geq 1}  \right]  \indicator{s\leq1} s^{\frac 52- \frac{\delta}{4}}  {\mathbb K}(r,s_0,r_c)  {\mathbb K}(s_0,s,r_c) \notag\\
&\quad + \indicator{r_c\leq 1} \left(\indicator{\rm strange} \frac{r^{\frac 12} \brak{s}^{\frac 12} }{s^{\frac{1}{2}}} + (1 - \indicator{\rm strange}) \right)   \left[\frac{\indicator{\frac{r_c}{2} \leq r\leq 1}}{r^{\frac 52  - \frac{ \delta}{2}}} + \frac{\indicator{r\geq 1}}{r^{1+ \delta}} \right]   \frac{\indicator{s\geq 1}}{s^{2 - \frac{3\delta}{4}}}  {\mathbb K}(r,s_0,r_c)  {\mathbb K}(s_0,s,r_c)
\end{align}
where the decomposition is based on $s\leq 1$ or $s\geq 1$. 

When $s\leq 1$, and either $s\leq 2r_c$ or $s\leq 2 r$, the quotient $(s/r)^{\frac{5}{2} - \frac{\delta}{4}}$ is bounded by a universal constant, so we are left to consider the case $s\geq 2r$ and $s\geq 2r_c$, which is precisely the support of $\indicator{\rm strange}$. Therefore, the boundedness of ${\mathbb J}_{r_c\leq 1}^{(1)}$ reduces to the boundedness of 
\begin{align*}
{\mathbb J}_{r_c\leq 1}^{(1)} \indicator{\rm strange}
&=  
 \indicator{\rm strange} \indicator{s\leq1}  \frac{s^{ 2 - \frac{\delta}{4}}  }{r^{2 -\frac{ \delta}{2}}}   
 {\mathbb K}(s_0,s,r_c) {\mathbb K}(r,s_0,r_c). 
\end{align*}
At this stage the specific form of ${\mathbb K}$ is useful to us.
By analyzing the product  ${\mathbb K}(s_0,s,r_c) {\mathbb K}(r,s_0,r_c)$ we note that
\begin{align*}
&\indicator{\frac{r_c}{2} \leq r \leq s \leq 1}    \indicator{r_c\leq 1} {\mathbb K}(s_0,s,r_c) {\mathbb K}(r,s_0,r_c)  \notag\\
&\quad = 
\left( \indicator{s_0 \leq r_c \leq r \leq s \leq 1}  \frac{r_c^4}{s^2 r^2} + \indicator{r_c \leq s_0 \leq r \leq s \leq 1} \frac{s_0^4}{s^2 r^2} + \indicator{r_c \leq r \leq s_0 \leq s \leq 1} \frac{r^2}{s^2} + \indicator{r_c \leq r \leq s \leq s_0 \leq 1} \frac{r^2 s^2}{s_0^4} + \indicator{r_c \leq r \leq s \leq 1 \leq s_0}  s^2 r^2 \right)
\notag\\
&  + \indicator{\frac{r_c}{2} \leq r} \left( \indicator{s_0  \leq r \leq r_c \leq s \leq 1}  \frac{r_c^2}{s^2}   + \indicator{ r \leq s_0 \leq r_c \leq s \leq 1}  \frac{r_c^2}{s^2}  + \indicator{ r \leq r_c \leq s_0 \leq s \leq 1}  \frac{r_c^2}{s^2} + \indicator{  r \leq r_c \leq s \leq s_0 \leq 1} \frac{s^2 r_c^2}{s_0^4} + \indicator{ r \leq r_c \leq s \leq 1 \leq s_0} s^2 r_c^2 \right)
\notag\\
&  +  \indicator{\frac{r_c}{2} \leq r} \left(  \indicator{s_0 \leq r \leq s  \leq r_c \leq 1}  +  \indicator{ r \leq s_0 \leq s  \leq r_c \leq 1}  +  \indicator{ r \leq s  \leq s_0 \leq r_c \leq 1}  +  \indicator{ r \leq s  \leq r_c \leq s_0 \leq 1} \frac{r_c^4}{s_0^4} +  \indicator{ r \leq s  \leq r_c \leq 1 \leq s_0} r_c^4 \right)
\notag\\
&\lesssim 
\left( \indicator{s_0 \leq r_c \leq r \leq s \leq 1}  \frac{r^2}{s^2} + \indicator{r_c \leq s_0 \leq r \leq s \leq 1} \frac{r^2}{s^2} + \indicator{r_c \leq r \leq s_0 \leq s \leq 1} \frac{r^2}{s^2} + \indicator{r_c \leq r \leq s \leq s_0 \leq 1} \frac{r^2}{s^2} +  \indicator{r_c \leq r \leq s \leq 1 \leq s_0}  s^2 r^2 \right)
\notag\\
&  + \indicator{\frac{r_c}{2} \leq r} \left( \indicator{s_0  \leq r \leq r_c \leq s \leq 1}  \frac{r^2}{s^2}   + \indicator{ r \leq s_0 \leq r_c \leq s \leq 1}  \frac{r^2}{s^2}  + \indicator{ r \leq r_c \leq s_0 \leq s \leq 1}  \frac{r^2}{s^2} + \indicator{  r \leq r_c \leq s \leq s_0 \leq 1} \frac{r^2}{s^2 } + \indicator{ r \leq r_c \leq s \leq 1 \leq s_0} s^2 r^2 \right)
\notag\\
&  +  \indicator{\frac{r_c}{2} \leq r} \left(  \indicator{s_0 \leq r \leq s  \leq r_c \leq 1}  \frac{r^2}{s^2}  +  \indicator{ r \leq s_0 \leq s  \leq r_c \leq 1} \frac{r^2}{s^2}  +  \indicator{ r \leq s  \leq s_0 \leq r_c \leq 1}  \frac{r^2}{s^2}  +  \indicator{ r \leq s  \leq r_c \leq s_0 \leq 1} \frac{r^2}{s^2} +  \indicator{ r \leq s  \leq r_c \leq 1 \leq s_0} r^4 \right)
\notag\\
&\lesssim \indicator{\frac{r_c}{2} \leq r \leq s \leq 1}    \indicator{r_c\leq 1}  \frac{r^2}{s^2} .
\end{align*}
Therefore, since $r\leq s \leq 1 $,  we are left with
\begin{align*}
{\mathbb J}_{r_c\leq 1}^{(1)} \indicator{\rm strange}
&\lesssim 1
\end{align*}
which is the needed estimate.

The case $s\geq 1$ is treated similarly. When $s\leq 2r$ or $s \leq 2 r_c$, then since $\frac{r_c}{2} \leq r \leq s$, $r$ is bounded from below, and thus there is no loss of $r^{-\frac 52 + \frac{ \delta}{2}}$. We are left to consider the support of $\indicator{\rm strange}$ and bound the term
\begin{align*}
{\mathbb J}_{r_c\leq 1}^{(2)} \indicator{\rm strange}
&=   \indicator{\rm strange} \indicator{s\geq1}  \frac{1}{r^{2 -\frac{ \delta}{2}} s^{2 - \frac{3 \delta}{4}} }   
 {\mathbb K}(s_0,s,r_c) {\mathbb K}(r,s_0,r_c)  
\end{align*}
 Here we use that 
\begin{align}
&\indicator{\frac{r_c}{2} \leq r\leq 1 < s}    \indicator{r_c\leq 1}{\mathbb K}(s_0,s,r_c) {\mathbb K}(r,s_0,r_c)    \notag\\
&\quad = 
\left( \indicator{s_0 \leq r_c \leq r \leq 1 < s} \frac{r_c^4}{r^2} + \indicator{r_c \leq s_0 \leq  r \leq 1 < s} \frac{s_0^4}{r^2} + \indicator{r_c \leq r \leq s_0 \leq 1 < s} r^2 + \indicator{r_c \leq r \leq 1 < s_0 \leq  s}  r^2  + \indicator{r_c \leq r \leq 1 < s \leq s_0}  r^2 \right) \notag\\
&\quad \qquad \qquad 
+   \indicator{\frac{r_c}{2} \leq r \leq r_c \leq 1 < s} r_c^2 \notag\\
&\quad \lesssim
\indicator{\frac{r_c}{2} \leq r\leq 1 < s}    \indicator{r_c\leq 1} r^2.
\end{align}
Therefore, it follows that 
\begin{align*}
{\mathbb J}_{r_c\leq 1}^{(2)} \indicator{\rm strange}
&=   \indicator{\rm strange} \indicator{s\geq1}  \frac{r^{\frac{ \delta}{2}}}{s^{2 - \frac{3 \delta}{4}} }
\lesssim 1
\end{align*}
which concludes the proof.
\end{proof}

A similar result may be obtained for the weight which has $\chi_2(r,r_c)$ instead of $\chi_1(r,r_c)$, but an additional argument has to be given to control the region in which $r_c$ is much larger than all the other parameters. We have:
\begin{corollary}
\label{cor:abstract:off:diagonal:1:a}
Let $0 \leq j \leq k-1$, and $0\leq \ell, \ell_1,\ell_2$ be such that $\ell + \ell_1 + \ell_2 \leq j$. 
Assume that the functions ${\mathfrak B}_{\eps,1}$ and ${\mathfrak B}_{\eps,2}$ in \eqref{eq:L:eps:1:def} are given by \eqref{eq:B1:2SIO:1DELTA:a}--\eqref{eq:B2:2SIO:1DELTA:a}, 
where $B_{\ell,\eps}^{(1)}$ is a suitable $(2\ell_1,\ell_1+\eta/2)$ kernel of type $I$ or $II$, and $B_{\ell,\eps}^{(2)}$ is a suitable $(2\ell_2,\ell_2+\eta/2)$ kernel of type $I$ or $II$.
Then the operator $L_{\eps,1}$ defined in \eqref{eq:L:eps:1:def} vanishes in $L^2(dr)$ as $\eps \to 0$.
\end{corollary}
\begin{proof}[Proof of Corollary~\ref{cor:abstract:off:diagonal:1:a}]
The proof is nearly identical to the proof of Corollary~\ref{cor:abstract:off:diagonal:1}, so we only emphasize here the points which are different. As noted below Theorem~\ref{thm:2SIO:1DELTA:a}, the main difference is that a factor of  $\frac{r^2 \brak{r_c}^4}{r_c^2}$ enters the estimates. Using the definition of $\indicator{\rm strange}$, we see that instead of checking the uniform boundedness of the expression in \eqref{eq:Merlot:**}, we are left to check the uniform boundedness of the new expression
 \begin{align*}
{\mathbb J}(r,s_0,s,r_c)  
&:= \indicator{|s_0-r_c|\geq \frac{r_c}{k}} \frac{r^2 \brak{r_c}^4}{r_c^2}   \beta(s_0) \brak{s_0}^{4+3\gamma} \indicator{2r\leq r_c}   \frac{r^{\frac 12} \brak{s}^{\frac 12} }{s^{\frac{1}{2}}}    \frac{\indicator{r\leq 1}}{r^{\frac 52  - \frac{ \delta}{2}}}   \left[ \indicator{s\leq1} s^{\frac 52- \frac{\delta}{4}}+ \frac{\indicator{s\geq 1}}{s^{2- \frac{3\delta}{4}}} \right] 
 {\mathbb K}(r,s_0,r_c)  {\mathbb K}(s_0,s,r_c) 
 \notag\\
 &\leq \indicator{|s_0-r_c|\geq \frac{r_c}{k}} \frac{ \brak{r_c}^4}{r_c^{2}}
 \beta(s_0) \brak{s_0}^{4+3\gamma}    \indicator{2r\leq r_c}    \indicator{r\leq 1} r^{\frac{ \delta}{2}}    \left[ \indicator{s\leq1} s^{2- \frac{\delta}{4}}+ \frac{\indicator{s\geq 1}}{s^{2 - \frac{3\delta}{4}}} \right] 
 {\mathbb K}(r,s_0,r_c)  {\mathbb K}(s_0,s,r_c).  
\end{align*}
 Since ${\mathbb K}(r,s_0,r_c)  {\mathbb K}(s_0,s,r_c)\leq 1$, the boundedness in the regions $s\leq r_c\leq 1$, and $1\leq r_c \leq 2 s$ follows immediately.  Moreover, for the  region $\indicator{r_c\leq 1} \indicator{s\geq r_c}$ we explicitly check that 
\begin{align*}
 &\indicator{r_c\leq 1} \indicator{s\geq r_c} \indicator{2r \leq r_c} \indicator{r\leq 1} \frac{s^2}{r_c^2 \brak{s}^4}{\mathbb K}(r,s_0,r_c)  {\mathbb K}(s_0,s,r_c) 
 \notag\\
 &\quad \lesssim \indicator{r_c\leq 1} \indicator{s\geq r_c} \indicator{2r \leq r_c} \indicator{r\leq 1} \left[ \indicator{s\leq 1} \left( 1+ s^4 + \indicator{s\leq s_0} \frac{s^4}{s_0^4} \right) +  \indicator{s\geq 1} \frac{1}{s^2} \right]  \qquad \lesssim 1.
\end{align*}
We are left to consider the region where $\indicator{r_c\geq 1} \indicator{2 s\leq r_c}$. If $r_c\leq 2s_0$, we can absorb the bad power of $r_c^2$ into and $s_0^2$, and use that $\beta(s_0) \brak{s_0}^{6+3\gamma} \leq 1$ to obtain the desired boundedness. However, in the case $r_c \gg  2r ,2s ,2s_0,2$, there is nothing to make the above term, and a different argument is needed. 

We recall that in the definition of $\chi_2(r,r_c)$ we have the cut-off function $\chi_I(r_c)$, which restricts our attention to $r_c \leq \eps^{-\frac{1}{2+\alpha}}$ for some $\alpha>0$. Here this information is used essentially. 
We start from the beginning of the proof of Lemma~\ref{lem:abstract:off:diagonal:1}, namely from \eqref{eq:Merlot:1}, and focus only on the remaining region $\indicator{r\leq 1}\indicator{2r\leq r_c}  \indicator{r_c \geq 1} \indicator{2s\leq r_c} \indicator{2s_0\leq r_c}$. We are instead left to consider the convergence as $\eps\to 0$ of 
\begin{align}
&\int_{\RR_+^4} \frac{|(u(r)-u(r_c)) u'(r_c)|}{(u(r)-u(r_c))^2+\eps^2}  \frac{|(u(s)-u(r_c)) u'(s)|}{(u(s)-u(r_c))^2 + \eps^2}  \frac{\eps |u'(s_0)|}{(u(s_0)-u(r_c))^2 +\eps^2}  \indicator{r\leq 1\leq r_c}\indicator{2r,2s,2s_0\leq r_c}  \notag\\
&\qquad \times \frac{\chi_I(r_c) r^{2+2j} r_c^2 \beta(s_0) w_{F,\frac{\delta}{4}+2\ell}(s){\mathcal B}(r,s_0) {\mathbb K}(r,s_0,r_c){\mathcal L}_{2\ell_1,\ell_1+\eta/2}(r,s_0) {\mathcal B}(s_0,s) {\mathbb K}(s_0,s,r_c){\mathcal L}_{2\ell_1,\ell_1+\eta/2}(s_0,s) }{w_{F,\delta}(r) } \notag\\
&\qquad \times|  f(s) \varphi(r) |\, \dd s  \dd s_0 \dd r_c  \notag\\
&=: \int_{\RR_+^4}  {\mathbb J}(r,s_0,s,r_c) |  f(s)  \varphi(r) |\, \dd s  \dd s_0 \dd r_c  \dd r. 
\label{eq:Super:Mario:1}
\end{align}
By appealing to \eqref{eq:Merlot:*} and the boundedness of ${\mathbb K}$, similarly to \eqref{eq:Merlot:**} we obtain that 
\begin{align*}
{\mathbb J}(r,s_0,s,r_c)
&\lesssim 
\frac{|u'(r_c)|}{\sqrt{(u(r)-u(r_c))^2+\eps^2}}  \frac{|u'(s)|}{\sqrt{(u(s)-u(r_c))^2 + \eps^2}}  \frac{\eps |u'(s_0)|}{(u(s_0)-u(r_c))^2 +\eps^2}   \notag\\
&\qquad \times   \indicator{r\leq 1\leq r_c}\indicator{2r,2s,2s_0\leq r_c}   \frac{\chi_I(r_c) r_c^2 \brak{s_0}^{4+\eta} s_0^{-2\eta} \beta(s_0)( \indicator{s\leq 1} s^{\frac 52 - \frac{\delta}{4} -2\eta} + \indicator{s\geq 1} s^{- \frac 32 + \frac{\delta}{4}+2\eta})}{r^{\frac 12 -  \delta + 2\eta}  },
\end{align*}
and using the definition of $\chi_I$ we obtain 
\begin{align*}
{\mathbb J}(r,s_0,s,r_c)
&\lesssim 
\frac{|u'(r_c)|}{\sqrt{(u(r)-u(r_c))^2+\eps^2}}  \frac{|u'(s)|}{\sqrt{(u(s)-u(r_c))^2 + \eps^2}}  \frac{\eps^{\frac{\alpha}{2+\alpha}} |u'(s_0)|}{(u(s_0)-u(r_c))^2 +\eps^2}   \notag\\
&\qquad \times   \indicator{r\leq 1\leq r_c}\indicator{2r,2s,2s_0\leq r_c}   \frac{\chi_I(r_c) \brak{s_0}^{4+\eta} s_0^{-2\eta}\beta(s_0) ( \indicator{s\leq 1} s^{\frac 52 - \frac{\delta}{4} -2\eta} + \indicator{s\geq 1} s^{- \frac 32 + \frac{\delta}{4}+2\eta})}{r^{\frac 12 - \delta+2\eta} }\notag\\
&\lesssim \eps^{\frac{\alpha}{4(2+\alpha)}} \frac{|u'(r_c)|}{|u(r)-u(r_c)|^{1 - \frac{\alpha}{4(2+\alpha)}}}  \frac{|u'(s)|}{|u(s)-u(r_c)|^{1- \frac{\alpha}{4(2+\alpha)}}}  \frac{ |u'(s_0)|}{|u(s_0)-u(r_c)|^{2 - \frac{\alpha}{4(2+\alpha)}}}   \notag\\
&\qquad \times   \indicator{r\leq 1\leq r_c}\indicator{2r,2s,2s_0\leq r_c}   \frac{ \brak{s_0}^{4+\eta} s_0^{-2\eta} \beta(s_0)( \indicator{s\leq 1} s^{\frac 52 - \frac{\delta}{4} -2\eta} + \indicator{s\geq 1} s^{- \frac 32 + \frac{\delta}{4}+2\eta})}{r^{\frac 12 - \delta+2\eta}  }
\notag\\
&\lesssim \eps^{\frac{\alpha}{4(2+\alpha)}} \frac{1}{\brak{r_c}^{1 + \frac{\alpha}{2(2+\alpha)}}}  \frac{s}{\brak{s}^{2 + \frac{\alpha}{4(2+\alpha)}}}  \frac{s_0}{\brak{s_0}^{\frac{\alpha}{2(2+\alpha)}}}   \notag\\
&\qquad \times   \indicator{r\leq 1\leq r_c}\indicator{2r,2s,2s_0\leq r_c}   \frac{ \brak{s_0}^{4+\eta} s_0^{-2\eta} \beta(s_0)( \indicator{s\leq 1} s^{\frac 52 - \frac{\delta}{4} -2\eta} + \indicator{s\geq 1} s^{- \frac 32 + \frac{\delta}{4}+2\eta})}{r^{\frac 12 - \delta +2\eta}  },
\end{align*}
by using properties of the cut-off  $\indicator{r\leq 1\leq r_c}\indicator{2r,2s,2s_0\leq r_c} $, and estimates \eqref{eq:u':info:1} and \eqref{eq:u':info:3} in the region $|\rho - t| \geq\frac{1}{10}$ relevant here.
As above, we first take care of the integral with respect to $s_0$
\begin{align*}
\int_0^{\infty}  \frac{s_0^{1-2\eta}\brak{s_0}^{4+2\eta} \beta(s_0)}{\brak{s_0}^{\frac{\alpha}{2(2+\alpha)}}} \lesssim 1
\end{align*}
in view of the decay rate of $\beta$. To conclude the proof and obtain the desired vanishing as $\eps \to 0$, we are left to show the boundedness of
\begin{align*}
&\int_{\RR_+^3} \frac{1}{\brak{r_c}^{1 + \frac{\alpha}{2(2+\alpha)}}}  \indicator{r\leq 1\leq r_c}\indicator{2r,2s\leq r_c}   \frac{ ( \indicator{s\leq 1} s^{\frac 72 - \frac{\delta}{4} - 2\eta} + \indicator{s\geq 1} s^{- \frac 52 + \frac{ \delta}{4}- \frac{\alpha}{4(2+\alpha)}+2\eta} )  }{r^{\frac 12 - \delta+2\eta} }  | f(s) \varphi(r) |\, \dd s   \dd r_c  \dd r
\notag\\
&\lesssim \norm{ f}_{L^2(ds)} \norm{\varphi}_{L^2(\dd r)} \norm{\indicator{s\leq 1} s^{\frac 72 - \frac{\delta}{4}-2\eta} + \indicator{s\geq 1} s^{- \frac 52 + \frac{ \delta}{4}- \frac{\alpha}{4(2+\alpha)}+2\eta} }_{L^2(ds)} \norm{\frac{\indicator{r\leq 1}}{r^{\frac 12 -  \delta+2\eta}}}_{L^2(dr)} \notag\\
&\lesssim \norm{ f}_{L^2(ds)} \norm{\varphi}_{L^2(dr)}.
\end{align*}
This concludes the proof of the corollary, upon passing $\eps \to 0$.
\end{proof}

\subsubsection{Identifying the leading order operator near the $s_0=r_c$ diagonal}
In this section we consider the set
\begin{align}
\left\{|s_0-r_c|\leq \frac{r_c}{k} \right\} 
\end{align}
and show that the contribution to the operator $L_\eps$ in \eqref{eq:L:eps:def} coming from the operators 
\begin{align}
L_{\eps,2} [f](r) 
&= \int_{0}^\infty \!\!\! \int_{0}^\infty \!\!\! \int_{0}^\infty \frac{(u(r)-u(r_c)) u'(r_c)}{(u(r)-u(r_c))^2+\eps^2}  \frac{(u(s)-u(r_c)) u'(s)}{(u(s)-u(r_c))^2 + \eps^2}  \frac{\eps u'(s_0)}{(u(s_0)-u(r_c))^2 +\eps^2} \notag\\
&\qquad \qquad \qquad \times \indicator{|s_0-r_c| \leq \frac{r_c}{k}}  \left( {\mathfrak B}_{\eps,1}(r,s_0,r_c) - {\mathfrak B}_{\eps,1}(r,r_c,r_c)\right){\mathfrak B}_{\eps,2}(s_0,s,r_c)  f(s) \, \dd s  \dd s_0 \dd r_c
\label{eq:L:eps:2:def}\\
L_{\eps,3} [f](r) 
&= \int_{0}^\infty \!\!\! \int_{0}^\infty \!\!\! \int_{0}^\infty \frac{(u(r)-u(r_c)) u'(r_c)}{(u(r)-u(r_c))^2+\eps^2}  \frac{(u(s)-u(r_c)) u'(s)}{(u(s)-u(r_c))^2 + \eps^2}  \frac{\eps u'(s_0)}{(u(s_0)-u(r_c))^2 +\eps^2} \notag\\
&\qquad \qquad \qquad \times \indicator{|s_0-r_c| \leq \frac{r_c}{k}}   {\mathfrak B}_{\eps,1}(r,r_c,r_c) \left({\mathfrak B}_{\eps,2}(s_0,s,r_c) -  {\mathfrak B}_{\eps,2}(r_c,s,r_c) \right) f(s) \, \dd s  \dd s_0 \dd r_c
\label{eq:L:eps:3:def}
\end{align}
vanish as $\eps\to0$ in  $L^2(dr)$. 
The goal is to establish a result which is similar to Lemma~\ref{lem:abstract:off:diagonal:1}. Once achieved, such a result shows that 
\begin{align}
L_{\eps,4}[f]
&= \int_{0}^\infty \!\!\! \int_{0}^\infty \!\!\! \int_{0}^\infty \frac{(u(r)-u(r_c)) u'(r_c)}{(u(r)-u(r_c))^2+\eps^2}  \frac{(u(s)-u(r_c)) u'(s)}{(u(s)-u(r_c))^2 + \eps^2}  \frac{\eps u'(s_0)}{(u(s_0)-u(r_c))^2 +\eps^2} \notag\\
&\qquad \qquad \qquad \times \indicator{|s_0-r_c| \leq \frac{r_c}{k}}   {\mathfrak B}_{\eps,1}(r,r_c,r_c) {\mathfrak B}_{\eps,2}(r_c,s,r_c)  f(s) \, \dd s  \dd s_0 \dd r_c
\label{eq:L:eps:4:def}
\end{align}
is the leading order operator with respect to $\eps$ in $L_{\eps}[f]$. Indeed, we note that 
\begin{align*}
L_{\eps}[f] - L_{\eps,4}[f] = L_{\eps,1}[f] + L_{\eps,2}[f] + L_{\eps,r}[f]
\end{align*}
where the right side vanishes in  $L^2(dr)$   as $\eps\to 0$.

\begin{lemma}
\label{lem:abstract:off:diagonal:2}
Let $\delta \in (0, \frac 12 )$ and assume that for some $\gamma \in(0, \frac{\delta}{4})$    we have that 
\begin{align}
\indicator{|s_0-r_c| \leq \frac{r_c}{k}} \frac{r_c^\gamma \left| {\mathfrak B}_{\eps,1}(r,s_0,r_c) - {\mathfrak B}_{\eps,1}(r,r_c,r_c) \right|}{|s_0-r_c|^\gamma} \, \left| {\mathfrak B}_{\eps,2}(s_0,s,r_c) \right| \lesssim {\mathfrak B}_0(r,s_0,s,r_c) 
\label{eq:BBB:cond:0:*}
\end{align}
holds for some $\eps$-independent function ${\mathfrak B}_0$ which obeys the bound \eqref{eq:BBB:cond:*}.
Then, if $ f \in L^2( ds)$, we have that  the operator $L_{\eps,2}[f]$, defined in \eqref{eq:L:eps:2:def}, vanishes  as $\eps\to 0$, in $L^2(  dr)$.
\end{lemma}
\begin{proof}[Proof of Lemma~\ref{lem:abstract:off:diagonal:2}]
The proof closely follows the proof of Lemma~\ref{lem:abstract:off:diagonal:1}.  Similarly to \eqref{eq:Merlot:1}, the lemma reduces to showing that the function
\begin{align*}
{\mathbb J}(r,s_0,s,r_c) &:= \frac{|(u(r)-u(r_c)) u'(r_c)|}{(u(r)-u(r_c))^2+\eps^2}  \frac{|(u(s)-u(r_c)) u'(s)|}{(u(s)-u(r_c))^2 + \eps^2}  \frac{\eps |u'(s_0)| \indicator{|s_0-r_c| \leq \frac{r_c}{k}} \frac{|s_0-r_c|^\gamma}{r_c^\gamma}}{(u(s_0)-u(r_c))^2 +\eps^2}   {\mathfrak B}_0(r,s_0,s,r_c)  
\end{align*}
obeys 
\begin{align}
\int_{\RR_+^4} {\mathbb J}(r,s,s_0,r_c) |\varphi(r)   f(s)| \dd s_0 \dd s \dd r \dd r_c  \lesssim \eps^{\frac{\gamma}{6}} \norm{\varphi}_{L^2} \norm{ f}_{L^2}.
\label{eq:Merlot:4}
\end{align}
The power law $\eps^{\frac{\gamma}{6}}$ is rather arbitrary.

\paragraph*{Case $r_c\geq 1$}
Note that estimates \eqref{eq:Warm:Merlot:1} and \eqref{eq:Warm:Merlot:2} hold as is, since they are independent of the bound $|s_0-r_c|\leq \frac{r_c}{k}$. Since $r_c\geq 1$ and $k \geq 2$, this restriction implies that $s_0 \geq \frac{r_c}{2} \geq \frac 12$, and we may thus replace \eqref{eq:Warm:Merlot:3} with
\begin{align}
\frac{\indicator{r_c\geq 1} \indicator{|s_0-r_c| \leq \frac{r_c}{k}} \eps^{1- \frac{5\gamma}{6}} |s_0-r_c|^{\gamma} |u'(s_0)|}{(u(s_0)-u(r_c))^2 +\eps^2} 
&\lesssim \indicator{r_c\geq 1} \indicator{|s_0-r_c| \leq \frac{r_c}{k}}  \frac{|s_0-r_c|^\gamma}{|u'(s_0)|^\frac{5\gamma}{6}} \left(\frac{|u'(s_0)|}{|u(s_0)-u(r_c)|}\right)^{1+\frac{5\gamma}{6}}
\notag\\
&\lesssim\brak{s_0}^{3\gamma} \left(   \frac{\indicator{r_c\geq 1}  \indicator{ |s_0-r_c| \geq \frac{1}{10} } }{k^\gamma \brak{s_0}^{1+\frac{\gamma}{3}} }    +
  \frac{\indicator{r_c\geq 1}  \indicator{|s_0-r_c|<\frac{1}{10}}}{ |s_0-r_c|^{1-\frac{\gamma}{6}}}  \right).
\label{eq:Warm:Merlot:7}
\end{align}
Estimate \eqref{eq:Warm:Merlot:7} is nearly identical to bound \eqref{eq:Warm:Merlot:3}, and in particular once the $\brak{s_0}^{3\gamma}$ is absorbed by the bound on ${\mathfrak B}_0$, the resulting object is integrable in $s_0$, with a bound that is $O(1)$ with respect to  $s,r_c$, and $r$. All the following arguments in the proof of Lemma~\ref{lem:abstract:off:diagonal:1}, for the case $r_c\geq 1$ follow line by line in this case too.  To avoid redundancy we omit these details.

\paragraph*{Case $r_c < 1$}
As before, the bounds \eqref{eq:Warm:Merlot:4} and \eqref{eq:Warm:Merlot:5} hold as is, since they are independent of the restriction $|s_0-r_c|\leq \frac{r_c}{k}$. For $r_c \leq 1$, the later restriction implies $|s_0-r_c|\leq \frac{1}{10}$ and also $s_0\leq \frac{11}{10}$, and thus \eqref{eq:Warm:Merlot:6} needs to be replaced by
\begin{align}
\frac{\indicator{r_c\leq 1} \indicator{|s_0-r_c| \leq \frac{r_c}{k}} \eps^{1- \frac{5\gamma}{6}} \frac{|s_0-r_c|^\gamma}{r_c^\gamma} |u'(s_0)|}{(u(s_0)-u(r_c))^2 +\eps^2}
&\lesssim   \frac{s_0^{\frac{5\gamma}{6}}  \indicator{r_c\leq 1} \indicator{ |s_0-r_c| \leq \frac{1}{10}}}{  r_c^\gamma |s_0-r_c|^{1-\frac{\gamma}{6}}}  \lesssim   \frac{\brak{s_0}^{3\gamma}  \indicator{r_c\leq 1} \indicator{ |s_0-r_c| \leq \frac{1}{10}}}{ r_c^{2\gamma}  |s_0-r_c|^{1-\frac{\gamma}{6}}} .
\label{eq:Warm:Merlot:8}
\end{align}
In the last inequality we have used that $r_c \leq 1$, with the purpose of showing that the right side of \eqref{eq:Warm:Merlot:8} is nearly identical to the first term on the right side of \eqref{eq:Warm:Merlot:6}. In particular, integrating \eqref{eq:Warm:Merlot:8} with respect to $s_0$ we obtain the same estimate as in the proof of Lemma~\ref{lem:abstract:off:diagonal:1}. All the following arguments in the proof of Lemma~\ref{lem:abstract:off:diagonal:1}, for the case $r_c\leq 1$ follow line by line in this case too.  To avoid redundancy we omit these details.
\end{proof}
The proof of Lemma~\ref{lem:abstract:off:diagonal:2} clearly implies, mutatis mutandi, also the following result:
\begin{lemma}
\label{lem:abstract:off:diagonal:3}
Let $\delta \in (0, \frac 12 )$ and assume that for some $\gamma \in(0, \frac{\delta}{4})$    we have that 
\begin{align}
\indicator{|s_0-r_c| \leq \frac{r_c}{k}} \left| {\mathfrak B}_{\eps,1}(r,r_c,r_c) \right| \, \frac{r_c^\gamma \left| {\mathfrak B}_{\eps,2}(s_0,s,r_c) - {\mathfrak B}_{\eps,2}(r_c,s,r_c) \right|}{|s_0-r_c|^\gamma}  \lesssim {\mathfrak B}_0(r,s_0,s,r_c) 
\label{eq:BBB:cond:0:**}
\end{align}
holds for some $\eps$-independent function ${\mathfrak B}_0$ which obeys the bound \eqref{eq:BBB:cond:*}.
Then, if $ f \in L^2( ds)$, we have that the operator $L_{\eps,3}[f]$, defined in \eqref{eq:L:eps:3:def}, vanishes  as $\eps\to 0$ in $L^2( dr)$.
\end{lemma}

\begin{remark}
\label{rem:BBB:cond:0:dual}
We note here that assumptions \eqref{eq:BBB:cond:0:*} and \eqref{eq:BBB:cond:0:**} may be replaced with the dual pair of conditions 
\begin{align}
\indicator{|s_0-r_c| \leq \frac{r_c}{k}} \left| {\mathfrak B}_{\eps,1}(r,s_0,r_c) \right| \, \frac{r_c^\gamma \left| {\mathfrak B}_{\eps,2}(s_0,s,r_c) - {\mathfrak B}_{\eps,2}(r_c,s,r_c) \right|}{|s_0-r_c|^\gamma}  &\lesssim {\mathfrak B}_0(r,s_0,s,r_c) 
\label{eq:BBB:cond:0:*:dual}\\
\indicator{|s_0-r_c| \leq \frac{r_c}{k}} \frac{r_c^\gamma \left| {\mathfrak B}_{\eps,1}(r,s_0,r_c) - {\mathfrak B}_{\eps,1}(r,r_c,r_c) \right|}{|s_0-r_c|^\gamma} \, \left| {\mathfrak B}_{\eps,2}(r_c,s,r_c) \right| &\lesssim {\mathfrak B}_0(r,s_0,s,r_c)
\label{eq:BBB:cond:0:**:dual}
\end{align}
where ${\mathfrak B}_0(r,s_0,s,r_c)$ obeys \eqref{eq:BBB:cond:*}.
\end{remark}

Lastly, similarly to Corollary~\ref{cor:abstract:off:diagonal:1}, we have that: 
\begin{corollary}
\label{cor:abstract:off:diagonal:2}
Let $0 \leq j \leq k-1$, and $0\leq \ell, \ell_1,\ell_2$ be such that $\ell + \ell_1 + \ell_2 \leq j$. 
Assume that the functions ${\mathfrak B}_{\eps,1}$ and ${\mathfrak B}_{\eps,2}$ in \eqref{eq:L:eps:2:def} and \eqref{eq:L:eps:3:def} are given by either  \eqref{eq:B1:2SIO:1DELTA}--\eqref{eq:B2:2SIO:1DELTA}, or \eqref{eq:B1:2SIO:1DELTA:a}--\eqref{eq:B2:2SIO:1DELTA:a}, 
where either $B_{\ell,\eps}^{(1)}$ is a suitable $(2\ell_1,\ell_1)$ kernel of type $II$ and $B_{\ell,\eps}^{(2)}$ is a suitable $(2\ell_2,\ell_2)$ kernel of type $I$, or $B_{\ell,\eps}^{(1)}$ is a suitable $(2\ell_1,\ell_1+\eta/2)$ kernel of type $I$ and $B_{\ell,\eps}^{(2)}$ is a suitable $(2\ell_2,\ell_2+\eta/2)$ kernel of type $II$.
Then the operators $L_{\eps,2}$ and $L_{\eps,3}$ defined in \eqref{eq:L:eps:2:def} and \eqref{eq:L:eps:3:def}, vanish in $L^2( dr)$ as $\eps \to 0$.
\end{corollary}
The proof of the corollary follows from the definition of being a kernel of type $I$, respectively $II$, and the proof of Corollaries~\ref{cor:abstract:off:diagonal:1} and \ref{cor:abstract:off:diagonal:1:a}. Indeed, the definitions of the kernel types precisely show that the conditions of Lemmas~\ref{lem:abstract:off:diagonal:2} and~\ref{lem:abstract:off:diagonal:3} are satisfied. Also here, we note that for the weight \eqref{eq:B1:2SIO:1DELTA:a}--\eqref{eq:B2:2SIO:1DELTA:a} a separate argument must be carried out in the region $\{ r \leq 1 \leq r_c\} \cap\{ 2r, 2s,2s_0 \leq r_c\}$, as in the proof of Corollary~\ref{cor:abstract:off:diagonal:1:a}.

\subsubsection{Convergence of the operator with $s_0=r_c$}
In this section we give the proof of Theorem~\ref{thm:BBB:main}. In view of Lemmas~\ref{lem:abstract:off:diagonal:1}, \ref{lem:abstract:off:diagonal:2}, and \ref{lem:abstract:off:diagonal:3}, we have shown that under the conditions \eqref{eq:BBB:cond:0}, \eqref{eq:BBB:cond:0:*}, and \eqref{eq:BBB:cond:0:**}, (see also Remark~\ref{rem:BBB:cond:0:dual} for a dual pair of conditions), we have that 
\begin{align}
\lim_{\eps\to0} \norm{L_{\eps} - L_{\eps,4}}_{ L^2 \to L^2   } =0
\label{eq:nasty:limit:operator}
\end{align}
where the operator $L_{\eps,4}$ is defined in \eqref{eq:L:eps:4:def}, and may be rewritten as
\begin{align}
L_{\eps,4}[f] 
&= \int_{0}^\infty \!\!\! \int_{0}^\infty \frac{(u(r)-u(r_c)) u'(r_c)}{(u(r)-u(r_c))^2+\eps^2}  \frac{(u(s)-u(r_c)) u'(s)}{(u(s)-u(r_c))^2 + \eps^2}    {\mathfrak B}_{\eps,1}(r,r_c,r_c) {\mathfrak B}_{\eps,2}(r_c,s,r_c) \notag\\
&\qquad \qquad \qquad \times \left( \arctan\left(\frac{u(\frac{k+1}{k} r_c )-u(r_c)}{\eps}\right) + \arctan\left(\frac{u(r_c )-u(\frac{k-1}{k}r_c )}{\eps}\right)\right)  f(s)  \, \dd s   \dd r_c
\label{eq:L:eps:4}
\end{align}
upon performing the explicit integration in $s_0$. Therefore, computing the limiting operator for $L_{\eps}$ reduces to computing the limiting operator for $L_{\eps,4}$.  For this purpose, we consider an arbitrary $\varphi \in L^2(dr)$ and compute as before
\begin{align}
&\left\langle   L_{\eps,4}[f](r) ,\varphi(r) \right \rangle
\notag\\
&= \int_{0}^\infty \left( \arctan\left(\frac{u(\frac{k+1}{k} r_c )-u(r_c)}{\eps}\right) + \arctan\left(\frac{u(r_c )-u(\frac{k-1}{k}r_c )}{\eps}\right)\right) 
\notag\\
&\qquad \times \left( \int_{0}^\infty  \frac{(u(r)-u(r_c)) u'(r_c)}{(u(r)-u(r_c))^2+\eps^2}  {\mathfrak B}_{\eps,1}(r,r_c,r_c)    \varphi(r) \, \dd r  \right)
\notag\\
&\qquad \times \left( \int_{0}^\infty  \frac{(u(s)-u(r_c)) u'(s)}{(u(s)-u(r_c))^2 + \eps^2}  {\mathfrak B}_{\eps,2}(r_c,s,r_c)  f(s)  \, \dd s  \right) \dd r_c \notag\\
&=: \int_0^\infty \left( \arctan\left(\frac{u(\frac{k+1}{k} r_c )-u(r_c)}{\eps}\right) + \arctan\left(\frac{u(r_c )-u(\frac{k-1}{k}r_c )}{\eps}\right)\right)  {\mathcal M}_{\eps,1}[\varphi](r_c) {\mathcal M}_{\eps,2}[ f](r_c) \,\dd r_c
\label{eq:L:eps:4:A}
\end{align}
where we have denoted the operators  
\begin{align}
{\mathcal M}_{\eps,1}[g] (r_c) &=\int_{0}^\infty  \frac{(u(r)-u(r_c)) u'(r)}{(u(r)-u(r_c))^2+\eps^2}    \frac{u'(r_c)  {\mathfrak B}_{\eps,1}(r,r_c,r_c)}{ {\mathfrak m}(r_c) u'(r)  }  g(r) \, \dd r 
\label{eq:Gluhwein:1:0}\\
{\mathcal M}_{\eps,2}[g] (r_c)  &= \int_{0}^\infty  \frac{(u(r)-u(r_c)) u'(s)}{(u(s)-u(r_c))^2 + \eps^2}       {\mathfrak m}(r_c) {\mathfrak B}_{\eps,2}(r_c,s,r_c)  g(s) \, \dd s
\label{eq:Gluhwein:2:0}
\end{align}
and the weight ${\mathfrak m}(r_c)$ is at our discretion. 
Assume for the moment that we can show for a fixed function $g \in L^2(\RR_+)$ that we have
\begin{align}
\lim_{\eps \to 0}  {\mathcal M}_{\eps,1}[g] (r_c) = {\mathcal M}_{0,1}[g](r_c) := p.v. \int_{0}^\infty  \frac{u'(r)}{u(r)-u(r_c)}    \frac{u'(r_c) {\mathfrak B}_{0,1}(r,r_c,r_c)}{{\mathfrak m}(r_c) u'(r) }  g(r) \, \dd r \quad \mbox{in} \quad L^2(dr_c)
\label{eq:Gluhwein:1}
\end{align}
and 
\begin{align}
\lim_{\eps \to 0}  {\mathcal M}_{\eps,2}[g] (r_c) = {\mathcal M}_{0,2}[g](r_c) := p.v. \int_{0}^\infty  \frac{u'(s)}{u(s)-u(r_c)}  {\mathfrak m}(r_c) {\mathfrak B}_{0,2}(r_c,s,r_c)  g(s) \, \dd s \quad \mbox{in} \quad L^2(dr_c)
\label{eq:Gluhwein:2}
\end{align}
where
\begin{align}
{\mathfrak B}_{0,1}(r,r_c,r_c) = \lim_{\eps\to 0} {\mathfrak B}_{\eps,1}(r,r_c,r_c) \quad \mbox{and} \quad {\mathfrak B}_{0,2}(r_c,s,r_c) = \lim_{\eps\to 0} {\mathfrak B}_{\eps,2}(r_c,s,r_c)
\label{eq:nasty:limit}
\end{align}
holds in a sense that is determined by Corollary~\ref{lem:abstract:4:a} below.
If \eqref{eq:Gluhwein:1} and \eqref{eq:Gluhwein:2} hold, then by \eqref{eq:L:eps:4:A}, the fact that $|\arctan(\cdot)| \leq \frac{\pi}{2}$, and the Dominated Convergence Theorem, we would have that 
\begin{align}
&\lim_{\eps \to 0} \left|\left\langle  L_{\eps,4}[f](r) ,\varphi(r) \right \rangle + \pi \int_0^\infty {\mathcal M}_{0,1}[\varphi](r_c) {\mathcal M}_{0,2}[  f](r_c) \dd r_c \right| \notag\\
&\qquad \lesssim \lim_{\eps \to 0}  \pi \norm{ {\mathcal M}_{\eps,2}[ f] - {\mathcal M}_{0,2}[ f]}_{L^2_{dr_c}} \norm{ {\mathcal M}_{\eps,1}[\varphi]}_{L^2_{dr_c}} + \lim_{\eps \to 0}  \pi \norm{ {\mathcal M}_{0,2}[ f]}_{L^2_{dr_c}} \norm{ {\mathcal M}_{\eps,1}[\varphi] - {\mathcal M}_{0,1}[\varphi] }_{L^2_{dr_c}}  \notag\\
&\quad \quad + \lim_{\eps\to 0} \int_0^\infty \left( \pi + \arctan\left(\frac{u(\frac{k+1}{k} r_c )-u(r_c)}{\eps}\right) + \arctan\left(\frac{u(r_c )-u(\frac{k-1}{k}r_c )}{\eps}\right) \right) \notag\\
&\qquad \qquad \qquad \times \left| {\mathcal M}_{0,1}[\varphi](r_c) {\mathcal M}_{0,2}[ f](r_c) \right| \dd r_c \notag\\
&\qquad = 0
\label{eq:L:eps:4:B}
\end{align}
and therefore we have identified the limit in  $L^2(dr)$ sense of $L_{\eps,4}[f]$, finishing the proof of the theorem.

It remains to show that \eqref{eq:Gluhwein:1} and \eqref{eq:Gluhwein:2} hold under a suitable convergence condition of the weight.
For this purpose it remains to consider the convergence as $\eps\to 0$, in $L^2(dr_c)$, of the model operator 
\begin{align}
{\mathcal M}_{\eps}[g](r_c) = \int_{0}^{\infty} \frac{(u(r)-u(r_c)) u'(r)}{ (u(r) -u(r_c))^2 + \eps^2} {\mathfrak M}_\eps(r,r_c) g(r) \dd r
\label{eq:MMM:eps:def}
\end{align}
where in view of \eqref{eq:Gluhwein:1:0}--\eqref{eq:Gluhwein:2:0} the weight ${\mathfrak M}_\eps(r,r_c)$ plays the role of either 
\begin{align*}
\frac{u'(r_c)  {\mathfrak B}_{\eps,1}(r,r_c,r_c)}{{\mathfrak m}(r_c) u'(r)  }
\qquad \mbox{or} \qquad 
 {\mathfrak m}(r_c) {\mathfrak B}_{\eps,2}(r_c,r,r_c) .
\end{align*}

We prove the $L^2(dr_c)$ convergence of ${\mathcal M}_{\eps}$ under suitable conditions on the difference ${\mathfrak M}_\eps(r,r_c) - {\mathfrak M}_0(r,r_c)$, and then check that these conditions hold for the specific weights that arise in \eqref{eq:Gluhwein:1:0}--\eqref{eq:Gluhwein:2:0}. A simple-to-work-with set of assumptions are:
\begin{lemma}
\label{lem:abstract:4}
Assume that there exists $0 < \zeta \leq \gamma < 1$ such that the following properties hold for all $r,r_c$:
\begin{subequations}
\label{eq:MMM:cond:all}
\begin{align}
|{\mathfrak M}_\eps(r,r_c) - {\mathfrak M}_0(r,r_c)| 
&\lesssim \eps^{\zeta}   {\mathfrak M}(r,r_c)
\label{eq:MMM:cond:*}
\\
|{\mathfrak M}_0(r,r_c)| 
&\lesssim   {\mathfrak M}(r,r_c)
\label{eq:MMM:cond:eps:*}
\\
\indicator{|r-r_c|\leq\frac{r_c}{k}}|{\mathfrak M}_0(r,r_c) - {\mathfrak M}_0(r_c,r_c)| 
&\lesssim \frac{k^\gamma |r-r_c|^\gamma}{r_c^\gamma}   {\mathfrak M}(r,r_c)
\label{eq:MMM:cond:eps:**}
\end{align}
\end{subequations}
where the $\eps$-independent function $ {\mathfrak M} \geq 0$ is defined by
\begin{align}
{\mathfrak M}(r,r_c) =  \frac{r_c^{\frac{3\zeta}{2}}}{\brak{r_c}^{3\zeta}} \left( \indicator{r_c\leq r} \frac{r^{\frac {1-\zeta}{2}}}{r_c^{\frac{1-\zeta}{2}}}  + \indicator{r<r_c} \frac{\brak{r}^{\frac{1+\zeta}{2}}}{\brak{r_c}^{\frac{1+\zeta}{2}}} \right)
\label{eq:Buffalo:1}.
\end{align}
Then, for any $g \in L^2(dr)$, the operator ${\mathcal M}_{\eps}[g]$ defined in \eqref{eq:MMM:eps:def} converges as $\eps\to0$, in $L^2(dr_c)$, to ${\mathcal M}_0[g]$, defined by
\begin{align}
{\mathcal M}_{0}[g](r_c) &= p.v. \int_{0}^{\infty} \frac{u'(r)}{u(r) -u(r_c)} {\mathfrak M}_0(r,r_c) g(r) \dd r \notag\\
&=: \lim_{\eps'\to0} \int_{|u(r)-u(r_c)|\geq \eps'} \frac{u'(r)}{u(r) -u(r_c)} {\mathfrak M}_0(r,r_c) g(r) \dd r
\label{eq:MMM:0:def}
\end{align}
and this limiting operator is bounded on $L^2(dr_c)$, with norm bounded from above by $k^{\zeta}$.
\end{lemma}
Under the assumptions \eqref{eq:MMM:cond:*}--\eqref{eq:Buffalo:1}, the proof of the above lemma is direct. We give it here for the sake of completeness. However, before giving the proof, we state an immediate corollary, which yields the proof of Theorem~\ref{thm:BBB:main}.
\begin{lemma}
\label{lem:abstract:4:a}
Assume that there exists $0<\zeta < \frac{\delta}{4}$, limiting weights ${\mathfrak B}_{0,1}(r,r_c,r_c)$, ${\mathfrak B}_{0,2}(r_c,r,r_c)$, and that we may choose a weight ${\mathfrak m}(r_c) > 0$, which obeys $r_c \partial_{r_c} {\mathfrak m}(r_c)/ {\mathfrak m}(r_c) \lesssim k$, such that the following conditions hold:
\begin{subequations}
\label{eq:BBB:cond:nasty:1}
\begin{align}
| {\mathfrak B}_{\eps,1}(r,r_c,r_c) - {\mathfrak B}_{0,1}(r,r_c,r_c)  |
&\lesssim \eps^\zeta \frac{r  \brak{r_c}^4  {\mathfrak m}(r_c)}{\brak{r}^{4} r_c} {\mathfrak M}(r,r_c)
\label{eq:BBB:cond:nasty:1:a}\\
|{\mathfrak B}_{0,1}(r,r_c,r_c) |
&\lesssim \frac{r \brak{r_c}^4  {\mathfrak m}(r_c)}{\brak{r}^{4} r_c} {\mathfrak M}(r,r_c)
\label{eq:BBB:cond:nasty:1:b}\\
\indicator{|r-r_c|\leq\frac{r_c}{k}}|{\mathfrak B}_{0,1}(r,r_c,r_c) - {\mathfrak B}_{0,1}(r_c,r_c,r_c)| 
&\lesssim \frac{k^\gamma |r-r_c|^\gamma}{r_c^\gamma} \frac{r \brak{r_c}^4  {\mathfrak m}(r_c)}{\brak{r}^{4} r_c}  {\mathfrak M}(r,r_c)
\label{eq:BBB:cond:nasty:1:c}
\end{align}
\end{subequations}
and
\begin{subequations}
\label{eq:BBB:cond:nasty:2}
\begin{align}
|{\mathfrak B}_{\eps,2}(r_c,s,r_c) - {\mathfrak B}_{0,2}(r_c,s,r_c) |
&\lesssim \eps^\zeta \frac{1}{ {\mathfrak m}(r_c)} {\mathfrak M}(s,r_c)
\label{eq:BBB:cond:nasty:2:a}\\
|{\mathfrak B}_{0,2}(r_c,s,r_c) |
&\lesssim \frac{1}{ {\mathfrak m}(r_c)} {\mathfrak M}(s,r_c)
\label{eq:BBB:cond:nasty:2:b}\\
\indicator{|s-r_c|\leq\frac{r_c}{k}}|{\mathfrak B}_{0,2}(r_c,s,r_c) - {\mathfrak B}_{0,2}(r_c,r_c,r_c)| 
&\lesssim \frac{k^\gamma |s-r_c|^\gamma}{r_c^\gamma} \frac{1}{ {\mathfrak m}(r_c)}  {\mathfrak M}(s,r_c)
\label{eq:BBB:cond:nasty:2:c}
\end{align}
\end{subequations}
where the function ${\mathfrak M}(r,r_c)$ is defined by \eqref{eq:Buffalo:1} above.
Then we have that  
\begin{align*}
\lim_{\eps\to 0} L_{\eps}[f](r) = L_{0}[f](r) \quad \mbox{as bounded operators} \quad L^2\to L^2,
\end{align*}
where the operator $L_0$ is defined in \eqref{eq:L:0:def}. Moreover, the operator norm of $L_{0}[f]$ in this space is bounded by $k^\zeta$.
\end{lemma}

Lemma~\ref{lem:abstract:4:a} is a direct consequence of Lemma~\ref{lem:abstract:4}, which implies that  \eqref{eq:L:eps:4:B} holds, and we conclude using \eqref{eq:nasty:limit:operator}. We omit these details. It thus remains to prove Lemma~\ref{lem:abstract:4}.
\begin{proof}[Proof of Lemma~\ref{lem:abstract:4}]
Let $\chi_c(r,r_c)$ be defined as in \eqref{def:chic}. 
For any test function $\varphi \in L^2(dr_c)$, we may decompose
\begin{align}
\left \langle {\mathcal M}_{\eps}[g] , \varphi \right \rangle
&=  \int_0^\infty \!\!\! \int_0^\infty  \frac{(u(r)-u(r_c))u'(r)}{(u(r) -u(r_c))^2+\eps^2 } \left( {\mathfrak M}_\eps(r,r_c) -  {\mathfrak M}_0(r,r_c) \right) g(r) \varphi(r_c) \dd r \dd r_c 
\notag\\
&\qquad + \int_0^\infty \!\!\! \int_0^\infty  \frac{u'(r)(u(r)-u(r_c)) (1- \chi_c(r,r_c)){\mathfrak M}_0(r,r_c)}{ (u(r) -u(r_c))^2 + \eps^2}   g(r)  \varphi(r_c) \dd r \dd r_c
\notag\\
&\qquad + \int_0^\infty \!\!\! \int_0^\infty  \frac{u'(r) (u(r)-u(r_c))\chi_c(r,r_c) ( {\mathfrak M}_0(r,r_c) -{\mathfrak M}_0(r_c,r_c)  )}{(u(r) -u(r_c))^2 + \eps^2}   g(r)   \varphi(r_c) \dd r \dd r_c
\notag\\
&\qquad +  \int_0^\infty \!\!\! \int_0^\infty  \frac{u'(r) (u(r)-u(r_c))\chi_c(r,r_c)}{(u(r) -u(r_c))^2 + \eps^2}   g(r)   {\mathfrak M}_0(r_c,r_c)  \varphi(r_c) \dd r \dd r_c
\notag\\
&=: \left\langle {\mathcal M}_{\eps}^{(0)}[g] , \varphi \right \rangle + \left\langle {\mathcal M}_{\eps}^{(1)}[g] , \varphi \right \rangle + \left\langle {\mathcal M}_{\eps}^{(2)}[g] , \varphi \right \rangle +\left\langle {\mathcal M}_{\eps}^{(3)}[g] , \varphi \right \rangle .
\label{eq:Barefoot:0}
\end{align}
Next compute the limit of the first three terms on the right side of \eqref{eq:Barefoot:0} vanish as $\eps \to 0$.

By \eqref{eq:u':info:1} and \eqref{eq:MMM:cond:all}--\eqref{eq:Buffalo:1}, we have
\begin{align}
&\left|\left\langle {\mathcal M}_{\eps}^{(0)}[g] , \varphi \right \rangle \right|
\notag\\
&\lesssim  k \eps^{\frac{\zeta}{2}} \int_0^\infty \!\!\! \int_0^\infty \frac{|u'(r)|}{|u(r) -u(r_c)|^{1 -\frac{\zeta}{2}}}  {\mathfrak M}(r,r_c)| g(r) \varphi(r_c)| \dd r \dd r_c 
\notag\\
&\lesssim  k \eps^{\frac{\zeta}{2}} \int_0^\infty \!\!\! \int_0^\infty \frac{r^\frac{\zeta}{2} r_c^{\frac{3\zeta}{2}}}{\brak{r}^{2\zeta} \brak{r_c}^{3\zeta}} \left[\frac{\indicator{|r-r_c|\leq \frac{1}{10}}}{|r-r_c|^{1- \frac{\zeta}{2}}}  + \frac{\indicator{|r-r_c|\geq \frac{1}{10}}}{\brak{r}^{1- \frac{\zeta}{2}}}\right] \left[ \indicator{r_c\leq r} \frac{r^{\frac {1-\zeta}{2}}}{r_c^{\frac{1-\zeta}{2}}}  + \indicator{r<r_c} \frac{\brak{r}^{\frac{1+\zeta}{2}}}{\brak{r_c}^{\frac{1+\zeta}{2}}} \right]| g(r) \varphi(r_c)| \dd r \dd r_c 
\notag\\
&\lesssim  k \eps^{\frac{\zeta}{2}} \int_0^\infty \!\!\! \int_0^\infty \left[\frac{\indicator{|r-r_c|\leq \frac{1}{10}}}{|r-r_c|^{1- \frac{\zeta}{2}}}  +  \frac{1}{\brak{r}^{\frac{1 + \zeta}{2}}\brak{r_c}^{\frac{1 + \zeta}{2}} } + \frac{\indicator{r_c\leq 1}}{\brak{r}^{\frac{1+\zeta}{2}} r_c^{\frac{1-\zeta}{2}}}  + \frac{\indicator{r_c\leq 1} \indicator{r\leq 1}}{r^{\frac{1-\zeta}{2}} r_c^{\frac{1-\zeta}{2}}} \right]| g(r) \varphi(r_c)| \dd r \dd r_c 
\notag\\
&\lesssim k \eps^{\frac{\zeta}{2}} \norm{g}_{L^2} \norm{\varphi}_{L^2} 
\left(\norm{\frac{\indicator{|\rho|\leq \frac{1}{10}}}{|\rho|^{1- \frac{\zeta}{2}}}}_{L^1_\rho} + \norm{\frac{1}{\brak{r}^{\frac{1 + \zeta}{2}}\brak{r_c}^{\frac{1 + \zeta}{2}} } }_{L^2_{r,r_c}} + \norm{\frac{\indicator{r_c\leq 1}}{\brak{r}^{\frac{1+\zeta}{2}} r_c^{\frac{1-\zeta}{2}}}}_{L^2_{r,r_c}} + \norm{\frac{\indicator{r_c\leq 1} \indicator{r\leq 1}}{r^{\frac{1-\zeta}{2}} r_c^{\frac{1-\zeta}{2}}}}_{L^2_{r,r_c}}\right)
\notag\\
&\lesssim k \eps^{\frac{\zeta}{2}} \norm{g}_{L^2} \norm{\varphi}_{L^2}  \rightarrow 0 \quad \mbox{as} \quad \eps \to 0.
\label{eq:Buffalo:1:*}
\end{align}
Similarly, using \eqref{eq:u':info:1} and \eqref{eq:MMM:cond:eps:*}--\eqref{eq:MMM:cond:eps:**}, which allows us to replace $\gamma$ with $\zeta$ in \eqref{eq:MMM:cond:eps:**} since $\zeta \leq \gamma$, we {get}
\begin{align}
&\left|  \left\langle {\mathcal M}_{\eps}^{(2)}[g] , \varphi \right \rangle - \left\langle {\mathcal M}_{0}^{(2)}[g] , \varphi \right \rangle \right|
\notag\\
&\lesssim  \eps^{\frac{\zeta}{2}} k^\zeta \int_0^\infty \!\!\! \int_0^\infty  \frac{|u'(r)| |r-r_c|^\zeta \indicator{|r-r_c| \leq \frac{r_c}{k}}}{r_c^\zeta |u(r) -u(r_c)|^{1 + \frac{\zeta}{2}} }  { \mathfrak M}(r,r_c) |g(r) \varphi(r_c)| \dd r \dd r_c
\notag\\
&\lesssim  \eps^{\frac{\zeta}{2}}k^\zeta  \int_0^\infty \!\!\! \int_0^\infty    
\frac{|r-r_c|^\zeta \brak{r}^{2\zeta} \indicator{|r-r_c|\leq \frac{r_c}{k}} }{r_c^\zeta r^{\frac{\zeta}{2}}} \left[ \frac{\indicator{|r-r_c|\leq \frac{1}{10}}}{|r-r_c|^{1+\frac{\zeta}{2}}} + \frac{\indicator{|r-r_c|\geq \frac{1}{10}}}{\brak{r}^{1+\frac{\zeta}{2}}} \right]
\notag\\
&\qquad \qquad \times 
\left[\indicator{r_c\leq r} \frac{r^{\frac {1-\zeta}{2}}}{r_c^{\frac{1-\zeta}{2}}}  + \indicator{r<r_c} \frac{\brak{r}^{\frac{1+\zeta}{2}}}{\brak{r_c}^{\frac{1+\zeta}{2}}} \right]
|g(r) \varphi(r_c)| \dd r \dd r_c
\notag\\
&\lesssim  \eps^{\frac{\zeta}{2}}k^\zeta  \int_0^\infty \!\!\! \int_0^\infty    
 \left[ \frac{\indicator{|r-r_c|\leq \frac{1}{10}}}{|r-r_c|^{1-\frac{\zeta}{2}}} +   \frac{\indicator{\frac{1}{10}\leq |r-r_c|\leq \frac{r_c}{k}}}{\brak{r}^{\frac{1+\zeta}{2}}\brak{r_c}^{\frac{1+\zeta}{2}}} 
\right]
|g(r) \varphi(r_c)| \dd r \dd r_c
\notag\\
&\lesssim  \eps^{\frac{\zeta}{2}} k^\zeta  \norm{g}_{L^2} \norm{\varphi}_{L^2} \rightarrow 0 \quad \mbox{as} \quad \eps\to0.
\label{eq:Buffalo:3}
\end{align}
Moreover, a bound similar to the above shows that
\begin{align}
\left|\left\langle {\mathcal M}_{0}^{(2)}[g] , \varphi \right \rangle \right| \lesssim k^\zeta  \norm{g}_{L^2} \norm{\varphi}_{L^2}
\label{eq:Buffalo:3:*}
\end{align}
so that the limiting operator ${\mathcal M}_0^{(2)}$ is bounded on $L^2$.

In a similar fashion, by using \eqref{eq:u':info:1}, \eqref{eq:u':info:3}, and \eqref{eq:MMM:cond:eps:*}, we arrive at
\begin{align}
&\left| \left\langle {\mathcal M}_{\eps}^{(1)}[g] , \varphi \right \rangle - \left\langle {\mathcal M}_{0}^{(1)}[g] , \varphi \right \rangle \right| 
\notag\\
&\lesssim  \eps^{\frac{\zeta}{2}}\int_0^\infty \!\!\! \int_0^\infty  \frac{|u'(r)| \indicator{|r-r_c| \geq \frac{r_c}{k}}}{|u(r) -u(r_c)|^{1 + \frac{\zeta}{2}} }  { \mathfrak M}(r,r_c) |g(r) \varphi(r_c)| \dd r \dd r_c
\notag\\
&\lesssim  \eps^{\frac{\zeta}{2}}\int_0^\infty \!\!\! \int_0^\infty    
\frac{ \brak{r}^{2\zeta} \indicator{|r-r_c|\geq \frac{r_c}{k}} }{r^{\frac{\zeta}{2}}} \left[ \frac{\indicator{|r-r_c|\leq \frac{1}{10}}}{|r-r_c|^{1+\frac{\zeta}{2}}}
\left( \frac{\indicator{r \leq 1} r^{1 + \frac{\zeta}{2}}}{(r+r_c)^{1 + \frac{\zeta}{2}}} + \indicator{r\geq 1}\right) + \indicator{|r-r_c|\geq \frac{1}{10}} \left( \indicator{r\leq 1} r^{1 + \frac{\zeta}{2}}+ \frac{\indicator{r\geq 1}}{\brak{r}^{1+\frac{\zeta}{2}}} \right)\right]
\notag\\
&\qquad \qquad \times \frac{r_c^{\frac{3\zeta}{2}}}{\brak{r_c}^{3\zeta}}
\left[\indicator{r_c\leq r} \frac{r^{\frac {1-\zeta}{2}}}{r_c^{\frac{1-\zeta}{2}}}  + \indicator{r<r_c} \frac{\brak{r}^{\frac{1+\zeta}{2}}}{\brak{r_c}^{\frac{1+\zeta}{2}}} \right]
|g(r) \varphi(r_c)| \dd r \dd r_c
\notag\\
&\lesssim  \eps^{\frac{\zeta}{2}}\int_0^\infty \!\!\! \int_0^\infty    
 \left[ \frac{\indicator{\frac{r_c}{k}\leq|r-r_c|\leq \frac{1}{10}}}{|r-r_c|^{1-\frac{\zeta}{2}}} +   \frac{1}{\brak{r}^{\frac{1+\zeta}{2}}\brak{r_c}^{\frac{1+\zeta}{2}}}   + \frac{\indicator{r\leq 1} \indicator{r_c\leq 1}}{r^{\frac{1-\zeta}{2}} r_c^{\frac{1-\zeta}{2}}}   + \frac{ \indicator{r_c\leq 1}}{\brak{r}^{\frac{1+\zeta}{2}} r_c^{\frac{1-\zeta}{2}}} 
\right]
|g(r) \varphi(r_c)| \dd r \dd r_c
\notag\\
&\lesssim  \eps^{\frac{\zeta}{2}} \norm{g}_{L^2} \norm{\varphi}_{L^2} \rightarrow 0 \quad \mbox{as} \quad \eps\to0.
\label{eq:Buffalo:2}
\end{align}
Moreover, a bound similar to the above shows that
\begin{align}
\left|\left\langle {\mathcal M}_{0}^{(1)}[g] , \varphi \right \rangle \right| \lesssim \norm{g}_{L^2} \norm{\varphi}_{L^2}
\label{eq:Buffalo:2:*}
\end{align}
so that the limiting operator ${\mathcal M}_0^{(1)}$ is bounded on $L^2$.

Thus, by \eqref{eq:Buffalo:1:*}, \eqref{eq:Buffalo:2}--\eqref{eq:Buffalo:2:*}, and \eqref{eq:Buffalo:3}--\eqref{eq:Buffalo:3:*}, we know that the first three terms on the right side of \eqref{eq:Barefoot:0} converge as $\eps\to 0$ to bounded operators on $L^2(dr_c)$.

Lastly, we need to consider the fourth operator on the right side of \eqref{eq:Barefoot:0}, namely
\begin{align*}
{\mathcal M}_{\eps}^{(3)}[g](r_c)  
&= {\mathfrak M}_0(r_c,r_c) 
\int_0^\infty \frac{u'(r) (u(r)-u(r_c))}{(u(r) -u(r_c))^2 + \eps^2}   \chi_c(r,r_c)g(r)  \dd r.
\end{align*}
We note that since $u \colon [0,\infty) \to (0,u(0)]$ is a bijection, hence upon making a change of variables, we have that 
\begin{align}
\norm{f(r_c)}_{L^2(dr_c)} = \norm{\frac{f(u^{-1}(c))}{\sqrt{|u'(u^{-1}(c))|}}}_{L^2(dc)}
\label{eq:Barefoot:1}
\end{align}
so that we may change variables in the formula for ${\mathcal M}_{\eps}^{(3)}[g](r_c)$ as
\begin{align*}
{\mathcal M}_{\eps}^{(3)}[g](r_c) &= {\mathfrak M}_0(r_c,r_c)  \sqrt{|u'(r_c)|}  
\int_0^\infty \frac{u'(r) (u(r)-u(r_c))}{(u(r) -u(r_c))^2 + \eps^2}  \left(\frac{\sqrt{|u'(r)|} \chi_c(r,r_c)}{\sqrt{|u'(r_c)|}} \right)\frac{g(r)}{\sqrt{|u'(r)|}}  \dd r 
\notag\\
&=   \sqrt{|u'(u^{-1}(c))|} {\mathfrak M}_0(u^{-1}(c),u^{-1}(c)) \notag\\
&\qquad \times
\int_0^{u(0)} \frac{(y-c)}{(y-c)^2 + \eps^2}  \left(\frac{\sqrt{|u'(u^{-1}(y))|} \chi_c(u^{-1}(y),u^{-1}(c))}{\sqrt{|u'(u^{-1}(c))|}} \right)\frac{g(u^{-1}(y))}{\sqrt{|u'(u^{-1}(y))|}}  dy
\notag\\
&=:\sqrt{|u'(u^{-1}(c))|}
 {\mathcal M}_{\eps}^{(4)}\left[\frac{g(u^{-1}(\cdot))}{\sqrt{|u'(u^{-1}(\cdot))|}}\right](c).
\end{align*}
Thus, if we can show that the operator ${\mathcal M}_{\eps}^{(4)}$ is bounded on $L^2((0,u(0)])$, uniformly in $\eps$, and that it has a certain limit in this space, then by applying \eqref{eq:Barefoot:1} twice, we have proven that 
\begin{align}
\norm{{\mathcal M}_{\eps}^{(3)}[g] - {\mathcal M}_{0}^{(3)}[g]}_{L^2(dr_c)} &= \norm{{\mathcal M}_{\eps}^{(4)}\left[\frac{g(u^{-1}(\cdot))}{\sqrt{|u'(u^{-1}(\cdot))|}}\right] - {\mathcal M}_{0}^{(4)}\left[\frac{g(u^{-1}(\cdot))}{\sqrt{|u'(u^{-1}(\cdot))|}}\right]}_{L^2(dc)}
\notag\\
&\leq \norm{{\mathcal M}_{\eps}^{(4)} - {\mathcal M}_{0}^{(4)}}_{L^2(dc) \to L^2(dc)} \norm{\frac{g(u^{-1}(\cdot))}{\sqrt{|u'(u^{-1}(\cdot))|}}}_{L^2(dc)}
\notag\\
&= \norm{{\mathcal M}_{\eps}^{(4)} - {\mathcal M}_{0}^{(4)}}_{L^2(dc) \to L^2(dc)} \norm{g}_{L^2(dr_c)}.
\label{eq:Barefoot:2}
\end{align}
Therefore, in order to conclude the proof of Lemma~\ref{lem:abstract:4}, it remains to show that 
\begin{align}
{\mathcal M}_{\eps}^{(4)} - {\mathcal M}_{0}^{(4)} \to 0 \quad\mbox{in}\quad  L^2(dc)\quad \mbox{as} \quad \eps \to 0,
\label{eq:Buffalo:4}
\end{align} 
and that the the limiting operator ${\mathcal M}_{0}^{(4)}$,  naturally defined as
\begin{align}
{\mathcal M}_{0}^{(4)}[\psi](c) = {\mathfrak M}_0(u^{-1}(c),u^{-1}(c)) 
p.v. \int_0^{u(0)} \left(\frac{\sqrt{|u'(u^{-1}(y))|} \chi_c(u^{-1}(y),u^{-1}(c))}{\sqrt{|u'(u^{-1}(c))|}} \right) \frac{\psi(y)}{y-c}  dy,
\label{eq:Barefoot:3}
\end{align}
obeys
\begin{align}
\norm{{\mathcal M}_{0}^{(4)}[\psi]}_{L^2(dc)} \lesssim \norm{\psi}_{L^2(dc)}
\label{eq:Buffalo:4:*}
\end{align}
Indeed, once this is achieved, we may change variables $c \mapsto r_c$, to obtain that
\begin{align*}
{\mathcal M}_{\eps}^{(3)}[g](r_c) \to {\mathcal M}_{0}^{(3)}[g](r_c) 
&= {\mathfrak M}_0(r_c,r_c) p.v. \int_0^\infty \frac{u'(r)}{u(r) -u(r_c)}   \chi_c(r,r_c)g(r)  \dd r\notag\\
& := {\mathfrak M}_0(r_c,r_c)  \lim_{\eps'>0}  
\int_0^\infty \frac{u'(r)}{u(r) -u(r_c)}   \chi_c(r,r_c)g(r)  \dd r
\end{align*}
in $L^2(dr_c)$, as $\eps\to 0$.
Combining the above with \eqref{eq:Buffalo:1:*}, \eqref{eq:Buffalo:3}--\eqref{eq:Buffalo:3:*},  \eqref{eq:Buffalo:2}--\eqref{eq:Buffalo:2:*}, \eqref{eq:Barefoot:1}, \eqref{eq:Barefoot:2}, and \eqref{eq:Buffalo:4}--\eqref{eq:Buffalo:4:*}, we obtain that 
\begin{align*}
{\mathcal M}_{\eps}[g](r_c) \to {\mathcal M}_{0}^{(1)}[g](r_c) + {\mathcal M}_{0}^{(2)}[g](r_c)  + {\mathcal M}_{0}^{(3)}[g](r_c)  = p.v. \int_0^\infty \frac{u'(r)}{u(r) -u(r_c)}  {\mathfrak M}_0(r,r_c)g(r)  \dd r
\end{align*}
and that the limiting operator is bounded on $L^2$, as desired.

In order to prove \eqref{eq:Buffalo:4} we note that for any $c>0$ we have
\begin{align*}
\lim_{y\to c} \frac{\sqrt{|u'(u^{-1}(y))|} \chi_c(u^{-1}(y),u^{-1}(c))}{\sqrt{|u'(u^{-1}(c))|}}  = 1
\end{align*}
which allows us to rewrite
\begin{align}
{\mathcal M}_{\eps}^{(4)}[\psi](c) 
&= {\mathfrak M}_0(u^{-1}(c),u^{-1}(c)) 
\int_0^{u(0)} \left(\frac{\sqrt{|u'(u^{-1}(y))|} \chi_c(u^{-1}(y),u^{-1}(c))}{\sqrt{|u'(u^{-1}(c))|}} - 1 \right) \frac{\psi(y) (y-c)}{(y-c)^2+\eps^2}  dy \notag\\
&\qquad + {\mathfrak M}_0(u^{-1}(c),u^{-1}(c)) 
p.v. \int_0^{u(0)}  \frac{\psi(y) (y-c)}{(y-c)^2+\eps^2}  dy
\notag\\
&=: {\mathcal M}_{\eps}^{(4,1)}[\psi](c)  + {\mathcal M}_{\eps}^{(4,2)}[\psi](c). 
\label{eq:Barefoot:4}
\end{align}
The convergence in $L^2(dc)$ of the second part in \eqref{eq:Barefoot:4}, namely
\begin{align*}
{\mathcal M}_{\eps}^{(4,2)}[\psi](c)  \to  {\mathcal M}_{0}^{(4,2)}[\psi](c) ={\mathfrak M}_0(u^{-1}(c),u^{-1}(c)) {\mathcal H}\left(\indicator{(0,u(0)]} \psi\right) (c)
\end{align*}
where ${\mathcal H}$ is the Hilbert transform is classical. Moreover, since ${\mathcal H}$ is unitary on $L^2$, and since by \eqref{eq:MMM:cond:eps:*} and \eqref{eq:Buffalo:1} we have
\begin{align*}
|{\mathfrak M}_0(u^{-1}(c),u^{-1}(c))| \lesssim {\mathfrak M}(r_c,r_c) \lesssim 1, 
\end{align*}
we have that 
\begin{align*}
\norm{{\mathcal M}_{0}^{(4,2)}[\psi]}_{L^2} \lesssim \norm{\psi}_{L^2}, 
\end{align*}
which is consistent with \eqref{eq:Buffalo:4:*}.
The convergence in $L^2(dc)$ of the first part in \eqref{eq:Barefoot:4} follows since this term is not a principle value anymore, and we have that for any test function $\varphi \in L^2(dc)$, 
\begin{align*}
&\left| \langle {\mathcal M}_{\eps}^{(4,1)}[\psi](c) - {\mathcal M}_{0}^{(4,1)}[\psi](c) , \varphi(c) \rangle \right| \notag\\
&\lesssim \eps^{\frac{\zeta}{2}} \int_0^{u(0)} \!\!\! \int_0^{u(0)} \min\left\{u^{-1}(c)^{\frac{3\zeta}{2}},  u^{-1}(c)^{-\frac{3\zeta}{2}}\right\} \left|\frac{\sqrt{|u'(u^{-1}(y))|} \chi_c(u^{-1}(y),u^{-1}(c))}{\sqrt{|u'(u^{-1}(c))|}} - 1 \right| \frac{|\psi(y)| |\varphi(c)|}{|y-c|^{1+\frac{\zeta}{2}}}  dy
\notag\\
&\lesssim \eps^{\frac{\zeta}{2}} {\mathcal K} \int_0^{u(0)} \!\!\! \int_0^{u(0)}  \frac{|\psi(y)| |\varphi(c)|}{|y-c|^{1-\frac{\zeta}{4}}}  dy
\notag\\
&\lesssim \eps^{\frac{\zeta}{2}} {\mathcal K} \norm{\psi}_{L^2} \norm{\varphi}_{L^2} \to 0 \quad\mbox{as}\quad \eps \to 0,
\end{align*}
where 
\begin{align}
{\mathcal K} &= \sup_{y,c \in (0,u(0)]} \left| \min\left\{u^{-1}(c)^{\frac{3\zeta}{2}},  u^{-1}(c)^{-\frac{3\zeta}{2}}\right\} \frac{\frac{\sqrt{|u'(u^{-1}(y))|} \chi_c(u^{-1}(y),u^{-1}(c))}{\sqrt{|u'(u^{-1}(c))|}} - 1}{|y-c|^{\frac{3\zeta}{4}}} \right| \notag\\
&= \sup_{r,r_c \in [0,\infty)} \left| \frac{r_c^{\frac{3\zeta}{2}}}{\brak{r_c}^{3\zeta}}  \frac{\frac{\sqrt{|u'(r)|} \chi_c(r,r_c)}{\sqrt{|u'(r_c)|}} - 1}{|u(r)-u(r_c)|^{\frac{3\zeta}{4}}} \right|   \lesssim k^{\zeta}.  
\label{eq:Barefoot:6}
\end{align}
To see that \eqref{eq:Barefoot:6} holds,
first note that 
\begin{align*}
\frac{\indicator{|r-r_c|\geq \frac{1}{10}}}{|u(r)-u(r_c)|} \lesssim \frac{1}{|u'(r_c)|} \min\{ r_c , r_c^{-1} \} \lesssim \brak{r_c}^2, 
\end{align*}
and 
\begin{align*}
\frac{\indicator{|r-r_c|\leq \frac{1}{10}}}{|u(r)-u(r_c)|} \lesssim \frac{1}{|u'(r_c)|} \frac{1}{|r-r_c|} \lesssim \frac{\brak{r_c}^4}{r_c |r-r_c|}, 
\end{align*}
from which it follows that
\begin{align*}
\frac{1 - \chi_c(r,r_c)}{|u(r)-u(r_c)|^{\frac{3\zeta}{4}}} \lesssim \brak{r_c}^{\frac{3\zeta}{2}} + k^{\frac{3\zeta}{4}} \frac{\brak{r_c}^{3\zeta}}{r_c^{\frac{3\zeta}{2}}}.
\end{align*}
This proves \eqref{eq:Barefoot:6} for the region $\chi_c(r,r_c) = 0$.
Second, due to the regularity of $u'$, we have that 
\begin{align*}
\chi_c(r,r_c) \frac{\frac{\sqrt{|u'(r)|} }{\sqrt{|u'(r_c)|}} - 1}{|u(r)-u(r_c)|^{\frac{3\zeta}{4}}} 
\lesssim \frac{\chi_c(r,r_c)  |u'(r)-u'(r_c)|}{|u'(r_c)|^{\frac 12 + \frac{3\zeta}{4} }\left( \sqrt{|u'(r_c)|} + \sqrt{|u'(r)|} \right) } \left( \frac{\indicator{|r-r_c|\leq \frac{1}{10}}}{|r-r_c|^{\frac{3\zeta}{4}}} + \frac{\indicator{|r-r_c|\geq \frac{1}{10}} r_c^{\frac{3\zeta}{4}}}{\brak{r_c}^{\frac{3\zeta}{2}}}\right) \lesssim \frac{ \brak{r_c}^{3\zeta}}{ r_c^{ \frac{3\zeta}{2} }},  
\end{align*}
which proves \eqref{eq:Barefoot:6} for the region $\chi_c(r,r_c) \neq 0$.

To conclude, the same argument as above shows that 
\begin{align*}
\left| \langle {\mathcal M}_{0}^{(4,1)}[\psi] , \varphi \rangle \right| \lesssim {\mathcal K} \norm{\psi}_{L^2} \norm{\varphi}_{L^2}, 
\end{align*}
which establishes the boundedness of ${\mathcal M}_{0}^{(4,1)}$ on $L^2$, and thus \eqref{eq:Buffalo:4:*} holds.
This finishes the proof of the lemma.
\end{proof}

To conclude this section, we show that  the conditions of Lemma~\ref{lem:abstract:4:a} hold for the specific weights we need to considering. First, the next Corollary concludes the proof of Theorem~\ref{thm:2SIO:1DELTA}.

\begin{corollary}
\label{cor:abstract:4}
Let $0 \leq j \leq k-1$, and $0\leq \ell, \ell_1,\ell_2$ be such that $\ell + \ell_1 + \ell_2 \leq j$. Let $0 < \zeta \leq \frac{\delta}{6}$. Consider the weights
\begin{align*}
{\mathfrak B}_{\eps,1}(r,r_c,r_c) 
&=   \chi_1(r,r_c) \frac{\beta(r) \min(r^2,r^{-2})^j}{u'(r_c)  w_{F,\delta}(r)}  B_{\ell,\eps}^{(1)}(r,r_c,r_c) \\
{\mathfrak B}_{\eps,2}(r_c,s,r_c) 
&=   \frac{\beta(r_c) w_{F,\frac{\delta}{4} + 2\ell}(s)}{u'(s)}  B_{\ell,\eps}^{(2)}(r_c,s,r_c)
\end{align*} 
where $B_{\ell,\eps}^{(1)}$ is a suitable $(2\ell_1,\ell_1+\eta/2)$ kernel of type $I$ or $II$, and $B_{\ell,\eps}^{(2)}$ is a suitable $(2\ell_2,\ell_2+\eta/2)$ kernel of type $I$ or $II$.
Then, letting
\begin{align}
{\mathfrak m}(r_c) =  \frac{r_c^{\frac{ 3\zeta}{2} } }{\brak{r_c}^{ 3\zeta } w_{F,\frac{\delta}{2}}(r_c) (\max\{r_c^2, r_c^{-2}\})^{j-\ell_1-\eta/2}}, 
\label{eq:mathrak:m:chi:1}
\end{align}
the conditions \eqref{eq:BBB:cond:nasty:1}--\eqref{eq:BBB:cond:nasty:2} of Lemma~\ref{lem:abstract:4:a},  with ${\mathfrak M}$ replaced by $k^{2\ell_1+ 2\ell_2} {\mathfrak M}$, and where the limiting weights are obtained by passing $\eps\to 0$ in $B_{\ell,\eps}^{(1)}$ and $B_{\ell,\eps}^{(2)}$. Also, the operator norm of $L_0$ on $L^2(dr)$ is bounded as $k^{\zeta + 2\ell_1 + 2\ell_2 }$. 
\end{corollary}
\begin{proof}[Proof of Corollary~\ref{cor:abstract:4}]
In view of the established pointwise (and in a H\"older class near the diagonal) convergence properties of type I and type II kernels $B_{\ell,\eps}^{(1)} (r,r_c,r_c) \to B_{\ell,0}^{(1)}(r,r_c,r_c)$ and $B_{\ell,\eps}^{(2)}(r_c,s,r_c) \to B_{\ell,0}^{(2)}(r_c,s,r_c)$ as $\eps \to 0$ (cf.~Definition~\ref{def:SuitableType1} and Definition~\ref{def:SuitableType2}), it is clear that checking conditions \eqref{eq:BBB:cond:nasty:1:b} and \eqref{eq:BBB:cond:nasty:2:b} is sufficient. Indeed, checking conditions \eqref{eq:BBB:cond:nasty:1:a} and \eqref{eq:BBB:cond:nasty:1:c}, respectively \eqref{eq:BBB:cond:nasty:2:a} and \eqref{eq:BBB:cond:nasty:2:c}, follows mutatis mutandis. For simplicity we only treat the case $\eta=0$. The case $0 < \eta \ll 1$ is treated similarly.  First, we verify \eqref{eq:BBB:cond:nasty:1:b}. We recall that 
\begin{align*}
|B_{\ell,0}^{(1)}(r,r_c,r_c)| 
&\lesssim |u'(r_c)| {\mathcal B}(r,r_c) {\mathbb K}(r,r_c,r_c) {\mathcal L}_{2\ell_1,
\ell_1}(r,r_c)
\notag\\
&= |u'(r_c)|  \left( \indicator{r_c<r} \frac{r_c^{k-\frac 12}}{r^{k-\frac 12}} + \indicator{r<r_c} \frac{r^{k+\frac 12}}{r_c^{k+\frac 12}}  \right) \brak{r_c}^4 \left( \indicator{r_c>1} + \indicator{r \leq r_c \leq 1} + \frac{r_c^2}{r^2}\indicator{r_c<r<1} + r_c^2 \indicator{r_c<1<r} \right)
\notag\\
&\qquad \qquad \times k^{2\ell_1} \left(\max\left\{\frac{1}{r^2},r^2,\frac{1}{r_c^2},r_c^2 \right\} \right)^{\ell_1},
\end{align*}
from which it follows, using the definition of $w_{F,\delta}$ in \eqref{eq:weightF}, that 
\begin{align*}
&\frac{r_c \brak{r}^{4}}{\brak{r_c}^4 {\mathfrak m}(r_c) r} |{\mathfrak B}_{0,1}(r,r_c,r_c)|
\notag\\
&\lesssim ( \indicator{\frac{r_c}{2} \leq r \leq 1} + \indicator{r\geq 1} )\frac{r_c^{1 - \frac{3\zeta}{2} } \brak{r_c}^{3\zeta } w_{F,\frac{\delta}{2}}(r_c)}{  r \brak{r}^{4} w_{F,\delta}(r) }
 \left( \indicator{r_c<r} \frac{r_c^{k-\frac 12}}{r^{k-\frac 12}} + \indicator{r<r_c} \frac{r^{k+\frac 12}}{r_c^{k+\frac 12}}  \right)   \notag\\
&\qquad \qquad \times \left( \indicator{r_c>1} + \indicator{r \leq r_c \leq 1} + \frac{r_c^2}{r^2}\indicator{r_c<r<1} + r_c^2 \indicator{r_c<1<r} \right) 
k^{2\ell_1} \frac{\left(\max\left\{\frac{1}{r^2},r^2,\frac{1}{r_c^2},r_c^2 \right\} \right)^{\ell_1} \left(\max\left\{\frac{1}{r_c^2},r_c^2 \right\} \right)^{j-\ell_1}}{\left(\max\left\{r^2,\frac{1}{r^2}\right\} \right)^{j}}
\notag\\
&\lesssim k^{2\ell_1} \frac{r_c^{\frac{3\zeta}{2}}}{\brak{r_c}^{3\zeta}} \Big[  \indicator{\frac{r_c}{2}<r<r_c<1} + \indicator{\frac{r_c}{2} < r<1<r_c} + \indicator{1<r<r_c} \frac{r^{2k-2j+\frac 12 - \delta}}{r_c^{2k-2j+ \frac12 + 4- \frac{\delta}{2} - 3\zeta}}
\notag\\
&\qquad \qquad + \indicator{r_c<r<1} \frac{r_c^{2k-2j+ \frac12 +5- \frac{\delta}{2} - 3\zeta}}{r^{2k-2j+\frac 12 +5 - \delta}} + \indicator{r_c<1<r} r^{\frac 12 - \delta} r_c^{2k -2j + \frac 12 + 5 - \frac{\delta}{2} - 3\zeta} + \indicator{1<r_c<r} \frac{r^{\frac 12 - \delta}}{r_c^{\frac 12  + 4 - \frac{\delta}{2} - 3 \zeta}} \Big]
\notag\\
& \lesssim  k^{2\ell_1} \frac{r_c^{\frac{3\zeta}{2}}}{\brak{r_c}^{3\zeta}} \Big[ 
\indicator{r_c\leq r } \frac{r^{\frac {1-\zeta}{2}}}{r_c^{\frac{1-\zeta}{2}}} + \indicator{r<r_c} \frac{\brak{r}^{\frac{1+\zeta}{2}}}{\brak{r_c}^{\frac{1+\zeta}{2}}}
\Big] = k^{2\ell_1} {\mathfrak M}(r,r_c)
\end{align*}
in view of our assumption on $\zeta$. Here we have used that $j \leq k$. Thus, \eqref{eq:BBB:cond:nasty:1:b} holds with the above definition of ${\mathfrak B}_{\eps,1}(r,r_c,r_c)$, upon multiplying ${\mathfrak M}$ by $k^{2\ell_1}$. Next, we verify \eqref{eq:BBB:cond:nasty:2:b}. 
We recall that 
\begin{align*}
|B_{\ell,0}^{(2)}(r_c,s,r_c)| 
&\lesssim |u'(s)| {\mathcal B}(r_c,s) {\mathbb K}(r_c,s,r_c) {\mathcal L}_{2\ell_2,\ell_2}(r_c,s)
\notag\\
&= |u'(s)| \left( \indicator{s<r_c} \frac{s^{k-\frac 12}}{r_c^{k-\frac 12}} + \indicator{r_c<s} \frac{r_c^{k+\frac 12}}{s^{k+\frac 12}}  \right) \brak{s}^4 \left( \indicator{r_c>1} + \indicator{s \leq r_c \leq 1} + \frac{r_c^2}{s^2}\indicator{r_c<s<1} + r_c^2 \indicator{r_c<1<s} \right)
\notag\\
&\qquad \qquad \times k^{2\ell_2} \left(\max\left\{\frac{1}{s^2},s^2,\frac{1}{r_c^2},r_c^2 \right\} \right)^{\ell_2}
\end{align*}
from which it follows, that upon using the definition of $w_{F,\delta}$ in \eqref{eq:weightF} we arrive at
\begin{align*}
&{\mathfrak m}(r_c) |{\mathfrak B}_{0,2}(r_c,s,r_c) |  \notag\\
&\lesssim \left( \indicator{s<r_c} \frac{s^{k-\frac 12}}{r_c^{k-\frac 12}} + \indicator{r_c<s} \frac{r_c^{k+\frac 12}}{s^{k+\frac 12}}  \right) \frac{w_{F,\frac{\delta}{4}+2\ell}(s) \brak{s}^{4} r_c^{\frac{3\zeta}{2}}}{\brak{r_c}^{6+ 3\zeta} w_{F,\frac{\delta}{2}}(r_c) } \left( \indicator{r_c>1} + \indicator{s \leq r_c \leq 1} + \frac{r_c^2}{s^2}\indicator{r_c<s<1} + r_c^2 \indicator{r_c<1<s} \right)
\notag\\
&\qquad \qquad \times k^{2\ell_2} \frac{\left(\max\left\{\frac{1}{s^2},s^2,\frac{1}{r_c^2},r_c^2 \right\} \right)^{\ell_2}}{ \left(\max\left\{r_c^2, \frac{1}{r_c^{2}}\right\}\right)^{j-\ell_1}}
\notag\\
&\lesssim k^{2\ell_2} \frac{r_c^{\frac{3\zeta}{2}}}{\brak{r_c}^{3\zeta}}
\Big[ 
\indicator{s<r_c<1}   \frac{s^{2k-2\ell-2\ell_2+\frac 12 -\frac{\delta}{4}}}{r_c^{2k - 2j + 2\ell_1 + \frac 12- \frac{\delta}{2}}}   
+ \indicator{r_c<s<1}  \frac{s^{-2\ell + \frac 12-\frac{\delta}{4}}}{r_c^{-2j +2\ell_2 + 2\ell_1 + \frac 12 - \frac{\delta}{2}}}  
+ \indicator{s<1<r_c} \frac{s^{2k - 2\ell +3 - \frac 12-\frac{\delta}{4}} \max\{s^{-2\ell_2},r_c^{2\ell_2}\} }{r_c^{2j - 2\ell_1 +  \frac 12 + \frac{\delta}{2}}}\notag\\
&\qquad  
+ \indicator{r_c<1<s} \frac{\max\{r_c^{-2\ell_2},s^{2\ell_2} \}}{r_c^{-2j + 2\ell_1 + \frac 12 - \frac{\delta}{2}} s^{2k - 2\ell+ \frac 12 +1 - \frac{\delta}{2}}}
+ \indicator{1<s<r_c} \frac{1}{r_c^{2j - 2\ell_1 - 2\ell_2 + 1+ \frac{\delta}{2}} s^{-2\ell + \frac 32 - \frac{\delta}{4}} }    
+ \indicator{1<r_c<s} \frac{r_c^{2k - 2j + 2\ell_1 - \frac 12- \frac{\delta}{2}}}{s^{2k-2\ell-2\ell_2+\frac 32 - \frac{\delta}{4}}}     \Big]
\notag\\
&\lesssim k^{2\ell_2} \frac{r_c^{\frac{3\zeta}{2}}}{\brak{r_c}^{3\zeta}} \left[ \indicator{r_c\leq s } \frac{s^{\frac {1-\zeta}{2}}}{r_c^{\frac{1-\zeta}{2}}} + \indicator{s<r_c} \frac{\brak{s}^{\frac{1+\zeta}{2}}}{\brak{r_c}^{\frac{1+\zeta}{2}}} \right] = k^{2\ell_2} {\mathfrak M}(s,r_c).
\end{align*}
Here we used that $\ell + \ell_1 + \ell_2 \leq j \leq k-1$.
Therefore,  \eqref{eq:BBB:cond:nasty:2:b} holds with ${\mathfrak M}$ multiplied by $k^{2\ell_2}$.
\end{proof}
Similarly, the next corollary concludes the proof of Theorem~\ref{thm:2SIO:1DELTA:a}.
\begin{corollary}
\label{cor:abstract:5} 
Let $0 \leq j \leq k-1$, and $0\leq \ell, \ell_1,\ell_2$ be such that $\ell + \ell_1 + \ell_2 \leq j$. Let $0 < \zeta \leq \frac{\delta}{6}$. Consider the weights
\begin{align*}
{\mathfrak B}_{\eps,1}(r,r_c,r_c) 
&=  \indicator{r\leq 1} \indicator{2r\leq r_c} \frac{\beta(r) \min(r^2,r^{-2})^j}{  r^{\frac 12} w_{f,\delta}(r) u'(r_c) r_c u'(r_c)}  B_{\ell,\eps}^{(1)}(r,r_c,r_c)  \\
{\mathfrak B}_{\eps,2}(r_c,s,r_c) 
&=   \frac{\beta(r_c) w_{F,\frac{\delta}{4} + 2\ell}(s)}{u'(s)}  B_{\ell,\eps}^{(2)}(r_c,s,r_c)
\end{align*} 
where $B_{\ell,\eps}^{(1)}$ is a suitable $(2\ell_1,\ell_1+\eta/2)$ kernel of type $I$ or $II$, and $B_{\ell,\eps}^{(2)}$ is a suitable $(2\ell_2,\ell_2+\eta/2)$ kernel of type $I$ or $II$.
Then,  letting
\begin{align}
{\mathfrak m}(r_c) =  \frac{r_c^{\frac{ 3\zeta}{2}} \brak{r_c}^2}{\brak{r_c}^{ 3\zeta} w_{F,\frac{\delta}{2}}(r_c) (\max\{r_c^2, r_c^{-2}\})^{j-\ell_1}} 
\label{eq:mathrak:m:chi:2}
\end{align}
the conditions \eqref{eq:BBB:cond:nasty:1}--\eqref{eq:BBB:cond:nasty:2} of  Lemma~\ref{lem:abstract:4:a} are satisfied, with ${\mathfrak M}$ replaced by $k^{ 2\ell_1 + 2\ell_2 } {\mathfrak M}$, and where the limiting weights are obtained by passing $\eps\to 0$ in $B_{\ell,\eps}^{(1)}$ and $B_{\ell,\eps}^{(2)}$. In particular, the norm on $L^2(dr)$ of the operator $L_0$ defined in \eqref{eq:L:0:def} is bounded as $k^{\zeta + 2\ell_1+ 2\ell_2 }$. 
\end{corollary}
\begin{proof}[Proof of Corollary~\ref{cor:abstract:5}]
The proof is essentially the same as the proof of Corollary~\ref{cor:abstract:4}, hence, we only emphasize the requisite modifications. We notice that the choice of ${\mathfrak m}(r_c)$ in \eqref{eq:mathrak:m:chi:2} is different from the one in \eqref{eq:mathrak:m:chi:1}, as we have multiplied by a factor of $\brak{r_c}^2$. Thus, when checking conditions \eqref{eq:BBB:cond:nasty:2:a}--\eqref{eq:BBB:cond:nasty:2:c} for the weight ${\mathfrak B}_{\eps,2}(r_c,s,r_c) $ we need to obtain bounds which are better by an $\brak{r_c}^2$. This however is automatic from the decay of $\beta(r_c)$ present in the definition of ${\mathfrak B}_{\eps,2}(r_c,s,r_c)$. In turn, when checking conditions \eqref{eq:BBB:cond:nasty:1:a}--\eqref{eq:BBB:cond:nasty:1:c} for ${\mathfrak B}_{\eps,1}(r,r_c,r_c)$, we can afford estimates which are  worse than those in Corollary~\ref{cor:abstract:4} by a factor of $\brak{r_c}^2$. This is natural, since recall that the ${\mathfrak B}_{\eps,1}(r,r_c,r_c) $ in this corollary, differs from the one in Corollary~\ref{cor:abstract:4} by a factor of $\frac{\brak{r_c}^4 r^2}{r_c^2}$, for $r\leq 1$ and $2r \leq r_c$. The worsening of the estimate by a factor of $\brak{r_c}^2$ is thus compensating this extra factor when  $r_c\geq 1$. On the other hand, in the  case $r_c\leq 1$, we have $\frac{\brak{r_c}^4 r^2}{r_c^2} \leq \frac{r^2}{r_c^2} \leq \frac 14$, and thus there is nothing additional to prove. With these changes in hand, we may now follow line by the proof of Corollary~\ref{cor:abstract:4}, and conclude the proof.
\end{proof}

\subsubsection{A useful product formula for weights}
\label{sec:product:for:weights}
\begin{lemma}
\label{lem:weights:product:1}
With the notation of \eqref{eq:script:B:def}, define
\begin{align}
{\mathbb W}(r,s,s_0)&:= 
\left[\frac{\indicator{r\leq 1}  }{r^{k+\frac 12}} +  \indicator{r\geq 1} r^{k- \frac 12} \right] 
\left[\indicator{s\leq1} s^{k + \frac 12 }+ \frac{\indicator{s\geq 1}}{s^{k - \frac 12}} \right] 
\frac{{\mathcal B}(r,s_0)}{\brak{s_0}^4} \frac{{\mathcal B}(s_0,s)}{\brak{s}^4}
\notag\\
&=
\left[\frac{\indicator{r\leq 1}  }{r^{k+\frac 12}} +  \indicator{r\geq 1} r^{k- \frac 12} \right] 
\left[\indicator{s\leq1} s^{k + \frac 12 }+ \frac{\indicator{s\geq 1}}{s^{k - \frac 12}} \right] \notag\\
&\qquad \qquad \times 
\left[ \indicator{s_0<r} \frac{s_0^{k-\frac 12}}{r^{k-\frac 12}} + \indicator{s_0>r} \frac{r^{k+\frac 12}}{s_0^{k+\frac 12}}  \right]   
\left[ \indicator{s<s_0} \frac{s^{k-\frac 12}}{s_0^{k-\frac 12}} + \indicator{s>s_0} \frac{s_0^{k+\frac 12}}{s^{k+\frac 12}}  \right].
\label{eq:W:Merlot:def}
\end{align}  
Then we have that
\begin{align*}
{\mathbb W}(r,s,s_0) \indicator{r<s<s_0} 
&=  
\indicator{1<r<s<s_0} \frac{  r^{2k}}{s_0^{2k }} 
+  \indicator{r<1<s<s_0} \frac{1}{ s_0^{2k}} 
+  \indicator{r<s<1<s_0} \frac{s^{2k}}{ s_0^{2k}}  
+  \indicator{r<s<s_0<1} \frac{s^{2k}}{s_0^{2k}}  
\\
{\mathbb W}(r,s,s_0) \indicator{s<r<s_0} 
&= 
\indicator{1<s<r<s_0} \frac{ r^{2k} }{s_0^{2k}} 
+ \indicator{s<1<r<s_0} \frac{r^{2k} s^{2k} }{ s_0^{2k}}  
+ \indicator{s<r<1<s_0} \frac{s^{2k} }{ s_0^{2k}}  
+ \indicator{s<r<s_0<1} \frac{s^{2k} }{ s_0^{2k}}  
\\
{\mathbb W}(r,s,s_0)  \indicator{r<s_0<s} 
&=  
\indicator{1<r<s_0<s} \frac{r^{2k}}{ s^{2k}} 
+ \indicator{r<1<s_0<s} \frac{1}{ s^{2 k }} 
+ \indicator{r<s_0<1<s} \frac{1}{s^{2 k}}  
+ \indicator{r<s_0<s<1}   
\\
{\mathbb W}(r,s,s_0) \indicator{s_0<r<s} 
&= 
\indicator{1<s_0<r<s} \frac{s_0^{2k }}{ s^{2k}} 
+  \indicator{s_0<1<r<s} \frac{s_0^{2k}}{s^{2 k}}
+ \indicator{s_0<r<1<s} \frac{s_0^{2k }}{r^{2k}  s^{2k}}  
+ \indicator{s_0<r<s<1} \frac{s_0^{2k}}{r^{2k}}  
\\
{\mathbb W}(r,s,s_0)  \indicator{s<s_0<r} 
&= 
\indicator{1<s<s_0<r}    
+ \indicator{s<1<s_0<r}   s^{2k}   
+ \indicator{s<s_0<1<r}   s^{2k} 
+ \indicator{s<s_0<r<1} \frac{s^{2k}  }{r^{2k} }  
\\
{\mathbb W}(r,s,s_0) \indicator{s_0<s<r} 
&= 
\indicator{1<s_0<s<r} \frac{ s_0^{2k }}{ s^{2k}} 
+ \indicator{s_0<1<s<r} \frac{ s_0^{2k}}{ s^{2k}} 
+ \indicator{s_0<s<1<r}  s_0^{2k} 
+ \indicator{s_0<s<r<1}\frac{ s_0^{2k}}{r^{2k}}   
\end{align*}
holds. In particular, note that 
\[
{\mathbb W}(r,s,s_0) \leq 1.
\]
\end{lemma}
\begin{proof}[Proof of Lemma~\ref{lem:weights:product:1}]
 The proof follows by inspection of each of the $24$ possible permutations of $\{ r,s,s_0,1\}$.
\end{proof}
\begin{lemma}
\label{lem:weights:product:2} 
Let $0 \leq j \leq k$, and $\ell + \ell_1 + \ell_2 \leq j$. Define
\begin{align*}
{\mathbb L}(r,s,s_0)&:= \left[\indicator{r\leq 1}  r^{2j} + \frac{\indicator{r\geq 1}}{r^{2j}} \right] \left[ \frac{\indicator{s\leq1}} {s^{ 2\ell}}+  \indicator{s\geq 1} s^{2\ell} \right] {\mathcal L}_{2\ell_1,\ell_1}(r,s_0) {\mathcal L}_{2\ell_2,\ell_2}(s_0,s).
\end{align*}
Then, we have that 
\begin{align*}
{\mathbb W}(r,s,s_0) {\mathbb L}(r,s,s_0) \lesssim 1.
\end{align*}
\end{lemma}
\begin{proof}[Proof of Lemma~\ref{lem:weights:product:2}]
Notice that for $j=0$, we must have $\ell = \ell_1 = \ell_2 =0$, and thus ${\mathbb L}= 1$. In this case the proof directly follows from Lemma~\ref{lem:weights:product:1}.

For $j\geq 1$, we use the precise formula for ${\mathbb W}$ in Lemma~\ref{lem:weights:product:1}, and use the definition of ${\mathbb L}$, which is also a function of $r,s$, and $s_0$, to obtain that 
\begin{align*}
\sqrt{{\mathbb W}(r,s,s_0) {\mathbb L}(r,s,s_0) \indicator{r<s<s_0}}
&=  
\indicator{1<r<s<s_0} \frac{r^{ k-j} s^{\ell}}{s_0^{ k -\ell_1 -\ell_2}} 
+  \indicator{r<1<s<s_0} \frac{s^{\ell} r^{j}}{ s_0^{k-\ell_2}} \left(\frac{1}{r^{\ell_1}} + s_0^{\ell_1} \right) 
\notag\\
&\qquad +  \indicator{r<s<1<s_0} \frac{r^j s^{k+\ell}}{ s_0^{ k}}   \left(\frac{1}{r^{\ell_1}} + s_0^{\ell_1} \right)  \left(\frac{1}{s^{\ell_2}} + s_0^{\ell_2} \right) 
+  \indicator{r<s<s_0<1} \frac{r^{j-\ell_1} s^{ k -\ell - \ell_2}}{s_0^{ k}}  
\\
\sqrt{{\mathbb W}(r,s,s_0) {\mathbb L}(r,s,s_0) \indicator{s<r<s_0}}
&= 
\indicator{1<s<r<s_0} \frac{ r^{k-j} s^{\ell}}{s_0^{ k-\ell_1-\ell_2}} 
+ \indicator{s<1<r<s_0} \frac{r^{ k-j} s^{ k-\ell} }{ s_0^{ k-\ell_1}}   \left(\frac{1}{s^{\ell_2}} + s_0^{\ell_2} \right) 
\notag\\
&\qquad 
+ \indicator{s<r<1<s_0} \frac{r^j s^{ k-\ell} }{ s_0^{ k}} \left(\frac{1}{r^{\ell_1}} + s_0^{\ell_1} \right)  \left(\frac{1}{s^{\ell_2}} + s_0^{\ell_2} \right) 
+ \indicator{s<r<s_0<1} \frac{r^j s^{ k-\ell} }{ s_0^{ k}}  
\\
\sqrt{{\mathbb W}(r,s,s_0) {\mathbb L}(r,s,s_0)  \indicator{r<s_0<s}}
&=  
\indicator{1<r<s_0<s} \frac{r^{k-j} s_0^{\ell_1}}{ s^{ k-\ell-\ell_2}} 
+ \indicator{r<1<s_0<s} \frac{r^j}{ s^{ k-\ell-\ell_2 }} \left( \frac{1}{r^{\ell_1}} + s_0^{\ell_1} \right)
\notag\\
&\qquad 
+ \indicator{r<s_0<1<s} \frac{r^{j-\ell_1}}{s^{ k-\ell}}   \left(\frac{1}{s_0^{\ell_2}} + s^{\ell_2} \right)
+ \indicator{r<s_0<s<1} \frac{r^{j-\ell_1}}{s^{\ell} s_0^{\ell_2}}  
\\
\sqrt{{\mathbb W}(r,s,s_0) {\mathbb L}(r,s,s_0) \indicator{s_0<r<s}}
&= 
\indicator{1<s_0<r<s} \frac{s_0^{ k }}{r^{j-\ell_1} s^{ k-\ell-\ell_2}} 
+  \indicator{s_0<1<r<s} \frac{s_0^{ k}}{r^j s^{ k-\ell}} \left(\frac{1}{s_0^{\ell_1}} + r^{\ell_1} \right)  \left(\frac{1}{s_0^{\ell_2}} + s^{\ell_2} \right)
\notag\\
&\qquad 
+ \indicator{s_0<r<1<s}  \frac{s_0^{ k-\ell_1 }}{r^{ k-j}  s^{ k-\ell}}  \left(\frac{1}{s_0^{\ell_2}} + s^{\ell_2} \right)
+ \indicator{s_0<r<s<1} \frac{s_0^{ k-\ell_1-\ell_2}}{s^{\ell} r^{ k-j}}  
\\
\sqrt{{\mathbb W}(r,s,s_0) {\mathbb L}(r,s,s_0)  \indicator{s<s_0<r}}
&= 
\indicator{1<s<s_0<r}  \frac{s^{\ell} s_0^{\ell_2}}{r^{j - \ell_1}}  
+ \indicator{s<1<s_0<r}   \frac{s^{ k-\ell} }{r^{j-\ell_1}}  \left(\frac{1}{s^{\ell_2}} + s_0^{\ell_2} \right)
\notag\\
&\qquad 
+ \indicator{s<s_0<1<r} \frac{s^{ k-\ell-\ell_2}}{r^{j}} \left(\frac{1}{s_0^{\ell_1}} + r^{\ell_1} \right)  
+ \indicator{s<s_0<r<1} \frac{s^{ k-\ell-\ell_2}  }{r^{ k-j} s_0^{\ell_1}}  
\\
\sqrt{{\mathbb W}(r,s,s_0) {\mathbb L}(r,s,s_0) \indicator{s_0<s<r} }
&= 
\indicator{1<s_0<s<r} \frac{ s_0^{ k }}{r^{j-\ell_1} s^{ k-\ell-\ell_2}} 
+ \indicator{s_0<1<s<r} \frac{ s_0^{ k}}{r^j s^{ k-\ell}} \left(\frac{1}{s_0^{\ell_1}} + r^{\ell_1} \right)  \left(\frac{1}{s_0^{\ell_2}} + s^{\ell_2} \right)
\notag\\
&\qquad 
+ \indicator{s_0<s<1<r}  \frac{s_0^{ k-\ell_2} }{r^j s^{\ell}} \left(\frac{1}{s_0^{\ell_1}} + r^{\ell_1} \right)  
+ \indicator{s_0<s<r<1}\frac{ s_0^{ k-\ell_1}}{r^{ k-j} s^{\ell +\ell_2}}   
\end{align*}
Inspecting each of these $24$ terms, by using the constraint $0 \leq \ell + \ell_1 + \ell_2 \leq j \leq k-1$ the proof of the lemma follows.
\end{proof}

\subsection{The remaining combinations of iterated operators}
Similar to the results in the previous section, namely Theorems~\ref{thm:BBB:main},~\ref{thm:2SIO:1DELTA}, and~\ref{thm:2SIO:1DELTA:a}, one can prove a number of results for passing $\eps\to0$ in other combinations of three integral operators. The aforementioned theorems deal with the most difficult case, of an approximate delta function combined with two approximate singular integrals. The remaining operators all have at least one approximate delta function, and at most one approximate singular integral, which are hence easier to treat. 
\begin{theorem}
\label{thm:other:integral:operators}
Let $0 \leq j \leq k-1$, and $0\leq \ell, \ell_1,\ell_2$ be such that $\ell + \ell_1 + \ell_2 \leq j$. Let $B_{\ell,\eps}^{(1)}$ be a suitable $(2\ell_1,\ell_1+\eta/2)$ kernel of type $I$ or $II$, and $B_{\ell,\eps}^{(2)}$ is a suitable $(2\ell_2,\ell_2+\eta/2)$ kernel of type $I$ or $II$. Let $0 < \zeta \leq \frac{\delta}{6}$, and consider the pairs of weights given in \eqref{eq:B1:2SIO:1DELTA}--\eqref{eq:B2:2SIO:1DELTA} or \eqref{eq:B1:2SIO:1DELTA:a}--\eqref{eq:B2:2SIO:1DELTA:a}. That is, either consider the pair
\begin{align*}
{\mathfrak B}_{\eps,1}(r,s_0,r_c) 
&=  \chi_1(r,r_c) \frac{\beta(r)\min(r^2,r^{-2})^j}{w_{F,\delta}(r) u'(s_0)}  B_{\ell,\eps}^{(1)}(r,s_0,r_c) \\
{\mathfrak B}_{\eps,2}(s_0,s,r_c) 
&=  \frac{\beta(s_0) w_{F,\frac{\delta}{4} + 2\ell}(s)}{u'(s)}   B_{\ell,\eps}^{(2)}(s_0,s,r_c) \, ,
\end{align*} 
or consider the pair
\begin{align*}
{\mathfrak B}_{\eps,1}(r,s_0,r_c) 
&=  \chi_2(r,r_c) \frac{\beta(r)\min(r^2,r^{-2})^j}{r^{\frac 12} w_{f,\delta}(r) u'(s_0) r_c u'(r_c)}  B_{\ell,\eps}^{(1)}(r,s_0,r_c) 
\\
{\mathfrak B}_{\eps,2}(s_0,s,r_c) 
&=  \frac{\beta(s_0) w_{F,\frac{\delta}{4} + 2\ell}(s)}{u'(s)}   B_{\ell,\eps}^{(2)}(s_0,s,r_c)\, .
\end{align*} 
For each such pair of weights, we have that the following limits hold, in the sense of bounded operators on $L^2(dr)$: 
\begin{align*}
&\int_{0}^\infty \!\!\! \int_{0}^\infty \!\!\! \int_{0}^\infty \frac{(u(r)-u(r_c)) u'(r_c)}{(u(r)-u(r_c))^2+\eps^2}  \frac{\eps u'(s)}{(u(s)-u(r_c))^2 + \eps^2}  \frac{\eps u'(s_0)}{(u(s_0)-u(r_c))^2 +\eps^2} \notag\\
&\qquad \qquad \qquad \times {\mathfrak B}_{\eps,1}(r,s_0,r_c) {\mathfrak B}_{\eps,2}(s_0,s,r_c) f(s) \, \dd s  \dd s_0 \dd r_c 
\notag\\
&\quad \longrightarrow \pi^2 p.v. \int_{0}^\infty  \frac{ u'(r_c)}{ u(r)-u(r_c) }  {\mathfrak B}_{0,1}(r,r_c,r_c) {\mathfrak B}_{0,2}(r_c,r_c,r_c) f(r_c) \,  \dd r_c 
\\
&\int_{0}^\infty \!\!\! \int_{0}^\infty \!\!\! \int_{0}^\infty \frac{\eps u'(r_c)}{(u(r)-u(r_c))^2+\eps^2}  \frac{(u(s)-u(r_c)) u'(s)}{(u(s)-u(r_c))^2 + \eps^2}  \frac{\eps u'(s_0)}{(u(s_0)-u(r_c))^2 +\eps^2} \notag\\
&\qquad \qquad \qquad \times {\mathfrak B}_{0,1}(r,s_0,r_c) {\mathfrak B}_{0,2}(s_0,s,r_c) f(s) \, \dd s  \dd s_0 \dd r_c 
\notag\\
&\quad \longrightarrow \pi^2 p.v. \int_{0}^\infty  \frac{ u'(s)}{ u(s)-u(r) }  {\mathfrak B}_{0,1}(r,r,r) {\mathfrak B}_{0,2}(r,s,r) f(s) \,  \dd s
\\
&\int_{0}^\infty \!\!\! \int_{0}^\infty \!\!\! \int_{0}^\infty \frac{\eps u'(r_c)}{(u(r)-u(r_c))^2+\eps^2}  \frac{\eps u'(s)}{(u(s)-u(r_c))^2 + \eps^2}  \frac{\eps u'(s_0)}{(u(s_0)-u(r_c))^2 +\eps^2} \notag\\
&\qquad \qquad \qquad \times {\mathfrak B}_{\eps,1}(r,s_0,r_c) {\mathfrak B}_{\eps,2}(s_0,s,r_c) f(s) \, \dd s  \dd s_0 \dd r_c 
\notag\\
&\quad \longrightarrow -\pi^3 {\mathfrak B}_{0,1}(r,r,r) {\mathfrak B}_{0,2}(r,r,r) f(r)\\
&\int_{0}^\infty \!\!\! \int_{0}^\infty \!\!\! \int_{0}^\infty \frac{(u(r)-u(r_c)) u'(r_c)}{(u(r)-u(r_c))^2+\eps^2} u'(s)  \frac{\eps u'(s_0)}{(u(s_0)-u(r_c))^2 +\eps^2} \notag\\
&\qquad \qquad \qquad \times {\mathfrak B}_{\eps,1}(r,s_0,r_c) {\mathfrak B}_{\eps,2}(s_0,s,r_c) f(s) \, \dd s  \dd s_0 \dd r_c 
\notag\\
&\quad \longrightarrow -\pi  p.v. \int_{0}^\infty \!\!\! \int_{0}^\infty  \frac{ u'(r_c)}{ u(r)-u(r_c) } u'(s) {\mathfrak B}_{0,1}(r,r_c,r_c) {\mathfrak B}_{0,2}(r_c,s,r_c) f(s) \,  \dd s \dd r_c 
\\
&\int_{0}^\infty \!\!\! \int_{0}^\infty \!\!\! \int_{0}^\infty \frac{\eps u'(r_c)}{(u(r)-u(r_c))^2+\eps^2} u'(s)  \frac{\eps u'(s_0)}{(u(s_0)-u(r_c))^2 +\eps^2} \notag\\
&\qquad \qquad \qquad \times {\mathfrak B}_{\eps,1}(r,s_0,r_c) {\mathfrak B}_{\eps,2}(s_0,s,r_c) f(s) \, \dd s  \dd s_0 \dd r_c 
\notag\\
&\quad \longrightarrow \pi^2 \int_{0}^\infty u'(s) {\mathfrak B}_{0,1}(r,r,r) {\mathfrak B}_{0,2}(r,s,r) f(s) \,  \dd s \\ 
&\int_{0}^\infty \!\!\! \int_{0}^\infty \frac{(u(r) - u(r_c) u'(r_c)}{(u(r)-u(r_c))^2+\eps^2} \frac{\eps u'(s)}{(u(s)-u(r_c))^2 +\eps^2} 
 {\mathfrak B}_{\eps,1}(r,s,r_c) f(s) \, \dd s  \dd r_c 
\notag\\
&\quad \longrightarrow -\pi p.v. \int_{0}^\infty \frac{u'(r_c)}{u(r) - u(r_c)} {\mathfrak B}_{0,1}(r,r_c,r_c) f(r_c) \,  \dd r_c \\
&\int_{0}^\infty \!\!\! \int_{0}^\infty \frac{\eps u'(r_c)}{(u(r)-u(r_c))^2+\eps^2} \frac{\eps u'(s)}{(u(s)-u(r_c))^2 +\eps^2} {\mathfrak B}_{\eps,1}(r,s,r_c) f(s) \, \dd s \dd r_c  \notag\\
&\quad \longrightarrow \pi^2  {\mathfrak B}_{0,1}(r,r,r) f(r) \,  \dd r_c \\ 
& \int_{0}^\infty \!\!\! \int_{0}^\infty \frac{\eps u'(r_c)}{(u(r)-u(r_c))^2+\eps^2} \frac{(u(s) - u(r_c) u'(s)}{(u(s)-u(r_c))^2 +\eps^2} {\mathfrak B}_{\eps,1}(r,s,r_c) f(s) \, \dd s  \dd s_0 \dd r_c 
\notag\\
&\quad \longrightarrow -\pi p.v. \int_0^\infty \frac{u'(s)}{u(s) - u(r)}  {\mathfrak B}_{0,1}(r,s,r) f(s) \,  \dd s \\
& \int_{0}^\infty \!\!\! \int_{0}^\infty \frac{\eps u'(r_c)}{(u(r)-u(r_c))^2+\eps^2} u'(s) {\mathfrak B}_{\eps,1}(r,s,r_c) f(s) \, \dd s  \dd s_0 \dd r_c 
\notag\\
&\quad \longrightarrow -\pi \int_0^\infty u'(s)  {\mathfrak B}_{0,1}(r,s,r) f(s) \,  \dd s. 
\end{align*}
Moreover, in each case the limiting operators are bounded on $L^2(dr)$  with norm $\lesssim k^{\zeta +  2\ell_1 + 2\ell_2 }$.
\end{theorem}

Proving each statement in Theorem~\ref{thm:other:integral:operators} amounts to following step-by-step the proof of Theorem~\ref{thm:BBB:main}~\ref{thm:2SIO:1DELTA}, and~\ref{thm:2SIO:1DELTA:a}. For every approximate delta function, first show that the contribution away from the diagonal vanishes, leaving one with the contribution close to the diagonal, in which one may pass to the limit because we have H\"older regularity of our weights in all variables not called $r_c$. Checking that the weights we consider are suitable for this procedure, i.e. that they give a bounded operator in the correctly weighted $L^2$ spaces,  was already done in the proof of Theorems~\ref{thm:2SIO:1DELTA} and~\ref{thm:2SIO:1DELTA:a}. This leaves us with the contribution from the joint diagonal, in which the remaining operators are either approximate singular integrals (note however that at most one singular integral may be present), or simple Hilbert-Schmidt kernels. For the approximate singular integral we again employ Lemma~\ref{lem:abstract:4:a} and Corollaries~\ref{cor:abstract:4} and~\ref{cor:abstract:5}, depending on which weight we choose. For passing to the limit in the operators with Hilbert-Schmidt kernels, we simply subtract the limiting kernel, so that the difference is again Hilbert-Schmidt but of size $\eps^\zeta$ for some $\zeta>0$. Here we lose lose a power of $k^\zeta / r_c^\zeta$, but since $\zeta>0$ is arbitrary, we may absorb this loss into the weights. This procedure again essentially uses the regularity of our weights in all variables not called $r_c$. The only non-trivial case left to discuss is the case in which we have an apparent approximate delta function in the $r_c$ variable. We treat this operator by duality, so that the approximate delta function in $r_c$, becomes an approximate delta function in the $r$ variable, acting on the $L^2$ test function. We pass to the limit in the dual pairing and thus obtain the formula and the boundedness of the limiting operator. Having available the estimates in the previous section, which were used to Theorems~\ref{thm:BBB:main},~\ref{thm:2SIO:1DELTA}, and~\ref{thm:2SIO:1DELTA:a}, implementing the above described program is tedious, but routine. In order to avoid this unnecessary redundancy, we omit these details.

\section{Properties of $H_0$ and $H_\infty$} \label{sec:H0HinfBds}

The next lemma concerns the asymptotic analysis of $R_{0,r}(z)$ and $R_{r,\infty}(z)$, which both arise in the Green's function.
\begin{lemma} \label{lem:asympsRE}
For $z = c\pm i \eps \in I_\alpha$ with $\eps$ sufficiently small, for $\abs{r-r_c} \lesssim r_c/k$ there holds 
\begin{subequations} \label{ineq:R0r}
\begin{align}
\abs{R_{0,r}^\eps(z)} \mathbf{1}_{\abs{r-r_c} \leq r_c/k} & \lesssim \mathbf{1}_{r_c \leq 1} \max(r_c^{-3},r_c^{5}) \left(k + \abs{\log k\frac{\abs{r-r_c}}{r_c}} \right) \\ 
\abs{R_{0,r}^\eps(z)}\mathbf{1}_{\abs{r-r_c} \geq r_c/k} & \lesssim k\mathbf{1}_{r_c \leq 1} \left( \mathbf{1}_{r < r_c} \frac{r^{2k-2}}{r_c^{2k+1}} + \mathbf{1}_{r_c < r} \frac{1}{r_c^3} \right) \\ 
& \quad + \mathbf{1}_{r_c \geq 1} \left( \mathbf{1}_{r < 1} \frac{r^{2k-2}}{r_c^{2k-1}} + \mathbf{1}_{1 \leq r < r_c} \frac{k r^{2k+4}}{r_c^{2k-1}} + \mathbf{1}_{r > r_c} kr_c^5\right),  
\end{align}
\end{subequations} 
and similarly
\begin{subequations} \label{ineq:Rrinf}
\begin{align}
\abs{R_{r,\infty}^\eps(z)} \mathbf{1}_{\abs{r-r_c} \lesssim r_c/k} & \lesssim \mathbf{1}_{r_c \leq 1} \max(r_c^{-3},r_c^{5}) \left(k + \abs{\log k\frac{\abs{r-r_c}}{r_c}} \right) \\ 
\abs{R_{r,\infty}^\eps(z)} \mathbf{1}_{\abs{r-r_c} \gtrsim r_c/k} & \lesssim \mathbf{1}_{r_c \leq 1} \left( \mathbf{1}_{r < r_c}\frac{k}{r_c^3} + \mathbf{1}_{r_c < r \leq 1}\frac{k r_c^{1+2k}}{r^{4+2k}} + \mathbf{1}_{r > 1} \frac{r_c^{2k+1}}{r^{2k-2}} \right) \\ 
& \quad + k\mathbf{1}_{r_c \geq 1}\left(\mathbf{1}_{r < r_c} r_c^5 + \mathbf{1}_{r > r_c}\frac{r_c^{2k+3}}{r^{2k-2}}\right).  
\end{align}
\end{subequations} 
Moreover, for all $\eta$ sufficiently small and $z \in I_\alpha$, we have that $\eps^{-\eta}\left(R_{0,r}^\eps(c\pm i \eps) - R_{0,r}(c)\right)$ also satisfies \eqref{ineq:R0r} 
and $\eps^{-\eta}\left(R_{r,\infty}^\eps(c\pm i \eps) - R_{r,\infty}(c)\right)$ also satisfies \eqref{ineq:Rrinf}. 
\end{lemma}
\begin{proof}[Proof of Lemma~\ref{lem:asympsRE}]
The proof is a straightforward variant of arguments applied in the proof of Lemma \ref{lem:Wronk} and is hence omitted for the sake of brevity. 
\end{proof} 

\begin{proof}[Proof of Lemma~\ref{lem:Hests}]
Recall \eqref{eq:H0HinfExp}, which will be of use in the critical layer region $\abs{r-r_c} < r_c/k$.
From Theorem \ref{thm:nonvan}, followed by Lemma \ref{lem:Trivrrc}, there holds for $\abs{r-r_c} < r_c/k$:  
\begin{align*}
\frac{1}{u'(r) P(r,c\pm i \eps)} \approx \frac{1}{u'(r_c)}\min\left(\frac{r^{k-1/2}}{r_c^{k-1/2}},\frac{r_c^{k+1/2}}{r^{k+1/2}}\right) \approx \max(r_c^{-1},r_c^3). 
\end{align*}
Consider next \eqref{ineq:phiRE}. 
From Lemma \ref{lem:asympsRE} and Lemma \ref{lem:Trivrrc} we have for $\abs{r-r_c} < r_c/k$ (for $z \in I_\alpha$ and $\eps$ sufficiently small), 
\begin{align*}
\abs{\phi(r,c \pm i \eps)\left(R^\eps_{0,r}(c \pm i \eps) \mp i E^\eps_{0,r}(c \pm i \eps)\right)} &  \lesssim \mathbf{1}_{r_c \leq 1} r_c \abs{r-r_c} r_c^{-3} \left(k + \abs{\log k \abs{\frac{r-r_c}{r_c}}} \right) \\ 
& \quad + \mathbf{1}_{r_c \geq 1} \frac{1}{r_c^3} \abs{r-r_c} r_c^5  \left(k +  \abs{\log k\abs{\frac{r-r_c}{r_c}}} \right),  
\end{align*}
which is the desired estimate of $H_0$; the treatment of $H_\infty$ is the same. 

Turn next to the estimates \eqref{ineq:H0bd} and \eqref{ineq:Hinfbd}. 
Consider just the estimates in \eqref{ineq:H0bd}; the estimate \eqref{ineq:Hinfbd} is analogous. 
For $r < r_c\left(1 - \frac{1}{k}\right)$,  (from \eqref{def:H0Hin}), 
\begin{align*}
\abs{H_0(r,z)}  & \leq \abs{\phi(r,z)} \int_0^r \frac{1}{\abs{\phi^2(s,z)}} \dd s \lesssim \abs{u-c} \frac{r_c^{k-1/2}}{r^{k-1/2}} \int_0^r \frac{1}{\abs{u-c}^2} \frac{s^{2k-1}}{r_c^{2k-1}} \dd s. 
\end{align*}
Consider the case $r_c \leq 1$. Then from \eqref{ineq:uNCL}, 
\begin{align*}
 \mathbf{1}_{r < r_c/2} \abs{H_0(r,z)} & \lesssim \frac{r_c^{k-1/2}}{r_c^2 r^{k-1/2}} \int_0^r \frac{s^{2k-1}}{r_c^{2k-1}} \dd s \approx \frac{1}{k} \frac{r^{k+1/2}}{r_c^{k+3/2}}, \\ 
\mathbf{1}_{r_c/2 < r < r_c(1 - \frac{1}{k})}\abs{H_0(r,z)} & \lesssim r_c \abs{r-r_c} \frac{r_c^{k-1/2}}{r^{k-1/2}} \left(\frac{r^{2k}}{k r_c^{2k+3}} + \int_{r_c/2}^{r} \frac{1}{\abs{u(s)-c}^2}\frac{s^{2k-1}}{r_c^{2k-1}} \dd s \right) \\
& \lesssim \frac{r^{k+1/2}}{k r_c^{k+3/2}} + \frac{\mathbf{1}_{r_c/2 < r < r_c(1 - \frac{1}{k})}}{r_c \abs{r-r_c}} \frac{1}{k} \frac{r^{k+1/2}}{r_c^{k-1/2}} \\
& \lesssim \frac{r^{k+1/2}}{r_c^{k+3/2}}.  
\end{align*} 
Once the integral crosses the critical layer, one cannot do better than the analogous upper bound on $M(z)$: 
\begin{align*}
\mathbf{1}_{r_c < r < 1}\abs{H_0(r,z)} & \lesssim  \mathbf{1}_{r_c < r < 1} \abs{u-c}\frac{r^{k+1/2}}{r_c^{k+1/2}} \frac{k}{r_c^{3}} \lesssim k\frac{r^2}{r_c^3}\frac{r^{k+1/2}}{r_c^{k+1/2}}, \\ 
\mathbf{1}_{r_c < 1} \mathbf{1}_{r \geq 1}\abs{H_0(r,z)} & \lesssim \mathbf{1}_{r_c < 1} \mathbf{1}_{r \geq 1} \frac{k}{r_c^3}\frac{r^{k+1/2}}{r_c^{k+1/2}}. 
\end{align*}
Analogous estimates are made for $r_c \geq 1$ and for $H_\infty$; these are omitted for the sake of brevity. 
This completes the boundedness estimates \eqref{ineq:H0bd} and \eqref{ineq:Hinfbd}. 

Next, we consider the estimation of $H_0(r,c\pm i \eps) - H_0(r,c)$. 
For $r < r_c(1-1/k)$, we write 
\begin{align*}
H_0(r,c\pm i \eps) - H_0(r,c) & = \left(\phi(r,c\pm i \eps) - \phi(r,c)\right) \int_0^r \frac{1}{\phi^2(s,z)} \dd s + \phi(r,c) \int_0^r \frac{\phi^2(s,c) - \phi^2(s,c\pm i\eps)}{\phi^2(s,c\pm i \eps) \phi^2(s,c)} \dd s. 
\end{align*}
In this region, we apply a proof analogous to those used in Lemma \ref{lem:Wronk}. 
For the region $\abs{r-r_c} < r_c/k$ we again use the complex integral expansion \eqref{eq:H0HinfExp}: 
\begin{align*}
H_0(r,c\pm i \eps) - H_0(r,c)  & = \frac{1}{u'(r) P(r,c\pm i \eps)} - \frac{1}{u'(r) P(r,c)} \notag \\ 
& \quad +  \phi(r,c \pm i \eps)\left(R^\eps_{0,r}(c \pm i \eps) \mp i E^\eps_{0,r}(c \pm i \eps)\right) - \phi(r,c)\left(R^\eps_{0,r}(c) \mp i E^\eps_{0,r}(c)\right). 
\end{align*}
The desired estimates then follow from the convergence of $R^\eps_{0,r}$ and $E^\eps_{0,r}$. 
Convergence for $r \gtrsim r_c$ follows as in Lemma \ref{lem:Wronk}. 
The estimates of $\eps^\eta \left(H_\infty(r,c\pm i \eps) - H_\infty(r,c)\right)$ follows  similarly as well. 
\end{proof}

\begin{proof}[Proof of Lemma~\ref{lem:dH}] 
The identities \eqref{eq:HderivCL} follow by direct calculation from \eqref{eq:H0HinfExp}. 
Next, observe that $r\partial_r \phi = ru'P + (u-c) r\partial_rP$. 
The estimate \eqref{ineq:drHCL} then follows fromm Lemma \ref{lem:asympsRE} and Theorem \ref{thm:nonvan}. 

By direct calculation, 
\begin{align}
\partial_r H_0(r,z) = -\partial_r \phi \int_0^r  \frac{1}{\phi(s,z)^2}\dd s - \frac{1}{\phi}  = \left(\frac{u'}{u-z} + \frac{P'}{P}\right)H_0 - \frac{1}{\phi}; \label{eq:drH0ID}
\end{align}
and analogously for $H_\infty$. 
Therefore, 
\begin{align}
\abs{r\partial_r H_\infty(r,z)}\mathbf{1}_{\abs{r-r_c} \geq r_c/k} & \lesssim \left(\frac{r \abs{u'}}{\abs{u-c}} + k\right)\abs{H_\infty(r,c\pm i \eps)} + r \min\left(\frac{r^{k-1/2}}{r_c^{k-1/2}},\frac{r_c^{k+1/2}}{r^{k+1/2}}\right) \label{ineq:rHinfNCL}\\ 
\abs{r\partial_r H_0(r,z)}\mathbf{1}_{\abs{r-r_c} \geq r_c/k} & \lesssim \left(\frac{r \abs{u'}}{\abs{u-c}} + k\right)\abs{H_0(r,c\pm i \eps)} + r \min\left(\frac{r^{k-1/2}}{r_c^{k-1/2}},\frac{r_c^{k+1/2}}{r^{k+1/2}}\right).  \label{ineq:rH0NCL} 
\end{align}
Then, note from \eqref{ineq:uNCL}, $\abs{\frac{ru'(r)}{\abs{u-c}}}\mathbf{1}_{\abs{r-r_c} > r_c/k} \lesssim k$ and moreover that
\begin{align*}
r \min\left(\frac{r^{k-1/2}}{r_c^{k-1/2}},\frac{r_c^{k+1/2}}{r^{k+1/2}}\right) & \lesssim  \mathbf{1}_{r_c \leq 1} \left( \mathbf{1}_{r < r_c} \frac{r^{k+1/2}}{r_c^{k+3/2}} + \mathbf{1}_{r_c < r < 1} k  \frac{r^{k+2+1/2}}{r_c^{k+2+3/2}} + \mathbf{1}_{r>1}k \frac{r^{k+1/2}}{r_c^{k+3+1/2}}\right) \\ 
& \quad +  \mathbf{1}_{r_c > 1} \left(\mathbf{1}_{r < 1} \frac{r^{k+1/2}}{r_c^{k-1/2}} + \mathbf{1}_{1 < r < r_c} \frac{r^{k+5/2}}{r_c^{k-1/2}} + \mathbf{1}_{r>r_c} k \frac{r^{k+1/2}}{r_c^{k-5/2}}\right) \\ 
r \min\left(\frac{r^{k-1/2}}{r_c^{k-1/2}},\frac{r_c^{k+1/2}}{r^{k+1/2}}\right) & \lesssim \mathbf{1}_{r_c \leq 1} \left( \mathbf{1}_{r<r_c} k \frac{r_c^{k-3/2}}{r^{k-1/2}} + \mathbf{1}_{r_c <r < 1} \frac{r_c^{k+1/2}}{r^{k+3/2}} + \mathbf{1}_{r>1} \frac{r_c^{k+1/2}}{r^{k-1/2}} \right) \\ 
& \quad + \mathbf{1}_{r_c > 1} \left(\mathbf{1}_{r<1} k \frac{r_c^{k+5-1/2}}{r^{k-1/2}} + \mathbf{1}_{1 \leq r < r_c} k \frac{r_c^{k+5-1/2}}{r^{k+3/2}} + \mathbf{1}_{r_c < r} \frac{r_c^{k+5/2}}{r^{k-1/2}}\right), 
\end{align*}
which completes the estimates of the $\partial_r$ derivatives. 
Moreover, convergence follows from  \eqref{eq:drH0ID} and convergence estimates already proved. 

Next, turn to studying derivatives involving $\partial_{r_c}$. 
Away from the critical layer there holds 
\begin{align*}
-\partial_{r_c} H_0(r,c) \mathbf{1}_{\abs{r-r_c} > r_c/k}  
& =   \partial_{r_c}\left((u-z) P\right) \int_0^r \frac{1}{\phi^2 (s,c)} \dd s +    \phi \int_0^r \chi_{\neq} \frac{u'(r_c)}{(u-z)^3 P^2} \dd s \\ 
& \quad + \phi \int_0^r \frac{u'(r_c)u'(s)}{(u-z)^2} \partial_G\left( \frac{\chi_c}{u' P^2}\right) \dd s + \phi \int_0^r \frac{1}{(u-z)^2} \partial_{r_c} \frac{ \chi_{\neq} }{P^2} \dd s. 
\end{align*} 
By adapting the arguments in Lemmas \ref{lem:Hests}, \ref{lem:Wronk}, \ref{lem:Wronkdrc} the desired boundedness and convergence estimates follow (note  that appropriate estimates on $\partial_{r_c} P$ follow from Theorem \ref{thm:nonvan}).
Further, from Theorem \ref{thm:nonvan}, we may even take an additional $r\partial_r$ derivative and obtain similar estimates (note crucially $\abs{r-r_c} \geq r_c/k$). 

Computing $\partial_G$ derivatives near the critical layer from Lemma \ref{lem:Complex}:  
\begin{align*}
\partial_G H_0 & = \frac{1}{u'(r)}\partial_r \left(\frac{1}{u'P(r,c\pm i \eps)}\right)  - \frac{1}{u'(r)}\partial_r\phi\left(R_{0,r}^\eps(z) \mp i E^\eps_{0,r}\right) - P E(r,c\pm i \eps) \\ 
& \quad + \frac{1}{u'(r_c)}\partial_{r_c}\left(\frac{1}{u' P}\right) - \frac{1}{u'(r_c)}\partial_{r_c} \phi \left(R_{0,r}^\eps(z) \mp i E^\eps_{0,r}(z)\right) \\ 
& \quad  + P \chi_c E(r,z) - \phi \int_0^r \frac{u'}{(u-z)} \partial_{G} \left(\chi_c E\right) \dd s \\ 
& \quad - \frac{1}{u'(r_c)} \phi \int_0^r  \chi_{\neq}\frac{u'(s) u'(r_c)}{(u-z)^2} E(s,z) - \frac{u'(s)}{(u-z)} \chi_{\neq} \partial_{r_c} E(s,z) \dd s. 
\end{align*} 
The cancellation which ultimately removes the logarithmic singularity is the fact that $\partial_G \phi = (u-z) \partial_G P$. Taking this into account gives the identity \eqref{eq:dGH0}.
The corresponding calculation on $H_\infty$ is analogous.  

Next, we prove \eqref{ineq:dGHCL} and \eqref{ineq:drdGHCL}.  
Consider one of the singular appearing terms in \eqref{eq:dGH0}: 
\begin{align*}
\phi \int_0^r \frac{u'(s)}{(u-z)} \partial_{G} \left(\chi_c E\right) \dd s & = \phi \int_0^r \frac{u'(s)}{(u-z)} \partial_{G} \left(\chi_c \frac{1}{u'(s)} \partial_s \left(\frac{1}{u' P^2}\right)     \right) \dd s. 
\end{align*}
Note that $\partial_G$ and $(u'(s))^{-1} \partial_s$ commute. 
Hence, this term is essentially the same as that treated in Lemma \ref{lem:Wronkdrc} (combining also with arguments in Lemma \ref{lem:Hests}). 
Similarly, the remaining terms in \eqref{eq:dGH0} are all easy variants of terms we have treated before in Lemmas \ref{lem:Wronkdrc}, \ref{lem:Hests}, and \ref{lem:asympsRE}. 
Hence, the details are omitted for brevity. The calculations involving $H_\infty$ are analogous and are also omitted. 
This proves \eqref{ineq:dGHCL}.
Similarly, convergence as $\eps \rightarrow 0$ asserted in part (c) is deduced as in previous arguments. 

Next, turn to \eqref{ineq:drdGHCL}. 
Taking a $\partial_r \partial_G$ derivative of \eqref{eq:dGH0} gives: 
\begin{align*}
\partial_r \partial_G H_0 & = \partial_r \partial_G \left(\frac{1}{u'P(r,c\pm i \eps)}\right)  - u'(\partial_GP )   \left(R_{0,r}^\eps(z) \mp i E^\eps_{0,r}\right) \\ 
& \quad - (u-z) (\partial_r\partial_GP )   \left(R_{0,r}^\eps(z) \mp i E^\eps_{0,r}\right) -  (\partial_GP ) u'(r) E(r,z) \\ 
& \quad  + (\partial_r\phi) \int_0^r \frac{u'(s)}{(u-z)} \partial_{G} \left(\chi_c E\right) \dd s + P u'(r) \partial_{G} \left(\chi_c E\right) \\ 
& \quad - \frac{1}{u'(r_c)} (\partial_r\phi) \int_0^r  \chi_{\neq}\frac{u'(s) u'(r_c)}{(u-z)^2} E(s,z) - \frac{u'(s)}{(u-z)}  \partial_{r_c}\left(\chi_{\neq} E(s,z)\right) \dd s. 
\end{align*}
By Theorem \ref{thm:nonvan} and the arguments used previously, we can again deduce the desired logarithmically singular upper bounds; the details are omitted as they are repetitive. 
\end{proof}

\section{Vanishing for \texorpdfstring{$k\geq2$}{TEXT} outside $I_\alpha$}\label{app:vanishingout}
We start by performing energy estimates on the solution to the inhomogeneous
Rayleigh problem
\begin{align}\label{eq:abstkg}
\left[\partial_{rr} + \frac{1/4-k^2}{r^2} + \frac{\beta(r)}{u(r) - c\mp i\eps}\right] Y_\eps(r,c)=\frac{F_\eps(r,c)}{u(r) - c\mp i\eps} +F_{\ast,\eps}(r,c),
\end{align}
with boundary conditions 
\begin{align}\label{eq:asumpt1}
Y_\eps(0,c)=0, \qquad \lim_{r\to \infty}Y_\eps(r,c)=0.
\end{align}
Problem \eqref{eq:abstkg} is a slight generalization of \eqref{eq:hphiray}.
Notice that, for every $\eps>0$, \eqref{eq:abstkg} is just a regular perturbation of Laplace's equation, and therefore
\begin{align}\label{eq:asumpt}
Y_\eps(r,c)=\mathcal{O}(r^{k+1/2})\quad \text{ as } r\to 0, \qquad Y_\eps(r,c)=\mathcal{O}(r^{1/2-k})\quad\text{ as } r\to \infty.
\end{align}
For this reason, we define the weight
\begin{align}\label{eq:weightRAY}
&w_Y(r)=\min\{ r^{k+1}, r^{-k+1}\},
\qquad
w_{Y,\gamma}(r)=\min\{ r^{k+1-\gamma}, r^{-k+1+\gamma}\},
\end{align}
and, using the notation as in \eqref{eq:weightpsi}, prove the following theorem.

\begin{theorem}\label{thm:boundsonY}
Let $Y_\eps$ be the solution to \eqref{eq:abstkg}-\eqref{eq:asumpt}, and fix any $\gamma\in (0,2k)$. Then $Y_\eps$ satisfies
the following bounds.
\begin{itemize}
\item If $c\in (u(0),u(0)+1]$, then 
\begin{align}\label{eq:Ybd1}
&\|Y_\eps(\cdot,c)\|^2_{L^\infty_{\psi,\gamma}}+\norm{Y_\eps(\cdot,c)}_{L_{Y,\gamma}^2}^2+\norm{r\de_rY_\eps(\cdot,c)}_{L_{Y,\gamma}^2}^2\lesssim_{\gamma}  \norm{\brak{r}^{2}F_\eps(\cdot,c)}_{L_{Y,\gamma}^2}^2+\| r^2 F_{\ast,\eps}(\cdot, c)\|^2_{L_{Y,\gamma}^2}.
\end{align}

\item If $c\in (0,u(0))$ is such that $r_c\leq  \eps^\frac{1}{2+\alpha}$,
then 
\begin{align}\label{eq:Ybd2}
\|Y_\eps(\cdot,c)\|^2_{L^\infty_{\psi,\gamma}}\!+\norm{Y_\eps(\cdot,c)}_{L_{Y,\gamma}^2}^2\!+\norm{r\de_rY_\eps(\cdot,c)}_{L_{Y,\gamma}^2}^2\!\lesssim_{\gamma}  \frac{1}{r_c^{2\alpha}} \left[\| \brak{r}^{2+\alpha}r^{-\alpha} F_{\eps}(\cdot, c)\|^2_{L_{Y,\gamma}^2}
+ \| r^2 F_{\ast,\eps}(\cdot, c)\|^2_{L_{Y,\gamma}^2}\right].
\end{align}

\item If $c>u(0)+1$, then 
\begin{align}\label{eq:Ybd3}
&\|Y_\eps(\cdot,c)\|^2_{L^\infty_{\psi,\gamma}}+\norm{Y_\eps(\cdot,c)}_{L_{Y,\gamma}^2}^2+\norm{r\de_rY_\eps(\cdot,c)}_{L_{Y,\gamma}^2}^2\lesssim_{\gamma}   \frac{1}{c^2}\| r^2 F_{\eps}(\cdot, c)\|^2_{L_{Y,\gamma}^2}+\| r^2 F_{\ast,\eps}(\cdot, c)\|^2_{L_{Y,\gamma}^2}. 
\end{align}
\item If $c\leq 0$, or $c\in (0,u(0))$ is such that $r_c\geq \frac12\eps^{-\frac{1}{2+\alpha}}$ , then for every $0<\tilde\alpha<\gamma$ we have
\begin{align}\label{eq:Ybd4}
&\|Y_\eps(\cdot,c)\|^2_{L^\infty_{\psi,\gamma}}\!\!+\norm{Y_\eps(\cdot,c)}_{L_{Y,\gamma}^2}^2\!\!+\norm{r\de_rY_\eps(\cdot,c)}_{L_{Y,\gamma}^2}^2\!\! \lesssim_{\gamma}   \norm{r^2\brak{r}^{2+\alpha}F_\eps(\cdot,c)}_{L_{Y,\gamma-\tilde\alpha}^2}^2+\| r^2 F_{\ast,\eps}(\cdot, c)\|^2_{L_{Y,\gamma-\tilde\alpha}^2}.
\end{align}
\item There exists $R_{\gamma}\gg 1$ such that if $c<-R_{\gamma}$, then
\begin{align}\label{eq:Ybd5}
&\|Y_\eps(\cdot,c)\|^2_{L^\infty_{\psi,\gamma}}+\norm{Y_\eps(\cdot,c)}_{L_{Y,\gamma}^2}^2+\norm{r\de_rY_\eps(\cdot,c)}_{L_{Y,\gamma}^2}^2\lesssim_{\gamma}\frac{1}{c^2}\| r^2 F_{\eps}(\cdot, c)\|^2_{L_{Y,\gamma}^2}+\| r^2 F_{\ast,\eps}(\cdot, c)\|^2_{L_{Y,\gamma}^2}.
\end{align}
\end{itemize}
\end{theorem}

The proof of Theorem \ref{thm:boundsonY} is based mostly on energy estimates, along with the Sobolev inequality
\begin{align}\label{eq:sobembi}
\|Y_\eps(\cdot,c)\|^2_{L^\infty}\leq \int_0^\infty |\de_{r}Y_\eps(r,c)|^2 r\dd r +\int_0^\infty \frac{1}{r^2}|Y_\eps(r,c)|^2 r\dd r.
\end{align}
We take the real and imaginary parts of \eqref{eq:abstkg}, to obtain the system
\begin{align}
\left[\partial_{rr} + \frac{1/4-k^2}{r^2}\right]\!\Re Y_\eps + \frac{\beta(u-c) \Re Y_\eps}{(u-c)^2 + \eps^2}
\mp \frac{\eps\beta \Im Y_\eps}{(u-c)^2 + \eps^2} & = \frac{(u-c) \Re F_\eps}{(u-c)^2 + \eps^2} \mp \frac{\eps \Im F_\eps}{(u-c)^2 + \eps^2} +   \Re F_{\ast,\eps}\label{eq:realparte}\\ 
\left[\partial_{rr} + \frac{1/4-k^2}{r^2}\right]\!\Im Y_\eps  + \frac{\beta(u-c) \Im Y_\eps }{(u-c)^2 + \eps^2} 
\pm \frac{\eps\beta \Re Y_\eps}{(u-c)^2 + \eps^2} & = \frac{(u-c) \Im F_\eps}{(u-c)^2 + \eps^2} \mp \frac{\eps \Re F_\eps}{(u-c)^2 + \eps^2}+\Im F_{\ast,\eps}.\label{eq:imagparte}
\end{align}
Let 
$$
\gamma_* \in (-k+1/2,k+1/2).
$$ 
We multiply \eqref{eq:realparte} by $\Re Y_\eps r^{1-2\gamma_*}$ and \eqref{eq:imagparte} by $\Im Y_\eps r^{1-2\gamma_*}$, integrate by parts using  \eqref{eq:asumpt}, and add the result to obtain
\begin{align}
&\int_0^\infty \abs{\partial_r \frac{Y_\eps (r,c)}{r^{\gamma_*}} }^2 r \dd r + \left[k^2 - \frac{1}{4} - \gamma_*(\gamma_*-1)\right]\int_0^\infty \frac{\abs{Y_\eps(r,c)}^2}{r^{2\gamma_*+2}} r \dd r + \int_0^\infty \frac{\beta(r)(c-u(r))}{(u(r)-c)^2 + \eps^2} \frac{\abs{Y_\eps(r,c)}^2}{r^{2\gamma_*}} r \dd r \notag\\
 &\qquad \qquad\lesssim \int_0^\infty \frac{\abs{F_\eps(r,c)} \abs{Y_\eps(r,c)}}{\sqrt{(u(r)-c)^2 + \eps^2} } \frac{1}{r^{2\gamma_*}}  r \dd r
 + \int_0^\infty \abs{F_{\ast,\eps}(r,c)} \abs{Y_\eps(r,c)} \frac{1}{r^{2\gamma_*}}  r \dd r \label{eq:estipro31}.
\end{align}
Since for any small $\kappa>0$ there holds
\begin{align}\label{eq:treatfstar}
\int_0^\infty \abs{F_{\ast,\eps}(r,c)} \abs{Y_\eps(r,c)} \frac{1}{r^{2\gamma_*}}  r \dd r\leq\kappa\int_0^\infty \frac{\abs{Y_\eps(r,c)}^2}{r^{2\gamma_*+2}}   r \dd r +\kappa^{-1}\int_0^\infty \frac{|F_{\ast,\eps}(r,c)|^2}{r^{2\gamma_*-3}} \dd r,
\end{align}
we obtain
\begin{align}
&\int_0^\infty \abs{\partial_r \frac{Y_\eps (r,c)}{r^{\gamma_*}} }^2 r \dd r + \left[k^2 - \frac{1}{4} - \gamma_*(\gamma_*-1)\right]\int_0^\infty \frac{\abs{Y_\eps(r,c)}^2}{r^{2\gamma_*+2}} r \dd r + \int_0^\infty \frac{\beta(r)(c-u(r))}{(u(r)-c)^2 + \eps^2} \frac{\abs{Y_\eps(r,c)}^2}{r^{2\gamma_*}} r \dd r \notag\\
 &\qquad \qquad\lesssim \int_0^\infty \frac{\abs{F_\eps(r,c)} \abs{Y_\eps(r,c)}}{\sqrt{(u(r)-c)^2 + \eps^2} } \frac{1}{r^{2\gamma_*}}  r \dd r
+\int_0^\infty \frac{|F_{\ast,\eps}(r,c)|^2}{r^{2\gamma_*-3}} \dd r\label{eq:estipro3}.
\end{align}
Cross multiplying \eqref{eq:realparte} and \eqref{eq:imagparte} and subtracting gives
\begin{align}
\eps\int_0^\infty \frac{ \beta(r) \abs{Y_\eps(r,c)}^2}{(u(r)-c)^2 +\eps^2} \frac{1}{r^{2\gamma_*}} r \dd r \lesssim \int_0^\infty \frac{\abs{F_\eps(r,c)} \abs{Y_\eps(r,c)}}{\sqrt{(u(r)-c)^2 + \eps^2}}  \frac{1}{r^{2\gamma_*}} r \dd r+ \int_0^\infty \frac{\abs{F_{\ast,\eps}(r,c)} \abs{Y_\eps(r,c)}}{r^{2\gamma_*-1}}  \dd r.  \label{ineq:imYctrl}
\end{align}
We use  \eqref{eq:estipro3} and \eqref{ineq:imYctrl} in different ways, depending on the various regimes considered.

\subsection{Estimates near \texorpdfstring{$c=u(0)$}{TEXT}}
We start by proving estimates \eqref{eq:Ybd1}-\eqref{eq:Ybd3}.

\subsubsection{Proof of (\ref{eq:Ybd1})}

When $c\in (u(0),u(0)+1]$,then
$$
c-u(r)\geq c-u(0)>0, \qquad \forall r\in[0,\infty).
$$
From  \eqref{eq:estipro3} and using standard arguments,
we obtain
\begin{align}
\int_0^\infty \frac{\abs{F_\eps(r,c)} \abs{Y_\eps(r,c)}}{\sqrt{(u(r)-c)^2 + \eps^2} } \frac{1}{r^{2\gamma_*}}  r \dd r
\leq  \kappa\int_0^\infty \frac{\abs{Y_\eps(r,c)}^2}{r^{2\gamma_*+2}} r \dd r 
+\frac1\kappa\int_0^\infty\frac{|F_\eps(r,c)|^2}{r^{2\gamma_*-2}(c-u(r))^2} r\dd r\label{eq:outsrsffdf},
\end{align}
for any $\kappa\in (0,1)$. Moreover,
since $u(0)-u(r)\sim r^2$ near $r=0$, we can write
\begin{align*}
\int_0^\infty\frac{|F_\eps(r,c)|^2}{r^{2\gamma_*-2}(c-u(r))^2} r\dd r
&\leq \int_0^1\frac{|F_\eps(r,c)|^2}{r^{2\gamma_*-2}(u(0)-u(r))^2} r\dd r
+\int_1^\infty\frac{|F_\eps(r,c)|^2}{r^{2\gamma_*-2}(u(0)-u(r))^2} r\dd r\notag\\
&\lesssim  \int_0^1\frac{|F_\eps(r,c)|^2}{r^{2\gamma_*+2}} r\dd r+ \int_1^\infty\frac{|F_\eps(r,c)|^2}{r^{2\gamma_*-2}}r\dd r
\lesssim \int_0^\infty \brak{r}^4\frac{|F_\eps(r,c)|^2}{r^{2\gamma_*+1}} \dd r.
\end{align*}
Hence, from the above estimates and \eqref{eq:estipro3}, we obtain by taking $\kappa\in (0,1)$ small enough, that
\begin{align}
&\int_0^\infty \abs{\partial_r \frac{Y_\eps (r,c)}{r^{\gamma_*}} }^2 r \dd r + \left[k^2 - \frac{1}{4} - \gamma_*(\gamma_*-1)\right]\int_0^\infty \frac{\abs{Y_\eps(r,c)}^2}{r^{2\gamma_*+2}} r \dd r + \int_0^\infty \frac{\beta(r)(c-u(r))}{(u(r)-c)^2 + \eps^2} \frac{\abs{Y_\eps(r,c)}^2}{r^{2\gamma_*}} r \dd r \notag\\
 &\qquad \qquad\lesssim \int_0^\infty \brak{r}^4\frac{|F_\eps(r,c)|^2}{r^{2\gamma_*+1}} \dd r+\int_0^\infty \frac{|F_{\ast,\eps}(r,c)|^2}{r^{2\gamma_*-3}} \dd r.\label{eq:outright1}
\end{align}
An application of \eqref{eq:sobembi} gives the estimate 
\begin{align*}
&\| r^{-\gamma_*} Y_\eps(\cdot,c)\|^2_{L^\infty}+
\int_0^\infty \abs{\partial_r \frac{Y_\eps (r,c)}{r^{\gamma_*}} }^2 r \dd r +\int_0^\infty \frac{\abs{Y_\eps(r,c)}^2}{r^{2\gamma_*+2}} r \dd r \lesssim_{\gamma_*} \int_0^\infty \brak{r}^4\frac{|F_\eps(r,c)|^2}{r^{2\gamma_*+1}} \dd r+\int_0^\infty \frac{|F_{\ast,\eps}(r,c)|^2}{r^{2\gamma_*-3}} \dd r.
\end{align*}
By considering the left-hand side above restricted to $(0,1)$ and choosing $\gamma_*=k+1/2-\gamma$, with $\gamma\in (0,k]$, 
we obtain
\begin{align}\label{eq:lastreally1}
&\| r^{-k-1/2+\gamma} Y_\eps(\cdot,c)\|^2_{L^\infty(0,1)}+\norm{Y_\eps(\cdot,c)}_{L_{Y,\gamma}^2(0,1)}^2+\norm{r\de_rY_\eps(\cdot,c)}_{L_{Y,\gamma}^2(0,1)}^2\notag\\
&\qquad\qquad\lesssim_{\gamma} \int_0^\infty \brak{r}^4\frac{|F_\eps(r,c)|^2}{r^{2\gamma_*+1}} \dd r+\int_0^\infty \frac{|F_{\ast,\eps}(r,c)|^2}{r^{2\gamma_*-3}} \dd r\notag\\
&\qquad\qquad\lesssim_{\gamma}  \norm{\brak{r}^{2}F_\eps(\cdot,c)}_{L_{Y,\gamma}^2}^2+\| r^2 F_{\ast,\eps}(\cdot, c)\|^2_{L_{Y,\gamma}^2}. 
\end{align}
Similarly, on   $(1,\infty)$ we choose $\gamma_*=-k+1/2+\gamma$, with $\gamma\in (0,k]$, and deduce that
\begin{align}\label{eq:lastreally2}
&\| r^{k-1/2-\gamma} Y_\eps(\cdot,c)\|^2_{L^\infty(1,\infty)}+\norm{Y_\eps(\cdot,c)}_{L_{Y,\gamma}^2(1,\infty)}^2+\norm{r\de_rY_\eps(\cdot,c)}_{L_{Y,\gamma}^2(1,\infty)}^2\notag\\
&\qquad\qquad\lesssim_{\gamma}  \norm{\brak{r}^{2}F_\eps(\cdot,c)}_{L_{Y,\gamma}^2}^2+\| r^2 F_{\ast,\eps}(\cdot, c)\|^2_{L_{Y,\gamma}^2}.
\end{align}
Adding the above two estimates together and recalling the shape of the weight \eqref{eq:weightRAY}, we deduce \eqref{eq:Ybd1}. Note
that we can extend the range of $\gamma$ to the interval $(0,2k)$, as the weight $w_{Y,\gamma}$ is symmetric about $\gamma=k$.

\subsubsection{Proof of (\ref{eq:Ybd2})}

When $c\in (0,u(0))$ complying with $r_c\leq  \eps^\frac{1}{2+\alpha}$,
from \eqref{eq:estipro31} and using a similar argument to that in \eqref{eq:outsrsffdf}, we have 
\begin{align*}
&\int_0^\infty \abs{\partial_r \frac{Y_\eps (r,c)}{r^{\gamma_*}} }^2 r \dd r + \left[k^2 - \frac{1}{4} - \gamma_*(\gamma_*-1)\right]\int_0^\infty \frac{\abs{Y_\eps(r,c)}^2}{r^{2\gamma_*+2}} r \dd r + \int_{r_c}^\infty \frac{\beta(r)(c-u(r))}{(u(r)-c)^2 + \eps^2} \frac{\abs{Y_\eps(r,c)}^2}{r^{2\gamma_*}} r \dd r \notag\\
 & \quad\lesssim \int_{0}^{r_c} \frac{\beta(r)(u(r)-c)}{(u(r)-c)^2 + \eps^2} \frac{\abs{Y_\eps(r,c)}^2}{r^{2\gamma_*}} r \dd r 
+ \int_0^\infty \frac{\abs{F_\eps(r,c)} \abs{Y_\eps(r,c)}}{\sqrt{(u(r)-c)^2 + \eps^2} } \frac{1}{r^{2\gamma_*}}  r \dd r
 + \int_0^\infty \frac{\abs{F_{\ast,\eps}(r,c)} \abs{Y_\eps(r,c)}}{r^{2\gamma_*-1}}  \dd r.\notag
\end{align*}
For $r \leq r_c\leq 1$ there holds $\abs{c-u(r)} \lesssim r_c^2$ and hence by \eqref{ineq:imYctrl} and therefore 
\begin{align*} 
\int_0^{r_c} \frac{ \beta(r) (u(r)-c)}{(u(r)-c)^2 +\eps^2} \frac{\abs{Y_\eps(r,c)}^2}{r^{2\gamma_*}} r \dd r \lesssim \frac{1}{r_c^\alpha}\left[ \int_0^\infty \frac{\abs{F_\eps(r,c)} \abs{Y_\eps(r,c)}}{\sqrt{(u(r)-c)^2 + \eps^2}}  \frac{1}{r^{2\gamma_*}} r \dd r+ \int_0^\infty \frac{\abs{F_{\ast,\eps}(r,c)} \abs{Y_\eps(r,c)}}{r^{2\gamma_*-1}}  \dd r\right],\notag
\end{align*}
and since $r_c\ll 1$, we obtain
\begin{align*}
&\int_0^\infty \abs{\partial_r \frac{Y_\eps (r,c)}{r^{\gamma_*}} }^2 r \dd r + \left[k^2 - \frac{1}{4} - \gamma_*(\gamma_*-1)\right]\int_0^\infty \frac{\abs{Y_\eps(r,c)}^2}{r^{2\gamma_*+2}} r \dd r + \int_{r_c}^\infty \frac{\beta(r)(c-u(r))}{(u(r)-c)^2 + \eps^2} \frac{\abs{Y_\eps(r,c)}^2}{r^{2\gamma_*}} r \dd r \notag\\
 &\qquad \qquad\lesssim \frac{1}{r_c^\alpha}\left[ \int_0^\infty \frac{\abs{F_\eps(r,c)} \abs{Y_\eps(r,c)}}{\sqrt{(u(r)-c)^2 + \eps^2}}  \frac{1}{r^{2\gamma_*}} r \dd r+ \int_0^\infty \abs{F_{\ast,\eps}(r,c)} \abs{Y_\eps(r,c)} \frac{1}{r^{2\gamma_*}}  r \dd r\right].
\end{align*}
 Now, for any $\kappa\in (0,1)$ we have
\begin{align*}
\frac{1}{r_c^\alpha} \int_0^\infty \frac{\abs{F_\eps(r,c)} \abs{Y_\eps(r,c)}}{\sqrt{(u(r)-c)^2 + \eps^2}}  \frac{1}{r^{2\gamma_*}} r \dd r \leq 
&\kappa \int_0^\infty \frac{\abs{Y_\eps(r,c)}^2}{r^{2\gamma_*+2}} r \dd r+\kappa^{-1}\frac{1}{r_c^{2\alpha}}\int_0^\infty \frac{\abs{F_\eps(r,c)}^2}{(u(r)-c)^2 + \eps^2}  \frac{1}{r^{2\gamma_*-2}} r \dd r.
\end{align*}
The  contribution from $F_\eps$ is controlled using that $r^{2+\alpha} \lesssim \brak{r}^{2+\alpha}\sqrt{(u-c)^2 + \eps^2}$, which implies that
\begin{align*}
\int_0^\infty \frac{\abs{F_\eps(r,c)}^2}{(u(r)-c)^2 + \eps^2}  \frac{1}{r^{2\gamma_*-2}} r \dd r
=\int_0^\infty \frac{r^{4+2\alpha}\abs{F_\eps(r,c)}^2}{ (u(r)-c)^2 + \eps^2}  \frac{1}{r^{2\gamma_*+2+2\alpha}} r \dd r
\lesssim \int_0^\infty  \brak{r}^{4+2\alpha}\frac{\abs{F_\eps(r,c)}^2}{r^{2\gamma_*+1+2\alpha}}   \dd r. 
\end{align*}
On $F_\ast$, we use again \eqref{eq:treatfstar}.
Therefore, choosing $\kappa \ll 1$  and arguing as in \eqref{eq:lastreally1}-\eqref{eq:lastreally2} yields the desired result.

\subsubsection{Proof of (\ref{eq:Ybd3})}
The starting point here is again \eqref{eq:estipro3}, together with \eqref{eq:outsrsffdf}, which allows us to write
\begin{align}
&\int_0^\infty \abs{\partial_r \frac{Y_\eps (r,c)}{r^{\gamma_*}} }^2 r \dd r + \left[k^2 - \frac{1}{4} - \gamma_*(\gamma_*-1)\right]\int_0^\infty \frac{\abs{Y_\eps(r,c)}^2}{r^{2\gamma_*+2}} r \dd r + \int_0^\infty \frac{\beta(r)(c-u(r))}{(u(r)-c)^2 + \eps^2} \frac{\abs{Y_\eps(r,c)}^2}{r^{2\gamma_*}} r \dd r \notag\\
 & \quad\lesssim \int_0^\infty\frac{|F_\eps(r,c)|^2}{r^{2\gamma_*-2}(c-u(r))^2} r\dd r+\int_0^\infty \frac{|F_{\ast,\eps}(r,c)|^2}{r^{2\gamma_*-3}} \dd r \lesssim \frac{1}{c^2}\int_0^\infty\frac{|F_\eps(r,c)|^2}{r^{2\gamma_*-3}} \dd r+\int_0^\infty \frac{|F_{\ast,\eps}(r,c)|^2}{r^{2\gamma_*-3}} \dd r,\label{eq:outoutright1}
\end{align}
and \eqref{eq:Ybd3} follows from similar arguments as in \eqref{eq:lastreally1}-\eqref{eq:lastreally2}.

\subsection{Estimates near \texorpdfstring{$c=0$}{TEXT}}

The core of this section lies in the proof of \eqref{eq:Ybd4}. As we shall see below, \eqref{eq:Ybd5} follows by the argument already used for \eqref{eq:Ybd3}. 

\subsubsection{Proof of (\ref{eq:Ybd4})} \label{sec:Ybd4}
In this case, the argument to prove uniform estimates is completely different. We aim to prove the following.
\begin{lemma}\label{lem:left}
Let $Y_\eps$ be the solution to \eqref{eq:abstkg}-\eqref{eq:asumpt}, and fix $0<\tilde\alpha<\gamma <2k$.
Then, for every $c\leq 0$, or any $c\in (0,u(0))$ such that
\begin{align}\label{eq:alpgjskfg}
r_c\geq \frac12\eps^{-\frac{1}{2+\alpha}},
\end{align}
we have that
\begin{align}\label{eq:L2outleft}
\| Y_\eps(\cdot,c)\|_{L^2_{Y,\gamma}}+\| r\de_rY_\eps(\cdot,c)\|_{L^2_{Y,\gamma}} \lesssim_\gamma \norm{r^2\brak{r}^{2+\alpha}F_\eps(\cdot,c)}_{L_{Y,\gamma-\tilde\alpha}^2}+\| r^2 F_{\ast,\eps}(\cdot, c)\|_{L_{Y,\gamma-\tilde\alpha}^2}.
\end{align}
Above, $\alpha$ can be take zero when $c\leq 0$.
In particular, using \eqref{eq:sobembi}, then \eqref{eq:Ybd4} holds.
\end{lemma}

The proof of this fact is split into different lemmas. To begin with, we need an estimate with sharper weights on $Y_\eps$ in terms of slightly weaker weights.
\begin{lemma}\label{lem:left1}
Let $Y_\eps$ be the solution to \eqref{eq:abstkg}-\eqref{eq:asumpt}.
Then, for every $c\leq 0$,
 and $0<\tilde\alpha<\gamma< 2k$ we have that
\begin{align}\label{eq:outleft}
\norm{Y_\eps(\cdot,c)}_{L_{Y,\gamma}^2}\lesssim \norm{r^2\brak{r}^2F_\eps(\cdot,c)}_{L_{Y,\gamma-\tilde\alpha}^2}+\| r^2 F_{\ast,\eps}(\cdot, c)\|_{L_{Y,\gamma-\tilde\alpha}^2}+\norm{Y_\eps(\cdot,c)}_{L_{Y,2\gamma}^2}
\end{align}
and
\begin{align}\label{eq:outleftder}
\norm{r\de_rY_\eps(\cdot,c)}_{L_{Y,\gamma}^2}\lesssim \norm{r^2\brak{r}^2 F_\eps(\cdot,c)}_{L_{Y,\gamma-\tilde\alpha}^2}+\| r^2F_{\ast,\eps}(\cdot, c)\|_{L_{Y,\gamma-\tilde\alpha}^2}+\norm{Y_\eps(\cdot,c)}_{L_{Y,2\gamma}^2}.
\end{align}
When $c\in (0,u(0))$ is such that
$$
r_c\geq \frac12\eps^{-\frac{1}{2+\alpha}},
$$
then
\begin{align}\label{eq:outleftbis}
\norm{Y_\eps(\cdot,c)}_{L_{Y,\gamma}^2}\lesssim \norm{r^2\brak{r}^{2+\alpha}F_\eps(\cdot,c)}_{L_{Y,\gamma-\tilde\alpha}^2}+\| r^2 F_{\ast,\eps}(\cdot, c)\|_{L_{Y,\gamma-\tilde\alpha}^2}+\norm{Y_\eps(\cdot,c)}_{L_{Y,2\gamma}^2}
\end{align}
and
\begin{align}\label{eq:outleftderbis}
\norm{r\de_rY_\eps(\cdot,c)}_{L_{Y,\gamma}^2}\lesssim \norm{r^2\brak{r}^{2+\alpha} F_\eps(\cdot,c)}_{L_{Y,\gamma-\tilde\alpha}^2}+\| r^2F_{\ast,\eps}(\cdot, c)\|_{L_{Y,\gamma-\tilde\alpha}^2}+\norm{Y_\eps(\cdot,c)}_{L_{Y,2\gamma}^2}.
\end{align}
\end{lemma}

\begin{proof}[Proof of Lemma~\ref{lem:left1}]
By using the Green's function \eqref{eq:greenlapl}, we can deduce from \eqref{eq:abstkg} that
\begin{align}\label{eq:THSTFD}
Y_\eps(r,c)= \int_0^\infty \mathcal{L}(r,\rho) \left[\frac{F_\eps(\rho,c)}{u(\rho)-c\mp i \eps}+F_{\ast,\eps}(\rho,c)\right] \dd\rho - \int_0^\infty \mathcal{L}(r,\rho) \frac{\beta(\rho)}{u(\rho) - c \mp i \eps} Y_\eps(\rho,c) \dd\rho. 
\end{align}
Considering the weights in \eqref{eq:weightRAY} and \eqref{eq:weightF}, 
we need to prove an $L^2$ bound for
\begin{align*}
\frac{Y_\eps(r,c)}{w_{Y,\gamma}(r)}&= \int_0^\infty \frac{w_{Y,\gamma-\tilde\alpha}(\rho)\rho^{-2}}{w_{Y,\gamma}(r)}\mathcal{L}(r,\rho) \left[\frac{1}{(u(\rho)-c\mp i \eps)} \frac{\rho^2F_\eps(\rho,c)}{w_{Y,\gamma-\tilde\alpha}(\rho)}+\frac{\rho^2F_{\ast,\eps}(\rho,c)}{w_{Y,\gamma-\tilde\alpha}(\rho)}\right] \dd\rho \notag\\
&\quad- \int_0^\infty \frac{w_{Y,2\gamma}(\rho)}{w_{Y,\gamma}(r)}\mathcal{L}(r,\rho) \frac{\beta(\rho)}{u(\rho) - c \mp i \eps} 
\frac{Y_\eps(\rho,c)}{w_{Y,2\gamma}(\rho)}\dd\rho. 
\end{align*}
If we consider the case $c\leq 0$ first, we observe that since $u(\rho)\sim \brak{\rho}^{-2}$ 
for the first term we need to prove that 
\begin{align}\label{eq:outleftSchur1}
\int_0^\infty \int_0^\infty \left|\frac{w_{Y,\gamma-\tilde\alpha}(\rho)\rho^{-2}}{w_{Y,\gamma}(r)}\mathcal{L}(r,\rho) \right|^2 \dd\rho\,\dd r<\infty,
\end{align}
which follows from a straightforward calculation.  
Similarly, we also have
\begin{align}\label{eq:outleftSchur2}
\int_0^\infty \int_0^\infty  \left|\frac{w_{Y,2\gamma}(\rho)}{w_{Y,\gamma}(r)}\mathcal{L}(r,\rho) \frac{\beta(\rho)}{u(\rho)}\right|^2\dd\rho \,\dd r, 
<\infty, 
\end{align}
where we need to exploit the fast decay of $\beta$ at infinity. The derivative estimate follows from applying $r\de_r$ to \eqref{eq:THSTFD}, and arguing in the same way as above.

When  $c\in (0,u(0))$ is such that
$$
r_c\geq \frac12\eps^{-\frac{1}{2+\alpha}},
$$
the proof is similar. In this case, the key observation is that 
 if $r>r_c/2$,
\begin{align}\label{eq:ciccio1}
\brak{r}^{2+\alpha}\sqrt{(u(r)-c)^2+\eps^2}\geq \brak{r}^{2+\alpha}\eps\gtrsim \frac{ \brak{r}^{2+\alpha}}{r_c^{2+\alpha}}\gtrsim 1,
\end{align}
while if $r\leq r_c/2$, then $|u(r)-c|\gtrsim \brak{r}^{-2}$, and thus again
\begin{align}\label{eq:ciccio2}
\brak{r}^{2+\alpha}\sqrt{(u(r)-c)^2+\eps^2}\geq \brak{r}^{2+\alpha}|u(r)-c|\gtrsim 1.
\end{align}
In view of this,
the weight on $F$ has an extra power of $\brak{r}^\alpha$. The proof is concluded.
\end{proof}

We also need a result on the homogeneous Rayleigh problem.
\begin{lemma}\label{eq:homoleft}
Fix $0<\gamma\ll 1$, and let $\phi \in L^2_{Y,\gamma}$  with $r\de_r\phi\in  L^2_{Y,\gamma}$ be a solution to the homogeneous Rayleigh problem \eqref{eq:Rayleigh}, for $z=c\leq 0$. Then $\phi\equiv0$.
\end{lemma}

\begin{proof}[Proof of Lemma~\ref{eq:homoleft}]
If $\phi\equiv0$ in a neighborhood of the origin, 
then, by unique continuation, $\phi\equiv0$ everywhere. Hence, we may assume that $\phi(r,c)>0$ for every $0<r\ll 1$.
Recall, the function $g(r)=r^{3/2}u(r)$ (Lemma \ref{rmk:steady}) satisfies
\begin{align*}
\partial_{rr} g + \left( -\frac{3}{4 r^2} + \frac{\beta(r)}{u(r)} \right) g = 0.
\end{align*}
Hence,
\begin{align}\label{eq:conjray}
g\partial_{rr} \phi-\phi\partial_{rr} g + \frac{1-k^2}{r^2}\phi g + \beta(r) \left(\frac{1}{u(r) - c}-\frac{1}{u(r)} \right) \phi g = 0.
\end{align}
Let
$$
\bar{r}=\sup\{r\in (0,\infty): \phi(r',c)>0, \quad \forall r'<r\}\in(0,\infty].
$$
We integrate \eqref{eq:conjray} on $(0,\bar{r})$. Notice that the functions involved are integrable due to the assumptions on $\phi$,
even if $\bar{r}=\infty$.
Using that 
$$
g\partial_{rr} \phi-\phi\partial_{rr} g=\de_r(g\partial_{r} \phi-\phi\partial_{r} g),
$$
and the fact that since $c\leq 0$ there holds
$$
\frac{1}{u(r) - c}-\frac{1}{u(r)} \leq 0,
$$
we obtain
\begin{align*}
g(\bar{r})\partial_{r} \phi(\bar{r},c)- \phi(\bar{r},c)\partial_{r} g \phi(\bar{r}) =\int_0^{\bar{r}}\left[ \frac{k^2-1}{r^2}\phi(r,c) g(r) - \beta(r) \left(\frac{1}{u(r) - c}-\frac{1}{u(r)} \right) \phi(r,c) g(r)\right]\dd r \geq 0\notag.
\end{align*}
Note that $\phi(\bar{r},c)= 0$, $\de_r\phi(\bar{r},c)\leq 0$ and $g>0$, so that the above implies (also in the case $\bar{r}=\infty$)
that
\begin{align*}
\int_0^{\bar{r}} \left[\frac{k^2-1}{r^2}\phi(r,c) g(r) - \beta(r) \left(\frac{1}{u(r) - c}-\frac{1}{u(r)} \right) \phi(r,c) g(r)\right]\dd r = 0,
\end{align*}
and therefore $\phi \equiv 0$, concluding the proof of the lemma.
\end{proof}

We can now complete the proof of Lemma \ref{lem:left}.

\begin{proof}[Proof of Lemma~\ref{lem:left}]
Assume for contradiction that \eqref{eq:L2outleft} does not hold, and let 
$$
M_{F,\eps}=\norm{r^2\brak{r}^{2+\alpha}F_\eps(\cdot,c)}_{L_{Y,\gamma-\tilde\alpha}^2}+\| r^2 F_{\ast,\eps}(\cdot, c)\|_{L_{Y,\gamma-\tilde\alpha}^2}. 
$$ 
Then there exists a sequence $\eps_j\to 0$ such that
\begin{align} 
\| Y_{\eps_j}(\cdot,c)\|_{L^2_{Y,\gamma}} > jM_{F,\eps_j}.
\end{align}
By replacing
$$
Y_{\eps_j}(r,c)\mapsto \frac{Y_{\eps_j}(\cdot,c)}{\| Y_{\eps_j}(\cdot,c)\|_{L^2_{Y,\gamma}} }, \qquad F_{\eps_j}(r,c)\mapsto \frac{F_{\eps_j}(r,c)}{\| Y_{\eps_j}(\cdot,c)\|_{L^2_{Y,\gamma}} }, \qquad F_{\ast,\eps_j}(r,c)\mapsto \frac{F_{\ast,\eps_j}(r,c)}{\| Y_{\eps_j}(\cdot,c)\|_{L^2_{Y,\gamma}} },
$$
we may assume that $\| Y_{\eps_j}(\cdot,c)\|_{L^2_{Y,\gamma}}=1$ and $\| F_{\eps_j}(\cdot,c)\|_{L^2_{Y,\gamma-\tilde\alpha}}+\| F_{\ast,\eps_j}(\cdot,c)\|_{L^2_{Y,\gamma-\tilde\alpha}} \to 0$. Hence, thanks to standard
compactness arguments, Lemma \ref{lem:left1} provides the existence of a subsequence (not relabeled) such that $Y_{\eps_j}$ weakly in $L^2_{Y,\gamma/2}$ and strongly in $L^2_{Y,\gamma}$, while 
$r\partial_r Y_{\eps_j}$ converges weakly in $L^2_{Y,\gamma/2}$. 
 Note that in that case $c\in (0,u(0))$, we have that $r_c\to \infty$ as $\eps\to 0$ thanks to \eqref{eq:alpgjskfg}, which is equivalent to say that $c\to 0$. 
Hence, the limit $Y\in H^1_{Y,\gamma/2}$, is a weak solution to the \emph{homogeneous} Rayleigh problem (due to $ F_{\eps_j}\to 0$), for some $c\leq 0$, and satisfies $\|Y\|_{L^2_{Y,\gamma}}=1$.  
However, Lemma \ref{eq:homoleft} implies that $Y=0$, which is a contradiction. With \eqref{eq:L2outleft} at our disposal, the derivative estimate  follows from \eqref{eq:outleftder}. 
\end{proof}

\subsubsection{Proof of (\ref{eq:Ybd5})}
Assume $c<-R_{\gamma_*}$, where $R_{\gamma_*}>0$ is fixed below.
Using the same ideas as in \eqref{eq:outoutright1}, we have
\begin{align}
&\int_0^\infty \abs{\partial_r \frac{Y_\eps (r,c)}{r^{\gamma_*}} }^2 r \dd r + \left[k^2 - \frac{1}{4} - \gamma_*(\gamma_*-1)\right]\int_0^\infty \frac{\abs{Y_\eps(r,c)}^2}{r^{2\gamma_*+2}} r \dd r  \notag\\
 &\qquad \qquad\lesssim  \int_0^\infty \frac{\beta(r)(u(r)-c)}{(u(r)-c)^2 + \eps^2} \frac{\abs{Y_\eps(r,c)}^2}{r^{2\gamma_*}} r \dd r 
 +\frac{1}{c^2}\int_0^\infty\frac{|F_\eps(r,c)|^2}{r^{2\gamma_*-3}} \dd r+\int_0^\infty \frac{|F_{\ast,\eps}(r,c)|^2}{r^{2\gamma_*-3}} \dd r.\label{eq:outoutleft1}
\end{align}
Now, since 
$$
\beta(r) r^2\leq C_\beta
$$
by standard estimates, 
\begin{align*}
\int_0^\infty \frac{\beta(r)(u(r)-c)}{(u(r)-c)^2 + \eps^2} \frac{\abs{Y_\eps(r,c)}^2}{r^{2\gamma_*}} r \dd r 
\leq\int_0^\infty \frac{\beta(r)r^2 }{u(r)-c} \frac{\abs{Y_\eps(r,c)}^2}{r^{2\gamma_*+2}} r\dd r \leq \frac{C_\beta}{R_{\gamma_*}}
\int_0^\infty \frac{\abs{Y_\eps(r,c)}^2}{r^{2\gamma_*+2}} r\dd r.
\end{align*}
Thus, by taking
$$
R_{\gamma_*}\gg \frac{C_\beta}{k^2 - \frac{1}{4} - \gamma_*(\gamma_*-1)},
$$
estimate \eqref{eq:Ybd5} follows from \eqref{eq:outoutleft1}.

\subsection{Vanishing of  \texorpdfstring{$f_E$}{TEXT}}\label{sub:j0E}
The function $f_E$ from \eqref{eq:fE}, can be rewritten using the notation introduced in \eqref{eq:XandA} as
\begin{align}\label{eq:appefE}
f_{E}^\eps(t,r)&=\frac{\beta(r) }{2\pi i\sqrt{r}}\int_{\RR} \e^{ik(u(r)-c)t}\left[\frac{u(r)-c}{(u(r)-c)^2+\eps^2}X(r,c,\eps)
+\frac{i\eps A(r,c,\eps)}{(u(r)-c)^2+\eps^2} \right](1-\chi_\sigma (c))\dd c,
\end{align}
where
\begin{align}\label{eq:Xfdasfa}
\Ray_{+}X=\frac{1}{u(r) - c- i\eps}\frac{2i\eps [F(r)-\beta(r)\hPhi(r,c-i\eps)]}{u(r) - c+i\eps},
\end{align}
and 
\begin{align}\label{eq:Afdasfa}
\Ray_{+}A=\frac{1}{u(r) - c- i\eps}\frac{2(u(r)-c)F(r)+2i\eps \beta(r)\hPhi(r,c-i\eps)}{u(r) - c+i\eps} +2F_{\ast}(r).
\end{align}
According to \eqref{eq:Ybd3} and \eqref{eq:Ybd5}, if $c>u(0)+1/2$ or $c< R_{\gamma}$ we use Lemma \ref{lem:BasicVort} and the fact that $1\lesssim|u(r)-c|^2$ to obtain that
\begin{align}
&\norm{  \frac{ X(\cdot,c,\eps)}{\min\{r^{k+1/2-\gamma},r^{-k+1/2+\gamma}\}}}^2_{L^\infty}+\norm{X(\cdot,c,\eps)}_{L_{Y,\gamma}^2}^2+\norm{r\de_rX(\cdot,c,\eps)}_{L_{Y,\gamma}^2}^2\notag\\
&\qquad\lesssim_{\gamma} \frac{\eps^2}{c^2}\| r^2 F\|^2_{L_{Y,\gamma}^2}+\frac{\eps^2}{c^2}\|  r^2 \beta\hPhi(\cdot,c-i\eps) \|^2_{L_{Y,\gamma}^2}
\notag\\
&\qquad\lesssim_{\gamma} \frac{\eps^2}{c^4} \| r^2 F\|^2_{L_{Y,\gamma}^2}+\frac{\eps^2}{c^2}\|  \hPhi(\cdot,c-i\eps) \|^2_{L_{Y,\gamma}^2}\notag\\
&\qquad\lesssim_{\gamma} \frac{\eps^2}{c^4}\left[ \| r^2 F\|^2_{L_{Y,\gamma}^2}+\| r^2 F_{\ast}\|^2_{L_{Y,\gamma}^2}\right]\label{eq:elia1},
\end{align}
and, similarly
\begin{align}
\norm{  \frac{ A(\cdot,c,\eps)}{\min\{r^{k+1/2-\gamma},r^{-k+1/2+\gamma}\}}}^2_{L^\infty}
&\qquad\lesssim_{\gamma}  \|r^2 F\|^2_{L_{Y,\gamma}^2}+\| r^2 F_{\ast}\|^2_{L_{Y,\gamma}^2}\label{eq:elia2}.
\end{align}

\begin{proposition}\label{prop:fEvanish}
For $\delta$ sufficiently small, there holds
\begin{align*}
\lim_{\eps\to0}\|f_{E}^\eps(t,\cdot)\|_{L^2_{f,\delta}}=0,
\end{align*}
for every $t\geq 0$.
\end{proposition}

\begin{proof}[Proof of Proposition~\ref{prop:fEvanish}]
Recalling the shape of the weight \eqref{eq:weightf}, we split into different cases. If $r\leq 1$,  we 
use \eqref{eq:elia1} and \eqref{eq:elia2} to obtain
\begin{align*}
\frac{|f_{E}^\eps(t,r)|}{r^{k+1/2-\delta}}
&\lesssim  \frac{\beta(r) }{r^{k+1-\delta}}\int_{\RR}\frac{|X(r,c,\eps)|}{|c|}(1-\chi_\sigma (c))\dd c+ \eps\frac{\beta(r) }{r^{k+1-\delta}}\int_{\RR}\frac{|A(r,c,\eps)|}{|c|^2}(1-\chi_\sigma (c))\dd c \notag\\
&\lesssim_{\delta,F,F_\ast}  \frac{\beta(r)}{r^{1/2-\delta/2}} \eps.
\end{align*}
If $r\geq 1$, we get in a similar manner that
\begin{align*}
\frac{|f_{E}^\eps(t,r)|}{r^{-k+1/2-4+\delta}}&\lesssim  \frac{\beta(r) }{r^{-k-3+\delta}}\int_{\RR}\frac{|X(r,c,\eps)|}{|c|}(1-\chi_\sigma (c))\dd c+ \eps\frac{\beta(r) }{r^{-k-3+\delta}}\int_{\RR}\frac{|A(r,c,\eps)|}{|c|^2}(1-\chi_\sigma (c))\dd c \notag\\
&\lesssim_{\delta,F,F_\ast}  \frac{\beta(r)}{r^{-7/2+\delta/2}} \eps,
\end{align*}
Upon squaring and integrating over $r\in(0,\infty)$ and using Lemma \ref{lem:BasicVort}, we deduce that
\begin{align*}
\|f_{E}(t,\cdot,\eps)\|_{L^2_{f,\delta}}\lesssim_{\delta,F,F_\ast} \eps,
\end{align*}
concluding the proof.
\end{proof}

\subsection{Vanishing of \texorpdfstring{$f_{S}$}{TEXT}}\label{sub:j0S}
We next treat $f_{S}$ in \eqref{eq:fS}, which we rewrite here 
\begin{align}\label{eq:fS1}
f_{S}^\eps(t,r)= \frac{\beta(r) }{2\pi i\sqrt{r}}\int_{-R_\delta}^{u(0)+1} \e^{ik(u(r)-c)t}\left[\frac{\hPhi(r,c+i\eps)}{u(r)-c-i\eps}-\frac{\hPhi(r,c-i\eps)}{u(r)-c+i\eps} \right]\chi_\sigma (c)(1-\chi_I(r_c))\dd c.
\end{align}
Recall that
\begin{align}\label{eq:rcarea}
1-\chi_I(r_c) \neq 0 \qquad \text{if}\qquad    r_c\leq  \eps^\frac{1}{2+\alpha}\qquad \text{or}\qquad r_c\geq \frac12\eps^{-\frac{1}{2+\alpha}}.
\end{align}
We prove the following result.

\begin{proposition}\label{prop:fSvanish}
For $\delta$ sufficiently small, there holds 
\begin{align*}
\lim_{\eps\to0}\|f_{S}^\eps(t,\cdot)\|_{L^2_{f,\delta}}=0,
\end{align*}
for every $t\geq 0$.
\end{proposition}

\begin{proof}[Proof of Proposition~\ref{prop:fSvanish}]
Since the cut-off functions in \eqref{eq:fS1} isolate different subsets of $\RR$, we proceed case by case, using the estimates 
provided by Theorem \ref{thm:boundsonY}. 
We first note that
\begin{align*}
\frac{|f_{S}^\eps(t,r)|}{w_{f,\delta}(r)}\lesssim J_1+J_2+J_3+J_4
\end{align*}
where
\begin{align}
J_1^\eps(t,r)&=  \frac{\beta(r) }{w_{f,\delta}(r) r^{1/2}}\abs{\int_{r_c\leq  \eps^\frac{1}{2+\alpha}}^{u(0)} \e^{ik(u(r)-c)t}\left[\frac{\hPhi(r,c+i\eps)}{u(r)-c-i\eps}-\frac{\hPhi(r,c-i\eps)}{u(r)-c+i\eps} \right]\dd c},\\
J_2^\eps(t,r)&=\frac{\beta(r) }{w_{f,\delta}(r) r^{1/2}}\abs{\int_{u(0)}^{u(0)+1} \e^{ik(u(r)-c)t}\left[\frac{u(r)-c}{(u(r)-c)^2+\eps^2}X(r,c,\eps)
+\frac{i\eps A(r,c,\eps)}{(u(r)-c)^2+\eps^2} \right]\dd c},\label{eq:onlyone}\\
J_3^\eps(t,r)&=\frac{\beta(r) }{w_{f,\delta}(r) r^{1/2}}\abs{\int_{-R_\delta}^0 \e^{ik(u(r)-c)t}\left[\frac{u(r)-c}{(u(r)-c)^2+\eps^2}X(r,c,\eps)
+\frac{i\eps}{(u(r)-c)^2+\eps^2} A(r,c,\eps)\right]\dd c},\\
J_4^\eps(t,r)&=\frac{\beta(r) }{w_{f,\delta}(r) r^{1/2}}\abs{\int_{0}^{r_c\geq \frac12\eps^{-\frac{1}{2+\alpha}}} \e^{ik(u(r)-c)t}\left[\frac{\hPhi(r,c+i\eps)}{u(r)-c-i\eps}-\frac{\hPhi(r,c-i\eps)}{u(r)-c+i\eps} \right]\dd c},
\end{align}
and $X$ and $A$ satisfy \eqref{eq:Xfdasfa} and \eqref{eq:Afdasfa}, respectively. 
Let us first consider $J_1$.
The key observation is that on the support of the integrands in $r_c$, there holds for $r \leq 1$, 
\begin{align*}
r^{\delta} \lesssim \left((u(r)-u(r_c))^2 + \eps^2\right)^{\frac{\delta}{2(2+\alpha)}},
\end{align*}
whereas for $r > 1$ and $r_c < 1/4$, there holds
\begin{align*}
1 \lesssim (u(r)-u(r_c))^2 + \eps^2. 
\end{align*}
Moreover, $u'(r_c)\approx r_c$. Hence, by \eqref{eq:Ybd2} if $r\leq 1$
we have
\begin{align*}
\left|\int_{r_c\leq  \eps^\frac{1}{2+\alpha}}^{u(0)} \e^{ik(u(r)-c)t}\frac{\hPhi(r,c+i\eps)}{u(r)-c-i\eps} (1-\chi_\sigma (c))\dd c\right|
&\lesssim \int_{0}^{\eps^\frac{1}{2+\alpha}} |u'(r_c)|\frac{|\hPhi(r,c+i\eps)|}{\sqrt{(u(r)-u(r_c))^2+\eps^2}}\dd r_c\notag\\
&\lesssim \int_{0}^{\eps^\frac{1}{2+\alpha}} r_c\frac{r^{-\delta/4}|\hPhi(r,c+i\eps)|}{((u(r)-u(r_c))^2+\eps^2)^{\frac12 -\frac{\delta}{8(2+\alpha)}}}\dd r_c\notag\\
&\lesssim \frac{1}{\eps^{1-\frac{\delta}{4(2+\alpha)}}}\int_{0}^{\eps^\frac{1}{2+\alpha}} r_c r^{-\delta/4}|\hPhi(r,c+i\eps)|\dd r_c\notag\\
&\lesssim \frac{r^{k+1/2-\delta/2}}{\eps^{1-\frac{\delta}{4(2+\alpha)}}}\int_{0}^{\eps^\frac{1}{2+\alpha}} \!\!\! r_c \|r^{-k-1/2+\delta/4}\hPhi(\cdot,c+i\eps)\|_{L^\infty(0,1)}\dd r_c\notag\\
&\lesssim_{F,F_*,\delta} r^{k+1/2-\delta/2}\eps^{\frac{\delta-8\alpha}{4(2+\alpha)}},
\end{align*}
while if $r>1$ there holds
\begin{align*}
\left|\int_{r_c\leq  \eps^\frac{1}{2+\alpha}}^{u(0)} \e^{ik(u(r)-c)t}\frac{\hPhi(r,c+i\eps)}{u(r)-c-i\eps} (1-\chi_\sigma (c))\dd c\right|
&\lesssim\int_{0}^{\eps^\frac{1}{2+\alpha}} r_c |\hPhi(r,c+i\eps)|\dd r_c\notag\\
&\lesssim r^{-k+1/2+\delta/2}\int_{0}^{\eps^\frac{1}{2+\alpha}} r_c \|r^{k-1/2-\delta/2}\hPhi(\cdot,c+i\eps)\|_{L^\infty(1,\infty)}\dd r_c\notag\\
&\lesssim_{F,F_*,\delta} r^{-k+1/2+\delta/2}\eps^{\frac{2-\alpha}{2+\alpha}}.
\end{align*}
and the same holds for $\hPhi(r,c-i\eps)$.
Hence we have 
\begin{align*}
|J_1^\eps(t,r)|\lesssim_{F,F_\ast,\delta} \begin{cases}
\displaystyle\beta(r)r^{-1/2-\delta/2} \eps^{\frac{\delta-8\alpha}{4(2+\alpha)}}, \quad &r\in (0,1],\\
\displaystyle \beta(r)r^{7/2-\delta/2} \eps^{\frac{2-\alpha}{2+\alpha}}, \quad &r>1.
\end{cases}
\end{align*}
Since $\alpha<\delta/8$, thanks to Lemma \ref{lem:BasicVort} we obtain
\begin{align*}
\|J_1^\eps(t)\|_{L^2}\lesssim_{F,F_\ast,\delta}  \eps^{\frac{\delta-8\alpha}{4(2+\alpha)}},
\end{align*}
Now, $J_4$ is very similar. On
its support, we use \eqref{eq:ciccio1}, \eqref{eq:ciccio2}, the fact that $u'(r_c)\sim r_c^{-3}$ and \eqref{eq:Ybd4}, to obtain
for $r\leq 1$ that
\begin{align*}
&\left|\int_{0}^{r_c\geq \frac12\eps^{-\frac{1}{2+\alpha}}} \e^{ik(u(r)-c)t}\frac{\hPhi(r,c+i\eps)}{u(r)-c-i\eps} (1-\chi_\sigma (c))\dd c\right|
\lesssim \int_{\frac12\eps^{-\frac{1}{2+\alpha}}}^{\infty} |u'(r_c)|\frac{|\hPhi(r,c+i\eps)|}{\sqrt{(u(r)-u(r_c))^2+\eps^2}}\dd r_c\notag\\
&\qquad\lesssim \brak{r}^{2+\alpha}\int_{\frac12\eps^{-\frac{1}{2+\alpha}}}^{\infty}\frac{|\hPhi(r,c+i\eps)|}{r_c^3}\dd r_c\notag\\
&\qquad\lesssim r^{k+1/2-\delta/2}\brak{r}^{2+\alpha}\int_{\frac12\eps^{-\frac{1}{2+\alpha}}}^{\infty} \frac{1}{r_c^3} \|r^{-k-1/2+\delta/2}\hPhi(\cdot,c+i\eps)\|_{L^\infty(0,1)}\dd r_c\notag\\
&\qquad\lesssim_{F,F_\ast,\delta} r^{k+1/2-\delta/2}\brak{r}^{2+\alpha}\eps^{\frac{2}{2+\alpha}},
\end{align*}
while for $r>1$, a similar argument implies
\begin{align*}
\left|\int_{0}^{r_c\geq \frac12\eps^{-\frac{1}{2+\alpha}}} \e^{ik(u(r)-c)t}\frac{\hPhi(r,c+i\eps)}{u(r)-c-i\eps} (1-\chi_\sigma (c))\dd c\right|
\lesssim_{F,F_\ast,\delta} r^{-k+1/2+\delta/2}\brak{r}^{2+\alpha}\eps^{\frac{2}{2+\alpha}}.
\end{align*}
Similar estimates hold for $\hPhi(r,c-i\eps)$, and thus in view of Lemma \ref{lem:BasicVort} we have
\begin{align*}
\|J_4^\eps(t)\|_{L^2} \lesssim_{F,F_\ast,\delta} \eps^{\frac{2}{2+\alpha}}.
\end{align*}
We now deal with $J_2$ and $J_3$. For $J_2$, in analogy with \eqref{eq:elia1} and \eqref{eq:elia2} and using \eqref{eq:Ybd1},  for $c\in (u(0),u(0)+1)$ we obtain that
\begin{align}
&\norm{  \frac{ X(\cdot,c,\eps)}{\min\{r^{k+1/2-\gamma},r^{-k+1/2+\gamma}\}}}^2_{L^\infty}+\norm{X(\cdot,c,\eps)}_{L_{Y,\gamma}^2}^2+\norm{r\de_rX(\cdot,c,\eps)}_{L_{Y,\gamma}^2}^2\notag\\
&\qquad\lesssim_{\gamma}
\eps^2\norm{\frac{\brak{r}^2F}{\sqrt{(u(r)-c)^2+\eps^2}}}_{L_{Y,\gamma}^2}^2+
\eps^2\norm{\frac{\brak{r}^2\beta \hPhi(\cdot,c-i\eps)}{\sqrt{(u(r)-c)^2+\eps^2}}}_{L_{Y,\gamma}^2}^2  \notag\\
&\qquad\lesssim_{\gamma}\frac{\eps^{\delta/4}}{|u(0)-c|^{\delta/4}}\left[ \norm{\brak{r}^{2}F}_{L_{Y,\gamma}^2}^2+\norm{\hPhi(\cdot,c-i\eps)}_{L_{Y,\gamma}^2}^2\right]\notag\\
&\qquad\lesssim_{\gamma}\frac{\eps^{\delta/4}}{|u(0)-c|^{\delta/4}}\left[ \norm{\brak{r}^{2}F}_{L_{Y,\gamma}^2}^2+ \norm{r^{2}F_\ast}_{L_{Y,\gamma}^2}^2\right],\label{eq:elia3}
\end{align}
and, similarly
\begin{align}
\norm{  \frac{ A(\cdot,c,\eps)}{\min\{r^{k+1/2-\gamma},r^{-k+1/2+\gamma}\}}}^2_{L^\infty}
&\lesssim_{\gamma} \norm{\brak{r}^{2}F}_{L_{Y,\gamma}^2}^2+ \norm{r^{2}F_\ast}_{L_{Y,\gamma}^2}^2\label{eq:elia4}.
\end{align}
If $r\leq 1$, taking into account that $u(r)-u(0)\sim r^2$ when $r\to 0$, we appeal to
 \eqref{eq:elia3} and \eqref{eq:elia4}  to obtain
\begin{align}
J_2^\eps(t,r)&\lesssim  \frac{\beta(r) }{r^{k+1-\delta}}\int_{u(0)}^{u(0)+1}\frac{|X(r,c,\eps)|}{|u(r)-u(0)|^{\delta/4}  |u(0)-c|^{1-\delta/4}  }\dd c\notag\\
&\quad+ \eps^{\delta/8}\frac{\beta(r) }{r^{k+1-\delta}}\int_{u(0)}^{u(0)+1}\frac{|A(r,c,\eps)|}{|u(r)-u(0)|^{\delta/4}  |u(0)-c|^{1-\delta/8} }\dd c\notag\\
&\lesssim  \frac{\beta(r) }{r^{1/2-\delta/4} }\int_{u(0)}^{u(0)+1}\frac{\|r^{-k-1/2+\delta/4}X(r,c,\eps)\|_{L^\infty(0,1)}}{|u(0)-c|^{1-\delta/4}  }\dd c\notag\\
&\quad+ \eps^{\delta/8} \frac{\beta(r)}{r^{1/2-\delta/4}}\int_{u(0)}^{u(0)+1}\frac{\|r^{-k-1/2+\delta/4}A(r,c,\eps)\|_{L^\infty(0,1)}}{ |u(0)-c|^{1-\delta/8}}\dd c\notag\\
&\lesssim_{\delta,F,F_\ast}  \frac{\beta(r)}{r^{1/2-\delta/4}} \eps^{\delta/8}.
\end{align}
If $r\geq 1$, we use an even simpler version of \eqref{eq:elia3} to get 
\begin{align*}
J_2^\eps(t,r)&\lesssim  \frac{\beta(r) }{r^{-k-3+\delta}}\int_{\RR}|X(r,c,\eps)|\dd c+ \eps\frac{\beta(r) }{r^{-k-3+\delta}}\int_{\RR}|A(r,c,\eps)|\dd c\lesssim_{\delta,F,F_\ast}  \frac{\beta(r)}{r^{-7/2+\delta/2}} \eps.
\end{align*}
As a consequence,
\begin{align*}
\|J_2^\eps(t)\|_{L^2}\lesssim_{F,F_\ast,\delta} \eps^{\delta/8}.
\end{align*}
The treatment of $J_3$ is similar. In this case,
using \eqref{eq:Ybd4} with $\alpha=0$,  we obtain that
\begin{align}
&\norm{  \frac{ X(\cdot,c,\eps)}{\min\{r^{k+1/2-\gamma},r^{-k+1/2+\gamma}\}}}^2_{L^\infty}+\norm{X(\cdot,c,\eps)}_{L_{Y,\gamma}^2}^2+\norm{r\de_rX(\cdot,c,\eps)}_{L_{Y,\gamma}^2}^2\notag\\
&\qquad\lesssim_{\gamma}  \eps^2 \norm{\frac{r^2\brak{r}^{2}F}{\sqrt{(u(r)-c)^2+\eps^2}}}_{L_{Y,\gamma-\tilde\alpha}^2}^2
+\eps^2 \norm{\frac{r^2\brak{r}^{2}\beta \hPhi(\cdot,c-i\eps)}{\sqrt{(u(r)-c)^2+\eps^2}}}_{L_{Y,\gamma-\tilde\alpha}^2}^2.
\end{align}
From $u(r)\sim r^{-2}$ when $r\to \infty$ and the fact that $c\leq0$, we have
\begin{align*}
\eps^2 \norm{\frac{r^2\brak{r}^{2}F}{\sqrt{(u(r)-c)^2+\eps^2}}}_{L_{Y,\gamma-\tilde\alpha}^2}^2
&\lesssim\eps^2\int_0^1 \frac{r^4|F(r)|^2}{r^{2k+2-2\gamma+2\tilde\alpha}((u(r)-c)^2+\eps^2)}\dd r\notag\\
&\quad+\eps^2\int_1^\infty \frac{r^8|F(r)|^2}{r^{-2k+2+2\gamma-2\tilde\alpha}((u(r)-c)^2+\eps^2)}\dd r\notag\\
&\lesssim_\gamma\eps^{\delta/8}\int_0^1 \frac{r^4|F(r)|^2}{r^{2k+2-2\gamma+2\tilde\alpha}}\dd r
+\eps^{\delta/8}\int_1^\infty \frac{r^{8+\delta/4}|F(r)|^2}{r^{-2k+2+2\gamma-2\tilde\alpha}}\dd r\notag\\
&\lesssim_{\gamma}\eps^{\delta/8}\norm{r^2\brak{r}^{2+\delta/8}F(r)}_{L_{Y,\gamma-\tilde\alpha}^2}^2.
\end{align*}
Analogously, using Lemma \ref{lem:BasicVort}, we have
\begin{align*}
\eps^2 \norm{\frac{r^2\brak{r}^{2}\beta \hPhi(\cdot,c-i\eps)}{\sqrt{(u(r)-c)^2+\eps^2}}}_{L_{Y,\gamma-\tilde\alpha}^2}^2\!\!\!\!\!\!\!&\lesssim_\gamma
\eps^{\delta/8}\int_0^1 \frac{r^{4-2\tilde\alpha}|\hPhi(r,c-i\eps)|^2}{r^{2k+2-2\gamma}}  \dd r+\eps^{\delta/8}\int_1^\infty \frac{\beta(r)^2r^{8+\delta/4+2\tilde\alpha}|\hPhi(r,c-i\eps)|^2}{r^{-2k+2+2\gamma}}  \dd r\notag\\
&\lesssim_\gamma \eps^{\delta/8} \|  \hPhi(\cdot,c-i\eps) \|^2_{L_{Y,\gamma}^2}
 \lesssim_{\gamma}\eps^{\delta/8} \left[\norm{r^2\brak{r}^{2}F(r)}_{L_{Y,\gamma-\tilde\alpha}^2}^2+\| r^2 F_{\ast}(r)\|^2_{L_{Y,\gamma-\tilde\alpha}^2}\right],
\end{align*}
implying that, with the choice $\tilde\alpha=\gamma/2=\delta/4$, we have the control
\begin{align}\label{eq:elia5}
\norm{  \frac{ X(\cdot,c,\eps)}{\min\{r^{k+1/2-\delta/2},r^{-k+1/2+\delta/2}\}}}^2_{L^\infty}\lesssim_\gamma \eps^{\delta/8}\left[\norm{r^2\brak{r}^{2+\delta/8}F(r)}_{L_{Y,\delta/4}^2}^2+\| r^2 F_{\ast}(r)\|^2_{L_{Y,\delta/4}^2}\right]
\end{align}
Regarding $A$, we argue in the same way and arrive at 
\begin{align}
\norm{  \frac{ A(\cdot,c,\eps)}{\min\{r^{k+1/2-\delta/2},r^{-k+1/2+\delta/2}\}}}^2_{L^\infty}
&\lesssim_{\gamma} \norm{r^2\brak{r}^{2}F(r)}_{L_{Y,\delta/4}^2}^2+\| r^2 F_{\ast}(r)\|^2_{L_{Y,\delta/4}^2}.\label{eq:elia6}
\end{align}
In order to control $J_3$, we again consider two cases. If $r\leq 1$,
\begin{align*}
J_3^\eps(t,r)&\lesssim  \frac{\beta(r) }{r^{k+1-\delta}}\int_{-R_\delta}^{0}|X(r,c,\eps)|\dd c+ \eps\frac{\beta(r) }{r^{k+1-\delta}}\int_{-R_\delta}^{0}|A(r,c,\eps)|\dd c\notag\\
&\lesssim \frac{\beta(r) }{r^{1/2-\delta/2} }\int_{-R_\delta}^{0}\|r^{-k-1/2+\delta/2}X(r,c,\eps)\|_{L^\infty}\dd c
+ \eps \frac{\beta(r)}{r^{1/2-\delta/2}}\int_{-R_\delta}^{0}\|r^{-k-1/2+\delta/2}A(r,c,\eps)\|_{L^\infty}\dd c\notag\\
&\lesssim_{\delta,F,F_\ast}  \frac{\beta(r)}{r^{1/2-\delta/2}} \eps^{\delta/16},
\end{align*}
while if $r\geq 1$, since $u(r)\sim r^{-2}$ as $r\to \infty$, we use \eqref{eq:elia5} and \eqref{eq:elia6} to deduce that
\begin{align*}
J_3^\eps(t,r)&\lesssim  \frac{\beta(r) }{r^{-k+1+\delta-4}}\eps^{-\delta/32}\int_{-R_\delta}^{0}\frac{|X(r,c,\eps)|}{|c|^{1-\delta/32}} \dd c
+ \eps^{\delta/16}\frac{\beta(r) }{r^{-k+1+\delta-4}}\int_{-R_\delta}^{0}\frac{|A(r,c,\eps)|}{|u(r)|^{\delta/8}  |c|^{1-\delta/16} }\dd c\notag\\
&\lesssim  \frac{\beta(r)\eps^{-\delta/32} }{r^{-7/2+\delta/2} }\int_{-R_\delta}^{0}\frac{\|r^{k-1/2-\delta/2}X(r,c,\eps)\|_{L^\infty}}{|c|^{1-\delta/32} }\dd c
+ \frac{\beta(r)\eps^{\delta/8} }{r^{-7/2+\delta/4}}\int_{-R_\delta}^{0}\frac{\|r^{k-1/2-\delta/2}A(r,c,\eps)\|_{L^\infty}}{ |c|^{1-\delta/16}}\dd c\notag\\
&\lesssim_{\delta,F,F_\ast}  \frac{\beta(r)}{r^{-7/2+\delta/4}} \eps^{\delta/32}.
\end{align*}
Hence,
\begin{align*}
\|J_3^\eps(t)\|_{L^2}\lesssim_{F,F_\ast,\delta} \eps^{\delta/32},
\end{align*}
which concludes the proof.
\end{proof}

\subsection{Vanishing of higher derivatives}
The aim of this section is to prove analogous results to those of Propositions \ref{prop:fEvanish}- \ref{prop:fSvanish}.
\begin{proposition}\label{prop:edvanish}
Let $\delta$ be fixed sufficiently small, and let $j\in \{1,\ldots,k\}$. Then
\begin{align*}
\lim_{\eps\to0}\left[\|(r\de_r)^jf_{E}^\eps(t,\cdot)\|_{L^2_{f,\delta}}+\lim_{\eps\to0}\|(r\de_r)^jf_{S}^\eps(t,\cdot)\|_{L^2_{f,\delta}}\right]=0,
\end{align*}
for every $t\geq 0$.
\end{proposition}

The proof is based on the iteration scheme laid out in Lemma \ref{lem:IterSchemeIntro}, and appropriately choosing $\gamma\in (0,2k)$ in Theorem \ref{thm:boundsonY}, depending on the derivative index $j$ and the small parameter $\delta$. We illustrate the proof only in one case to handle higher derivatives of $f_S$, for $k=2$ and $j=1,2$. As in \eqref{eq:rdrf1f2}, we have
\begin{align}
(r\de_r)^j f_{S}^\eps(t,r)= &\frac{1 }{2\pi i}\int_{-R_\delta}^{u(0)+1} \e^{ik(u(r)-c)t}\Bigg[\frac{1}{u(r)-c-i\eps}(ru'(r)\de_G)^j\left(\frac{\beta(r)}{\sqrt{r}}\hPhi(r,c+i\eps)\chi_\sigma (c)(1-\chi_I(r_c))\right)\notag\\
&-\frac{1}{u(r)-c+i\eps}(ru'(r)\de_G)^j\left(\frac{\beta(r)}{\sqrt{r}}\hPhi(r,c-i\eps)\chi_\sigma (c)(1-\chi_I(r_c))\right) \Bigg]\dd c\label{eq:derfS1}.
\end{align}

\subsubsection{Proof for \texorpdfstring{$j\in \{1,\ldots, k-1\}$}{TEXT}}\label{sub:jlessk}
Arguing as in the proof of Proposition \ref{prop:fSvanish} and taking $j=1$, 
an important term to bound (somewhat analogous to \eqref{eq:onlyone})
reads as
\begin{align}
J_2^\eps(t,r)= \frac{\beta(r) ru'(r)}{w_{f,\delta}(r) r^{1/2}}\abs{\int_{u(0)}^{u(0)+1} \e^{ik(u(r)-c)t}\left[\frac{u(r)-c}{(u(r)-c)^2+\eps^2}\de_GX(r,c,\eps)
+\frac{i\eps  \de_GA(r,c,\eps)}{(u(r)-c)^2+\eps^2}\right]\dd c}\label{eq:corvo}.
\end{align}
According to Lemma \ref{lem:IterSchemeIntro}, we have
\begin{align}
\Ray_+ \de_G X=&\frac{2i\eps}{(u(r)-c)^2 + \eps^2 }\Bigg(\frac{1}{u'(r)}\de_rF(r)-\beta(r)\de_G\hPhi(r,c-i\eps)-\frac{\beta'(r)}{u'(r)}\hPhi(r,c-i\eps)
\\ &\quad -2\frac{u''(r)}{(u'(r))^2}[F(r)-\beta(r)\hPhi(r,c-i\eps)] \Bigg) \label{eq:XFtermout}\\
&+ \frac{1}{u-c-i\eps}\left( 2\frac{u''(r)\beta(r)}{(u'(r))^2}-\frac{\beta'(r)}{u'(r)} \right)X \label{eq:XEtermout}\\
&+ \de_{rr}\left(\frac{1}{u'(r)}\right)\de_rX+2\frac{1/4-k^2}{r^3(u'(r))^2}\left(u'(r)+ru''(r)\right)X\label{eq:XRtermout}
\end{align}
and
\begin{align}
\Ray_{-} \partial_G\hPhi  =& \frac{1}{u-c+ i\eps}\left(\left( 2\frac{u''(r)\beta(r)}{(u'(r))^2}-\frac{\beta'(r)}{u'(r)} \right)\hPhi+\frac{1}{u'(r)}\de_rF(r)-2\frac{u''(r)}{(u'(r))^2}F(r)\right)\label{eq:phiEtermout}\\ 
&+ \de_{rr}\left(\frac{1}{u'(r)}\right)\de_r\hPhi+2\frac{1/4-k^2}{r^3(u'(r))^2}\left(u'(r)+ru''(r)\right)\hPhi -2\frac{u''(r)}{(u'(r))^2}F_\ast(r)+\de_GF_\ast(r)\label{eq:phiRtermout}. 
\end{align}
Next, we show that the terms in \eqref{eq:XRtermout} and \eqref{eq:phiRtermout} are bounded in the norm 
$\| r^2 \cdot\|_{L_{Y,\gamma}^2}$ as they correspond to the term $F_{\ast,\eps}$ in Theorem \ref{thm:boundsonY}. The terms containing 
$F_\ast$ are harmless, since they are compactly supported away from $r=0$. As for the other terms, notice that from Lemma \ref{lem:RFErecurse}, we have
\begin{align*}
\frac{\left(u'(r)+ru''(r)\right)}{r^3(u'(r))^2}\sim
\begin{cases}
r^{-4}, \quad  &r\to 0,\\
1, &r\to\infty,
\end{cases}
\qquad 
\de_{rr}\left(\frac{1}{u'(r)}\right)\sim
\begin{cases}
r^{-3}, \quad  &r\to 0,\\
r, &r\to\infty.
\end{cases}
\end{align*}
Thus, we can choose $\gamma=2+\gamma'\in (0,2k)$, deduce that
\begin{align}
\norm{r^2 \left[\de_{rr}\left(\frac{1}{u'(r)}\right)\de_r\hPhi+2\frac{1/4-k^2}{r^3(u'(r))^2}\left(u'(r)+ru''(r)\right)\hPhi\right]}^2_{L_{Y,2+\gamma'}^2}
\lesssim_\delta\| r\de_r\hPhi\|^2_{L_{Y,\gamma'}^2}+\|\hPhi\|^2_{L_{Y,\gamma'}^2}\label{eq:corvo0}
\end{align}
and
\begin{align}
\norm{r^2 \left[\de_{rr}\left(\frac{1}{u'(r)}\right)\de_rX+2\frac{1/4-k^2}{r^3(u'(r))^2}\left(u'(r)+ru''(r)\right)X\right]}^2_{L_{Y,2+\gamma'}^2}
\lesssim_\delta\| r\de_rX\|^2_{L_{Y,\gamma'}^2}+\|X\|^2_{L_{Y,\gamma'}^2}\label{eq:corvo1}.
\end{align}
To control the terms in \eqref{eq:XEtermout} and \eqref{eq:phiEtermout}, we note Lemma \ref{lem:RFErecurse}, 
\begin{align*}
2\frac{u''(r)\beta(r)}{(u'(r))^2}-\frac{\beta'(r)}{u'(r)} \sim
\begin{cases}
r^{-2}, \quad  &r\to 0,\\
r^{-4}, &r\to\infty,
\end{cases}
\quad 
\frac{u''(r)}{(u'(r))^2}\sim
\begin{cases}
r^{-2}, \quad  &r\to 0,\\
r^2, &r\to\infty,
\end{cases}
\quad 
\frac{1}{u'(r)}\sim
\begin{cases}
r^{-1}, \quad  &r\to 0,\\
r^3, &r\to\infty.
\end{cases}\notag
\end{align*}
The point is that these coefficient are less singular at the origin by a power of $r^{-2}$, and therefore all the norms appearing in Theorem 
 \ref{thm:boundsonY} for $F_\eps$ are finite even without a gain of $r^2$ at the origin, upon choosing again $\gamma=2+\gamma'$. 
In the specific case when $c\in (u(0),u(0)+1)$, we use \eqref{eq:Ybd1} to have
\begin{align}
&\norm{\brak{r}^2 \left(\left( 2\frac{u''(r)\beta(r)}{(u'(r))^2}-\frac{\beta'(r)}{u'(r)} \right)\hPhi+\frac{1}{u'(r)}\de_rF(r)-2\frac{u''(r)}{(u'(r))^2}F(r)\right)}^2_{L_{Y,2+\gamma'}^2}\notag\\
&\qquad\qquad\lesssim_\delta\| r\de_r\hPhi\|^2_{L_{Y,\gamma'}^2}+\|\hPhi\|^2_{L_{Y,\gamma'}^2}+\|\brak{r}^2 r\de_rF\|^2_{L_{Y,\gamma'}^2}+\|\brak{r}^2F\|^2_{L_{Y,\gamma'}^2},\label{eq:corvo2}
\end{align}
and
\begin{align}
\norm{\brak{r}^2 \left( 2\frac{u''(r)\beta(r)}{(u'(r))^2}-\frac{\beta'(r)}{u'(r)} \right)X}^2_{L_{Y,2+\gamma'}^2}. 
\lesssim_\delta\| r\de_rX\|^2_{L_{Y,\gamma'}^2}+\|X\|^2_{L_{Y,\gamma'}^2}. \label{eq:corvo3}
\end{align}
Lastly, the terms in \eqref{eq:XFtermout} are treated as in \eqref{eq:elia3}. Accordingly, using \eqref{eq:corvo1} we obtain the bounds
\begin{align*}
&\norm{  \frac{ \de_G\hPhi(\cdot,c-i\eps)}{\min\{r^{k+1/2-2-\gamma'},r^{-k+1/2+2+\gamma'}\}}}^2_{L^\infty}+\norm{\de_G\hPhi(\cdot,c-i\eps)}_{L_{Y,2+\gamma'}^2}^2+\norm{r\de_r\de_G\hPhi(\cdot,c-i\eps)}_{L_{Y,2+\gamma'}^2}^2\notag\\
&\qquad\lesssim \| r\de_r\hPhi\|^2_{L_{Y,\gamma'}^2}+\|\hPhi\|^2_{L_{Y,\gamma'}^2}+\|\brak{r}^2 r\de_rF\|^2_{L_{Y,\gamma'}^2}+\|\brak{r}^2F\|^2_{L_{Y,\gamma'}^2},
\end{align*}
and
\begin{align}
&\norm{  \frac{ \de_GX(\cdot,c,\eps)}{\min\{r^{k+1/2-2-\gamma'},r^{-k+1/2+2+\gamma'}\}}}^2_{L^\infty}+\norm{\de_GX(\cdot,c,\eps)}_{L_{Y,2+\gamma'}^2}^2+\norm{r\de_r\de_GX(\cdot,c,\eps)}_{L_{Y,2+\gamma'}^2}^2\notag\\
&\qquad\lesssim \frac{\eps^{\delta/4}}{|u(0)-c|^{\delta/4}}\left[\| \brak{r}^2F\|^2_{Y,\gamma'}+\| \brak{r}^2r\de_rF\|^2_{Y,\gamma'}+\|\hPhi\|^2_{Y,\gamma'}+ \|\de_G\hPhi\|^2_{Y,\gamma'}\right]+\| r\de_rX\|^2_{L_{Y,\gamma'}^2}+\|X\|^2_{L_{Y,\gamma'}^2},
\end{align}
which, using  \eqref{eq:elia3} and \eqref{eq:elia4}, they imply that
\begin{align*}
&\norm{  \frac{ \de_G\hPhi(\cdot,c-i\eps)}{\min\{r^{k+1/2-2-\gamma'},r^{-k+1/2+2+\gamma'}\}}}^2_{L^\infty}+\norm{\de_G\hPhi(\cdot,c-i\eps)}_{L_{Y,2+\gamma'}^2}^2+\norm{r\de_r\de_G\hPhi(\cdot,c-i\eps)}_{L_{Y,2+\gamma'}^2}^2\notag\\
&\qquad\lesssim \|\brak{r}^2 r\de_rF\|^2_{L_{Y,\gamma'}^2}+\|\brak{r}^2F\|^2_{L_{Y,\gamma'}^2},
\end{align*}
and
\begin{align}
&\norm{  \frac{ \de_GX(\cdot,c,\eps)}{\min\{r^{k+1/2-2-\gamma'},r^{-k+1/2+2+\gamma'}\}}}^2_{L^\infty}+\norm{\de_GX(\cdot,c,\eps)}_{L_{Y,2+\gamma'}^2}^2+\norm{r\de_r\de_GX(\cdot,c,\eps)}_{L_{Y,2+\gamma'}^2}^2\notag\\
&\qquad\lesssim \frac{\eps^{\delta/4}}{|u(0)-c|^{\delta/4}}\left[\| \brak{r}^2F\|^2_{Y,\gamma'}+\|\brak{r}^2 r\de_rF\|^2_{Y,\gamma'}\right].\label{eq:atrae}
\end{align}
Similarly, by essentially arguing that $A(\cdot,c,\eps)=X(\cdot,c,\eps)+2\hPhi(\cdot,c-i\eps)$, we obtain
\begin{align*}
&\norm{  \frac{ \de_GA(\cdot,c,\eps)}{\min\{r^{k+1/2-2-\gamma'},r^{-k+1/2+2+\gamma'}\}}}^2_{L^\infty}+\norm{\de_GA(\cdot,c,\eps)}_{L_{Y,2+\gamma'}^2}^2+\norm{r\de_r\de_GA(\cdot,c,\eps)}_{L_{Y,2+\gamma'}^2}^2\notag\\
&\qquad\lesssim \|\brak{r}^2 r\de_rF\|^2_{L_{Y,\gamma'}^2}+\|\brak{r}^2F\|^2_{L_{Y,\gamma'}^2}.
\end{align*}
Now, going back to \eqref{eq:corvo} and noticing that $ru'(r)\sim r^2$ as $r\to 0$ and $ru'(r)\sim r^{-2}$ as $r\to \infty$, and collecting all the estimates above, we complete the proof of Proposition \ref{prop:edvanish} in the case $j=1$ by choosing $\gamma'=\delta/4$.

It is worth mentioning that the proof for $k > j>1$ is analogous: the main idea is to use Lemma \ref{lem:RFErecurse}, combined
with the choice $\gamma=2j+\delta<2k$. 
Clearly this imposes the constraint $j\leq k-1$, which is why the case $j=k$ is treated differently below.

\subsubsection{Proof for \texorpdfstring{$j=k$}{TEXT}}\label{sub:jequalk}
We now deal with the case $j=k$. For the sake of simplicity, we consider the case $k=2$; the others are analogous. 
 We begin from \eqref{eq:derfS1} with $j=1$ and take an additional $r\de_r$ derivative, but in this case
we do not exploit the $\de_G$ derivative. Let us only deal with the term containing $X$, namely
\begin{align*}
r\de_r f_{S,X}^\eps(t,r)&=\frac{1 }{2\pi i}\int_{\RR} \e^{ik(u(r)-c)t}\frac{u(r)-c}{(u(r)-c)^2+\eps^2} ru'(r)\de_G \left(\frac{\beta(r)}{\sqrt{r}}X(r,c,\eps)\right)(1-\chi_\sigma (c))\dd c.
\end{align*}
Then 
\begin{align*}
&(r\de_r)^2 f_{S,X}^\eps(t,r)=
\frac{kt }{2\pi }\int_{\RR} \e^{ik(u(r)-c)t}\frac{(u(r)-c)^2}{(u(r)-c)^2+\eps^2} r^2u'(r)\de_G \left(\frac{\beta(r)}{\sqrt{r}}X(r,c,\eps)\right)\chi_\sigma (c)(1-\chi_I(r_c))\dd c\notag\\
&\quad+\frac{1 }{2\pi i}\int_{\RR} \e^{ik(u(r)-c)t}\frac{\eps^2-(u(r)-c)^2}{((u(r)-c)^2+\eps^2)^2} r^2(u'(r))^2\de_G \left(\frac{\beta(r)}{\sqrt{r}}X(r,c,\eps)\right)\chi_\sigma (c)(1-\chi_I(r_c))\dd c\notag\\
&\quad+\frac{1 }{2\pi i}\int_{\RR} \e^{ik(u(r)-c)t}\frac{u(r)-c}{(u(r)-c)^2+\eps^2} r(u'(r)+ru''(r))\de_G \left(\frac{\beta(r)}{\sqrt{r}}X(r,c,\eps)\right)\chi_\sigma (c)(1-\chi_I(r_c))\dd c\notag\\
&\quad+\frac{1 }{2\pi i}\int_{\RR} \e^{ik(u(r)-c)t}\frac{u(r)-c}{(u(r)-c)^2+\eps^2} r^2u'(r)\de_r\de_G \left(\frac{\beta(r)}{\sqrt{r}}X(r,c,\eps)\right)\chi_\sigma (c)(1-\chi_I(r_c))\dd c. 
\end{align*}
We show how to deal with the four terms above, when all the derivatives land on $X$, in the case when $c\in (u(0),u(0)+1)$ and $r\leq 1$,
when using the appropriate weight \eqref{eq:weightf}.
For the first term, we bound it using \eqref{eq:atrae} and
\begin{align*}
\frac{kt}{r^{k+1/2-\delta}}\int_{u(0)}^{u(0)+1} r^2u'(r)\frac{1}{\sqrt{r}}|\de_G X(r,c,\eps)|\dd c&\lesssim \frac{kt}{r^{1/2-\delta/2}}\int_{u(0)}^{u(0)+1} r^3\frac{|\de_G X(r,c,\eps)|}{r^{k+1/2-\delta/2}}\dd c\notag\\
&\lesssim \frac{kt}{r^{1/2-\delta/2}}\int_{u(0)}^{u(0)+1} \frac{|\de_G X(r,c,\eps)|}{r^{k+1/2-2-\delta/2}}\dd c\notag\\
&\lesssim kt \eps^{\delta/4}\left[\| \brak{r}^2F\|^2_{Y,\delta/2}+\|\brak{r}^2 r\de_rF\|^2_{Y,\delta/2}\right],
\end{align*}
The second term is similar, only slightly more delicate. It suffices to bound the following, using  \eqref{eq:atrae} once more:
\begin{align*}
&\frac{1}{r^{k+1/2-\delta}}\int_{u(0)}^{u(0)+1}\frac{1}{(u(r)-c)^2+\eps^2} r^2(u'(r))^2 \frac{1}{\sqrt{r}}|\de_GX(r,c,\eps)|\dd c\notag\\
&\qquad\lesssim \frac{1}{r^{k+1/2-\delta}}\int_{u(0)}^{u(0)+1}\frac{r^4}{|u(r)-u(0)|^{1+2\delta/7}}\frac{1}{|u(0)-c|^{1-2\delta/7}}\frac{1}{\sqrt{r}}|\de_GX(r,c,\eps)|\dd c\notag\\
&\qquad\lesssim\frac{1}{r^{k+1/2-\delta}}\int_{u(0)}^{u(0)+1}\frac{r^{2-4\delta/7}}{|u(0)-c|^{1-2\delta/7}}\frac{1}{\sqrt{r}}|\de_GX(r,c,\eps)|\dd c\notag\\
&\qquad\lesssim\frac{1}{r^{1/2-\delta/7}}\int_{u(0)}^{u(0)+1}\frac{1}{|u(0)-c|^{1-2\delta/7}}\frac{|\de_GX(r,c,\eps)|}{r^{k+1/2-2-2\delta/7}}\dd c\notag\\
&\qquad\lesssim \eps^{\delta/4}\left[\| \brak{r}^2F\|^2_{Y,2\delta/7}+\|\brak{r}^2 r\de_rF\|^2_{Y,2\delta/7}\right].
\end{align*}
The next term is treated similarly, using that
\begin{align*}
&\frac{1}{r^{k+1/2-\delta}}\int_{u(0)}^{u(0)+1} \frac{1}{\sqrt{(u(r)-c)^2+\eps^2}} r(u'(r)+ru''(r)) \frac{1}{\sqrt{r}}|\de_GX(r,c,\eps)|\dd c\notag\\
&\qquad\lesssim \frac{1}{r^{k+1/2-\delta}}\int_{u(0)}^{u(0)+1} \frac{r^2}{|u(r)-u(0)|^{2\delta/7}}\frac{1}{|u(r)-c|^{1-2\delta/7}}\frac{1}{\sqrt{r}}|\de_GX(r,c,\eps)|\dd c\notag\\
&\qquad\lesssim \frac{1}{r^{1/2-\delta/7}}\int_{u(0)}^{u(0)+1}\frac{1}{|u(r)-c|^{1-2\delta/7}}\frac{|\de_G X(r,c,\eps)|}{r^{k+1/2-2-2\delta/7}}\dd c\notag\\
&\qquad\lesssim \eps^{\delta/4}\left[\| \brak{r}^2F\|^2_{Y,2\delta/7}+\|\brak{r}^2 r\de_rF\|^2_{Y,2\delta/7}\right].
\end{align*}
Lastly, we have
\begin{align}
&\frac{1}{r^{k+1/2-\delta}}\int_{u(0)}^{u(0)+1} \frac{1}{\sqrt{(u(r)-c)^2+\eps^2}} r^2u'(r) \frac{1}{\sqrt{r}}|\de_r\de_GX(r,c,\eps)|\dd c\notag\\
&\qquad\lesssim \frac{1}{r^{k+1/2-\delta}}\int_{u(0)}^{u(0)+1} \frac{1}{|u(r)-u(0)|^{2\delta/7}}\frac{1}{|u(0)-c|^{1-2\delta/7}} r^2u'(r) \frac{1}{\sqrt{r}}|\de_r\de_GX(r,c,\eps)|\dd c\notag\\
&\qquad\lesssim \frac{1}{r^{1/2-\delta/7}}\int_{u(0)}^{u(0)+1} \frac{1}{|u(0)-c|^{1-2\delta/7}}\frac{|r\de_r\de_GX(r,c,\eps)|}{r^{k+1/2-2-2\delta/7}}\dd c.\label{eq:corvo4}
\end{align}
Hence, to bound this last term we need a proper estimate $L^\infty$ on $r\de_r\de_GX$. By Sobolev embeddings and the fact that
$(r\de_r)^2=r^2\de_{rr}+r\de_r$, this follows from a proper $L^2$ bound on $r^2\de_{rr}$. Referring to \eqref{eq:XFtermout}-\eqref{eq:XRtermout},
we have that
\begin{align}
r^2 \de_{rr} \de_G X=-\left(\frac14-k^2\right)\de_G X- \frac{\beta(r) r^2}{u(r)-c-i\eps}\de_G X +r^2\left[\eqref{eq:XFtermout}+\eqref{eq:XEtermout}+\eqref{eq:XRtermout}\right].
\end{align}
Since for $c\in (u(0),u(0)+1)$ we have $r^2\lesssim\sqrt{(u(r)-c)^2+\eps^2}$ if $r\leq 1$, we use Lemma \ref{lem:BasicVort} to deduce that
\begin{align*}
\|r^2 \de_{rr} \de_G X\|^2_{L^2_{Y,2+\gamma'}}
&\lesssim\|\de_G X\|^2_{L^2_{Y,2+\gamma'}}\!\!\! +\norm{\frac{\beta(r) r^2}{\sqrt{(u(r)-c)^2+\eps^2}}\de_G X}_{L^2_{Y,2+\gamma'}}^2
\!\!\! +\norm{r^2\left[\eqref{eq:XFtermout}+\eqref{eq:XEtermout}+\eqref{eq:XRtermout}\right]}_{L^2_{Y,2+\gamma'}}^2\notag\\
&\lesssim\|\de_G X\|^2_{L^2_{Y,2+\gamma'}}
+\norm{r^2\left[\eqref{eq:XFtermout}+\eqref{eq:XEtermout}+\eqref{eq:XRtermout}\right]}_{L^2_{Y,2+\gamma'}}^2.
\end{align*}
Now, in view of \eqref{eq:corvo1}, \eqref{eq:corvo3} and \eqref{eq:atrae}, we end up with the higher order estimate
\begin{align}
&\norm{  \frac{ r\de_r\de_GX(\cdot,c,\eps)}{\min\{r^{k+1/2-2-\gamma'},r^{-k+1/2+2+\gamma'}\}}}^2_{L^\infty}+\norm{r\de_r\de_GX(\cdot,c,\eps)}_{L_{Y,2+\gamma'}^2}^2 \!\!\! +\norm{(r\de_r)^2\de_GX(\cdot,c,\eps)}_{L_{Y,2+\gamma'}^2}^2\notag\\
&\qquad\lesssim \frac{\eps^{\delta/4}}{|u(0)-c|^{\delta/4}}\left[\| \brak{r}^2F\|^2_{Y,\gamma'}+\|\brak{r}^2 r\de_rF\|^2_{Y,\gamma'}\right].\label{eq:atrae2}
\end{align}
Going back to \eqref{eq:corvo4} we now set $\gamma'=2\delta/7$ as in the other terms. Arguing in a similar manner for $r\geq 1$, we deduce that
\begin{align*}
\|(r\de_r)^2 f_{S,X}^\eps(t,r)\|_{L^2_{f,\delta}}\lesssim_{\delta,t} \eps^{\delta/4},
\end{align*}
which is what we wanted. The treatment of all the other cases is similar, following the ideas of Propositions \ref{prop:fEvanish} and \ref{prop:fSvanish}. The proof of Proposition \ref{prop:edvanish} is therefore concluded.

%
\def\cprime{$'$}
\def\bibfont{\large}


\end{document}